\documentclass[12pt]{amsbook}
\usepackage{amsmath} 
\usepackage{amssymb}
\usepackage{enumerate}
\usepackage{braket}
\usepackage[matrix,arrow]{xy}
\usepackage{tikz}
\usepackage{pgf}
\usepackage{calc}
\newtheorem{theorem}{Theorem}[chapter]
\newtheorem*{theorem*}{Theorem}

\newtheorem{lemma}{Lemma}[chapter]
\newtheorem{auxlemma}{Auxiliary Lemma}
\newtheorem{proposition}{Proposition}[chapter]
\newtheorem{corollary}{Corollary}[chapter]
\theoremstyle{definition}
\newtheorem{definition}{Definition}%[chapter]
\theoremstyle{remark}
\newtheorem{conclusion}{Conclusion}[chapter]
\newtheorem{remark}{Remark}[chapter]
\newtheorem*{example}{Example}

\newcommand{\Red}{black}
 \newcommand{\Black}{red}

 \newcommand{\mP}{\mathbb{P}}
 \newcommand{\A}{\mathbb{A}}
 \newcommand{\C}{\mathbb{C}}
 \newcommand{\G}{\mathbb{G}}
 \newcommand{\N}{\mathbb{N}}
 \renewcommand{\P}{\mathbb{P}}
 \newcommand{\Q}{\mathbb{Q}}
 \newcommand{\R}{\mathbb{R}}
 \newcommand{\Z}{\mathbb{Z}}
 \newcommand{\HH}{\mathbf{H}}
 \newcommand{\GG}{\mathbf{G}}
 \newcommand{\CC}{\mathbf{C}}
 \newcommand{\cA}{\mathcal{A}}
 \newcommand{\cB}{\mathcal{B}}
 \newcommand{\cC}{\mathcal{C}}
 \newcommand{\cF}{\mathcal{F}}
 \newcommand{\cH}{\mathcal{H}}
 \newcommand{\cI}{\mathcal{I}}
 \newcommand{\cJ}{\mathcal{J}}
 \newcommand{\cK}{\mathcal{K}}
 \newcommand{\cL}{\mathcal{L}}
 \newcommand{\cM}{\mathcal{M}}
 \newcommand{\cN}{\mathcal{N}}
 \newcommand{\cO}{\mathcal{O}}
 \newcommand{\cP}{\mathcal{P}}
 \newcommand{\cR}{\mathcal{R}}
 \newcommand{\cX}{\mathcal{X}}
 \renewcommand{\Im}{\mathrm{Im}}
 \newcommand{\Grass}{\operatorname{Grass}}
 \newcommand{\Hilb}{\operatorname{Hilb}}
 \newcommand{\vep}{\varepsilon}

 \newcommand{\reg}{\mathrm{reg}}
 \newcommand{\Pic}{\operatorname{Pic}}
 
 \newcommand{\Rat}{\operatorname{Rat}}
 \newcommand{\Spec}{\operatorname{Spec}}
 \newcommand{\colength}{\mathrm{colength}}
 \newcommand{\down}{\mathrm{down}}
 \newcommand{\up}{\mathrm{up}}
 \newcommand{\grade}{\mathrm{grade}}
 
 \newcommand{\mwedge}{{\mathop{\bigwedge}\limits^m}}
 \newcommand{\pal}{\psi_{\alpha}}
 \newcommand{\inn}{\mathit{in}}
 \newcommand{\itin}{\mathit{in}}
 \newcommand{\itbis}{{\mbox{\it bis\/}}}  %
 \newcommand{\ol}{\overline}
 \newcommand{\bigsum}{\sum}
 \newcommand{\bigudot}{{\mathop{\bigcup}\limits^{\cdot}}}
 \newcommand{\modulo}{\text{ modulo }}

\begin{document}
 \numberwithin{equation}{chapter}
 \baselineskip=16pt

 \title[Hilbert scheme of space curves]{Computation of the First Chow group
   of\\  a Hilbert scheme of space curves}

%  \begin{abstract}
%    An earlier wrong formula for the dimension of $ A_1( H_{d,g})\otimes\Q$ is
%    corrected.
%  \end{abstract}

 \author[G. Gotzmann]{Gerd Gotzmann}

%\date{2. March 2011}

 \begin{titlepage}
   \begin{center}
     \vspace{4.5cm}
     {
      \fontsize{20}{20}
      \textbf{
        Computation of the first Chow group of\\[3mm] 
        a Hilbert scheme of space curves}
      }\\
    \vspace{2cm}
     {\Large
       Gerd Gotzmann\\[5mm]
     }
  \vspace{1.5cm}
    {\large \textsc{Abstract}}\\[3mm]
    {\it An earlier wrong formula for the dimension of $ A_1(
      H_{d,g})\otimes\Q$ is corrected.}
   \end{center}

\vspace{1.5cm}
% \end{titlepage}

% \pagestyle{empty}
% \maketitle
%\tableofcontents
% Preface
%\frontmatter
% \mainmatter
% \pagestyle{plain}
%\thispagestyle{empty}
%\chapter*{Introduction} \label{intro}
\centerline{\textbf{\Large Introduction}}
\vspace{5mm}
 The results stated in ([T4], pp. 1) have to be corrected as follows: Let
   $\HH=H_{d,g}=\Hilb^P(\mP_{\C}^3)$ be the Hilbert scheme , which parametrizes
   the curves in $\mP_{\C}^3$ of degree $d$ and genus $g$ (i.e., the closed
   subschemes of $\mP_{\C}^3$ with Hilbert polynomial $P(T)=d T-g+1$). It is
   always assumed that $d\ge 3$ and $g$ is not maximal, i.e. that
   $g<(d-1)(d-2)/2$.

 \begin{theorem}\label{I} %%%%%%%%%%%%%%%%%%% Theorem 0.1.
   Let be $g(d):=(d-2)^2/4$. Then $\dim_{\Q}A_1(H_{d,g})\otimes\Q=3$ (resp.
   $=4$), if $g\le g(d)$ (resp. if $g>g(d)$).
 \end{theorem}

 \begin{corollary}\label{1}  %%%%%%%%%%%%%%%%% Corollary 0.2.
   $NS(\HH)\simeq\Z^{\rho}$ and $\Pic(\HH)\simeq\Z^{\rho}\oplus\C^r$, where
   $r:=\dim_{\C}H^1(\HH,\cO_{\HH})$ and $\rho=3$, if $g\le g(d)$. If $g>g(d)$,
   then $\rho=3$ or $\rho=4$.
 \end{corollary}

 \begin{theorem}\label{II}  %%%%%%%%%%%%%% Theorem II
   Let $\CC\hookrightarrow\HH\times\mP^3$ be the universal curve over $\HH$.
   Then
   $\dim_{\Q}A_1(\CC)\mathop{\otimes}\limits_{\Z}\Q=\dim_{\Q}A_1(\HH)\mathop{\otimes}\limits_{\Z}\Q+1$.
 \end{theorem}

 \begin{corollary}\label{2}  %%%%%%%%%%%%%%%% Corollary 2
   $NS(\CC)=\Z^{\rho+1}$ and $\Pic(\CC)\simeq\Z^{\rho+1}\oplus\C^s$, where
   $s:=\dim_{\C} H^1(\CC,\cO_{\CC})$ and $\rho$ is defined as in Corollary 1.
 \end{corollary}

 That means, the formula $ (d-2)(d-3)/2 $ for the bound $g(d)$ in ([T4], p.1) is wrong and has
 to be replaced by the above formula.
 
\vfill
\centerline{{March 2, 2011} }
\newpage{}
\end{titlepage}

%\newpage
%\clearpage
%\newpage
%{}
%\newpage

\frontmatter
\pagestyle{empty}
\tableofcontents

%%%%%%%%%%%%%%%%%%%%%%%%%%%%%%%% CHAPTER 1 %%%%%%%%%%%%%%%%%%%%%%%%%%%%%%%%%%%

\mainmatter
\pagestyle{plain}
 \chapter{Summary of earlier results}\label{1}

 \section{Description of the starting situation}\label{1.1}

 The result are the same as in [T1]-[T4] and are summed up in Appendix A.
 The ground field is $\C$, and $\HH=H_{d,g}$ is the Hilbert scheme which
 parametrizes the curves $C\subset\mP^3_{\C}$ of degree $d$ and genus $g$ (i.e.
 the closed subschemes of $\mP^3_{\C}$ with Hilbert polynomial $P(T)=dT-g+1$).
 According to F.S.Macaulay, $H_{d,g}$ is not empty if and only if the
 ``complementary'' Hilbert polynomial $Q(T)={T+r\choose r}-P(T)$ either has the
 form $Q(T)={T-1+3\choose 3}+{T-a+2\choose 2}$ or the form $Q(T)={T-1+3\choose
   3}+{T-a+2\choose 2}+{T-b+1\choose 1}$, where $a$ is an integer $\ge 1$,
 respectively
 $a$ and $b$ are integers (Macaulay coefficients), such that $2\le a\le b$.
 Between the degree and genus on the one hand and the Macaulay
 coefficients on the other hand, one has the following relations
 $d=a,g=(d-1)(d-2)/2$, if $Q(T)={T-1+3\choose 3}+{T-a+2\choose 2}$, and
 $d=a-1,g=(a^2-3a+4)/2-b$, if $Q(T)={T-1+3\choose 3}+{T-a+2\choose
   2}+{T-b+1\choose 1}$, respectively. One sees that the first case occurs if
 and only if one is dealing with plane curves, in which case the groups
 $A_1(\HH)$ and $NS(\HH)$ both have the rank 2 (cf. [T1], Satz 2a, p. 91).
 Therefore in the following we always suppose that $d\ge 3$ and
 $g<(d-1)(d-2)/2$, that means, the complementary Hilbert polynomial has the
 form $Q(T)={T-1+3\choose 3}+{T-a+2\choose 2}+{T-b+1\choose 1}$, where $4\le
 a\le b$.

 We also write $\HH_Q$ instead of $H_{d,g}$ in order to express that this
 Hilbert scheme likewise parametrizes the ideals $\cI\subset\cO_{\mP^3}$ with
 Hilbert polynomial $Q(T)$, or equivalently, the saturated graded ideals in
 $\C[x,y,z,t]$ with Hilbert polynomial $Q(T)$.

 In [T1]-[T4] it was tried to describe the first Chow group $A_1(\HH)$, where
 we always take rational coefficients, and we write $A_1(\HH)$ instead of
 $A_1(\HH)\mathop{\otimes}\limits_{\Z}\Q$. The starting point is the following
 consideration: If the Borel group $B=B(4;k)$ operates on $\HH=\HH_Q$ in the
 obvious way, then one can deform each 1-cycle on $\HH$ in a 1-cycle, whose
 prime components are $B$-invariant, irreducible, reduced and closed curves on
 $\HH$. It follows that $A_1(\HH)$ is generated by such $B$-invariant 1-prime
 cycles on $\HH$. This is a partial statement of a theorem of Hirschowitz.
  (Later on we will have to use the general statement, whereas the partial
 statement can be proved in a simple way, see [T1], Lemma 1, p.~6.)  Now
 such a $B$-invariant 1-prime cycle (i.e. closed, irreducible and reduced
 curve) $C$ on $\HH$ can be formally described as follows: Either each point of
 $C$ is invariant under $\Delta:=U(4;k)$, or one has
 $C=\overline{\G_a^i\cdot\eta}$, where $\eta$ is a closed point of $\HH$, which
 is invariant under $T=T(4;k)$ and the group $G_i, i\in\{ 1,2,3\}$. Here
 $\G_a^i$ is the group $\G_a$, acting by
 \[
 \begin{array}{l}
   \psi_{\alpha}^1:x\longmapsto x,\; y\longmapsto y,\; z\longmapsto z,\;
   t\longmapsto\alpha z+t \\
   \psi_{\alpha}^2:x\longmapsto x,\; y\longmapsto y,\; z\longmapsto\alpha
   y+z,\; t\longmapsto t \\
   \psi_{\alpha}^3:x\longmapsto x,\; y\longmapsto\alpha x+ y,\; z\longmapsto z,\; t\longmapsto t,
   \end{array} 
 \]
 respectively, on $P=k[x,y,z,t]$, and $G_i$ is the subgroup of $\Delta$, which is complementary to $\G_a^i$, that means, one defines
 \[
 G_1:=\left\{ \left( \begin{array}{cccc}
       1&*&*&*\\0&1&*&*\\0&0&1&0\\0&0&0&1\end{array}\right)\right\},\quad
 G_2:=\left\{ \left( \begin{array}{cccc}
       1&*&*&*\\0&1&0&*\\0&0&1&*\\0&0&0&1\end{array}\right)\right\},\quad
 G_3:=\left\{ \left( \begin{array}{cccc}
       1&0&*&*\\0&1&*&*\\0&0&1&*\\0&0&0&1\end{array}\right)\right\}.
 \]

 If $C$ has this form, then $C$ is called a curve or a 1-cycle of type $i$,
 where $i\in \{ 1,2,3\}$. $\cA(\HH):=\Im(A_1(\HH^{\Delta})\to
 A_1(\HH))$ is called the ``algebraic part'' and
 $\overline{A}_1(\HH):=A_1(\HH)/\cA(\HH)$ is called the ``combinatorial part''
 of the first Chow group of $\HH$. Here $\HH^{\Delta}$ denotes the fixed point
 scheme which, just as all other fixed point schemes that will occur later on,
 is supposed to have the induced reduced scheme structure. (This convention is
 valid also for the Hilbert scheme $H^d:=\Hilb^d(\mP^2_{\C})$, see below.)

 In order to formulate the results obtained so far, one has to introduce the
 following "tautological" 1-cycles on $\HH$:
 \begin{align*}
   C_1&=\Set{(x,y^a,y^{a-1}z^{b-a}(\alpha z+t))|\alpha\in k }^-\\
   C_2&=\{(x,y^{a-1}(\alpha y+z),y^{a-2}z^{b-a+1}(\alpha y+z))|\alpha\in k\}^-\\
   C_3&=\{(x^a,\alpha x+y,x^{a-1}z^{b-a+1})|\alpha\in k\}^-\\
   D&=\{(x^2,xy,y^{a-1},z^{b-2a+4}(y^{a-2}+\alpha xz^{a-3}))|\alpha\in k\}^-\\
   E&=\{(x^2,xy,xz,y^a,y^{a-1}z^{b-a+1},xt^{b-2}+\alpha
   y^{a-1}z^{b-a})|\alpha\in k \}^-
 \end{align*}
 For the sake of simplicity, we now suppose $d\ge 5$ (i.e. $a\ge 6$). (The
 cases $d=3$ and $d=4$ will be treated separately in Chapter 16.) Then one has
 the following results: \\[2mm]
 1. If $b<2a-4$, i.e. if $g>\gamma(d):=(d-2)(d-3)/2$, then $\cA(\HH)$ is
 generated by $E$, and $A_1(\HH)$ is generated by $E,C_1,C_2,C_3$. \\[2mm]
 2. If $b\ge 2a-4$, i.e if $g\le\gamma (d)$, then $\cA(\HH)$ is generated by
 $E$ and $D$ and $A_1(\HH)$ is generated by $E,D,C_2$ and $C_3$ ( see [T1], Satz
 2, p. 91; [T3], Proposition 4, p.\,22; [T4], Satz 1 and Proposition 2, p.\,26).

 From reasons of degree it follows that $[C_2]$ can not lie in the vector space
 spanned by $[E],[D],[C_3]$, so the problem is to decide, if
 $[C_3]\in\cA(\HH)$.

 In ([T4], Proposition 3, p.\,32) it was \emph{erroneously} claimed that
 $[C_3]\in\cA(\HH)$, if $b\ge 2a-4$. (The error is the wrong computation of the
 degree in ([T4], p.\,28, line 21 to line 30.) Therefore the bound for the genus
 in ([T4], p.\,1) is wrong.

 Actually, in ([T2], 3.3.2) it had been proved, that $[C_3]\in\cA(\HH)$, if
 $a\ge 6$ is even and $b\ge a^2/4$, i.e. if $d\ge 5$ is odd and $g\le
 (d-1)(d-3)/4$. In the case $d\ge 6$ even , in Conclusion 14.3 it will
 follow that $[C_3]\in\cA(\HH)$, if $g\le (d-2)^2/4$. (This means the bound of 
  [T2], 3.3.3 is valid if $ d \geq 6 $, already . ). One sees that the
  condition for $g$ in both  cases can be summed up to $g\le (d-2)^2/4$.

 The major part of the following text serves for the proof that this 
 sufficient condition is a necessary condition, too (cf. Conclusion 14.1).

 \section{Technical tools}\label{1.2}

 The formulas in ([T2], p. 134) and of ([T3], Anhang 2, p. 50) show that it is
 not possible to decide by means of the computation of degrees, whether $
 [C_3]$ lies in $\cA(\HH)$. Therefore we try to get a grasp of the relations
 among the $B$-invariant 1-cycles on $\HH$ with the help of the theorem of
 Hirschowitz ([Hi], Thm. 1, p. 87). We sketch the procedure.

 \subsection{The Theorem of Hirschowitz}\label{1.2.1}

 There is a closed and reduced subscheme $Z=Z(\HH)$ of $\HH$, such that
 $Z(k)=\{x\in\HH(k)|\dim\Delta\cdot x\le 1\}$ (cf. [Ho], p. 412 and [T3], Lemma
 1, p. 35). Then one can show, with the help of the theorem of Hirschowitz,
 that $A_1(Z)\mathop{\to}\limits^{\sim} A_1(\HH)$ (cf. [T2], Lemma 24, p. 121).
 As was explained in (1.1), $A_1(\HH)$ has a generating system consisting of
 $B$-invariant 1-cycles which lie in $Z$, automatically. As $\Delta$ is
 normalized by $B$, $B$ operates on $Z$ and therefore one can form the so
 called equivariant Chow group $A_1^B(Z)$, which is isomorphic to $A_1(Z)$
 ([Hi], loc. cit.). And the relations among $B$-invariant 1-cycles on $Z$ are
 generated by relations among such cycles, which lie on $B$-invariant surfaces
 $V\subset Z$ ( see [Hi], Mode d'emploi, p. 89).

 \subsection{The Restriction morphism}\label{1.2.2}

 Let $U_t\subset\HH$ be the open subset consisting of the ideals
 $\cI\subset\cO_{\mP^3}$ with Hilbert polynomial $Q$, such that $t$ is not a
 zero divisor of $\cO_{\mP^3}/\cI$. Then there is a so called
 restriction-morphism $h:U_t\to H^d:=\Hilb^d(\mP^2_{\C})$, defined by
 $\cI\mapsto\cI':=\cI+t\cO_{\mP^3}(-1)/t\cO_{\mP^3}(-1)$.  E.g., if 
 %  $G:=
 % \left\{
 %   \left(
 %     \begin{array}{cccc}
 %      1&0&0&*\\0&1&0&*\\0&0&1&*\\0&0&0&1
 %    \end{array}\right)\right\}
 % <\Delta$,
  $G:=
 \left\{\left(
  \begin{smallmatrix}
   1&0&0&* \\ 0&1&0&* \\ 0&0&1&* \\0&0&0&1   
  \end{smallmatrix}\right)\right\}
 <\Delta$,
 then the fixed point scheme $F:=\HH^G$ is contained in $U_t$, and the
 restriction of $h$ to $F$ is denoted by $h$, again. In ([G4], Abschnitt 6, p.
 672f) the following description of $\Im(h)$ is given:

 (i) There is a finite set $\cF$ of Hilbert functions of ideals of colength $d$
 on $\mP^2$ such that \\ \quad $\Im(h)= \bigcup\{H_{\ge\varphi}|\varphi\in\cF\}$.

 (ii) If $k=\overline{k}$ and if $\cI\subset\cO_{\mP^2_k}$ is an ideal of
 colength $d$ and Hilbert function $\varphi$, then\\ \quad $\cI\in \Im(h) \iff
 g^\ast(\varphi):=\sum\limits^d_0 \varphi (n)-{d+3\choose 3}+d^2+1\ge g$.

 (iii) If $\varphi\in\cF$ and if $\psi$ is the Hilbert function of an ideal on
 $\mP^2$ of colength $d$ such that $\varphi(n)\le\psi (n)$ for all $n\in\N$,
 then $\psi\in\cF$.

 (iv) Let be $\varphi\in\cF$ and $\cI\subset\cO_{\mP^2_k}$ an ideal with
 Hilbert function $\varphi$. Let $\cI^*$ be the ideal in $\cO_{\mP^3_k}$
 defined by $H^0(\cI^*(n))=\mathop{\oplus}\limits^n_{i=0}t^{n-i}H^0(\cI(i))$,
 then $V_+(\cI^*)\subset\mP^3_k$ is a curve of degree $d$ and genus
 $g^*(\varphi)$.

 Here $H_{\varphi}\subset H^d$ is the locally closed subscheme (with the
 reduced induced scheme structure), which parametrizes the ideals
 $\cI\subset\cO_{\mP^2}$ of colength $d$ with Hilbert function $\varphi$, and
 $H_{\ge\varphi}:=\bigcup \{ H_{\psi}|\psi\ge\varphi\}$ is a closed subscheme
 ( with the induced reduced scheme structure). The image of $C_3$ under $h$ is
 the 1-cycle $c_3:=\{(x^d,\alpha x+y)|\alpha\in k\}^-$. One has

 \begin{theorem} %%%%%%%%%%%%% im Original keine Numerierung angegeben, im
                 %%%%%%%%%%%%% Ausdruck erscheint Theorem 1.1.
   Let be $d \geq 5 , \cH : = \bigcup \{H_{\varphi}\subset H^d | g^*(\varphi)>g(d)\}$
   and \\$\cA(\cH): =\mathrm{Im} (A_1(\cH^{U(3;k)})\to A_1(\cH))$.  Then
   $ [c_3]\notin\cA(\cH)$.
 \end{theorem}

 The proof extends over the chapters 2 to 10 and essentially rests on the apparently strong condition for an ideal $\cI$ to have a Hilbert function  $\varphi$ such that $g^*(\varphi)>g(d)$.

 \subsection{Standard cycles on $H^d$}\label{1.2.3}

 It has been shown, respectively it will be shown that $[C_3]\in\cA(H_{d,g})$,
 if $g\le g(d)$ (cf. 1.1). Therefore, we can suppose that $g>g(d)$. If $\cJ\in
 U_t$ and the restriction ideal $\cI:=\cJ'$ has the Hilbert function $\varphi$,
 then from (ii) in (1.2.2) it follows that $g^\ast(\varphi)>g(d)$. It will be shown in
 Chapter 2 that this implies there is a linear form $\ell\in S_1-(0)$, an ideal
 $\cK\subset\cO_{\mP^2}$ of colength $c$ and a form $f\in H^0( \cK(m))$ such
 that $\cI=\ell\cK(-1)+f\cO_{\mP^2}(-m), c+m=d$ and $m\ge c+2$.

 Let be $C=\overline{\G_a\cdot\eta}\subset\HH$ a 1-cycle of type 3 and let be
 $\cJ\leftrightarrow\eta$ the corresponding ideal in $\cO_{\mP^3}$ with Hilbert
 polynomial $Q$. Then the ideal $\cI:=\cJ'\leftrightarrow\eta':=h(\eta)$ is invariant
 under $T(3;k)$ and $\Gamma:= \left\{\left(
 \begin{array}{ccc}
    1 & 0 & * \\
    0 & 1 & * \\
    0 & 0 & 1
 \end{array}
 \right)\right\} < U(3,k)$.  It follows that either
 $\cI=x\cK(-1)+y^m\cO_{\mP^2}(-m)$ or $\cI=y\cK(-1)+x^m\cO_{\mP^2}(-m)$, where
 $\cK$ is a monomial ideal. We say, $\cI$ has $x$-standard form or $\cI$ has
 $y$-standard form, respectively, and we call $C':=\overline{\G_a\cdot\eta'}$ a
 $x$-standard cycle or $y$-standard cycle on $H^d$, respectively. With the help
 of the theorem of Hirschowitz one can again try to describe the relations
 between $B(3;k)$-invariant $y$-standard cycles on $\cH$, and one obtains that
 such relations cannot make the $y$-standard cycle $c_3$ disappear modulo
 $\cA(\cH)$ (cf. Proposition 9.1) from which Theorem 0.1 will follow.

 \subsection{1-cycles of proper type 3}\label{1.2.4}

 Let be $C=\overline{\G_a\cdot\eta}, \eta\leftrightarrow\cJ$, be a 1-cycle of type 3 on
 $\HH=H_{d,g}$, such that $d\ge 5$ and $g>g(d)$. $C$ is called a 1-cycle of
 proper type 3, if $C'=\overline{\G_a\cdot\eta'}$ is a $y$-standard cycle on
 $\cH$. Corresponding to Hirschowitz's theorem one has to consider
 $B(4;k)$-invariant surfaces $V\subset Z(\HH)$, which contain a 1-cycle of
 proper type 3. It turns out that then $V$ is pointwise invariant under $G_3$
 and therefore $V$ is contained in $U_t$. Then one can map relations between
 $B$-invariant 1-cycles on $V$ by $h_*$ into relations between
 $B(3;k)$-invariant 1-cycles on $h(V)$, and one obtains with the aid of 
 Proposition 9.1 the main result of the second part of the paper ( Theorem
 14.1), which corresponds to Theorem 0.1. In Chapter 15 there is complete
 description of $A_1(H_{d,g})$ if $d\ge 5$, and in Chapter 16 this is done in
 the cases $d=3$ and $d=4$ (Theorem 15.1 and Theorem 16.1, respectively).

 %%%%%%%%%%%%%%%%%%%%%%%%%%%%CHAPTER 2%%%%%%%%%%%%%%%%%%%%%%%%%%%%%%%%%%%%%%%%%

 \chapter{Subschemes of points in $\mP^2$ and their Hilbert functions}\label{2}

 \section{General properties}\label{2.1}

 The ground field is $\C$ and $k$ denotes an extension field. A closed
 subscheme $Z\subset\mP^2_k$ of length $d>0$ is defined by an ideal
 $\cI\subset\cO_{\mP^2_k}$ with Hilbert polynomial $Q(n)={n+2\choose 2}-d$. If
 the Hilbert function $h^0(\cI(n))=\dim_kH^0(\mP^2_k; \cI(n)), n\in\N$, of
 $\cI$ is denoted by $\varphi (n)$, then $\varphi'(n):=\varphi(n)-\varphi
 (n-1),n\in\N$, denotes the difference function. If
 $\varphi:\N\longrightarrow\N$ is any function, such that
 $\varphi(n)={n+2\choose 2}-d$ for $n\gg 0$, then the ideals
 $\cI\subset\cO_{\mP^2_k}$ with Hilbert function $\varphi$ form a locally
 closed subset $H_{\varphi}$ of the Hilbert scheme $H^d=\Hilb^d(\mP^2_{\C})$,
 and we take $H_\varphi$ as a subscheme of $H^d$ with the induced reduced
 scheme structure.

 Iarrobino has shown ([I], Lemma 1.3, p.8) that $H_{\varphi}\neq\emptyset$ if and only if the difference function fulfils the following two conditions:
 \begin{enumerate}[\quad (a)]
 \item 
  $\varphi'(n)\le n+1$, for all $n\in\N$ and
 \item 
  $\varphi'(n)\le\max (\varphi'(n+1)-1,0)$, for all $n\in\N$.
 \end{enumerate}

 \noindent
 If $\alpha=\alpha(\varphi):=\min \{ n\in\N\;|\; \varphi(n)>0\}$, then (b) is
 equivalent to:

 \begin{enumerate}[\quad (b')]
 \item $\varphi'(n)+1\le \varphi'(n+1)$, for all $n\ge\alpha$.
 \end{enumerate}

 The (Mumford-)regularity $e$ of an ideal $\cI$ with Hilbert function $\varphi$
 as before is characterized by $e=\reg(\varphi)=\min
 \{n\in\N\;|\;\varphi'(n+1)=n+1\}$ (cf. Appendix B, Lemma 2). In principle, the
 graph of $\varphi'$ has the shape of Fig. 2.1. If $\emptyset\neq
 H_{\varphi}\subset H^d$, then $d$ is determined by the condition $\varphi(n)=
 {n+2\choose 2}-d, n\gg 0$, and we call $d$ the colength of $\varphi$. It is
 known  that $\reg (\varphi)\le d$ ([G1], Lemma 2.9, p. 65), and
 $\reg(\varphi)=d$ is equivalent with $\cI$ being generated by a linear form
 $\ell\in S_1$ and a form $f\in S_d$, not divisible by $\ell$. Here
 $S=k[x,y,z]$ is the graded polynomial ring. Another characterization of $\reg
 (\varphi)=d$ is that the graph of $\varphi'$ has the shape of Fig.2.2. One
 notes that the colength of $\varphi$ is equal to the number of ''monomials"
 between the graph of $\varphi'$ and the line $y=x+1$. (For this and other
 properties , see [T1]-[T4].) In the following we write $\mP^2$ instead of
 $\mP^2_k$ and denote by $\cI$ an ideal in $\cO_{\mP^2}$, whose finite colength
 (resp. whose Hilbert function) usually is denoted by $d$ (resp. by $\varphi$).

 \section{Numerical and algebraic properties}\label{2.2}

 \begin{lemma} %%%%%%%%%%%% Lemma 2.1.
   Let be $k=\ol{k}, \cI\subset\cO_{\mP^2}$ an ideal with colength $d$, Hilbert
   function $\varphi$ and regularity $m$. We assume that there is a number
   $\vep\in\N, 0\le\vep <m-2$, such that $\varphi'(n)=n$ for all $n\in\N$, such
   that $\vep+1\le n\le m-1$. Then there is a linear form $\ell=S_1$, an ideal
   $\cK\subset\cO_{\mP^2}$ of colength $c$ and a form $f\in H^0(\cK(m))$, such
   that $\cI=\ell \cK(-1)+f\cO_{\mP^2}(-m)$. If $\ell_1,\ell_2$ are any linear
   forms in $S_1$ such that $\ell,\ell_1,\ell_2$ is a basis of the $k$-vector
   space $S_1$ and if $R:=k[\ell_1,\ell_2]$ is the subring of $S$, isomorphic
   to $k[x,y]$, then $d=c+m$ and 
 \[
 H^0(\cI(n))=
 \begin{cases} 
   \ell H^0(\cK(n-1)) &  \text{if $n<m$,}\\ 
   \ell H^0(\cK(n-1))\oplus fR_{n-m} & \text{if $n\ge m$}.
 \end{cases}
  \]
 \end{lemma}

 \begin{proof}
   By assumption, the graph of $\varphi'$ has the shape as in Fig.2.3. Then
   there is a $\ell\in S_1-(0)$ and an ideal $\cK\subset\cO_{\mP^2}$ of
   regularity $\le\vep$ such that $H^0(\cI(n))=\ell H^0(\cK(n-1))$ for all
   $n\le m-1$ (cf. [G4] and Appendix B, Lemma 1). If $\psi$ is the Hilbert
   function of $\cK$, then $\varphi'(n)=\psi'(n-1)$ for $1\le n\le m-1$, and
   because of the shape of the graphs of $\varphi'$ and $\psi'$ it follows that
   $\varphi(n)=\psi(n-1)+(n-m+1)$ for all $n\ge m$. Therefore $H^0(\cI(m))=\ell
   H^0(\cK(m-1))\oplus f\cdot k$, where $f\in H^0(\cI(m))$ is a suitable
   section. Because of the $m$-regularity of $\cI$ it follows that
   $H^0(\cI(n))=\ell H^0(\cK(n-1))+f S_{n-m}, n\ge m$. If $n=m+1$, then from
   $\varphi (m+1)=\psi(m)+2$ it follows that $S_1f\cap\ell H^0(\cK(m))$ has the
   dimension 1. Thus there is a $h\in S_1-(0)$, such that $hf\in\ell
   H^0(\cK(m))$. If $\ell$ would be a divisor of $f$, then it would follow that
   $\cI\subset\ell\cO_{\mP^2}(-m)$ and thus $\cI$ would not have a finite
   colength in $\cO_{\mP^2}$. Therefore we may suppose that $h=\ell$, and it
   follows that $f\in H^0(\cK(m))$. We choose $\ell_1,\ell_2\in S_1$ such that
   $\ell,\ell_1,\ell_2$ are linear independent and we put
   $R:=k[\ell_1,\ell_2]$. If there would be a $r\in R_{\nu}-(0)$ such that
   $rf\in \ell H^0(\cK(m+\nu-1))$, then it would follow that $\ell$ is a
   divisor of $f$, contradiction. Between the graph of $\psi'(n-1)$ and the
   line $y=x$ there are exactly $c:=\colength(\psi)$ monomials, and
   therefore $d=c+m$ (cf. Fig.2.3).
 \end{proof}

 \begin{corollary}\label{1}    %
 The assumptions and notations are as in Lemma 2.1. Then one has:
 \begin{enumerate}[(i)]
 \item $\kappa:=\reg (\cK)\le\vep$, especially $\kappa\le m-3.$
 \item $\cK$ (respectively the linear form $\ell$) is uniquely determined
   (respectively uniquely up to a factor out of $k$ different from zero).
 \item $f$ is uniquely determined up to a factor out of $k$ different from
   zero, modulo $\ell H^0(\cK(m-1))$.
 \item $\kappa$ and $\vep$ are uniquely determined by $\varphi$.
 \end{enumerate}
 \end{corollary}

 \begin{proof}
   (i) $\vep+1=\varphi'(\vep+1)=\psi'(\vep)$. From (Appendix B, Lemma 2) it
   follows that $\kappa\le\vep$. \\
   (ii) The regularity only depends on the Hilbert function, and therefore
   $\kappa=\reg(\cK_1)=\reg(\cK_2)<m-1$. Thus from $\ell_1
   H^0(\cK_1(m-2))=\ell_2H^0(\cK_2(m-2))$ it follows that
   $\ell_1\cK_1(-1)=\ell_2\cK_2(-1)$. \\
   If $\ell_1$ would not be a divisor of $\ell_2$ then one would have
   $\cK_2\subset\ell_1\cO_{\mP^2}(-1)$ contradiction. From this assertion (ii)
  does follow, and (iii) and (iv) are clear.
 \end{proof}

 \begin{remark}   %%%%%%%%%%%%%%%%%%% Remark 2.3.
   If $\varphi$ and $\psi$ are two Hilbert functions of colength $d$, then from
   $\varphi<\psi$ (that means $\varphi(n)\le\psi(n)$ for all $n\in\N$ and
   $\varphi(n)<\psi(n)$ for at least one $n\in\N$) it follows that
   $g^*(\varphi)<g^*(\psi)$. This follows immediately from the definition of
   $g^*(\varphi)$ in (1.2.2).
 \end{remark}

 \begin{remark}  %%%%%%%%%%%%%%%%%%%%% Remark 2.4.
   If $e:=\reg(\varphi), d:=\colength(\varphi)$ and
   $s(\varphi):=\sum\limits^{e-2}_{i=0}\varphi(i)$, then
   $g^*(\varphi)=s(\varphi)-{e+1\choose 3}+d(e-2)+1$. Because of
   $\varphi(n)={n+2\choose 2}-d$ for all $n\ge e-1$ this follows from a simple
   computation with binomial coefficients.
 \end{remark}

 \subsection{Hilbert functions of colength $\le 4$}\label{2.2.2}

 We use the formula of Remark 2.2 and orientate ourselves by the figures 2.4--2.7.\\[2mm]
 \fbox{$d=1$} There is only one Hilbert function (cf. Fig. 2.4). 
 \[ 
  e=1,\ s(\varphi)=0,\ g^*(\varphi)=0-{2\choose 3}+1\cdot (1-2)+1=0.
  \]
 $\fbox{$d=2$}$ There is again only one Hilbert function (cf. Fig. 2.5). 
 \[ 
  e=2,\ s(\varphi)=0,\ g^*(\varphi)=0-{3\choose 3}+2\cdot 0+1=0.
 \]
 $\fbox{$d=3$}$ There are two Hilbert functions (Fig. 2.6 a and Fig. 2.6
 b).
 \begin{gather*}
   e_1=2,\ s(\varphi_1)=0,\ g^*(\varphi_1)=0-\binom{3}{3}+3\cdot 0+1=0,\\
   e_2=3,\ s(\varphi_2)=1,\ g^*(\varphi_2)=1-{4\choose 3}+3\cdot 1+1=1.
 \end{gather*}
 $\fbox{$d=4$}$ There are two Hilbert functions (Fig. 2.7 a and Fig. 2.7 b).
  \begin{gather*}
 e_1=3,\ s(\varphi_1)=0,\ g^*(\varphi_1)=0-{4\choose 3}+4\cdot 1+1=1,\\
 e_2=4,\ s(\varphi_2)=4,\ g^*(\varphi_2)=4-{5\choose 3}+4\cdot 2+1=3.
 \end{gather*}

 \subsection{Two special ideals}\label{2.2.3}

 First case: If $d\ge 6$ is even, then let be $e:=d/2+1$ and
 $\cI:=(x^2,xy^{e-2},y^e)$. The Hilbert function $\chi$ can be read from
 Fig. 2.8. One notes that $\colength(\cI)$ and $\reg(\cI)$ really are equal to
 $d$ and $e$, respectively, and $\chi(n)=\sum\limits^{n-1}_1i={n\choose 2}$, if
 $1\le n\le e-2$. Therefore $s(\chi)=\sum\limits^{e-2}_1\tbinom{i}{2}=\tbinom{e-1}{3}$
 % % \[
 % % s(\chi)=\sum\limits^{e-2}_1{i\choose 2}={e-1\choose 3}
 % % \]
 % \[
 % s(\chi)=\sum\limits^{e-2}_1\tbinom{i}{2}=\tbinom{e-1}{3}
 % \]
  and it follows that 
 \begin{align*}
   g^*(\chi) & =\tbinom{e-1}{3}-\tbinom{e+1}{3}+2(e-1)(e-2)+1
     =\tbinom{e-1}{3}-\tbinom{e}{3}+\tbinom{e}{3}-\tbinom{e+1}{3}+2e^2-6e+5 \\
    & =-\frac{1}{2}(e-1)(e-2)-\frac{1}{2}e(e-1)+2e^2-6e+5  =e^2-4e+4 \\
    & =(e-2)^2=\frac{1}{4}(d-2)^2.
 \end{align*}
 Second case: If $d\ge 5$ is odd, then let be $e:=(d+1)/2$ and
 $\cI:=(x^2,xy^{e-1},y^e)$. \\[2mm]
 The Hilbert function $\chi$ can be read from Fig. 2.9. One notes that colength
 $(\cI)$ and $\reg(\cI)$ are equal to $d$ and $e$, respectively, and
 $\chi(n)={n\choose 2}$, if $1\le n<e$.\\
 Therefore
 $s(\chi)=\sum\limits^{e-2}_2{i\choose 2}={e-1\choose 3}$ and it follows that
 \begin{align*}
   g^*(\chi) & =\tbinom{e-1}{3}-\tbinom{e+1}{3}+(2e-1)(e-2)+1  =-\tbinom{e-1}{2}-\tbinom{e}{2}+(2e-1)(e-2)+1 \\
   & =-(e-1)^2+2e^2-5e+3=e^2-3e+2  =\frac{1}{4}(d+1)^2-\frac{3}{2}(d+1)+2\\
   & =\frac{1}{4}(d^2-4d+3).
 \end{align*}

 \begin{definition}   %%%%%%%%%%%%%%%%%%%% Definition 2.5.
 If $d\ge 5$, then we set 
 \[
 g(d):=
 \begin{cases}
  \frac{1}{4}(d-2)^2 & \text{if $d\ge 6$ is even},\\
  \frac{1}{4}(d-1)(d-3) & \text{if $d\ge 5$ is odd}.
 \end{cases}
 \]
 $g(d)$ is called the \emph{deformation bound} for ideals in $\cO_{\mP^2}$ of
 colength $d$. 
 \end{definition}

 The rest of the article is to justify this notation.

 \subsection{}\label{2.2.3}  %%%%%%%%%%%% im Ausdruck steht 2.2.3 (ohne folgenden weiteren Text)
 %%%%%%%%%%%%%%%%%%%%%%%%%%%%%%%%%%%%%%%%%
 %%%%%%%%%%%%%%%%%%%%%%%%%%%%%%%%%%%%%%%%%  ?????????????????????????????????????????????????
 %%%%%%%%%%%%%%%%%%%%%%%%%%%%%%%%%%%%%%%%

 \begin{lemma} 
   Let be $k=\overline{k}, \cI\subset\cO_{\mP^2}$ an ideal of colength $d\ge 5$
   and regularity $m$. Let be $\varphi$ the Hilbert function of $\cI$. If
   $g^*(\varphi)>g(d)$, then the assumptions of Lemma 2.1 are fulfilled by $\cI$.
 \end{lemma}

 \begin{proof}
   Let be $\chi$ the Hilbert function defined by Fig. 2.8 and Fig. 2.9,
   respectively. Let be $m=\reg(\varphi)$. If $\varphi'(m)-\varphi'(m-1)>2$,
   then $\varphi'(i)\le\chi'(i)$ and therefore $\varphi(i)<\chi(i)$ for all
   $i$, and it would follow that $g^*(\varphi)\le g(\chi)$ (Remark 2.1). If
   $\varphi'(m)-\varphi'(m-1)=1$, then
   $\varphi'(m-1)=\varphi'(m)-1=(m+1)-1=(m-1)+1$, therefore $\reg(\cI)\le m-1$
   (cf. Appendix B, Lemma 2). It follows that $\varphi'(m)-\varphi'(m-1)=2$,
   therefore $\varphi'(m-1)=m-1$. If $\varphi'(m-2)=m-2$, as well, then the
   assumptions of Lemma 2.1 are fulfilled with $\vep:=m-3$, for instance. Thus
   without restriction of generality one can assume $\varphi'(m-2)\le m-3$. \\[2mm]
   Case 1: $\varphi'(m-2)<m-3$. Figure 2.10 represents the Hilbert function
   $\varphi$ as well as the $B(3;k)$-invariant ideal $\cM$ with Hilbert
   function $\varphi$. Then one makes the deformation 
 \[ 
     E(H^0(\cM(m)))\mapsto E(H^0(\cM(m)))-u\cup v=:E(H^0(\cN(m))),
 \] 
 where $\cN$ is a $B(3;k)$-invariant ideal with Hilbert function
   $\psi>\varphi$. But then it follows $g^*(\varphi)<g^*(\psi)\le
   g^*(\chi)=g(d)$, contradiction.

 \noindent Case 2: $\varphi'(m-2)=m-3$. If the graph of $\varphi'$ would have a
 shape different from that in Fig. 2.11 a, i.e., if the graph of $\varphi'$
 would have a ``jumping place'' $n<m-2$ (cf. the terminology in [T1], p. 72),
 then as in the first case one could make a deformation 
 \[
    E(H^0(\cM(m)))\mapsto E(H^0(\cM(m)))-u\cup v=:E(\cN(m)))
 \] 
 (cf. Fig. 2.11b) and would get a
 contradiction, again. It only remains the possibility represented in
 Fig. 2.11a.
  But then $\varphi=\chi$, which contradicts the assumption
 $g^*(\varphi)>g(d)$.
 \end{proof} 

 \section{Numerical conclusions from $g^*(\varphi)>g(d)$}\label{2.3}

 \subsection{}\label{2.3.1}

 At first we describe the starting situation: In this section we suppose that
 $g(d)$ is defined, i.e. $d\ge 5$. Moreover let be $\varphi$ a Hilbert function
 such that $H_{\varphi}\neq\emptyset$, colength$(\varphi)=d, \reg
 (\varphi)=m$ and $g^*(\varphi)>g(d)$. Then the assumptions of Lemma 2.2 are
 fulfilled for an ideal $\cI$, which can be supposed to be monomial. Therefore
 the assumption $k=\overline{k}$ is superfluous. As $m$ and $d$ are uniquely
 determined by $\varphi, c:=d-m$ is uniquely determined, too.

 The aim in this section is to prove the inequality $m\ge c+2$. By Lemma 2.1 and
 Lemma 2.2, respectively, one can write $\cI=\ell\cK(-1)+f\cO_{\mP^2}(-m)$, and
 $c$ is equal to the colength of the Hilbert function $\psi$ of $\cK$. (As
 $\cI$ is monomial, $\cK$ is monomial, too, and without restriction one has
 $\ell\in\{ x,y,z \}$.)

 \begin{lemma}   %%%%%%%%%%%%%%%%%% Lemma 2.3.
   Let be $\psi$ the Hilbert function of an ideal $\cK$ of colength $c\ge 5,
   \kappa:=\reg (\psi)$, and $m\ge\kappa+2$ an integer. If one defines $\varphi$ by
   $\varphi'(n):=\psi'(n-1),0\le n\le m-1,\varphi'(n):=n+1,n\ge m$, then
   $H_{\varphi}\neq\emptyset$, $\colength(\varphi)=c+m,\reg(\varphi)=m$ and
   $g^*(\varphi)=g^*(\psi)+\frac{1}{2}m(m-3)+c$.
 \end{lemma}

 \begin{proof}
   We orientate ourselves by Figure 2.3, but the weaker assumption
   $m\ge\kappa+2$ takes the place of the assumption $m\ge\kappa+3$.

   Without restriction one can assume that $\cK$ is $B(3;k)$-invariant. Then
   $y^m\in H^0(\cK(m))$ and $\cI:=x\cK(-1)+y^m\cO_{\mP^2}(-m)$ has the Hilbert
   function $\varphi$, the regularity $m$ and the colength $c+m$. This follows
   by considering the figure mentioned above (and has been shown in a more
   general situation in [G4], Lemma 4, p. 660). We compute $g^*(\varphi)$ (cf.
   Remark 2.2):
 \begin{align*}  
   s(\varphi) & = \sum^{m-2}_{i=0}\varphi(i)
   =\sum^{\kappa-1}_{i=0}\psi(i-1)+\sum^{m-2}_{i=\kappa}\psi(i-1)\\
   & =\sum^{\kappa-2}_{i=0}\psi(i)+\sum^{m-3}_{i=\kappa-1}\psi(i)
   =s(\psi)+\sum^{m-3}_{i=\kappa-1}\left[ \tbinom{i+2}{2}-c\right]\\
   &
   =s(\psi)+\sum^{m-3}_{i=0}\tbinom{i+2}{2}-\sum^{\kappa-2}_{i=0}\tbinom{i+2}{2}-(m-\kappa-1)c.
 \end{align*}
 By Remark 2.2 it follows that:
 \begin{align*}
   g^*(\varphi)& =s(\psi)+\tbinom{m}{3}-\tbinom{\kappa+1}{3}
       -(m-\kappa-1)c-\tbinom{m+1}{3}+(c+m)(m-2)+1 \\
   & =s(\psi)-\tbinom{\kappa+1}{3}+c(\kappa-2)+1+\tbinom{m}{3}-\tbinom{m+1}{3}
   -mc+c+(c+m)m-2m\\
   & =g^*(\psi)-\tbinom{m}{2}+ m^2-2m+c=g^*(\psi)+\frac{1}{2}m(m-3)+c. \qedhere
 \end{align*}
 \end{proof}

 \subsection{The cases $c\le 4$}\label{2.3.2}

 By Lemma 2.2 (resp. by Corollary 2.1 of Lemma 2.1) one has
 $\varphi'(n)=\psi'(n-1),0\le n\le m-1,\varphi'(n)=n+1,n\ge m$, and
 $\kappa=\reg (\cK)\le m-3$. Then the assumptions of Lemma 2.3 are fulfilled.

 We use the formula given there and orientate ourselves by the Figures 2.4 -
 2.7. The regularity of the Hilbert function considered each time will now be
 denoted by $\kappa$.
 \begin{enumerate}[\  ]
 \item If $c\in\{0,1\}$, then because of $d=m+c$ it follows that $m\ge 4$.
 \item If $c=2$, then $\kappa=2$ and $m\ge\kappa+3=5$.
 \item If $c=3$ and $\kappa=2$ or $\kappa=3$, then $m\ge\kappa+3=5$.
 \item If $c=4$, then $\kappa=3$ or $\kappa=4$ and $m\ge\kappa+3\ge 6$.
 \end{enumerate}
  Thus in the cases $0\le c\le 4$ one has $m\ge c+2$.

 \subsection{The case $g^*(\psi)\le g(c)$}\label{2.3.3}

 This notation implies that $c\ge 5$. If $\kappa$ is the regularity and $c$ is
 the colength of any Hilbert function, then because of $1+2+\cdots+\kappa\ge
 c$, one always has ${\kappa+1\choose 2}\ge c$, and therefore
 $\kappa\ge\sqrt{2c}-1$. By Lemma 2.2 the assumptions of Lemma 2.1 are fulfilled,
 therefore by Corollary 2.1 it follows that $m\ge\kappa+3>5.16$. \\[2mm]
 1st case: $c$ and $m$ are even. \\
 By the formulas for $g(d)$ and $g^*(\varphi)$ it follows that:
 % \begin{align*}
 %   \tfrac{1}{4}(c^2-4c+4)+\tfrac{1}{2}m(m-3)+c & >\tfrac{1}{4}[(c+m)^2-4(c+m)+4]\\
 %   \iff \quad \tfrac{1}{2}m(m-3)+c & >\tfrac{1}{4}[2cm+m^2-4m]\\
 %  \quad  \iff 
 %   m^2-2(c+1)m+4c & >0.
 % \end{align*}
 \begin{multline*}
   \tfrac{1}{4}(c^2-4c+4)+\tfrac{1}{2}m(m-3)+c  >\tfrac{1}{4}[(c+m)^2-4(c+m)+4]\\
   \iff \quad \tfrac{1}{2}m(m-3)+c  >\tfrac{1}{4}[2cm+m^2-4m]\\
  \quad  \iff 
   m^2-2(c+1)m+4c  >0.
 \end{multline*}
  The solutions of the corresponding quadratic equation are 0 and
  $2c$.\\
  Therefore $m\ge 2c+1>c+2$. 

 \noindent
 2nd case: $c$ is even, $m$ is odd.\\
 One obtains the inequality:
 \begin{multline*}
  \qquad \tfrac{1}{4}(c^2-4c+4)+\tfrac{1}{2}m(m-3)+c  >\tfrac{1}{4}[(c+m)^2-4(c+m)+3]\\
   \iff\quad m^2-2(c+1)m+4c+1>0.\qquad
 \end{multline*}
 The solutions of the corresponding quadratic equation are
 $m=c+1\pm\sqrt{c^2-2c}\ge 0$. Because of $c+1-\sqrt{c^2-2c}<3$, if $c\ge 5$,
 it follows that $m\ge c+1+\sqrt{c^2-2c}$. Because of $c+1+\sqrt{c^2-2c}>2c-1$,
 if $c\ge 5$, it follows that $m\ge 2c$, therefore $m\ge 2c+1>c+2$. \\[2mm]
 3rd case: $c$ is odd, $m$ is even. \\
 One obtains the inequality :
 \begin{align*}
   \tfrac{1}{4}(c^2-4c+3)+\tfrac{1}{2}m(m-3)+c & >\tfrac{1}{4}[(c+m)^2-4(c+m)+4] \\
  \iff \quad  m^2-2(c+1)m+4c & >0.
 \end{align*}
 It follows that $m\ge 2c+1>c+2$.\\
 \noindent
 4th case: $c$ and $m$ are odd.\\
 One obtains the inequality: 
 \begin{align*}
   \tfrac{1}{4}(c^2-4c+3)+\tfrac{1}{2}m(m-3)+c & >\tfrac{1}{4}[(c+m)^2-4(c+m)+4]\\
   \iff\quad m^2-2(c+1)m+4c-1 & >0.
 \end{align*}
 The solutions of the corresponding quadratic
 equation are $m=c+1\pm\sqrt{(c-1)^2+1}$. Because of $c+1-\sqrt{(c-1)^2+1}<2$
 it follows that $m\ge c+1+\sqrt{(c-1)^2+1}>2c$, therefore $m\ge 2c+1\ge c+2$.

 \subsection{The case $g^*(\psi)\ge g(c)$}\label{2.3.4}

 As in the proof of Lemma 2.3 one can write $\cI=x\cK(-1)+y^m\cO_{\mP^2}(-m),\cK$
 a $B(3;k)$-invariant ideal of colength $c,
 \kappa=\reg(\cK), d=\colength(\cI)=c+m$, and again $m\ge\kappa+3$ (Corollary
 2.1). We represent the Hilbert function $\psi$ by the ideal
 $\cK=x\cJ(-1)+y^{\kappa}\cO_{\mP^2}(-\kappa), \cJ$ a $B(3;k)$-invariant ideal
 of colength $b$ and of regularity $\vep$, where $\kappa\ge\vep+3$ (cf.
 Corollary 2.1). If the Hilbert function of $\cJ$ is denoted by $\vartheta$, then
 in principle one has the situation represented by Fig. 2.12. If one assumes
 that $(m-1)-\kappa\le \colength(\cJ)=b$, then one could bring the monomials
 denoted by $\mathbf{1,2,3,\ldots}$ in the positions denoted by $1,2,3,\ldots$
 (cf. Fig. 2.12). In the course of this the Hilbert function increases and
 therefore $g^*(\varphi)<g^*(\varphi_1)<\cdots<g(d)$, contradiction. Thus one
 has $(m-1)-\kappa>b$, i.e., $m\ge\kappa+b+2=c+2$.

 \subsection{Summary}\label{2.3.5}

 \begin{lemma} %%%%%%% Lemma 2.8.
   (Notations as in Lemma 2.1 and Lemma 2.2) If $g(d)<g^*(\varphi)$, then $m\ge
   c+2$.\hfill $\Box$ 
 \end{lemma}

 \noindent
 From the proof of Lemma 2.4 we can conclude one more statement:

 \begin{corollary} %%%%%%%%%%%%%%%% Corollary 2.9.
   Let be $g^*(\psi)\le g(c)$ (which notation implies $c\ge 5$). Then $m\ge
   2c+1$. \hfill $\Box$
 \end{corollary}

 \section{Additional group operations}\label{2.4}

 \subsection{General auxiliary lemmas}\label{2.4.1}
 The group $Gl(3;k)$ operates on $S=k[x,y,z]$, and therefore on
 $H^d=\Hilb^d(\mP^2_k)$. If $\rho=(\rho_0,\rho_1,\rho_2)\in\Z^3$ is a vector
 such that $\rho_0+\rho_1+\rho_2=0$, then
 $T(\rho):=\Set{(\lambda_0,\lambda_1,\lambda_2)\in
 (k^*)^3 | \lambda_0^{\rho_0}\lambda_1^{\rho_1}\lambda_2^{\rho_2}=1}$ is a
 subgroup of $T=T(3;k)$, and 
 $\Gamma:=\left\{\left(
   \begin{smallmatrix}
 1&0&*\\0&1&*\\0&0&1    
 \end{smallmatrix}
 \right) \right\}$ is a
 subgroup of $U(3;k)$.
  We let the additive group $\G_a$ operate on $S$ by
 \[
 \psi_{\alpha}:x\mapsto x,\ y\mapsto \alpha x+y,\ z\mapsto z,\ \alpha\in k,
 \]
  and
 $\sigma:\G_m\to T$ nearly always denotes the operation
 $\sigma(\lambda):x\mapsto x,y\mapsto y,z\mapsto\lambda z,\lambda\in k^*$.

 \noindent
 %{\bf Auxiliary lemma 1.}
 \begin{auxlemma}
   If $V\subset S_d$ is a vector space, invariant under $G:=\Gamma\cdot
   T(\rho)$ where $\rho=(0,-1,1)$ or $\rho=(-1,0,1)$, then $V$ is monomial,
   i.e. invariant under $T$.
 \end{auxlemma}
 \begin{proof}
   We first consider the case $\rho=(0,-1,1)$. We take a standard basis 
    of
   $V$ consisting of $T(\rho)$-semi-invariants (see Appendix E). Assuming that the assertion above is wrong , we conclude that there is
   a $T(\rho)$-semi-invariant $f\in V$, such that the monomials occurring in
   $f$ are not in $V$. Then there is such a form with smallest $z$-degree. From
   the invariance of $V$ under $T(\rho)$ it follows that $V$ is invariant under
   the $\G_m$-action $\tau(\lambda):x\mapsto\lambda x,y\to y,z\mapsto
   z,\lambda\in k^*$. We write $f=Mp$, where
   $M=x^{\ell}y^m,p=\sum^n_{i=0}a_iy^{n-i}z^i,\ell+m+n=d,n\ge 1$ and $a_n\neq
   0$. It follows that $y\partial f/\partial z=yM\sum^n_1ia_iy^{n-i}z^{i-1}\in
   V$. Now $y\partial f/\partial z$ is also a $T(\rho)$-semi-invariant with
   smaller $z$-degree than $f$. According to the choice of $f$ it follows that
   $g:=Myz^{n-1}\in V$, therefore $y\partial g/\partial z=(n-1)My^2z^{n-2}\in
   V$, etc. One gets $My^iz^{n-i}\in V, 1\le i\le n$, therefore $My^n\in V$,
   contradiction.\\
   In the case $\rho=(-1,0,1)$ we write
   $f=x^{\ell}y^m\sum^n_{i=0}a_ix^{n-i}z^i$. Because of $x\partial f/\partial
   z=xM\sum^n_{i=1}ia_ix^{n-i}z^{i-1}$ we can argue as before.
 \end{proof}

 %\noindent 
 \begin{auxlemma}
 %{\bf Auxiliary lemma 2.} 
 Let be $\cI\subset\cO_{\mP^2}$ an ideal of
 colength $d$, which is invariant under $G:=\Gamma\cdot T(\rho)$. If
 $\rho_0+\rho_1+\rho_2=0,\rho_0<0$ and $\rho_1<0$, then $\cI$ is invariant
 under $T$.
 \end{auxlemma}
 \begin{proof}
   Let be $n$ the smallest natural number, such that $H^0(\cI(n))$ is not
   $T$-invariant. Then we have without restriction that $n\ge 1$. As
   $H^0(\cI(n))$ has a standard basis 
   , there is a
   proper semi-invariant in $H^0(\cI(n))$, i.e. a form $f\in H^0(\cI(n))$ of
   the shape $f=M(1+a_1X^{\rho}+a_2X^{2\rho}+\cdots+a_rX^{r\rho})$, $M$ a
   monomial, $a_r\neq 0,r\ge 1$, and no monomial $MX^{i\rho}$ is in
   $H^0(\cI(n))$, if $a_i\neq 0$. If $M$ would be divisible by $z$, then
   $g:=z^{-1}f\in H^0(\cI(n-1))$ would be a proper semi-invariant, too, because
   from $z^{-1}MX^{i\rho}\in H^0(\cI(n-1))$ it follows that $MX^{i\rho}\in
   H^0(\cI(n))$. Therefore, $M$ is not divisible by $z$. From the proper
   semi-invariants of $H^0(\cI(n))$ we chose one, say $f$, such that the
   $z$-degree is minimal. Now from $f\in H^0(\cI(n))$, because of the
   $\Gamma$-invariance, it follows that $x\partial f/\partial
   z=xM(a_1\rho_2z^{-1}X^{\rho}+\cdots+ra_r\rho_2z^{-1}X^{r\rho})$ and
   $y\partial f/\partial
   z=yM(a_1\rho_2z^{-1}X^{\rho}+\cdots+ra_r\rho_2z^{-1}X^{r\rho})$ is in
   $H^0(\cI(n))$, i.e., $g:=xMX^{\rho}z^{-1}p$ and $h:=yMX^{\rho}z^{-1}p$ are
   in $H^0(\cI(n))$, where
   $p(X^{\rho}):=a_1\rho_2+2a_2\rho_2X^{\rho}+\cdots+ra_r\rho_2X^{(r-1)\rho}$.
   As the $z$-degree of $g$ and of $h$ is smaller than the $z$-degree of $f,g$
   and $h$ are no longer proper semi-invariants, i.e. the monomials which occur
   in $g$ or in $h$, all are in $H^0(\cI(n))$. It follows that
   $u:=z^{-1}xMX^{r\rho}$ and $v:=z^{-1}yMX^{r\rho}$ are in $H^0(\cI(n))$. From
   the $\Gamma$-invariance it follows by applying the operators
   $x\partial/\partial z$ and $y\partial/\partial z$ repeatedly, that
   $\frac{x^{|\rho_0|}}{z^{|\rho_0|}}\cdot\frac{y^{|\rho_1|}}{z^{|\rho_1|}}\cdot MX^{r\rho}\in H^0(\cI(n))$. \\
   Now $X^{\rho}=x^{-|\rho_0|}y^{-|\rho_1|}z^{\rho_2}$ and
   $\rho_0+\rho_1+\rho_2=0$, therefore $MX^{(r-1)\rho}\in H^0(\cI(n))$.
   Applying the operators mentioned before again gives $MX^{(r-2)\rho}\in
   H^0(\cI(n))$, etc. It follows that $MX^{i\rho}\in H^0(\cI(n)),0\le i\le
   r-1$, and therefore $MX^{r\rho}\in H^0(\cI(n))$, contradiction.
 \end{proof}

 \subsection{Successive construction of $\Gamma$-invariant ideals}\label{2.4.2}

 At first we consider a general situation: Let be $\cK\subset\cO_{\mP^2}$ an
 ideal of colength $c$ and of regularity $e$; $z$ is supposed to be a non-zero
 divisor of $\cO_{\mP^2}/\cK$; let be $R:=k[x,y]$ and $\ell\in R_1$ a linear
 form. Let be $m>e$ an integer and $f\in H^0(\cK(m))$ a section, whose leading
 term is not divisible by $\ell$, i.e., if one writes
 $f=f^0+zf^1+\cdots+z^mf^m$, where $f^i\in R_{m-i}$, then $f^0$ is not
 divisible by $\ell$.

 \begin{lemma}
 The ideal $\cI:=\ell\cK(-1)+f\cO_{\mP^2}(-m)$ has the following properties:
 \begin{enumerate}[(i)]
 \item $z$ is not a zero-divisor of $\cO_{\mP^2}/\cI$.

 \item  $H^0(\cI(n))=\ell H^0(\cK(n-1))$, if $n<m$, and \\
   $H^0(\cI(n))=\ell H^0(\cK(n-1))\oplus fk[x,z]_{n-m}$, if $n\ge m$ and
   $\ell=\alpha x+y$ \\
   (respectively $H^0(\cI(n))=\ell H^0(\cK(n-1))\oplus fk[y,z]_{n-m}$, if $n\ge
   m$ and $\ell=x$).

 \item $\colength(\cI)=c+m,\reg(\cI)=m$.
 \end{enumerate}
 \end{lemma}

 \begin{proof}
   If $\ell=x$, these are the statements of ([G4], Lemma 4, p. 660). If
   $\ell=\alpha x+y,\alpha\in k^*$, then let be $u$ the automorphism
   $x\mapsto\ell,y\mapsto y,z\mapsto z$ of $S$. By applying (loc.cit) to
 \[
   \overline{\cK}:=u^{-1}(\cK),\ \overline{\cI}:=u^{-1}(\cI),\
   \overline{f}:=u^{-1}(f)
 \] one gets
   \[
 H^0(\overline{\cI}(n))=xH^0(\overline{\cK}(n-1)))\oplus\overline{f}k[y,z]_{n-m}.
 \]
   Now applying $u$ gives 
 \[
  H^0(\cI(n))=\ell H^0(\cK(n-1))\oplus fk[y,z]_{n-m}.
 \]
   As $k[y,z]=k[\ell-\alpha x,z]$ and $\ell f\in\ell H^0(\cK(m))$, the
   statement (ii) follows, if $\alpha\neq 0$.\\ If $\alpha=0$, we take the
   automorphism $x\mapsto y,y\mapsto x,z\mapsto z$ and argue as before.
 \end{proof}

 We would like to put the section $f$ in a certain normal form. We first
 consider the case $\ell=\alpha x+y$. Then we can write $f^0=x^m+\ell u,u\in
 R_{m-1}$, without restriction. As $m-1\ge e$, there is
 $v=v^0+zv^1+z^2v^2+\cdots\; \in H^0(\cK(m-1))$ such that $v^0=u$. As $f$ is
 determined only modulo $\ell H^0(\cK(m-1))$, we can replace $f$ by
 $\tilde{f}:=f-\ell v$, therefore we can assume without restriction, that
 $f=f^0+zf^1+\cdots+z^mf^m$, where $f^0=x^m$.

 We now suppose that $\cK$ is invariant under $\Gamma$, and will formulate
 conditions that $\cI$ is $\Gamma$-invariant, too. This is equivalent to the
 condition that $f$ is $\Gamma$-invariant modulo $\ell H^0(\cK(m-1))$. By
 ([T2], Hilfssatz 1, p.~142) this is equivalent to the condition that $\langle
 x,y\rangle\partial f/\partial z\subset\ell H^0(\cK(m-1))$. It follows that
 $\ell$ is a divisor of $f^i,1\le i\le n$, i.e., one has $f=x^m+\ell zg,g\in
 S_{m-2}$.

 Write $g=g^0+zg^1+z^2g^2+\cdots$, where $g^i\in R_{m-2-i}$ and choose
 $u=u^0+zu^1+\cdots\in H^0(\cK(m-2))$ such that $u^0=g^0$. This is possible, if
 $m-2\ge e$. As $f$ is determined only modulo $\ell H^0(\cK(m-1))$, one can
 replace $f$ by $\tilde{f}=f-\ell zu$. It follows that one can assume without
 restriction $f=x^m+\ell z^2g$, where $g\in S_{m-3}$.

 Choose $u\in H^0(\cK(m-3))$, where $u^0=g^0$; this is possible, if $m-3\ge
 e$. If this is the case, replace $f$ by $\tilde{f}=f-\ell z^2u$. It follows
 that one can assume without restriction $f=x^m+\ell z^3g$, where $g\in
 S_{m-4}$, etc. Finally one obtains $f=x^m+z^{m-e}\ell g,\ell=\alpha x+y,g\in
 S_{e-1}$, and the $\Gamma$-invariance of $f$ modulo $\ell H^0(\cK(m-1))$ is
 equivalent to $\langle x,y\rangle [(m-e)z^{m-e-1}\ell g+z^{m-e}\ell\partial g/\partial z]\subset \ell H^0(\cK(m-1))$, i.e. equivalent to:
 \begin{equation}\label{1}
   \langle x,y\rangle [(m-e)g+z\partial g/\partial z]\subset H^0(\cK(e))
 \end{equation}

 In the case $\ell=x$, because of $R_m=xR_{m-1}\oplus y^m\cdot k$, one can
 write $f^0=y^m+xu$, and the same argumentation shows that one can write
 $f=y^m+z^{m-e}xg,g\in S_{e-1}$, and the $\Gamma$-invariance can again be
 expressed by the inclusion (2.1).

 \subsection{Standard forms}\label{2.4.3}

 Let $\cI\subset\cO_{\mP^2}$ have the colength $d$ and Hilbert function
 $\varphi$, and let $\cI$ be invariant under $G:=\Gamma\cdot T(\rho)$, where
 $\rho_2>0$. Moreover, we assume that $g^*(\varphi)>g(d)$. By Lemma 2.2 it
 follows that $\cI=\ell\cK(-1)+f\cO_{\mP^2}(-m)$, if $k=\overline{k}$ is
 supposed. As $H^0(\cI(m-1))=\ell H^0(\cK(m-2))$ is then invariant under $G$
 and $m-1>e=\reg (\cK)$, it follows that $\langle\ell\rangle$ and $\cK$ are
 $G$-invariant. Assume that $\langle ax+by+z\rangle$ is $\Gamma$-invariant.
 But then $\langle ax+by+z\rangle=\langle
 (a+\alpha)x+(b+\beta)y+z\rangle,\;\forall\;\alpha,\beta\in k$, which is not
 possible. Thus we have $\ell=ax+by$. From $\langle
 \lambda_0ax+\lambda_1by\rangle=\langle
 ax+by\rangle,\;\forall\;(\lambda_0,\lambda_1,\lambda_2)\in T(\rho)$ it
 follows that $\lambda_0/\lambda_1=1\;\forall\;
 (\lambda_0,\lambda_1,\lambda_2)\in T(\rho)$, if $a$ and $b$ both were
 different from 0. But then it would follow $T(\rho)\subset T(1,-1,0)$, and
 therefore $\rho_2=0$, contradiction. Therefore we have $\ell=x$ or $\ell=y$,
 without restriction.

 We consider the case $\ell=x$, for example. As it was shown in (2.4.2) we can
 write $f=y^m+z^{m-e}xg, e=\reg(\cK), g\in S_{e-1}$.

 From Appendix E it follows that $xH^0(\cK(m-1))$
 has a standard basis of $T(\rho)$-semi-invariants $f_i=m_ip_i(X^{\rho})$,
 i.e., $m_i$ is a monomial, $p_i$ is a polynomial in one variable with constant
 term 1, and such that $m_i$ does not occur in $f_j$ any longer, if $i\neq
 j$. Now each $f_i$ is divisible by $x$, therefore $y^m$ does not occur in
 $f_i$. If the initial monomial $m_i$ of $f_i$ appears in $f$, then $m_i$ has
 to appear in $z^{m-e}xg$. By choosing $\alpha\in k$ in a suitable way, one
 can achieve that $m_i$ does not occur in $\tilde{f}:=f-\alpha f_i$. As
 $\rho_2>0$ and $f_i$ is divisible by $x$, $\tilde{f}$ still has the shape
 $y^m+z^{m-e}x\tilde{g},\tilde{g}\in S_{e-1}$. By repeating this procedure
 one can achieve that none of the $m_i$ does occur in $f=y^m+z^{m-e}xg$
 (and $f$ is still invariant under $\Gamma$ modulo $xH^0(\cK(m-1))$. The same
 argumentation as in the proof of the lemma in Appendix E then shows that
 $f$ is automatically a $T(\rho)$-semi-invariant with initial monomial $y^m$,
 and $f$ together with the $f_i$ forms a standard basis of $H^0(\cI(m))$. We
 summarize:

 \begin{lemma}
   Let be $\cI\subset\cO_{\mP^2_k}$ an ideal of colength $d$, with Hilbert
   function $\varphi$, which is invariant under $G=\Gamma\cdot T(\rho)$,
   where $\rho_2>0$. Assume that $g(d)<g^*(\varphi)$. (It is not assumed that
   $k=\overline{k}$.) Then $\cI=x\cK(-1)+f\cO_{\mP^2}(-m)$ or
   $\cI=y\cK(-1)+f\cO_{\mP^2}(-m)$, where $\cK$ is a $G$-invariant ideal with
   $\colength(\cK)=c , \reg (\cK)=e$, and $c+m=d$. Moreover $f\in
   H^0(\cK(m))$ can be written in the form $f=y^m+z^{m-e}xg$ respectively in
   the form $f=x^m + z^{m-e}yg $, where $g\in S_{e-1}$. We have $\langle
   x,y\rangle\partial f/\partial z\subset x H^0(\cK(m-1))$ or $\langle
   x,y\rangle\partial f/\partial z\subset y H^0(\cK(m-1))$, respectively, and
   each of these inclusions are equivalent to the inclusion (2.1) in section
   (2.4.2). One has
   \begin{align*}
 H^0(\cI(n))  & =
 \begin{cases}
   xH^0(\cK(n-1)) & \text{if $n<m$},\\
   xH^0(\cK(n-1))\oplus fk[y,z]_{n-m} & \text{if $n\ge m$},
 \end{cases}
 \intertext{respectively}
 H^0(\cI(n)) & = 
 \begin{cases}
   yH^0(\cK(n-1))  & \text{if $n\le m$},\\
       yH^0(\cK(n-1))\oplus fk[x,z]_{n-m} & \text{if $n\ge m$}.
     \end{cases}
 \end{align*}
 \end{lemma}

 If one chooses a standard basis 
  $\{f_i\}$ of $xH^0(\cK(m-1))$ or of
 $yH^0(\cK(m-1))$, respectively, then one can choose $f$ in such a way that $f$
 has the form and the properties mentioned above and together with the $f_i$'s
 forms a standard basis of $H^0(\cI(m))$. 

 \begin{proof}
   If $k=\overline{k}$ this follows from the foregoing discussion. One has,
   e.g. $\cI\mathop{\otimes}\limits_k\overline{k}=y\overline{\cK}(-1)+\overline{f}\cO_{\mP^2}(-m)$, where $\overline{\cK}\subset\cO_{\mP^2\otimes\overline{k}}$ and $\overline{f}\in H^0(\overline{\cK}(m))$ have the properties mentioned above.
 One sees that $\overline{\cK}=\cI\otimes\overline{k}:y\cO_{\mP^2\otimes\overline{k}}(-1)$ and therefore one has the exact sequence: 
   \[
   0\to (\cO_{\mP^2\otimes\overline{k}}/\overline{\cK})(-1)\to
   \cO_{\mP^2\otimes\overline{k}}/\cI\otimes\overline{k}\to\cO_{\mP^2\otimes\overline{k}}/\cI\otimes\overline{k}+y\cO_{\mP^2\otimes\overline{k}}(-1)\to
   0.
 \]
   If $\cK:=\cI:y\cO_{\mP^2}(-1)$, then the sequence 
 \[
 0\to
 (\cO_{\mP^2}/\cK)(-1)\to\cO_{\mP^2}/\cI\to\cO_{\mP^2}/\cI+y\cO_{\mP^2}(-1)\to
 0
 \]
 is exact, too. Tensoring this sequence with $\overline{k}$ one obtains a
 commutative diagram
 \[
 \xymatrix{
   0 \ar[r] & (\cO_{\mP^2\otimes\ol{k}} / \cK\otimes\ol{k})(-1)\ar[r]^{\cdot y}\ar[d]  
            & \cO_{\mP^2\otimes\ol{k}}/ \cI\otimes\ol{k} \ar[r]\ar@{=}[d]  
            & \cO_{\mP^2\otimes\ol{k}} /
               \cI\otimes\ol{k}+y\cO_{\mP^2\otimes\ol{k}}(-1) \ar[r]\ar@{=}[d] 
            & 0  \\
   0 \ar[r] & \cO_{\mP^2\otimes\ol{k}}/\ol{\cK} (-1) \ar[r]  
            & \cO_{\mP^2\otimes\ol{k}} /\cI\otimes\ol{k}\ar[r]  
            & \cO_{\mP^2\otimes\ol{k}} /
                  \cI\otimes\ol{k}+y\cO_{\mP^2\otimes\ol{k}}(-1) \ar[r] 
            & 0 
 }
 \]
 with exact rows, where the first vertical arrow is obtained from the canonical
 injection $\cK\otimes \overline{k}\hookrightarrow\overline{\cK}$ by tensoring
 with $\cO_{\mP^2\otimes\overline{k}}(-1)$. It follows from this, that
 $\cK\otimes\overline{k}\stackrel{\sim}{\to}\overline{\cK}$. Because of
 $H^0(\cI(m))\otimes\overline{k}\stackrel{\sim}{\to}
 H^0(\cI(m)\otimes\overline{k})$ one obtains a standard basis of
 $T(\rho)$-semi-invariants of $H^0(\cI(m)\otimes\overline{k})$ by tensoring a
 standard basis of $H^0(\cI(m))$ with
 $\mathop{\otimes}\limits_k 1_{\overline{k}}$.
  As the elements of a standard basis
 are uniquely determined up to constant factors, it follows that
 $\overline{f}=f\otimes_k 1_{\overline{k}}$, where $f\in H^0(\cK(m))$.
 Therefore $f$ has the form $x^m+z^{m-e}yg,g\in S_{e-1}$, if $\overline{f}$ has
 the form $x^m+z^{m-e}y\overline{g},\overline{g}\in
 S_{e-1}\otimes\overline{k}$. For reasons of dimension it follows
 $H^0(\cI(n))=yH^0(\cK(n-1)),n<m$, and
 $H^0(\cI(n))=yH^0(\cK(n-1))+fk[x,z]_{n-m},n\ge m$. As the $G$-invariance of
 $\cK$ follows from the $G$-invariance of $\cI$, the remaining statements of
  Lemma 2.6 follow by the same argumentation as in the case $k=\overline{k}$.
 \end{proof}

 \begin{definition}\label{2}
   The (uniquely determined) decomposition $\cI=x\cK(-1)+f\cO_{\mP^2}(-m)$ or
   $\cI=y\cK(-1)+f\cO_{\mP^2}(-m)$ of Lemma 2.6 is called $x$-standard form or
   $y$-standard form of $\cI$, respectively.( N.B. For the Hilbert function $\varphi$ of $\cI$
   this definition implies that $g(d) < g^*(\varphi)$ is fulfilled.)
 \end{definition}

 \begin{corollary}
   Let be $R = k[x,y]$. If $ \cI$ has x-standard form (resp. y-standard form),
   then $xR_{m-2}$ (resp. $yR_{m-2} $) is contained in $H^0(\cI(m-1))$ and thus
   $ xR_{m-1}$ (resp. $yR_{m-1}$) is contained in $H^0(\cI(m))$.
 \end{corollary}
 \begin{proof}
   One has $ m-2\ge c$ by Lemma 2.4 and thus $xR_{m-2}\subset xH^0(\cK(m-2))$
   (resp. $yR_{m-2}\subset yH^0(\cK(m-2))$) by Appendix C, Remark 2.
 \end{proof}

 \begin{remark}
   If $\cI$ has $x$-standard from or $y$-standard form, respectively, and if
   $\G_m$ acts on $S$ by $\sigma(\lambda):x\mapsto x,y\mapsto y,z\mapsto\lambda
   z,\lambda\in k^*$, then $\cI_0:=\lim\limits_{\lambda\to
     0}\sigma(\lambda)\cI$ again has $x$-standard form or $y$-standard form,
   respectively. This follows from $h^0(\cI_0(n))=h^0(\cI(n)),n\in\N$ (cf.
   [G2], Lemma 4, p. 542), because then $H^0(\cI_0(n))$ is generated by the
   initial monomials of a standard basis of $H^0(\cI(n)),n\in\N$. The Figures
   2.13a and 2.13b show the typical shape of the pyramid formed by the initial
   monomials.
 \end{remark}

 \begin{remark}\label{4}
   An ideal cannot have $x$-standard form and $y$-standard form at the same
   time. For $m=\reg(\cI)$ and $c=\colength(\cI)$ are determined by the
   Hilbert function $\varphi$ of $\cI$. If $\cI$ would have $x$-standard form
   as well as $y$-standard form, then $E(H^0(\cI_0(d)))$ has the form shown in
   Figure 2.14a. As $\cI$ and $\cI_0$ have the same Hilbert function, $\varphi$
  has the form shown in Figure 2.14b, and therefore $g^*(\varphi)\le
  g^*(\chi)=g(d)$, where $\chi$ is the Hilbert function of (2.2.2).
\end{remark}

\begin{remark}\label{5}
  (Notations and assumption as in Remark 2.3)
  $\cI_{\infty}:=\lim\limits_{\lambda\to\infty}\sigma(\lambda)\cI$ has again
  $x$-standard form or $y$-standard form, respectively. The reasoning is as
  follows: The Hilbert function $\vartheta$ of $\cI_{\infty}$ (respectively
  the regularity $\mu$ of $\cI_{\infty}$) is $\ge\varphi$ (respectively $\ge
  m$). This follows from the semicontinuity theorem. By Lemma 2.4 it
  follows that $m\ge c+2$, therefore $R_{m-2}\subset H^0(\cK(m-2))$ (cf.
  Appendix C, Remark 2). If $\cI$ has $x$-standard form (respectively
  $y$-standard form), then $xR_{m-2}\subset H^0(\cI(m-1))$ (respectively
  $yR_{m-2}\subset H^0(\cI(m-1)))$ follows. Therefore $xR_{m-2}\subset
  H^0(\cI_{\infty}(m-1))$ (respectively $yR_{m-2}\subset
  H^0(\cI_{\infty}(m-1)))$. It follows that $xR_{\mu-2}\subset
  H^0(\cI_{\infty}(\mu-1))$ (respectively $yR_{\mu-2}\subset
  H^0(\cI_{\infty}(\mu-1)))$. Now $\cI_{\infty}$ has $x$-standard form or
  $y$-standard form, in any case, and the above inclusions show that $\cI$
  and $\cI_{\infty}$ have the same kind of standard form.
\end{remark}

\section{The type of a $G$-invariant ideal}\label{2.5}

Let $\cI\subset\cO_{\mP^2_k}$ be an ideal of colength $d$, with Hilbert
function $\varphi$ and invariant under $G=\Gamma\cdot T(\rho)$, where
$\rho_2>0$ and $k$ is an extension field of $\C$.

\subsection{}\label{2.5.1}
%{Definition 3}
\begin{definition}
$1^{\circ}$ $\cI$ has the type $(-1)$, if one of the following cases occurs:\\
\quad 1st case: $g^*(\varphi)\le g(d)$, where $d\ge 5$ by convention.\\
\quad 2nd case: $0\le d\le 4$

$2^{\circ}$ We now assume $g^*(\varphi)>g(d)$, which notation implies $d\ge
5$. Then one has $\cI=\ell_0\cI_1(-1)+f_0\cO_{\mP^2}(-m_0)$, where
$(\ell_0,\cI_1,f_0,m_0)$ is as in Lemma 2.6. If $\cI_1$ has type $(-1)$, then we
say $\cI$ has type 0. If $\cI_1$ is not of type $(-1)$, then
$\cI_1=\ell_1\cI_2(-1)+f_1\cO_{\mP^2}(-m_1)$, where $(\ell_1,\cI_2,f_1,m_1)$
is as in Lemma 2.6. If $\cI_2$ has the type $(-1)$, then we say $\cI$ has the
type 1, etc. As $d=\colength(\cI)=\colength(\cI_1)+m_0$, etc., the
colengths of the ideals in question decrease and the procedure will
terminate. We have
\end{definition}

%% Eingebettes Lemma in Definition!?!
\begin{lemma}\label{7}
If $\cI$ has not the type $(-1)$, then one has a sequence of decompositions
%
% \begin{equation}\label{2}
%   \begin{array}{c}
%   \cI=:\cI_0=\ell_0\cI_1(-1)+f_0\cO_{\mP^2}(-m_0)\\
%   \cI_1=\ell_1\cI_2(-1)+f_1\cO_{\mP^2}(-m_1)\\
%   \cdots\cdots\cdots\cdots\cdots\cdots\cdots\\
%   \cI_{r-1}=\ell_{r-1}\cI_r(-1)+f_{r-1}\cO_{\mP^2}(-m_{r-1})\\
%   \cI_r=\ell_r\cK(-1)+f_r\cO_{\mP^2}(-m_r).
%   \end{array}
% \end{equation}
%
\begin{equation}\label{2}
  \begin{aligned}
  \cI=:\cI_0 & =\ell_0\cI_1(-1)+f_0\cO_{\mP^2}(-m_0),\\
  \cI_1& =\ell_1\cI_2(-1)+f_1\cO_{\mP^2}(-m_1),\\
  \cdots\cdots & \cdots\cdots\cdots\cdots\cdots\cdots\cdots\cdots\cdots\\
  \cI_{r-1}& =\ell_{r-1}\cI_r(-1)+f_{r-1}\cO_{\mP^2}(-m_{r-1}),\\
  \cI_r& =\ell_r\cK(-1)+f_r\cO_{\mP^2}(-m_r).
\end{aligned}
\end{equation}
For a given ideal $\cI_i , (\ell_i,\cI_{i+1},f_i,m_i)$ is defined as in Lemma 2.6,
 where $0\le i\le r$ and $\cI_{r+1}:=\cK$. If $d_i$ and $\varphi_i$ is the
colength and the Hilbert function, respectively, of $\cI_i$, then the
inequality $g(d_i)<g^*(\varphi_i)$ is fulfilled. The ideal $\cK$ has the type
$(-1)$. The colength (the Hilbert function, the regularity) of $\cK$ is
denoted by $c$ (by $\psi$ and $\kappa$, respectively).
\end{lemma}
%\end{proof}

$3^{\circ}$ We have already noted in Corollary 2.1 that the
decompositions of $\cI_0,\cI_1,\cdots$, are uniquely determined, in essence.
Therefore the number $r$ is determined uniquely. It is called the $\emph{type}$ of
$\cI$. The types of the ideals occurring in (2.2), their Hilbert functions and
the numbers $m_0,\cdots,m_r$ are uniquely determined by the Hilbert function
$\varphi$. \hfill $\qed$

\subsection{Properties of an ideal of type $r\ge 0$}\label{2.5.2}

The assumptions and notations are as before, and we assume that $\cI$ has the
type $r\ge 0$.

\begin{lemma} In the decompositions (2.2) of $\cI$ one has:%\\[-8mm]
  \begin{enumerate}[(a)]
%(a) 
  \item $\colength(\cI_i)=\colength(\cI_{i+1})+m_i$, i.e.,
$\colength(\cI_i)=c+m_r+\cdots+m_i, 0\le i\le r$.
%(b)
\item If $r=0$, then $m_0\geq c+2$, where $c=\colength(\cK)$.
%(c)
\item If $r\ge 1$, then $m_0\ge c+2+m_r+\cdots+m_1=\colength(\cI_1)+2$.
%(d)
\item If $r\ge 0$, then $m_0\ge 2^r(c+2)$. 
\end{enumerate}
\end{lemma}

\begin{proof}
  (a) follows from the decompositions (2.2) and from Lemma 2.6.\\
  (b) follows from Lemma 2.4.\\
  (c) As the statement only depends on the Hilbert functions of
  $\cI_0,\cI_1,\cdots,\cI_{r+1}=\cK$ , one can assume without restriction
  $\ell_i=x,0\le i\le r$, and $\cI_i$ is $B(3;k)$-invariant, $0\le i\le
  r+1$. Then one has  
\[
  \cI=x\cI_1(-1)+y^{m_0}\cO_{\mP^2}(-m_0)\quad \text{and} \quad
  \cI_1=x\cI_2(-1)+y^{m_1}\cO_{\mP^2}(-m_1),
\]
  where $\cI_2=\cK$, if $r=1$.

  We argue as in (2.3.4) and we orientate ourselves by Figure 2.15. If one
  would have $m_0-1-m_1\le \colength(\cI_2)$, then one could make the
  deformations 
\[
  \fbox{$\mathbf{1}$} \mapsto 1,\ \ldots\ ,\  
   \fbox{$\mathbf{m_0-1-m_1}$}\mapsto m_0-1-m_1.
\]
  Then we would get
  $g^*(\varphi)<\cdots\le g(d)$, contradiction, because from type $r\ge 0$
  it follows that $g^*(\varphi)>g(d)$. Therefore one has $m_0-1-m_1>
  \colength(\cI_2)$. Now $\colength(\cI_2)=c$, if $r=1$ (respectively
  $\colength(\cI_2)=c+m_r+\cdots+m_2$, if $r\ge 2$) as was shown in (a).\\

  (d) If $r=0$, this is statement (b). If $r=1$, then by (b) and (c) it
  follows that $m_0\ge (c+2)+m_1\ge 2(c+2)$. Now we assume $r\ge 2$. We argue
  by induction and assume that $m_r\ge c+2,m_{r-1}\ge 2(c+1),\ldots,m_1\ge
  2^{r-1}(c+2)$. By (c) it follows that
\[
    m_0\ge (c+2)+(c+2)+2(c+2)+\cdots+2^{r-1}(c+2)=2^r(c+2).
\]

  In the case $r=1$, the statement (c) is valid even if $\cK=\cO_{\mP^2}$,
  i.e., if $c=0$. Because then one can assume without restriction again
\[
  \cI=x\cI_1(-1)+y^{m_0}\cO_{\mP^2}(-m_0) \quad \text{and} \quad
  \cI_1=x\cO_{\mP^2}(-1)+y^{m_1}\cO_{\mP^2}(-m_1),
\]
 and then, because of
  $\colength(\cI_1)=m_1$, by Lemma 2.4 it follows that $m_0\ge m_1+2$,
  i.e., one gets the statement (c).
\end{proof} 

\begin{corollary}[of the proof of Lemma 2.8] \label{3}
If $\cI$ has type $r\ge 1$, then $m_j+j<m_i+i-1$ for all $0\le i<j\le r$.
\end{corollary}

\begin{proof}
  We refer to Lemma 2.8 and use the same notations. If the sequence of
  decompositions (2.2) in (2.5.1) begins with $\cI_i$ instead of $\cI_0$, then
  from Lemma 2.8c it follows that $m_{i-1}\ge c+2+m_r+\cdots+m_i$. It follows
  that $m_{i-1}-m_i\ge 2$, therefore $m_{i-1}+(i-1)>m_i+i$. One gets
\[
  m_r+r<m_{r-1}+(r-1)<\cdots <m_1+1<m_0.
\]
If $m_j+j=m_i+(i-1)$, then it would follow $j=i+1$ and therefore
$m_{i+1}=m_i-2$. We will show, that this is not possible.

  The ideal $\cI_i$ has the type $r-i\ge 1$, the colength
  $c+m_r+\cdots+m_i=d_i$, and the Hilbert function $\varphi_i$ of $\cI_i$
  fulfils the inequality $g(d_i)<g^*(\varphi_i)$ (cf. Lemma 2.7). As to the
  Hilbert function $\varphi_i$, one obtains the situation of Figure 2.16,
  which corresponds to the Hilbert function $\varphi$. But then it follows
  that $g^*(\varphi_i)\le g^*(\chi')=g(d_i)$, where the graph of $\chi'$ is
  denoted by a dotted line (cf. the first case in (2.2.2), Figure 2.8, and the
  argumentation in the proof of Lemma2.2).
\end{proof}
\newpage
%%%%%%%%%%%%%%%%%%%%%%%%%%%%%%%%%%%%%%%%%%%%%%%%%%%%%%%%%%%%%%%%%%%%%%%%%%%%%%
%%
%%  Graphiken fuer Kapitel 2
%%
%%
% ---- fig2.1 ------
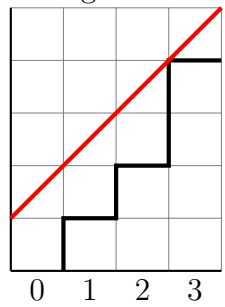
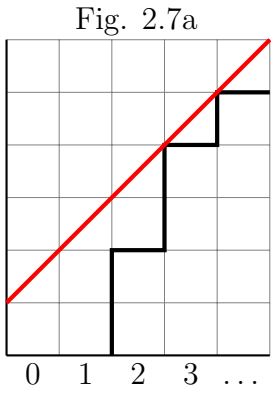
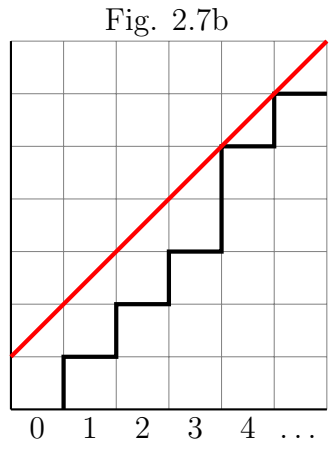
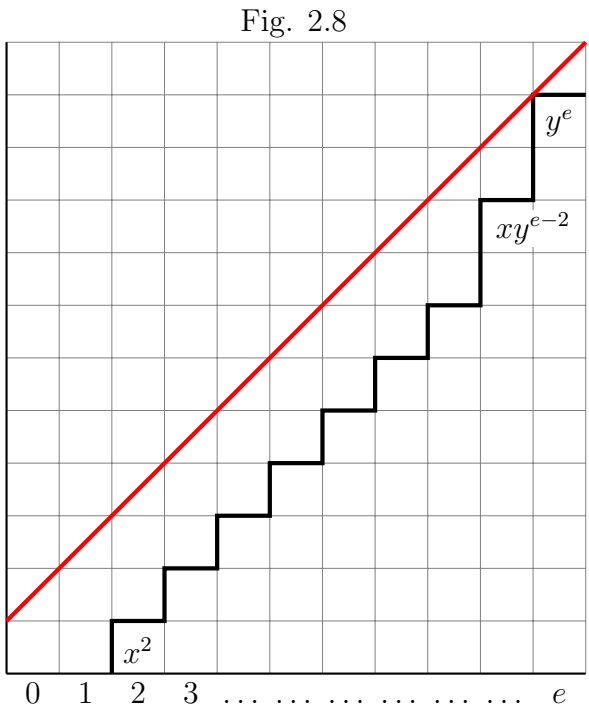
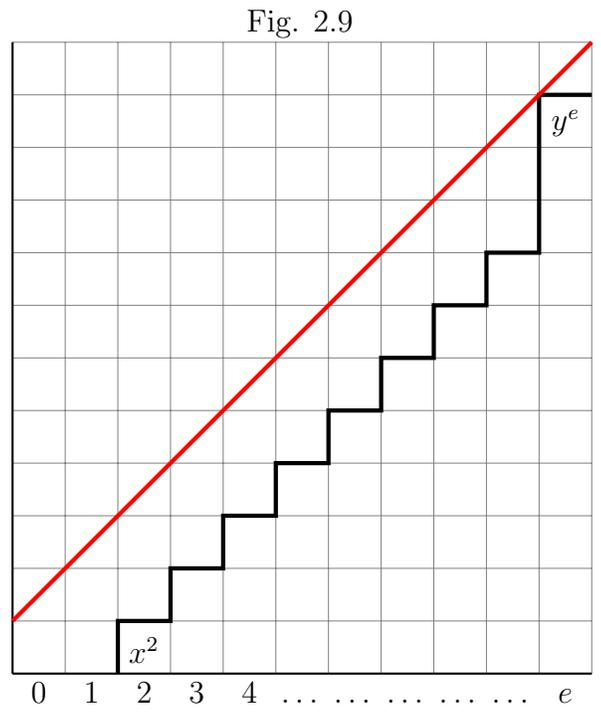
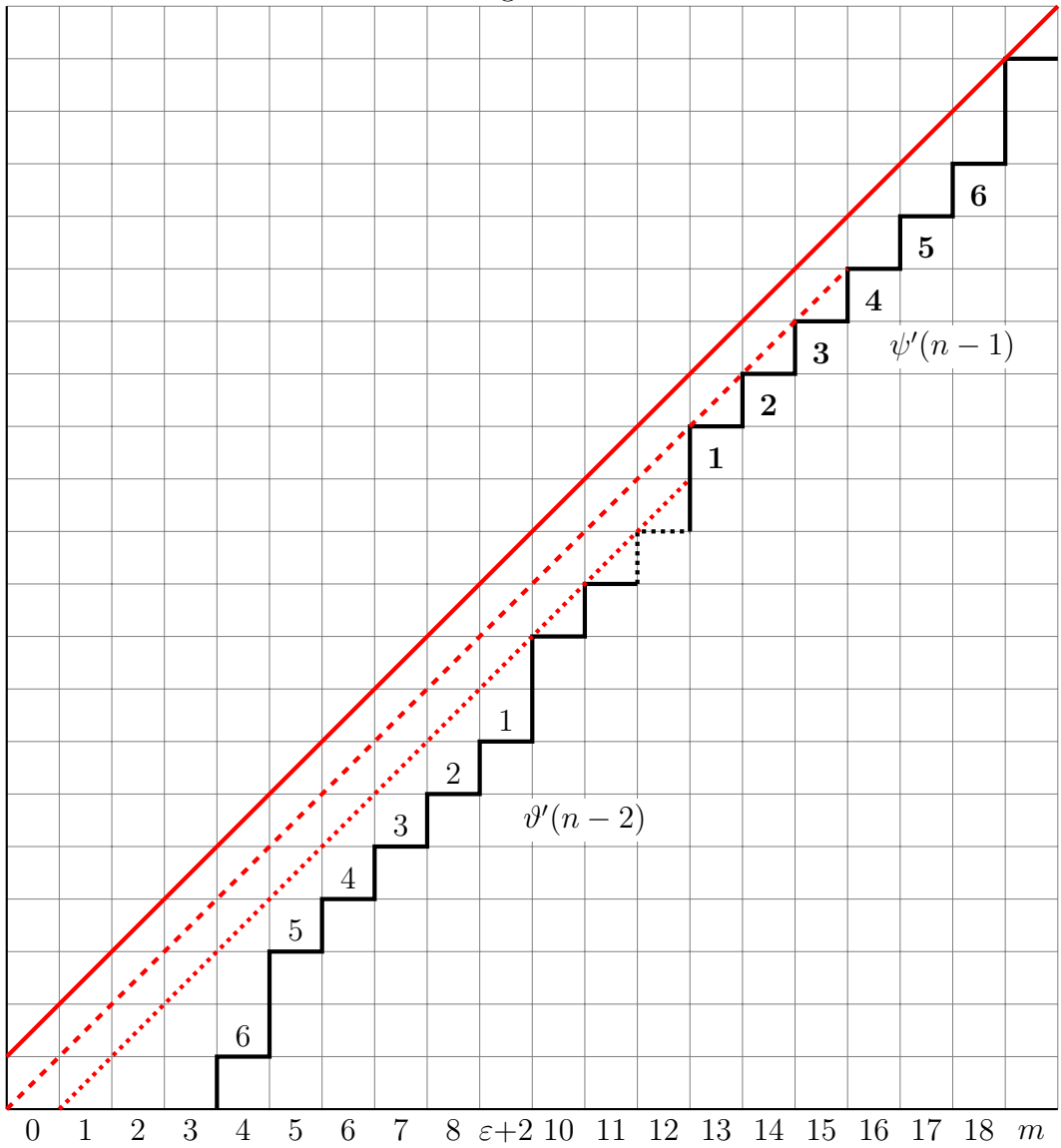
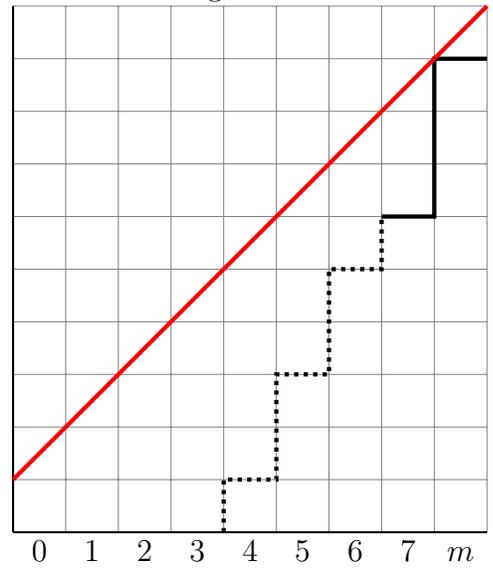
\begin{minipage}{10cm*\real{0.7}}
  \label{fig:2.1}
 \centering Fig. 2.1
 \begin{tikzpicture}[scale=0.7]
  \draw[style=help lines] grid (10,11);
  % axes
  \draw[thick] (0,0) -- (0,11); % node[anchor=north east] {$y$};
  \draw[thick] (0,0) -- (10,0); % node[anchor=south east] {$x$};
  % ticks
  { \pgftransformxshift{0.5cm} 
    \pgftransformyshift{-0.55cm}
  \foreach \x in {0,1,2,3} \draw (\x,0) node[anchor=base] {$\x$};
    \draw (4,0) node[anchor=base] {$\alpha$}; 
     \foreach \x in {5,6,7,8} \draw (\x,0) node[anchor=base] {$\dots$};
     \draw (9,0)  node[anchor=base] {$e$}; 
  }
  % steps
  \draw[\Red,ultra thick] (4,0) -- (4,1) -- (5,1) -- (5,3) -- (6,3) -- (6,4)
  -- (7,4) -- (7,5) -- (8,5) -- (8,5) -- (8,7) -- (9,7) --(9,10) -- (10,10);
  % line % ultra thick
  \draw[\Black, ultra thick] (0,1) -- (10,11);
 \end{tikzpicture}
%  \caption{Fig. 2.1} \label{fig:2.1}
\end{minipage}
\hfill
% ---- fig2.2 ------
\begin{minipage}{10cm*\real{0.7}}
 \label{fig:2.2}
 \centering Fig. 2.2  
\begin{tikzpicture}[scale=0.7]
 \draw[style=help lines]  grid (10,11);
 % axes 
 \draw[thick] (0,0) -- (0,11); 
 \draw[thick] (0,0) -- (10,0);
 % ticks
 {
 \pgftransformxshift{0.5cm}
 \pgftransformyshift{-0.55cm}
  \foreach \x in {0,1,2,3,4} \draw (\x,0) node[anchor=base] {$\x$}; 
  \foreach \x in {5,6,7,8} \draw (\x,0) node[anchor=base] {$\dots$}; 
  \draw[anchor=base] (9,0) node {$d$};
 }
 % steps 
 \draw[\Red,ultra thick] (1,0) -- (1,1) -- (2,1) -- (2,2) -- (3,2) -- (3,3) --
 (4,3) -- (4,4) -- (5,4) -- (5,5) -- (6,5) -- (6,6) -- (7,6) -- (7,7) -- (8,7)
 -- (8,8) -- (9,8) -- (9,10) -- (10,10);
 % line
  \draw[\Black, ultra thick] (0,1) -- (10,11);  
\end{tikzpicture}
\end{minipage}
\par
% 
%
% ---- fig2.3 ------
\begin{minipage}{11cm*\real{0.7}} \label{fig:2.3}
  \centering Fig. 2.3 
 \begin{tikzpicture}[scale=0.7]
 \draw[style=help lines]  grid (11,12);
% axes
 \draw[thick] (0,0) -- (0,12); 
 \draw[thick] (0,0) -- (11,0); 
% ticks
 {
 \pgftransformxshift{0.5cm}
 \pgftransformyshift{-0.55cm}
  \foreach \x in {0,1,2,3,4} \draw (\x,0) node[anchor=base] {$\x$}; 
 \draw[anchor=base] (5,0) node {$\varepsilon$};
 \draw[anchor=base] (6,0) node {$\varepsilon{+}1$};
 \foreach \x in {7,8,9} \draw (\x,0) node[anchor=base] {$\dots$}; 
 \draw[anchor=base] (10,0) node {$m$};
 }
% steps
 \draw[\Red,ultra thick,dotted] (4,0) -- (4,1) -- (5,1) -- (5,3)
   -- (6,3) -- (6,4);
 \draw[\Red,ultra thick] (6,4) -- (6,6) -- (7,6);
 \draw[\Red,ultra thick,dotted] (7,6) -- (7,7) -- (8,7) -- (8,8) --
                  (9,8);
 \draw[\Red,ultra thick] 
            (9,8) -- (9,9) -- (10,9) -- (10,11) -- (11,11);  
% line   % ultra thick
 \draw[\Black, ultra thick] (0,1) -- (11,12);  
 \draw[\Black, ultra thick,dashed] (0,0) -- (9,9);  
\end{tikzpicture}
\end{minipage}
\hfill
\begin{minipage}{4cm}
% ---- Fig:2.4
  \begin{minipage}{3cm*\real{0.7}}\label{fig:2.4}
\centering Fig. 2.4
\begin{tikzpicture}[scale=0.7] 
 \draw[style=help lines]  grid (3,4);
 % axes 
 \draw[thick] (0,0) -- (0,4); 
 \draw[thick] (0,0) -- (3,0);
 % ticks
 {
 \pgftransformxshift{0.5cm}
 \pgftransformyshift{-0.55cm}
  \foreach \x in {0,1,2} \draw (\x,0) node[anchor=base] {$\x$}; 
 }
 % steps
 \draw[\Red,ultra thick] (1,0) -- (1,2) -- (2,2) -- (2,3) -- (3,3);
 % line
 \draw[\Black, ultra thick] (0,1) -- (3,4);  
\end{tikzpicture}
%\caption{Fig. 2.4} 
\end{minipage}\\
\vspace{1cm}

% ---- fig2.5 ------
\begin{minipage}{3cm*\real{0.7}}\label{fig:2.5}
 \centering Fig. 2.5
\begin{tikzpicture}[scale=0.7]
 \draw[style=help lines]  grid (3,4);
 % axes 
 \draw[thick] (0,0) -- (0,4); 
 \draw[thick] (0,0) -- (3,0);
 % ticks
 {
 \pgftransformxshift{0.5cm}
 \pgftransformyshift{-0.55cm}
  \foreach \x in {0,1,2} \draw (\x,0) node[anchor=base] {$\x$}; 
 }
 % steps 
  \draw[\Red,ultra thick] (1,0) -- (1,1) -- (2,1) -- (2,3) -- (3,3);
 % line
  \draw[\Black, ultra thick] (0,1) -- (3,4);  
\end{tikzpicture}
\end{minipage}
\end{minipage}

% ---- fig2.6a -------
\begin{minipage}{3cm*\real{0.7}} \label{fig:2.6a}
  \centering Fig. 2.6a
  \begin{tikzpicture}[scale=0.7]
    \draw[style=help lines] grid (4,5);
    % axes
    \draw[thick] (0,0) -- (0,5); 
    \draw[thick] (0,0) -- (4,0);
    % ticks
    { 
    \pgftransformxshift{0.5cm} 
    \pgftransformyshift{-0.55cm}
    \foreach \x in {0,1,2} \draw (\x,0)  node[anchor=base] {$\x$}; 
    \draw (3,0)  node[anchor=base] {$\dots$};
    }
    % steps
    \draw[\Red,ultra thick] (2,0) -- (2,3) -- (3,3) -- (3,4) -- (4,4);
    % line
    \draw[\Black, ultra thick] (0,1) -- (4,5);
  \end{tikzpicture}
\end{minipage}
\hspace{2cm}
%
% ---- fig2.6b ------
  \begin{minipage}{4cm*\real{0.7}} \label{fig:2.6b}
    \centering  Fig. 2.6b
    \begin{tikzpicture}[scale=0.7]
      \draw[style=help lines] grid (4,5);
      % axes
      \draw[thick] (0,0) -- (0,5); \draw[thick] (0,0) -- (4,0);
      % ticks
      { 
       \pgftransformxshift{0.5cm} 
       \pgftransformyshift{-0.55cm}
        \foreach \x in {0,1,2,3} \draw (\x,0) node[anchor=base] {$\x$};
      }
      % steps
      \draw[\Red,ultra thick] (1,0) -- (1,1) -- (2,1) -- (2,2) -- (3,2) --
      (3,4) -- (4,4);
      % line
      \draw[\Black, ultra thick] (0,1) -- (4,5);
    \end{tikzpicture}
  \end{minipage}

\newpage

\begin{minipage}{11cm}
  % ---- fig2.7a ------
\begin{minipage}[b]{5cm*\real{0.7}} \label{fig:2.7a}
\centering Fig. 2.7a 
\tikzstyle{help lines}=[gray,very thin]
\begin{tikzpicture}[scale=0.7]
 \draw[style=help lines]  grid (5,6);
 % axes 
 \draw[thick] (0,0) -- (0,6); 
 \draw[thick] (0,0) -- (5,0);
 % ticks
 {
 \pgftransformxshift{0.5cm}
 \pgftransformyshift{-0.55cm}
  \foreach \x in {0,1,2,3} \draw (\x,0) node[anchor=base] {$\x$}; 
  \draw (4,0) node[anchor=base] {$\dots$};
 }
 % steps 
 \draw[\Red,ultra thick] (2,0) -- (2,2) -- (3,2) -- (3,4) -- (4,4) -- (4,5) -- (5,5);
 % line
  \draw[\Black, ultra thick] (0,1) -- (5,6);  
\end{tikzpicture}
\end{minipage}
\hfill
% ---- fig2.7b ------
\begin{minipage}[b]{6cm*\real{0.7}} \label{fig:2.7b}
\centering Fig. 2.7b 
\tikzstyle{help lines}=[gray,very thin]
\begin{tikzpicture}[scale=0.7]
 \draw[style=help lines]  grid (6,7);
 % axes 
 \draw[thick] (0,0) -- (0,7); 
 \draw[thick] (0,0) -- (6,0);
 % ticks
 {
 \pgftransformxshift{0.5cm}
 \pgftransformyshift{-0.55cm}
  \foreach \x in {0,1,2,3,4} \draw (\x,0) node[anchor=base] {$\x$}; 
  \draw (5,0) node[anchor=base] {$\dots$};
 }
 % steps 
 \draw[\Red,ultra thick] (1,0) -- (1,1) -- (2,1) -- (2,2) -- (3,2) -- (3,3) --
 (4,3) -- (4,5) -- (5,5) -- (5,6) -- (6,6);
 % line
  \draw[\Black, ultra thick] (0,1) -- (6,7);  
\end{tikzpicture}
\end{minipage}
\end{minipage}

\begin{minipage}{16.5cm}
  \centering
\begin{minipage}[b]{11cm*\real{0.7}} \label{fig:2.8}
\centering Fig. 2.8 
\tikzstyle{help lines}=[gray,very thin]
\begin{tikzpicture}[scale=0.7]
 \draw[style=help lines]  grid (11,12);
 % axes 
 \draw[thick] (0,0) -- (0,12); 
 \draw[thick] (0,0) -- (11,0);
 % ticks
 {
 \pgftransformxshift{0.5cm}
 \pgftransformyshift{-0.55cm}
  \foreach \x in {0,1,2,3} \draw (\x,0) node[anchor=base] {$\x$}; 
  \foreach \x in {4,5,6,7,8,9} \draw (\x,0) node[anchor=base] {$\dots$}; 
  \draw (10,0) node[anchor=base] {$e$};
 }
{
 \pgftransformxshift{0.5cm}
  \draw (2,0) node[above] {$x^2$};
  \draw (10,10) node[above] {$y^e$};
  \draw[above] (9.5,8.1) node[fill=white,inner sep=1pt] {$xy^{e-2}$};
%   \draw (9,8) node[anchor=south west,fill=white,inner sep=1pt] {$xy^{e-2}$};
 
}
% steps
\draw[\Red,ultra thick] (2,0) -- (2,1) -- (3,1) -- (3,2) -- (4,2) -- (4,3) --
(5,3) -- (5,4) -- (6,4) -- (6,5) -- (7,5) -- (7,6) -- (8,6) -- (8,7) -- (9,7)
-- (9,9) -- (10,9) -- (10,11) -- (11,11);
 % line
  \draw[\Black, ultra thick] (0,1) -- (11,12);  
 \end{tikzpicture}
\end{minipage}
\hfill
% ---- fig2.9 ------
\begin{minipage}[b]{11cm*\real{0.7}} \label{fig:2.9}
\centering Fig. 2.9 
\tikzstyle{help lines}=[gray,very thin]
\begin{tikzpicture}[scale=0.7]
 \draw[style=help lines]  grid (11,12);
 % axes 
 \draw[thick] (0,0) -- (0,12); 
 \draw[thick] (0,0) -- (11,0);
 % ticks
 {
 \pgftransformxshift{0.5cm}
 \pgftransformyshift{-0.55cm}
  \foreach \x in {0,1,2,3,4} \draw (\x,0) node[anchor=base] {$\x$}; 
  \foreach \x in {5,6,7,8,9} \draw (\x,0) node[anchor=base] {$\dots$}; 
  \draw (10,0) node[anchor=base] {$e$};
}
{
 \pgftransformxshift{0.5cm}
  \draw (2,0) node[above] {$x^2$};
  \draw (10,10) node[above] {$y^e$};
 }
 % steps
 \draw[\Red,ultra thick] (2,0) -- (2,1) -- (3,1) -- (3,2) -- (4,2) -- (4,3) --
 (5,3) -- (5,4) -- (6,4) -- (6,5) -- (7,5) -- (7,6) -- (8,6) -- (8,7) -- (9,7)
 -- (9,8) -- (10,8) -- (10,11) -- (11,11);
 % line
  \draw[\Black, ultra thick] (0,1) -- (11,12);  
\end{tikzpicture}
\end{minipage}
\end{minipage}

% ---- fig2.10 ------
\begin{minipage}{13cm*\real{0.7}} \label{fig:2.10}
\centering Fig. 2.10 \\
\tikzstyle{help lines}=[gray,very thin]
\begin{tikzpicture}[scale=0.7]
 \draw[style=help lines]  grid (13,14);
 % axes 
 \draw[thick] (0,0) -- (0,14); 
 \draw[thick] (0,0) -- (13,0);
 % ticks
 {
 \pgftransformxshift{0.5cm}
 \pgftransformyshift{-0.55cm}
  \foreach \x in {0,1} \draw (\x,0) node[anchor=base] {$\x$}; 
 \foreach \x in {2,3,4,5,6,7,8,9,10} \draw (\x,0) node[anchor=base] {$\dots$}; 
 \draw (11,0) node[anchor=base] {$m{-}1$};
 \draw (12,0) node[anchor=base] {$m$};
}
{
 \pgftransformxshift{0.5cm}
 \draw (10,7) node[above=2pt] {$v$};
 \draw (11,10) node[above=2pt] {$u$};
 }

 % steps 
 \draw[\Red,ultra thick,dotted] (5,0) -- (5,1) -- (6,1) -- (6,2) -- (7,2) --
 (7,3) -- (8,3) -- (8,5) -- (9,5) -- (9,6) -- (10,6);
 \draw[\Red,ultra thick] (10,6) -- (10,7) -- (11,7) -- (11,9);
 \draw[\Red,ultra thick,dotted] (11,9) -- (11,10);
 \draw[\Red,ultra thick] (11,10) -- (11,11) -- (12,11) -- (12,13) -- (13,13);
 % line
  \draw[\Black, ultra thick] (0,1) -- (13,14);  
\end{tikzpicture}
\end{minipage}

\newpage

\begin{minipage}{15cm}
  % ---- fig2.11a ------
\begin{minipage}{9cm*\real{0.7}} \label{fig:2.11a}
\centering Fig. 2.11a \\
\tikzstyle{help lines}=[gray,very thin]
\begin{tikzpicture}[scale=0.7]
 \draw[style=help lines]  grid (9,10);
 % axes 
 \draw[thick] (0,0) -- (0,10); 
 \draw[thick] (0,0) -- (9,0);
 % ticks
 {
 \pgftransformxshift{0.5cm}
 \pgftransformyshift{-0.55cm}
  \foreach \x in {0,1,2} \draw[anchor=base] (\x,0) node {$\x$}; 
  \foreach \x in {3,4,5,6} \draw[anchor=base] (\x,0) node {$\cdots$}; 
  \draw[anchor=base] (7,0) node {$m{-}1$};
  \draw[anchor=base] (8,0) node {$m$};
 }
 % steps 
 \draw[\Red,ultra thick] (2,0) -- (2,1) -- (3,1) -- (3,2) -- (4,2) -- (4,3) -- (5,3) -- (5,4) -- (6,4) -- (6,5) -- (7,5) -- (7,7) -- (8,7) -- (8,9) -- (9,9);
 % line
  \draw[\Black, ultra thick] (0,1) -- (9,10);  
\end{tikzpicture}
\end{minipage}
\hfill
%
% ---- fig2.11b ------
\begin{minipage}{9cm*\real{0.7}} \label{fig:2.11b}
\centering Fig. 2.11b \\
\tikzstyle{help lines}=[gray,very thin]
\begin{tikzpicture}[scale=0.7]
 \draw[style=help lines]  grid (9,10);
 % axes 
 \draw[thick] (0,0) -- (0,10); 
 \draw[thick] (0,0) -- (9,0);
 % ticks
 {
 \pgftransformxshift{0.5cm}
 \pgftransformyshift{-0.55cm}
  \foreach \x in {0,1} \draw[anchor=base] (\x,0) node {$\x$}; 
  \foreach \x in {2,3} \draw[anchor=base] (\x,0) node {$\dots$}; 
 \draw[anchor=base] (4,0) node {$n$};
  \foreach \x in {5,6,7} \draw[anchor=base] (\x,0) node {$\dots$}; 
  \draw[anchor=base] (8,0) node {$m$};
 }
 {
 \pgftransformxshift{0.5cm}
  \draw (4,2) node[above=2pt] {$v$};
  \draw (7,6) node[above=2pt] {$u$};
 }
 % steps 
 \draw[\Red,ultra thick] (3,0) -- (3,1) -- (4,1);
 \draw[\Red,ultra thick,dotted] (4,1) -- (4,2);
 \draw[\Red,ultra thick] (4,2) -- (5,2) -- (5,4);
 \draw[\Red,ultra thick,dotted] (5,4)-- (6,4) -- (6,5);
 \draw[\Red,ultra thick] (6,5) -- (7,5) -- (7,7) -- (8,7) -- (8,9) -- (9,9);
 % line
  \draw[\Black, ultra thick] (0,1) -- (9,10);  
\end{tikzpicture}
\end{minipage}
\end{minipage}

%
% ---- fig2.12 ------
\begin{minipage}{21cm*\real{0.7}} \label{fig:2.12}
\centering Fig. 2.12 \\
\tikzstyle{help lines}=[gray,very thin]
\begin{tikzpicture}[scale=0.7]
 \draw[style=help lines]  grid (20,21);
 % axes 
 \draw[thick] (0,0) -- (0,21); 
 \draw[thick] (0,0) -- (20,0);
 % ticks
 {
 \pgftransformxshift{0.5cm}
  \foreach \x in {0,1,2,3,4,5,6,7,8} \draw (\x,0) node[anchor=north] {$\x$}; 
  \draw (9,0) node[below] {$\varepsilon{+}2$};
  \foreach \x in {10,11,12,13,14,15,16,17,18} \draw (\x,0) node[anchor=north]
  {$\x$}; 
\draw (19,0) node[below=2pt] {$m$};
 }
{
 \pgftransformxshift{0.5cm}
\draw (4,1) node[above] {$6$};
\draw (5,3) node[above] {$5$};
\draw (6,4) node[above] {$4$};
\draw (7,5) node[above] {$3$};
\draw (8,6) node[above] {$2$};
\draw (9,7) node[above] {$1$};

\draw (13,12) node[above] {$\mathbf{1}$};
\draw (14,13) node[above] {$\mathbf{2}$};
\draw (15,14) node[above] {$\mathbf{3}$};
\draw (16,15) node[above] {$\mathbf{4}$};
\draw (17,16) node[above] {$\mathbf{5}$};
\draw (18,17) node[above] {$\mathbf{6}$};
}
\draw (11,5) node[above=2pt,fill=white,inner sep=2pt] {$\vartheta'(n-2)$};
\draw (18,14) node[above=2pt,inner sep=2pt,fill=white] {$\psi'(n-1)$};
 % steps 
 \draw[\Red,ultra thick] (4,0) -- (4,1) -- (5,1) -- (5,3) -- (6,3) -- (6,4) --
 (7,4) -- (7,5) -- (8,5) -- (8,6) -- (9,6) -- (9,7) -- (10,7) -- (10,9) --
 (11,9) -- (11,10) -- (12,10);
 \draw[\Red,ultra thick,dotted] (12,10) -- (12,11) -- (13,11);
 \draw[\Red,ultra thick] (13,11) -- (13,13) -- (14,13) -- (14,14) -- (15,14) -- (15,15) -- (16,15) -- (16,16) -- (17,16) -- (17,17) -- (18,17) -- (18,18) -- (19,18) -- (19,20) -- (20,20);
 % line
  \draw[\Black, ultra thick] (0,1) -- (20,21);  
  \draw[\Black, ultra thick,dashed] (0,0) -- (16,16);  
  \draw[\Black, ultra thick,dotted] (1,0) -- (13,12);  
\end{tikzpicture}
\end{minipage}

\newpage

\begin{minipage}{1.0\linewidth}
% ---- fig2.13a ------
\begin{minipage}{9cm*\real{0.7}} \label{fig:2.13a}
\centering Fig. 2.13a\\ 
\tikzstyle{help lines}=[gray,very thin]
\begin{tikzpicture}[scale=0.7]
 \draw[style=help lines]  grid (9,10);
 % axes 
 \draw[thick] (0,0) -- (0,10); 
 \draw[thick] (0,0) -- (9,0);
 % ticks
 {
 \pgftransformxshift{0.5cm}
 \pgftransformyshift{-0.55cm}
  \foreach \x in {0,1,2,3,4,5,7,8} \draw[anchor=base] (\x,0) node {$\x$}; 
  \draw[anchor=base] (6,0) node {$\kappa{+}1$};
 
 }
 {
 \pgftransformxshift{0.5cm}
 \draw (8,8) node[above] {$y^m$};
 }
 % steps 
 \draw[\Red,ultra thick] (5,0) -- (5,1) -- (4,1) -- (4,3) -- (5,3) -- (5,4) --
 (6,4) -- (6,6) -- (7,6) -- (7,7) -- (8,7) -- (8,9) -- (9,9);
 % line
  \draw[\Black, ultra thick] (0,1) -- (9,10);  
\end{tikzpicture}
\end{minipage}
\hfill
% ---- fig2.13b ------
\begin{minipage}{10cm*\real{0.7}} \label{fig:2.13b}
\centering Fig. 2.13b\\ 
\tikzstyle{help lines}=[gray,very thin]
\begin{tikzpicture}[scale=0.7]
 \draw[style=help lines]  grid (10,10);
 % axes 
 \draw[thick] (0,0) -- (0,10); 
 \draw[thick] (0,0) -- (10,0);
 % ticks
 {
 \pgftransformxshift{0.5cm}
 \pgftransformyshift{-0.55cm}
  \foreach \x in {0,1,2,3,4,5,6,8} \draw[anchor=base] (\x,0) node {$\x$}; 
 \draw[anchor=base] (7,0) node {$\kappa{+}1$};
 }
 {
 \pgftransformxshift{0.5cm}
 \draw (9,0) node[above] {$x^m$};
 }
 % steps 
 \draw[\Red,ultra thick] (9,0) -- (9,1)--(6,1)--(6,2)--
(5,2) -- (5,4) -- (6,4) -- (6,6) -- (7,6) -- (7,8) -- (8,8) -- (8,9) -- (9,9);
 % line
  \draw[\Black, ultra thick] (0,1) -- (9,10);  
\end{tikzpicture}
\end{minipage}  
\end{minipage}

\vspace{3cm}
\begin{minipage}{1.0\linewidth}
% ---- fig2.14a ------
\begin{minipage}{9cm*\real{0.7}} \label{fig:2.14a}
\centering Fig. 2.14a\\ 
\tikzstyle{help lines}=[gray,very thin]
\begin{tikzpicture}[scale=0.7]
 \draw[style=help lines]  grid (9,10);
 % axes 
 \draw[thick] (0,0) -- (0,10); 
 \draw[thick] (0,0) -- (9,0);
 % ticks
 {
 \pgftransformxshift{0.5cm}
 \pgftransformyshift{-0.55cm}
  \foreach \x in {0,1,2,3,4,5,6,7} \draw[anchor=base] (\x,0) node {$\x$}; 
  \draw[anchor=base] (8,0) node {$m$};
 }
 % steps 
 \draw[\Red,ultra thick] (8,0)--(8,1)--(4,1); 
 \draw[\Red,ultra thick,dotted] (4,1) -- (4,2) -- (5,2) -- (5,4) -- (6,4) --
 (6,6) -- (7,6) -- (7,7);
 \draw[\Red,ultra thick] (7,7)--(8,7) -- (8,9) -- (9,9);
 % line
  \draw[\Black, ultra thick] (0,1) -- (9,10);  
\end{tikzpicture}
\end{minipage}
\hfill
% ---- fig2.14b ------
\begin{minipage}{9cm*\real{0.7}} \label{fig:2.14b}
\centering Fig. 2.14b\\ 
\tikzstyle{help lines}=[gray,very thin]
\begin{tikzpicture}[scale=0.7]
 \draw[style=help lines]  grid (9,10);
 % axes 
 \draw[thick] (0,0) -- (0,10); 
 \draw[thick] (0,0) -- (9,0);
 % ticks
 {
 \pgftransformxshift{0.5cm}
 \pgftransformyshift{-0.55cm}
  \foreach \x in {0,1,2,3,4,5,6,7} \draw[anchor=base] (\x,0) node {$\x$}; 
 \draw[anchor=base] (8,0) node {$m$};
 }
 % steps
 \draw[\Red,ultra thick,dotted] (4,0) -- (4,1) -- (5,1) -- (5,3) -- (6,3) --
 (6,5) -- (7,5) -- (7,6);
 \draw[\Red,ultra thick] (7,6)--(8,6) -- (8,9) -- (9,9);
 % line
  \draw[\Black, ultra thick] (0,1) -- (9,10);  
\end{tikzpicture}
\end{minipage}
\end{minipage}

\newpage

% ---- fig2.15 ------
\begin{minipage}{21cm*\real{0.7}} \label{fig:2.15}
\centering Fig. 2.15\\ 
\tikzstyle{help lines}=[gray,very thin]
\begin{tikzpicture}[scale=0.7]
 \draw[style=help lines]  grid (21,21);
 % axes 
 \draw[thick] (0,0) -- (0,21); 
 \draw[thick] (0,0) -- (21,0);
 % ticks
 {
 \pgftransformxshift{0.5cm}
 \pgftransformyshift{-0.55cm}
  \foreach \x in {0,1,2,3,4,5,6,7,8,10,11,13,14,15,16,17,19} 
     \draw[anchor=base] (\x,0) node {$\x$}; 
 \draw[anchor=base] (9,0) node[fill=white,rotate=90] 
        {$m_2{+}2$ or $\kappa{+}2$};
% \draw[anchor=base] (12,0) node[fill=white,rotate=90,above=3pt] {$m_1{+}1$};
 \draw[anchor=base] (12,0) node  {$\scriptstyle{m_1{+}1}$};
 \draw[anchor=base] (18,0) node {$m_0$};
 }
 {
 \pgftransformxshift{0.5cm}
 \draw (7,5) node[above] {$2$};
 \draw (8,6) node[above] {$1$};
 \draw (12,11) node[above] {$1$};
 \draw (13,12) node[above] {$2$};
 \draw (12,5) node[above,fill=white,inner sep=0pt] {graph of $\varphi'_2(n-2)$};
 \draw (18,13) node[above,fill=white,inner sep=0pt] {graph of $\varphi'_1(n-1)$};
 \draw (18,16) node[above,fill=white,inner sep=0pt] {$m_0{-}1{-}m_1$};
}
 % steps 
\draw[\Red,ultra thick] (4,0) -- (4,1) -- (5,1) -- (5,3) -- (6,3) -- (6,4) --
(7,4) -- (7,5) -- (8,5) -- (8,6) -- (9,6) -- (9,8) -- (10,8) -- (10,9) --
(11,9) -- (11,10) -- (12,10) -- (12,12) -- (13,12) -- (13,13) -- (14,13) --
(14,14) -- (15,14) -- (15,15) -- (16,15) -- (16,16) -- (17,16) -- (17,17) --
(18,17) -- (18,19) -- (19,19);
 % line
  \draw[\Black, ultra thick] (0,1) -- (20,21); 
  \draw[\Black, ultra thick,dashed] (0,0) -- (17,17);
  \draw[gray, ultra thick] (1,0) -- (11,10); 
  \draw[gray, ultra thick,dashed] (2,0) -- (7,5);
 
\end{tikzpicture}\\
$\#\{\text{monomials between graph of } \varphi'_2(n-1) \text{ and line } 
 y=x-1\}
= \colength(\cI_2)$
\end{minipage}

\vspace{1cm}
% ---- fig2.16 ------
\begin{minipage}{8cm*\real{0.7}} \label{fig:2.16}
\centering Fig. 2.16\\ 
\tikzstyle{help lines}=[gray,very thin]
\begin{tikzpicture}[scale=0.7]
 \draw[style=help lines]  grid (8,9);
 % axes 
 \draw[thick] (0,0) -- (0,9); 
 \draw[thick] (0,0) -- (8,0);
 % ticks
 {
 \pgftransformxshift{0.5cm}
 \draw[anchor=west] (6,0) node[fill=white, rotate=90,inner sep=0pt] 
  {$m_{i+1}{+}1$};
 \draw[anchor=west] (7,0) node[fill=white, rotate=90,inner sep=0pt] {$m_i$};
 }
 % steps 
 \draw[\Red,ultra thick,dotted] (3,1) -- (3,2) -- (4,2) -- (4,3) -- (5,3) -- (5,4) -- (6,4);
 \draw[\Red,ultra thick] (6,4) -- (6,6) -- (7,6) -- (7,8) -- (8,8);
 % line
  \draw[\Black, ultra thick] (0,1) -- (8,9);  
\end{tikzpicture}
\end{minipage}

%%%%%%%%%%%%%%%%%%%%%%%%%%%CHAPTER 3%%%%%%%%%%%%%%%%%%%%%%%%%%%%%%%%%%%%%%%%%%

\chapter{A rough description of ideals invariant under $\Gamma\cdot T(\rho)$}\label{3}

It seems impossible to characterize ideals of colength $d$ in $\cO_{\mP^2}$,
which are invariant under $G:=\Gamma\cdot T(\rho)$. If $\cI$ is an ideal of
type $r$ in the sense of the definition in (2.5.1), however, then one can give
a rough description of the forms $f_0,\cdots, f_r$, which occur in the
decomposition (2.2) in (2.5.1). This description will be used later on for
estimating the so called $\alpha$-grade of $\cI$ which will be defined in the
next chapter.

\section{Assumptions}\label{3.1}
Let $\cI\subset\cO_{\mP^2}$ be an ideal of colength $d$, invariant under
$\Gamma\cdot T(\rho)$, where $\rho_2>0$, as usual. We assume that $\cI$ has the
type $r\ge 0$. Then according to Lemma 2.7 one has a decomposition:
% \begin{equation}\label{Z}
%   \begin{array}{c}
%   \cI=:\cI_0=\ell_0\cI_1(-1)+f_0\cO_{\mP^2}(-m_0)\\
%   \cI_1=\ell_1\cI_2(-1)+f_1\cO_{\mP^2}(-m_1)\\
%   \cdots\cdots\cdots\cdots\cdots\cdots\cdots\\
%   \cI=\ell_i\cI_{i+1}(-1)+f_i\cO_{\mP^2}(-m_i)\\
%   \cdots\cdots\cdots\cdots\cdots\cdots\cdots\\
%   \cI_r=\ell_r\cI_{r+1}(-1)+f_r\cO_{\mP^2}(-m_r).
%   \end{array}
% \end{equation}
%
\begin{equation}\label{Z}
  \begin{aligned}
    \cI=:\cI_0 & =\ell_0\cI_1(-1)+f_0\cO_{\mP^2}(-m_0),\\
    \cI_1 & =\ell_1\cI_2(-1)+f_1\cO_{\mP^2}(-m_1),\\
    \cdots&\cdots\cdots\cdots\cdots\cdots\cdots\cdots\cdots\cdots\\
    \cI_i&=\ell_i\cI_{i+1}(-1)+f_i\cO_{\mP^2}(-m_i),\\
    \cdots&\cdots\cdots\cdots\cdots\cdots\cdots\cdots\cdots\cdots\\
    \cI_r&=\ell_r\cI_{r+1}(-1)+f_r\cO_{\mP^2}(-m_r).
  \end{aligned}
\tag{Z}
\end{equation}
Using the notations of (loc.cit.) the ideal $\cI_{r+1}:=\cK$ has the colength
$c$ and regularity $\kappa$. If $\ell_i=x$, then
\[
  H^0(\cI_i(n))=xH^0(\cI_{i+1}(n-1))+f_ik[y,z]_{n-m_i}
\]
and if $\ell_i=y$, then
\[
  H^0(\cI_i(n))=yH^0(\cI_{i+1}(n-1))+f_ik[x,z]_{n-m_i}
\]
(cf. Lemma 2.6). From $(Z)$ follows
% \[
%   \begin{array}{l}
%   H^0(\cI_1(m_i+i-1))=\ell_1H^0(\cI_2(m_i+i-2))\\[2mm]
%   H^0(\cI_2(m_i+i-2))=\ell_2H^0(\cI_3(m_i+i-3))\\[2mm]
%   \cdots\cdots\cdots\cdots\cdots\cdots\cdots\\[2mm]
%   H^0(\cI_{i-1}(m_i+1))=\ell_{i-1}H^0(\cI_i(m_i))\\[2mm]
%   H^0(\cI_i(m_i))=\ell_iH^0(\cI_{i+1}(m_i-1)+\langle f_i\rangle,
%   \end{array}
% \]
%
\begin{align*}
  H^0(\cI_1(m_i+i-1))&=\ell_1H^0(\cI_2(m_i+i-2)),\\
  H^0(\cI_2(m_i+i-2))& =\ell_2H^0(\cI_3(m_i+i-3)),\\
  \cdots\cdots\cdots\cdots&\cdots\cdots\cdots\cdots\cdots\cdots\cdots\\
  H^0(\cI_{i-1}(m_i+1))& =\ell_{i-1}H^0(\cI_i(m_i)),\\
  H^0(\cI_i(m_i)) & =\ell_iH^0(\cI_{i+1}(m_i-1))+\langle f_i\rangle,
\end{align*}

for by Corollary 2.4  $m_i+2<m_{i-1},i=1,\ldots,r$. If one starts with
$H^0(\cI_1(m_i+i-2))$, then one obtains a similar system of equations, whose
last line is
\[
  H^0(\cI_i(m_i-1))=\ell_iH^0(\cI_{i+1}(m_i-2)).
\]
\begin{conclusion}\label{1}
  If $2\le i\le r$ (if $1\le i\le r$, respectively), then
  $H^0(\cI_1(m_i+i-1))=\ell_1\cdots\ell_{i-1}H^0(\cI_i(m_i))\;
  (H^0(\cI_1(m_i+i-2))=\ell_1\cdots\ell_iH^0(\cI_{i+1}(m_i-2))$,
  respectively).\hfill $\Box$
\end{conclusion}

\section{Notations}\label{3.2}

We orientate ourselves by Figure 3.1, which shows the initial monomials of
$H^0(\cI(m_0))$. The set of all monomials in $S_{m_0}$ with $z$-degree $\ge
m_0-(c+r)$ (with $z$-degree $\le m_0-(c+r+1)$, respectively) is called the
left domain and is denoted by $\cL\cB$ (is called right domain and is denoted
by $\cR\cB$, respectively). The monomials in $S_{c-1}-$ {\it in}
$(H^0(\cK(c-1)))$ form a basis of $S_{c-1}/H^0(\cK(c-1))$. If we put
$\ell=\ell_0\cdots\ell_r$, then $\ell[S_{c-1}/H^0(\cK(c-1))]$ has a basis
consisting of the $c$ monomials in $\ell S_{c-1}-\ell{\cdot}$ {\it in}
$(H^0(\cK(c-1)))$.

The initial monomial of $\ell_0\cdots\ell_{i-1}f_i\cdot z^{m_0-m_i-i}$ is
$M_i^{\rm up}:=x^{i-\iota(i)}y^{m_i+\iota(i)}z^{m_0-m_i-i}$ or $M_i^{\rm
  down}:=x^{m_i+i-\iota(i)}y^{\iota(i)}z^{m_0-m_i-i}$, if $\ell_i=x$ or if
$\ell_i=y$, respectively. Here we have put, if $1\le i\le r+1, \iota(i):=$
number of indices $0\le j<i$ such that $\ell_j=y$. We also put $\iota(0)=0$,
so the formulas give the right result for $i=0$, too.

For instance, in Figure 3.1 we have
\[
r=5,\ \ell_0=x,\ \ell_1=y,\ \ell_2=\ell_3=x,\ \ell_4=\ell_5=y,
\]
 and therefore
\[
\iota(1)=0,\ \iota(2)=\iota(3)=\iota(4)=1,\ \iota(5)=2,\ \iota(6)=3.
\] 
(N.B. $\ell_0\cdots\ell_{i-1}=x^{i-\iota(i)}y^{\iota(i)}$, if $0<i\le r+1$.)

As always we assume $\rho_2>0$. If $M_i^{\down}$ occurs, then
$\cI_i=y\cI_{i+1}(-1)+f_i\cO_{\mP^2}(-m_i)$. If $M_i^{\up}$ occurs, then
$\cI_i=x\cI_{i+1}(-1)+f_i\cO_{\mP^2}(-m_i)$. By Lemma 2.8  $m_i\ge
(c+2)+m_r+\cdots+m_{i+1}$, if $r\ge 1$ $(m_0\ge c+2$, if $r=0$).The colength of
$\cI_{i+1}$ equals $c+m_r+\cdots+m_{i+1}$, and therefore $R_n\subset
H^0(\cI_{i+1}(n))$ if $n\ge m_i-2\ge c+m_r+\cdots+m_{i+1}$ in the case $r\ge 1
(R_n\subset H^0(\cK(n))$ if $n\ge c$ in the case $r=0$). This follows from
Remark 2 in Appendix C. The next generating element $f_i$ has the initial
monomial $M_i^{\up}$ or $M_i^{\down}$. 

\begin{conclusion}\label{2}
  Suppose $\rho_2>0$ and $\rho_1$ is arbitrary.Then the columns of total 
 degree $m_i+i-2$ and $m_i+i-1$ (in the variables $x$ and $y$ ),which 
 occur in {\it in}$(H^0(\cI(m_0)))$ , also occur in $H^0(\cI(m_0))$.
Thus $xM_i^{\up}$ or $yM_i^{\down}$, when it exists, is contained in $\cI$.
\end{conclusion}
\begin{proof}
  This follows from 
   , Conclusion 3.1, the decompositions
  $(Z)$ and the foregoing discussion.
\end{proof}
  As always we assume $\rho_2>0$. We have to distinguish between two cases 
(cf. 2.4.1 Auxiliary Lemma 1 and Auxiliary Lemma 2):

%
% \begin{enumerate}[{Main Case} I]
% \item ${\rho_0} > 0 $ and ${\rho_1} < 0$
% \item ${\rho_0} < 0$ and ${\rho_1} >0$
% \end{enumerate}
\noindent
\textbf{Main Case \phantom{I}I}: ${\rho_0} > 0 $ and ${\rho_1} < 0$ \\[4mm]
\textbf{Main Case II}: ${\rho_0} < 0$ and ${\rho_1} >0$

\section{Description of $f_0$ in Case I}\label{3.3}

If the initial monomial of $f_0$ equals $x^{m_0}$, then $f_0=x^{m_0}$ and
$\ell_0=y$. Therefore we may assume that $f_0$ is a proper semi-invariant with
initial monomial $y^{m_0}$. Then $\cI=\cI_0=x\cI_1(-1)+f_0\cO_{\mP^2}(-m_0)$,
and we write $f_0=y^{m_0}+z^{\mu} x {\cdot} G$, $\mu\in\N$ maximal, $G$ a
$T(\rho)$-semi-invariant (cf. 2.4.3). If $G^0$ is the initial monomial of $G$,
then $N : \;=z^{\mu}\cdot x \cdot G^0$ is called the \emph{vice-monomial} of $f_0$, which by
assumption is not equal to zero .

The inclusion (2.1) of (2.4.2) reads
\begin{equation}\label{1}
  \langle x,y\rangle [\mu G+z\;\partial G/\partial z]\subset H^0(\cI_1(m_0-\mu))
\end{equation}

As for the initial monomial $G^0$ of $G$, one has
\begin{equation}\label{2}
  \langle x,y\rangle G^0\subset {\it in} (H^0(\cI_1(m_0-\mu)))
=H^0({\it in} (\cI_1(m_0-\mu)))
\end{equation}

where ${\it in}$ denotes the subspace or the initial ideal, respectively,
generated by the initial monomials. (Equality follows, for example, from ${\it
  in}(\cI)=\lim\limits_{\lambda\to 0}\sigma(\lambda)\cI$ and [G2], Lemma 3 and
4.) Without restriction we may assume that $G^0\notin \inn(
H^0(\cI_1(m_0-\mu-1)))$, because we can reduce $f_0$ modulo $x
H^0(\cI_1(m_0-1))$.

\noindent
(\text{N.B.}: The same statements are analogously valid in the Case II,
that means, if $\rho_1>0$, $\cI=y\cI_1(-1)+f_0\cO_{\mP^2}(-m_0),
f_0=x^{m_0}+z^{\mu}y G$, etc.)

It follows that one of the following cases occurs, where $1\le i\le r$:\\[3mm]
$1^{\circ}\; N=N_i^{\down}:=(z/x)M_i^{\down}=x^{m_i+i-\iota(i)-1}y^{\iota(i)}z^{m_0-m_i-i+1}$\\[3mm]
$2^{\circ}\; N=N_i^{\up}:=(z/y)M_i^{\up}=x^{i-\iota(i)}y^{m_i+\iota(i)-1}z^{m_0-m_i+i+1}$\\[3mm]
$3^{\circ}\; \langle x,y\rangle N\subset\ell \inn(H^0(\cK(m_0-r)))$, where $\ell:=\ell_0\cdots\ell_r$ and $N$ is a monomial in
\[
  \cL:=\ell S_{m_0-(r+1)}-\ell \inn(H^0(\cK(m_0-r-1))).
\]
 \textbf{Notabene}: One has $\cL= [\ell S_{c-1}-\ell \cdot
 \inn(H^0(\cK(c-1)))]{\cdot}z^{m_0-r-c}$, because of
%\begin{equation*}
\[ 
   \bigoplus^{m_0-r-1}_{\nu=c}z^{m_0-r-1-\nu}R_{\nu}\subset H^0(\cK(m_0-r-1)).
\]
% \end{equation*}
As $\mu$ is maximal and $\rho_2>0, G^0$ is not divisible by $z$. In the cases
$1^{\circ}$ and $2^{\circ}$ we therefore have $\mu=m_0-m_i-i+1$. The case
$3^{\circ}$ will be treated in (3.3.6), and we assume the case $1^{\circ}$ or
$2^{\circ}$. Then (3.1) can be written as
\[
\langle x,y\rangle [\mu G+z\;\partial G/\partial z]\subset
H^0(\cI_1(m_i+i-1)).
\]
If we put $h:=\ell_1\cdots\ell_{i-1}$, then $h=x^{i-1-\iota(i)}y^{\iota(i)}$,
because of $\ell_0=x$ one has $\iota(i)$ = number of indices j such that $
0<j<i$ and $\ell_j=y  $. If $i=1$, then $h:=1$. By Conclusion 1 it follows from
(3.1), that $G$ is divisible by $h$, that means, one has $G=hg,g\in S_{m_i-1}$.
Therefore we can write
\[
  f_0=y^{m_0}+z^{\mu}xhg
\]
and (3.1) is equivalent to:
\begin{equation}\label{3}
  \langle x,y\rangle [\mu g+z\;\partial g/\partial z]\subset H^0(\cI_i(m_i)).
\end{equation}

\subsection{}\label{3.3.1}
We assume the case $1^{\circ}$. As $N_i^{\down}$ occurs ,
$\cI_i=y\cI_{i+1}(-1)+x^{m_i}\cO_{\mP^2}(-m_i)$. If $g^0$ is the initial
monomial of the form $g$ (cf. the inclusion (3.3)), then
$z^{\mu}xG^0=hg^0xz^{\mu}=\\x^{i-\iota(i)}y^{\iota(i)}z^{\mu}g^0=N_i^{\down}$.
Because of $\mu=m_0-m_i-i+1$ it follows that $g^0=x^{m_i-1}$. Representing the
initial monomials of $H^0(\cI_i(m_i))$ in the same way as the initial
monomials of $H^0(\cI_0(m_0))$, one sees that southwest of $N_i^{\down}$ there
is no further monomial which occurs in $f_0$.

\begin{conclusion}\label{3}
  Let be $\rho_1<0$. If the vice-monomial of $f_0$ equals $N_i^{\down}$, then
  $f_0=y^{m_0} + \alpha N_i^{\down}$, where $\alpha\in k$. As
  $f_i=M_i^{\down}$ is a monomial, by Conclusion 3.2 it follows that all
  monomials which have the same $z$-degree as $M_i$ and are elements of 
$\inn(H^0(\cI(m_0)))$ also are elements of $H^0(\cI(m_0))$. Therefore we
  have $(x,y)y^{m_0}\subset\cI$.\hfill $\Box$
\end{conclusion}

\subsection{}\label{3.3.2}
We assume the case $2^{\circ}$. Then $y^{m_0}X^{\nu\rho}=N_i^{\up}$, where
$\nu>0$ and $1\le i\le r$. This statement is equivalent to the equations:
\begin{equation}\label{4}
  \nu\rho_0=i-\iota(i),\ \nu\rho_1=m_i+\iota(i)-1-m_0,\ \nu\rho_2=m_0-m_i-i+1.
\end{equation}
\begin{lemma}\label{9}
  If one has $f_0=y^{m_0}+\alpha N_i^{\up}+\cdots$, where $1\le i\le r$ and
  $\alpha\in k-(0)$, then $\rho_2>m_{i+1}=\reg(\cI_{i+1})$, where we put
  $\cI_{r+1}:=\cK$ and $m_{r+1}:=\kappa=\reg (\cK)$. Especially $\cI_j$ is a
  monomial ideal for all $j>i$.
\end{lemma}
\begin{proof}
  As we have assumed $\rho_1<0$, by (2.4.1 Auxiliary Lemma 2) one has
  $\rho_0>0$. Therefore $\nu\le i$. On the other hand
  $\rho_2=(m_0-m_i-i+1)/\nu\ge (m_0-m_i-i+1)/i$; thus it suffices to show
  $(m_0-m_i-i+1)/i>m_{i+1}$, i.e.
\begin{equation}\label{5}
  m_0>m_i+i\cdot m_{i+1}+(i-1).
\end{equation}
We start with $i=r$, in which case one has to show $m_0>m_r+r\kappa+(r-1)$.
This is valid if $r=0$, so we may assume $r\ge 1$. By Lemma 2.8 one has
$m_0\ge c+2+m_r+\cdots+m_1$, thus it suffices to show
$c+2+m_r+\cdots+m_1 > m_r +r\kappa+(r-1)$. If $r=1$, 
then this inequality is
equivalent to $c+2>\kappa$, and this is true. Therefore one can assume without
restriction that $r>1$, and has to show
\[
  c+2+m_{r-1}+\cdots+m_1>r\kappa+(r-1).
\]
Because of $\kappa\le c$ it suffices
to show $2+m_{r-1}+\cdots+m_1>(r-1)(\kappa+1)$ which is true because of
$\kappa\le c<m_{r-1}< \cdots <m_1$ (cf. Lemma 2.4 and Corollary 2.4).\\
Now we assume $1\le i<r$, in which case it suffices to show: 
\begin{multline*}
  c+2+m_r+\cdots+m_1>m_i+i\cdot m_{i+1}+(i-1)\\ 
\iff \qquad
  (c+2)+(m_r+\cdots+m_{i+2})+(m_{i-1}+\cdots+m_1)>(i-1)m_{i+1}+(i-1).
\end{multline*}
On the left side of this inequality the second summand (the third summand,
respectively) does not occur, if $i+1=r$ (if $i=1$, respectively). If $i=1$
and $r=2$, the inequality reads $c+2>0$. If $i=1$ and $r\ge 3$, the inequality
reduces to $(c+2)+m_r+\cdots+m_3>0$. Thus the case $i\ge 2$ remains, and it
suffices to show $m_{i-1}+\cdots +m_1>(i-1)(m_{i+1}+1)$, which is true because
of $m_{i+1}<m_i<\cdots <m_1$ (loc.cit.).
\end{proof}

\subsection{Description of the ideal $\cI_i$} \label{3.3.3}

The assumptions are the same as in Lemma 3.1, but we slightly change the
notations and write $\cI=\cI_i,\cK=\cI_{i+1}, m=m_i,e=\reg (\cK)$. (Thus
$e=m_{i+1}$ or $e=\kappa$.) We have the following situation:
$\cI=x\cK(-1)+f\cO_{\mP^2}(-m)$ is $\Gamma\cdot T(\rho)$-invariant, $\cK$ is
monomial, $\rho_1<0,\rho_2>e=\reg (\cK)$. From the results of (2.4.2) and
 Lemma 2.6 it follows that $f=y^m+z^{m-e}xg$, where $g\in S_{e-1}$ and
\begin{equation}\label{6}
 \langle x,y\rangle [(m-e)g+z\, \partial g/\partial z]\subset H^0(\cK(e))
\end{equation}
where $f\in H^0(\cK(m))$ is determined only modulo $x H^0(\cK(m-1) $ and we can assume that $f$ is a
$T(\rho)$-semi-invariant. Then $g$ is a $T(\rho)$-semi-invariant, too. Thus
we can write $g=N(1+a_1X^{\rho}+\cdots)$, where $N\in S_{e-1}$ is a
monomial. If we assume $a_1\neq 0$ for example, then $NX^{\rho}\in S_{e-1}$
would have a $z$-degree $\ge\rho_2$. As $\rho_2>e$ (cf. Lemma 3.1 ), this is
impossible.

\begin{conclusion}\label{4}
  Under the assumptions mentioned above, the form $g$ in (3.6) equals a monomial
  $N \in S_{e-1}-H^0(\cK(e-1))$ such that $\langle x,y\rangle N\subset
  H^0(\cK(e))$. It follows that $(x,y)y^m\subset\cI$.\hfill $\Box$
\end{conclusion}

\subsection{}\label{3.3.4}

We go back to the notations of Lemma 3.1, i.e. one has $N=N_i^{\up}$. As in
the\\  %%%  FEhler bad
case $1^{\circ}$ one has $\mu=m_0-m_i-i + 1$. As $y^{m_0}$ is the
initial monomial of $f_0 , \ell_0$ equals $x$ and $h:=\ell_1\cdots\ell_{i-1}$
equals $x^{i-1-\iota(i)}y^{\iota(i)}$ as before. If one puts $\overline{g}=\mu
g+z\partial g/\partial z$, then (3.3) can be written as
\begin{equation}\label{7}
  \langle x,y\rangle \overline{g}\subset xH^0(\cI_{i+1}(m_i-1))+\langle f_i\rangle\end{equation}
where $f_i$ has the initial monomial $y^{m_i} $; thus
\begin{equation}\label{8}
  \overline{g}\in H^0(\cI_{i+1}(m_i-1)).
\end{equation}
The initial monomial of $\overline{g}$ equals the initial monomial of $g$; by
construction this is equal to $y^{m_i-1}$. (Proof: The initial monomial of
$xhgz^{\mu}=xGz^{\mu}$ is equal to $N_i^{\up}$ by assumption. Thus the initial
monomial $g^0$ of $g$ fulfils the equation
\[
  xhg^0z^{\mu}=x^{i-\iota(i)}y^{m_i+\iota(i)-1}z^{m_0-m_i-i+1}.
\]
As $xh=x^{i-\iota(i)}y^{\iota(i)}$ and $\mu=m_0-m_i-i+1$, it follows that
$g^0=y^{m_i-1}$.) From (3.7) it follows that
\begin{equation}\label{9}
  y\overline{g}=xF+\alpha f_i,\; F\in H^0(\cI_{i+1}(m_i-1)),\; \alpha \in k^*.
\end{equation}
Now we can write $f_i=y^{m_i}+z^{m_i-m_{i+1}}xu$ (cf.~the last sentence in
(2.4.2)), where $u$ is either a monomial in
$S_{m_{i+1}-1}-H^0(\cI_{i+1}(m_{i+1}-1))$ (cf. Conclusion 2.4), or $u=0$.

We consider the first case. Then it follows that $z^{m_i - m_{i+1}}xu=\beta
y^{m_i}X^{\nu\rho}$, where $\beta\in k^*$ and $\nu\ge 1$. As $f_i$ can be
reduced modulo $xH^0(\cI_{i+1}(m_i-1))$, we have $z^{m_i-m_{i+1}}xu\notin
xH^0(\cI_{i+1}(m_i-1))$ without restriction. From (3.9) it follows that in
$\overline{g}$, except for $y^{m_i-1}$, the monomial
$\overline{u}:=z^{m_i-m_{i+1}}xu/y$ occurs.

\indent Suppose, there is another monomial $v$ in $\overline{g}$. Then one
would have $yv\in\\ xH^0(\cI_{i+1}(m_i-1))$, and from (3.8) it would follow that
$v$ is an element of the monomial subspace $H^0(\cI_{i+1}(m_i-1))$. Figure 3.2
shows $H^0(\cI_i(m_i))=xH^0(\cI_{i+1}(m_i-1))+\langle f_i\rangle$ marked 
with \textbf{---} in black 
and $H^0(\cI_{i+1}(m_i))$ marked with \textbf{---} in blue. Suppose that $v\in
H^0(\cI_i(m_i-1))=xH^0(\cI_{i+1}(m_i-2))$. By construction $v$ occurs in $g$,
therefore $xhv=\ell_0\cdots\ell_i\; v$ occurs in $xG$ and $z^{\mu}xhv$ occurs
in $z^{\mu}xG$. On the other hand, $z^{\mu}xhv\in z^{\mu} x\ell_1\cdots\
\ell_{i-1}xH^0(\cI_{i+1}(m_i-2))=z^{\mu}xH^0(\cI_1(m_i+i-2))\subset
xH^0(\cI_1(m_0-1))$, because $\ell_i=x$ and so Conclusion 3.1 can be applied. As
$f_0$ can be reduced modulo $xH^0(\cI_1(m_0-1))$, one can assume without
restriction, that $v$ does not occur in $g$ and therefore does not occur in
the inclusions (3.7) or (3.8), which have to be fulfilled. Thus we can assume
without restriction that $v\in H^0(\cI_{i+1}(m_i-1))-H^0(\cI_i(m_i-1))$. From
$yv\in xH^0(\cI_{i+1}(m_i-1))$ it follows that $v$ is equal to one of the
monomials in Figure 3.2, which are denoted by ?. Therefore the $z$-degree of
$v$ is $\ge m_i-m_{i+1}$.

By construction, the $z$-degree of $\overline{u}$ is $\ge m_i-m_{i+1}$, too.
As $\overline{u}$ and $v$ occur in the semi-invariant $\overline{g}$, both
monomials differ by a factor of the form $X^{\nu\rho}$. As $\rho_2>m_{i+1}$
(Lemma 3.1), it follows that $\nu=0$, i.e., $\overline{u}$ and $ v $ differ by a
constant factor. Therefore one has 
\[
  \overline{g}=\beta y^{m_i-1}+\gamma z^{m_i-m_{i+1}}xu/y\;\; ; \beta,\gamma\in k^*.
\]
We have to describe the position of $xu$ more exactly. From $xu\notin
xH^0(\cI_{i+1}(m_{i+1}-1))$ but $\langle x,y\rangle xu\in
xH^0(\cI_{i+1}(m_{i+1}))$ and from the $\Gamma$-invariance of $\cI_{i+1}$ it
follows that $z^{m_0-m_{i+1}}xu$ equals $N_j^{\down}$ or $N_j^{\up}$, where
$j$ is an index $r\ge j>i$, or equals a monomial $L\in\cL$ (cf. 3.3).

\indent Suppose $  z^{m_0-m_{i+1}}\cdot xu = N_j^{\down}$. Then
$\overline{u}=z^{m_i-m_{i+1}}xu/y$ is southwest of $N_j^{\down}/z^{m_0-m_i}$
and does not occur in the monomial subspace $H^0(\cI_{i+1}(m_i-1))$, which
contradicts (3.8). Finally we note that the monomials of $\overline{g}$ agree
with the monomials of $g$.

\begin{conclusion}\label{5}
  Assume that $f_0$ has the vice-monomial $N_i^{\up}$ and that $f_i$ is not a
  monomial. Then (up to a constant factor) $f_i=M_i^{\up}+\alpha N_j^{\up}$,
  where $1\le i<j\le r$ and $\alpha\in k^*$, or $f_i=M_i^{\up}+\alpha L$,
  where $L\in\cL$ is a monomial such $(x,y)L\subset \ell\cK(-r-1)$, and
  $\alpha\in k^*$. 

 Then it follows that $f_0=y^{m_0}+\beta N_j^{\up}+\gamma
  N_j^{\up}\cdot (z/y)$ or $f_0=y^{m_0}+\beta N_i^{\up}+\gamma L\cdot (z/y)$,
  respectively. Here $\beta$ and $\gamma$ are elements of $k^*$. Finally, all
  monomials, which occur in $x\, f_0$ or in $y^2f_0$ also occur in $\cI$.
\end{conclusion}

\begin{proof}
  The shape of $f_0$ and of $f_i$ results from the forgoing argumentation.
  Conclusion 3.4 gives $(x,y)L\subset\ell\cK(-r-1)$. The statements concerning
  the monomials in $xf_0$ follow from Conclusion 3.2. By Lemma 3.1, $\cI_j$ is
  monomial, thus $yN_j^{\up}=z\, M_j^{\up}\in\cI$ and therefore $y\,
  M_i^{\up}\in\cI$. Furthermore $y^2N_j^{\up}\cdot (z/y)=y\, N_j^{\up}z\in\cI$
  and as well $y^2L\cdot (z/y)=yLz\in\cI$. From this the assertion concerning
  the monomials of $y^2f_0$ follows.
\end{proof}

Now we come to the case that $f_i=y^{m_i}$. The above reasoning had shown that
$y^{m_i-1}$ occurs in $\overline{g}$. Suppose that another monomial $v$
occurs in $\overline{g}$. The same argumentation as before shows that $v$
equals one of the monomials in Figure 3.2 denoted by ?. Thus $v$ is equal to
$N_j^{\up}$ or $N_j^{\down}$ for an index $i<j\le r$, or is equal to a
monomial $L\in\cL$, such that $(x,y)L\subset\ell\cK(-r-1)$. Furthermore,
there can be only one such monomial $v$.

\begin{conclusion}\label{6}
  If $f_0$ has the vice-monomial $N_i^{\up}$ and if $f_i=y^{m_i}$, then
  $f_0=y^{m_0}+\alpha N_i^{\up}+\beta N_j^{\up}$ or $f_0=y^{m_0}+\alpha
  N_i^{\up}+\beta N_j^{\down}$, where $1\le i<j\le r$, or $f_0=y^{m_0}+\alpha
  N_i^{\up}+\beta L$, where $L\in\cL$ is a monomial such that
  $(x,y)L\subset\ell\cK(-r-1)$. All monomials occurring in $xf_0$ or $yf_0$
  also occur in $\cI$.
\end{conclusion}

\begin{proof}
  The statements concerning the shape of $f_0$ follow from the foregoing
  discussion. As $\cI_{i+1}$ is monomial by Lemma 3.1 and as $f_i$ is a
  monomial, $\cI_i$ is monomial and $(x,y)L, (x,y)N_i^{\up}$ and
  $(x,y)N_j^{\down}$ are contained in $\cI$.
\end{proof}

\subsection{}\label{3.3.5}
Suppose in the forms $f_0$ and $f_j$, where $1\le j\le r$ , as in Conclusion 3.5
or Conclusion 3.6 there are three monomials with coefficients different from
zero. We call $f_0$ (and $f_j$, respectively) a $ \emph{trinomial}$ and we then have
$f_0=y^{m_0}+\alpha N_i^{\up}+\beta E_0$ and $f_j=M_j^{\up}+\gamma
N_k^{\up}+\delta E_j$, where the ``final monomials'' $E_0$ and $E_j$ have the
shape described in Conclusion 3.5 and Conclusion 3.6, and where $\alpha,
\beta,\gamma,\delta\in k^*$. If $i>k$, then $m_i\le m_{k+1}<\rho_2$ (Lemma 3.1).
As we are in the case $\rho_0>0,\rho_2>0$, it follows that $|\rho_1|>m_i$, and
as $N_i^{\up}$ and $E_0$ both occur in the semi-invariant $f_0$, one has
$E_0=N_i^{\up}\cdot X^{\nu\rho}$, where $\nu\ge 1$. Looking at Figure 3.2, one
sees that then $E_0$ cannot be an element of $S_{m_0}$, contradiction. In the
same way it follows that $i<k$ is not possible, and thus we have $i=k$. As
$f_0$ and $f_j$ are $T(\rho)$-semi-invariants, it follows that $y^{m_0}\cdot
X^{\nu\rho}=M_j^{\up}$, where $\nu\ge 1$. We conclude from this that
\[
 \nu\rho_0=j-\iota(j),\ \nu\rho_1=m_j+\iota(j)-m_0,\ \nu\rho_2=m_0-m_j-j,
\]
 where $1\le j\le r$. We want to show that this implies $\rho_2>m_{j+1}$.
 Similarly as in the proof of Lemma 3.1 it suffices to show
 $m_0>m_j+jm_{j+1}+j$. We start with the case $j=r$. Then again
 $m_{r+1}:=\kappa$ and one has to show $m_0>m_r+r\kappa +r$.  The case $r=0$
 cannot occur. If $r=1$, the inequality reads $m_0>m_1+\kappa+1$, and because
 of $m_0\geq c+2+m_1$ (Lemma 2.8) and $\kappa\le c$ this is right.  Therefore
 we can assume $r\ge 2$.

 \noindent Because of (loc.cit.) it suffices to show
 $c+2+m_r+\cdots+m_1>m_r+r\kappa +r$. As $\kappa\le c$ it suffices to show
 $2+m_{r-1}+\cdots+m_1>(r-1)\kappa+r\; \iff\; 1+m_{r-1}+\cdots +
 m_1>(r-1)(\kappa+1)$. Because of $\kappa\le c<m_r<\cdots<m_1$ this is true
 (cf.Lemma 2.4 and Corollary 2.4).

\noindent We now assume $1\le j<r$, in which case it suffices to show:
\[
  c+2+m_r+\cdots +m_1>m_j+jm_{j+1}+j
\]
If $j=1$ this inequality reads $c+2+m_r+\cdots+m_1>m_1+m_2+1$, and this is
true, because $r\ge 2$. Thus we can assume $2\le j<r$ and the inequality which
has to be shown is equivalent to
\[
  (c+1)+(m_r+\cdots+m_{j+2})+(m_{j-1}+\cdots+m_1)>(j-1)(m_{j+1}+1)
\]
Because of $m_{j+1}<m_j<\cdots<m_1$ (loc.cit.), this is true.

\indent We thus have proved that from $y^{m_0}X^{\nu\rho}=M_j^{\up}$
\begin{equation}\label{9}
  \rho_2\; > \; m_{j+1}
\end{equation}
follows. As $k>j$ by definition, we have $m_k\le m_{j+1}<\rho_2$ and the same
argumentation as before carried out with $N_i^{\up}$ and $E_0$ shows that
$N_k^{\up}$ and $E_j$ cannot simultaneously occur in $f_j$.

\begin{conclusion}\label{7}
  Among the forms $f_i,\; 0\le i\le r$, which have the shape described in
  Conclusion 3.5 and Conclusion 3.6, there can only occur at most one trinomial.
\end{conclusion}

\subsection{}\label{3.3.6}

We now consider the case $3^{\circ}$. The following argumentation is first of
all independent of the sign of $\rho_1$. We write $f_0=f^0+z^ug$, where
$f^0=y^{m_0}$ if $\ell_0=x$ and $f^0=x^{m_0}$, if $\ell_0=y$. If we choose
$\mu$ maximal then we have $\mu\ge m_0-(\kappa+r)$. For as
$R_{\kappa}\subset\inn(H^0(\cK(\kappa)))$, the initial monomial of $g$ has a
$z$-degree $>m_0-(\kappa+r+1)$, i.e., the initial monomial occurs in a column
of total degree in $x$ and $y$ smaller or equal $ \kappa+r$.

\indent From the $\Gamma$-invariance of $f_0$ modulo $\ell_0H^0(\cI_1(m_0-1))$
it follows that $\langle x,y\rangle \partial f/\partial z=\langle x,y\rangle\;
[\mu z^{\mu-1}g+z^{\mu}\partial g/\partial z]\subset\ell_0H^0(\cI_1(m_0-1))$.
From the decompositions in $(Z)$ (cf. 3.1) we conclude that
$H^0(\cI_i(n))=\ell_iH^0(\cI_{i+1}(n-1))$ if $n<m_i$. Now 
\[
   m_0-\mu<\kappa +r+1<m_r+r<\cdots<m_1+1<m_0
\] (cf. Corollary 2.4) and thus
% \[
%   \begin{array}{l}
%   H^0(\cI_1(m_0-\mu))=\ell_1H^0(\cI_2(m_0-\mu-1))\\
%   H^0(\cI_2(m_0-\mu-1))=\ell_2H^0(\cI_3(m_0-\mu-2))\\
%   \dotfill\\
%   H^0(\cI_{r-1}(m_0-\mu-r+2))=\ell_{r-1}H^0(\cI_r(m_0-\mu-r+1))\\
%   H^0(\cI_r(m_0-\mu-r+1))=\ell_rH^0(\cK(m_0-\mu-r)).
%   \end{array}
% \]
\begin{align*}
  H^0(\cI_1(m_0-\mu)) & =\ell_1H^0(\cI_2(m_0-\mu-1)),\\
  H^0(\cI_2(m_0-\mu-1))& =\ell_2H^0(\cI_3(m_0-\mu-2)),\\
 \cdots\cdots\cdots\cdots\cdots\cdots\cdots&\cdots\cdots\cdots\cdots\cdots\cdots\cdots\cdots\cdots\\
  %\dotfill& \dotfill\\
  H^0(\cI_{r-1}(m_0-\mu-r+2))& =\ell_{r-1}H^0(\cI_r(m_0-\mu-r+1)),\\
  H^0(\cI_r(m_0-\mu-r+1)) & =\ell_rH^0(\cK(m_0-\mu-r)).
\end{align*}
\noindent It follows that
$H^0(\cI_1(m_0-\mu))=\ell_1\cdots\ell_rH^0(\cK(m_0-\mu-r))$ and therefore
\[
 \langle x,y\rangle\, [\mu g+z\, \partial g/\partial z]\subset\ell
H^0(\cK(m_0-\mu-r))\quad 
 \text{where}\  \ell:=\ell_0\cdots\ell_r=x^ay^b
\]
and $a$ (respectively $b$) is the number of $\ell_i=x$ (respectively the
number of $\ell_i=y$). This implies that $g$ is divisible by $\ell$, and
changing the notation we can write $f_0=f^0+\ell z^{\mu}g$, where
$\ell=\ell_0\cdots\ell_r, \mu\ge m_0-(\kappa +r)$ is maximal, $g\in
S_{m_0-\mu-r-1}$ and
\[
 \langle x,y\rangle\; [\mu g+z\,\partial g/\partial z]\subset
  H^0(\cK(m_0-\mu-r)).
\]

\noindent Now $f_0\in H^0(\cK(m_0))$ or $f_0\in H^0(\cI_1(m_0))$ and $m_0\ge
\colength(\cK)+2$, if $r=0$ or $m_0\ge \colength(\cI_1)+2$, if $r\ge 1$,
respectively (cf. Lemma 2.6 and Lemma 2.8). Therefore $R_{m_0}\subset
H^0(\cK(m_0))$ or $R_{m_0}\subset H^0(\cI_1(m_0))$, respectively (cf. Appendix~C, Remark 2). It follows $\ell z^{\mu}g\in H^0(\cK(m_0))$ (or $\ell
z^{\mu}g\in H^0(\cI_1(m_0))$ and thus $\ell g\in H^0(\cK(m_0-\mu))$ (or $\ell
g\in H^0(\cI_1(m_0-\mu))=\ell_1\cdots\ell_rH^0(\cK(m_0-\mu-r))$,
respectively).

\noindent From this we conclude that $\ell_0 g\in H^0(\cK(m_0-\mu - r))$, in any
case. We also note that $g\notin H^0(\cK(m_0-\mu-r-1))$ without restriction,
because otherwise 
\[
\ell g\in\ell H^0(\cK(m_0-\mu-r-1))=\ell_0 H^0(\cI_1(m_0-\mu-1)),
\] and thus $z^{\mu}\ell g\in\ell_0 H^0(\cI_1(m_0-1))$
would follow. But then $\ell g$ could be deleted.

% \begin{conclusion}\label{8}
%   If the vice-monomial of $f_0$ is different from all monomials $N_i^{\up}$
%   and $N_j^{\down}$, then we can write $f_0=f^0+z^{\mu}\ell g$, where $\mu\ge
%   m_0-(\kappa +r)$ is maximal, $\ell:=\ell_0\cdots\ell_r, f^0=y^{m_0}$, if
%   $\ell_0=x$ (or $f=x^{m_0}$, if $\ell_0=y$, respectively),\\ $g\in
%   S_{m_0-\mu-r-1}-H^0(\cK(m_0-\mu-r-1)), \ell_0g\in H^0(\cK(m_0-\mu-r))$ and
%   $\langle x,y\rangle\, [\mu g+z\,\partial g/\partial z] \\ \subset
%   H^0(\cK(m_0-\mu-r))$.
% \end{conclusion}

\begin{conclusion}\label{8}
  If the vice-monomial of $f_0$ is different from all monomials $N_i^{\up}$
  and $N_j^{\down}$, then we can write
  \begin{gather*}
f_0=f^0+z^{\mu}\ell g,\ \text{where}\ \mu\ge m_0-(\kappa+r) \text{is maximal},\
\ell:=\ell_0\cdots\ell_r,\\
 f^0=y^{m_0}, \text{if}\ 
  \ell_0=x \text{ (or}\ f=x^{m_0}, \text{if}\ \ell_0=y, \text{respectively)},\\
g\in
  S_{m_0-\mu-r-1}-H^0(\cK(m_0-\mu-r-1)),\ \ell_0g\in H^0(\cK(m_0-\mu-r))
  \end{gather*}
and $\langle x,y\rangle\, [\mu g+z\,\partial g/\partial z] \subset
H^0(\cK(m_0-\mu-r)) .\qed $
\end{conclusion}

\section{Order of $f_0$ in the case $3^\circ$}\label{3.4}

We keep the notations of Conclusion 3.8, but now we write
$h:=\ell_0\cdots\ell_r$. In order to simplify the notations , the cases
$\rho_1>0$ and $\rho_1<0$ will be treated separately. As always we assume
$\rho_2>0$.

\subsection{}\label{3.4.1}

Let be $\rho_1>0$ (and thus $\rho_0<0$). Conclusion 3.8 is still valid, if $x$
and $y$ are interchanged. Thus one can write $f_0=x^{m_0}+z^{\mu}hg,\ell_0=y$
and
 \[
 H^0(\cI(d))=yH^0(\cI_1(d-1))\oplus\langle \{ f_0x^nz^{d-m_0-n}|0\le n\le
 d-m_0\}\rangle,\] where $\cI_i:=\cK$ if $r=0$ (cf. Lemma 2.6).

\noindent Suppose there is a \;$0\le\nu\le d-m_0$ such that $x^{\nu+1}f_0\equiv
x^{m_0+\nu+1}$ modulo $hH^0(\cK(m_0+\nu-r))$. Because of
\[
hH^0(\cK(m_0+\nu-r))=y\ell_1\cdots\ell_rH^0(\cK(m_0+\nu-r))\subset
yH^0(\cI_1(m_0+\nu))
\]
 then follows that 
\begin{multline*}
  H^0(\cI(d))=yH^0(\cI_1(d-1)) \\ \oplus \langle \{ f_0x^nz^{d-m_0-n}|0\le
  n\le\nu\}\rangle\oplus \langle \{ x^{m_0+n}z^{d-m_0-n}| \nu+1\le n\le
  d-m_0\}\rangle.
\end{multline*}

\noindent If one wants to determine the so called $\alpha$-grade of
$\dot{\wedge}\psi_{\alpha}(H^0(\cI(d))$ (cf. Chapter 4), then it is easier to
estimate the contribution coming from the monomials in $H^0(\cI(d))$. In the
following definition we assume $\rho_2>0$ and the sign of $\rho_1$ is
arbitrary.

\begin{definition}\label{4}
  The \emph{order} of $f_0$ is the smallest natural number $ \nu$ such that
  $x^{\nu+1}f_0\equiv x^{m_0+\nu+1}$ modulo $hH^0(\cK(m_0+\nu-r))$ (such that
  $y^{\nu+1}f_0\equiv y^{m_0+\nu+1}$ modulo $hH^0(\cK(m_0+\nu-r))$,
  respectively) if $\rho_1>0$ (if $\rho_1<0$, respectively). \\
  Here we have put $h:=\ell_0\cdots\ell_r$ (cf. $(Z)$ in 3.1).
\end{definition}

\begin{remark}\label{1}
  If $\rho_1>0$, then $\rho_0<0$, and from Lemma 2.6 it follows that
  $f_0$ has the initial monomial $M_0^{\down}=x^{m_0}$. Then Conclusion
 3.2 gives $yx^{m_0}\in\cI$. On the other hand, if $\rho_1<0$ the same
  argumentation shows that $f_0$ has the initial monomial $y^{m_0}$ and
    $xy^{m_0}\in\cI$.\hfill $\Box$
\end{remark}

\noindent We now take up again the assumption $\rho_1>0$ from the beginning of
(3.4.1).\\
\textit{Subcase 1:} $\ell_0=x$. As $\rho_0<0$ one has $f_0=y^{m_0}$.\\
\textit{Subcase 2:} $\ell_0=y$. Putting $h:=x^ay^b$ one can write
\[
f_0=x^{m_0}+z^{\mu}hg=x^{m_0}(1+x^{a-m_0}y^bz^{\mu}g).
\]
 As $f_0$ is a $T(\rho)$-semi-invariant, one has
\[
x^{a-m_0}y^bz^{\mu}g=X^{\gamma\rho}p(X^\rho), \text{ where } \gamma\in\N -(0),\
  p(X^{\rho})=a_0+a_1X^{\rho}+\ldots+a_sX^{s\rho} \text{ and } a_i\in k.
\]
\noindent We now assume $f_0$ is not a monomial. If both $\mu$ and $\gamma$
  are chosen maximal, then $a_0\neq 0$, and we can write
$g=Mp(X^{\rho}),\
  \text{where}\  M:=x^{m_0-a+\gamma\rho_0}y^{\gamma\rho_1-b} \text{ and } \mu=\gamma\rho_2$.
  From the $\Gamma$-invariance of $f_0$ modulo $\ell_0H^0(\cI_1(m_0-1))$ it
  follows that
\[ 
\langle x,y\rangle hz^{\mu-1}[\mu g+z\partial g/\partial
  z]\subset\ell_0H^0(\cI_1(m_0-1)),
\]
 thus 
\[
\langle x,y\rangle
  \ell_1\cdots\ell_r[\mu g+z\partial g/\partial z]\subset
  H^0(\cI_1(m_0-\mu)).
\]

We had already obtained
$H^0(\cI_1(m_0-\mu))=\ell_1\cdots\ell_rH^0(\cK(m_0-\mu-r))$ in (3.3.6).\\
We conclude that
\[
\langle x,y\rangle [\mu g+z\partial g/\partial z]\subset
  H^0(\cK(m_0-\mu-r))
\]
 and therefore:
\[
  xg\equiv -\tfrac{1}{\mu}xM[\rho_2a_1X^{\rho}+2\rho_2a_2X^{2\rho}+\cdots +s\rho_2a_sX^{s\rho}]\;\mbox{modulo}\; H^0(\cK(m_0-\mu-r))
\]

\noindent Assume that $s=$ degree $(p)>0$. Then $a_s\neq 0$ and $j:=$ min $\{
i>0|a_i\neq 0\}$ is an integer between 1 and $s$ inclusive. Then one can write
:
\[
  xg\equiv -\tfrac{1}{\mu}xMX^{j\rho}[j\rho_2a_j+(j+1)\rho_2a_{j+1}X^{\rho}+\cdots+s\rho_2a_sX^{(s-j)\rho}]\; \text{modulo}\; H^0(\cK(m_0-\mu-r))
\]
From this it follows that  
\[
xf_0=x^{m_0+1}+h\, z^{\mu}xg\equiv\tilde{f} \modulo
hH^0(\cK(m_0-r)),
\]
 where $\tilde{f}:=x^{m_0+1}+h\, z^{\tilde{\mu}}\tilde{M}\tilde{p} (X^{\rho})$ and
$\tilde{\mu}:=\mu+j\rho_2>\mu,\ \tilde{M}:=xMx^{j\rho_0}y^{j\rho_1}$,
\[ 
\tilde{p}(X^{\rho}):=
  \tilde{a}_0+\tilde{a}_1X^{\rho}+\cdots+\tilde{a}_{s-j}X^{(s-j)\rho}
\]
and finally $\tilde{a}_0:=-\frac{1}{\mu}j\rho_2a_j\neq
0,\cdots,\tilde{a}_{s-j}:=-\frac{1}{\mu}s\rho_2a_s\neq 0$.

\begin{remark}\label{2}
  $\tilde{f}$ is again a $T(\rho)$-semi-invariant, because 
  \begin{align*}
    x^{-(m_0+1)}hz^{\tilde{\mu}}\tilde{M} & =
    x^{-m_0-1}x^ay^bz^{\mu+j\rho_2}x\cdot x^{m_0-a+\gamma\rho_0}\cdot
    y^{\gamma\rho_1-b}\cdot x^{j\rho_0}\cdot y^{j\rho_1}\\
    & = x^{(\gamma+j)\rho_0}y^{(\gamma+j)\rho_1}z^{(\gamma+j)\rho_2}. \qed
  \end{align*}
\end{remark}

\begin{remark}\label{3}
  $\tilde{f}$ is only determined modulo $\ell_0H^0(\cI_1(m_0))$, as
  $hH^0(\cK(m_0-r))\subset\ell_0H^0(\cI_1(m_0))$.\hfill $\qed$
\end{remark}

\indent We continue the above discussion and get at once $xf_0\equiv
x^{m_0+1}$ modulo $hH^0(\cK(m_0-r))$, if $s=0$. In any case we have degree
$(\tilde{p})<$ degree $(p)$, and continuing in this way, we get finally
$\tilde{p}^{\cdot (s+1)} = 0 $. This means, there is an integer
$0\le\nu\le s$ such that
\begin{equation}\label{10}
  x^{\nu+1}f_0\equiv x^{m_0+\nu+1} \modulo  hH^0(\cK(m_0-r+\nu)).
\end{equation} 

\subsection{}\label{3.4.2}

We now assume $\rho_1<0$. If $f_0$ is a monomial, then the order of $f_0$ is
zero by definition. Thus we may assume without restriction that $\rho_0>0$ 
(c.f. 2.4.1 Auxiliary Lemma 2). \\
\noindent \textit{Subcase 1:} $\ell_0=y$. Then $f_0=x^{m_0}$ (Lemma 2.6). \\
\noindent \textit{Subcase 2:} $\ell_0=x$. With notations analogous to that of
(3.4.1) one has: $f_0=y^{m_0}(1+x^ay^{b-m_0}z^{\mu}g)$ is a semi-invariant,
therefore $x^ay^{b-m_0}z^{\mu}g=X^{\gamma\rho}p(X^{\rho}),\gamma\in\N-(0)$. If
$f_0$ is not a monomial and $\mu$ and $\gamma$ are chosen maximal, then one
can write $g=Mp(X^{\rho})$, where $\mu=\gamma\rho_2,
M=x^{\gamma\rho_0-a}y^{m_0-b+\gamma\rho_1}$ and \;
$p(X^{\rho})=a_0+a_1X^{\rho}+\cdots +a_sX^{s\rho}$. Furthermore one gets
$yf_0\equiv\tilde{f}$ modulo $hH^0(\cK(m_0-r))$, where
$\tilde{f}:=y^{m_0+1}+hz^{\tilde{\mu}}\tilde{M}\tilde{p}(X^{\rho}),\tilde{\mu}:=\mu+j\rho_2,
\tilde{M}:=yMx^{j\rho_0}y^{j\rho_1}$ and $\tilde{p}$ is defined as in (3.4.1).
It is clear that one gets the corresponding statements as in (3.4.1), if $x$
and $y$ as well as $a$ and $b$ are interchanged. This means, there is an
integer $0\le\nu\le s$ such that
\begin{equation}\label{11}
  y^{\nu+1}f_0\equiv y^{m_0+\nu+1} \ \text{modulo} \ hH^0(\cK(m_0-r+\nu)).
\end{equation} 

\subsection{Estimate of the order of $f_0$}\label{3.4.3}

The order has been defined as the smallest natural number $\nu$ such that
$x^{\nu+1}f_0\equiv x^{m_0+\nu+1}$ (respectively $y^{\nu+1}f_0\equiv
y^{m_0+\nu+1}$) \text{modulo} $hH^0(\cK(m_0-r+\nu))$. We keep the notations of
(3.4.1) and (3.4.2). From $\gamma\rho_2=\mu \geq \, m_0-(\kappa+r)$ (cf. 3.3.6) it
follows that $\gamma\ge [m_0-(\kappa+r)]/\rho_2$. On the other hand,
$x^{m_0-a+\gamma\rho_0}\cdot x^{s\rho_0}$ has to be a polynomial in case that
$\rho_1>0$ ( $y^{m_0-b+\gamma\rho_1}\cdot y^{s\rho_1}$ has to be a monomial in
case that $\rho_1<0$, respectively). That means $m_0-a+(\gamma+s)\rho_0\ge 0$
($m_0-b+(\gamma+s)\rho_1\ge 0$, respectively). Now one has $\rho_0<0$, if
$\rho_1>0$, and $\rho_0>0$, if $\rho_1<0$ (cf. 2.4.1 Auxiliary Lemma 2). It
follows that
\begin{align*}
  \gamma+s \le (m_0-a)/|\rho_0| \quad & \text{if}\ \rho_1>0,  \\
\gamma+s\le  (m_0-b)/|\rho_1|  \quad &  \text{if}\ \rho_1<0.
\end{align*}
% \[
%   \gamma+s \le 
%   \begin{cases}
%     (m_0-a)/|\rho_0| & \text{if } \rho_1>0,\\
%     (m_0-b)/|\rho_1| & \text{if } \rho_1<0.
%   \end{cases}
% \]
Here $a$ (respectively $b$) is the number of $\ell_i, 0\le i\le r$, such that $\ell_i=x$ (respectively $\ell_i=y$).

In case that $\rho_1>0$ we obtain , because of $|\rho_0|=\rho_1+\rho_2$:
\[
  s\le\frac{m_0-a}{\rho_1+\rho_2} - \frac{m_0-(\kappa+r)}{\rho_2}
\]
And in case that $\rho_1<0$, because of $|\rho_1|=\rho_0+\rho_2$, we obtain
\[
  s\le\frac{m_0-b}{\rho_0+\rho_2} - \frac{m_0-(\kappa+r)}{\rho_2}.
\]
 From the congruences (3.11) and (3.12) we get 

 \begin{conclusion}\label{9}
The order $s$ of $f_0$ fulfils the inequality:
\[
s\le 
\begin{cases}
  \frac{\kappa}{\rho_2}
  -m_0\left(\frac{1}{\rho_2}-\frac{1}{\rho_1+\rho_2}\right)
  +\frac{r}{\rho_2}-\frac{a}{\rho_1+\rho_2} & \text{if } \rho_1>0,\\
  \frac{\kappa}{\rho_2}
  -m_0\left(\frac{1}{\rho_2}-\frac{1}{\rho_0+\rho_2}\right)
  +\frac{r}{\rho_2}-\frac{b}{\rho_0+\rho_2} & \text{if } \rho_1<0.
\end{cases}
\]
\end{conclusion}

\section{Summary in the Case I}\label{3.5}
 We recall that this means $\rho_0>0, \rho_2>0$ and $\rho_1<0$.

\begin{proposition}\label{1}
  Let be $\rho_2>0$ and $\rho_1<0$.\\
  (a) The following cases can a priori occur:\\
  1st case: One of the $f_i$, say $f_{i_0}$ has as vice-monomial an ``upper
  empty corner'' $N_{j_0}^{\up}$. Then $\rho_2>\kappa$, thus $\cK$ is
  monomial, and the $f_i$ have one of the following forms:
\begin{enumerate}
\item $f_i=M_i^{\down}$
\item $f_i=M_i^{\up}+\alpha N_j^{\down}$
\item $f_i=M_i^{\up}+\alpha N_j^{\up}+\beta N_k^{\up}$
\item $f_i=M_i^{\up}+\alpha N_j^{\up}+\beta N_k^{\down}$
\item $f_i=M_i^{\up}+\alpha N_j^{\up}+\beta L, L\in\cL$ a monomial \\
  such that 
\[
(x,y)\cdot L  \subset\ell\cK(-r-1),\ \ell:=\ell_0\cdots\ell_r.
\]
\item $f_i=M_i^{\up}+\alpha N_j^{\up}+\beta N_k^{\up}\cdot (z/y),\alpha,\beta\in k^*$
\item $f_i=M_i^{\up}+\alpha N_j^{\up}+\beta L\cdot (z/y),L\in\cL$ a monomial \\
  such that 
\[ 
  (x,y)\cdot L \subset \ell\cK(-r-1),\ \alpha,\beta\in k^*.
\]
\end{enumerate}

Here $\alpha$ and $\beta$ are (a priori) arbitrary elements of $k$ and $0\le
i<j<k\le r$. 

\noindent
2nd case: Not any of the $f_i$ has as vice-monomial an ``upper empty corner'' $N_j^{\up}$. Then each of the $f_i$ has one of the following forms:
\begin{enumerate}
\item $f_i=M_i^{\up}+\alpha N_j^{\down}$
\item $f_i=M_i^{\down}$
\item $f_i=M_i^{\up}+F, F=\sum \alpha_j L_j,\ L_j\in\cL$ monomial,
  $\alpha_j\in k^*$, such that $xF\in\ell\cK(-r-1)$, and for suitable
  $\beta_j\in k^*$ and $G:=\sum\beta_jL_j$ one has $(x,y)\cdot
  G\subset\ell\cK(-r-1)$.
\end{enumerate} 
(b) In the 1st case there is only one trinomial at most, i.e. there is only
one $f_i$ at most, which has the shape as in one of the cases 1.3 - 1.7, with
$\alpha$ and $\beta$ different from zero. \\
(c) Let be $\cM:=\bigcup\limits_{n\ge 0} \{\text{monomials in } H^0(\cI(n))\}$
and let be $\langle\cM\rangle$ the subspace of $S$ generated by these
monomials. In the 1st case one has:
\begin{itemize} 
\item $xf_i\in\langle\cM\rangle$, for all $0\le i\le r$;
\item $yf_i\in\langle\cM\rangle$, if $f_i$ has the shape as in one of the
  cases 1.1--1.5
\item $y^2f_i\in\langle\cM\rangle$, if $f_i$ has the shape as in one of the
  cases 1.6 and 1.7.
\end{itemize} 
(d) In the 2nd case one has $\langle x,y\rangle f_i\subset \langle\cM\rangle$,
if $f_i$ has the shape as in case 2.1. \\
If $f_i$ has the shape as in case 2.3, then $xM_i^{\up}\in\cM$ , and 
the order $s$ of $f_i$ fulfils the inequality 
\[
s\le \frac{\kappa}{\rho_2}-m_i \left(
  \frac{1}{\rho_2}-\frac{1}{\rho_0+\rho_2}\right) + \frac{r}{\rho_2},
\]
where $0\le i\le r$.
\end{proposition}

\begin{proof}
  (1.1) and (1.2) result from Conclusion 3.2 and Conclusion 3.3; (1.3)--(1.5)
  result from Conclusion 3.6; and (1.6), and (1.7) results from Conclusion
  3.5. (2.1) results again from Conclusion 3.2 and Conclusion 3.3; (2.2) is clear,
  and (2.3) results from Conclusion 3.8, for if $M_i^{\up}$ occurs, then
  $\ell_i=x$, and this corresponds to the linear form $\ell_0$ in Conclusion
  3.8. \\
  (b) results from Conclusion 3.7.\\
  (c) results from the Conclusions 3.3, 3.5 and 3.6.\\
  (d) results in the case (2.1) from the fact that $f_j=M_j^{\down}$ is a
  monomial and therefore all monomials with the same $z$-degree as
  $M_j^{\down}$, which occur in $\inn(H^0(\cI(m_0)))$, also occur in %% \inn
  $H^0(\cI(m_0))$ (Conclusion 3.2). In the case (2.3) the first part of the
  statement results from Conclusion 3.2, and the statement concerning the order
  results from Conclusion 3.9.
\end{proof}

\begin{remark}\label{4}
  If $r=0$, then only the cases (2.2) and (2.3) can occur, i.e., one has
  either $f_0=x^{m_0}$ or $f_0=y^{m_0}+F$, where $F$ has the properties
  mentioned above.
\end{remark}

\section{Summary in the Case II}\label{3.6}
We recall that this means $\rho_0<0, \rho_1>0, \rho_2>0$. 

\noindent  We translate formally the results of (3.5): As
always $\rho_2>0$ is assumed, one has $\rho_0<0$. By definition $\iota(i)=\#
\{\ell_j=y|0\le j<i\}$ and therefore $i-\iota (i)=\# \{\ell_j=x|0\le j<i\}$. \\
We conclude that if one interchanges $x$ and $y$ in the monomial $M_i^{\up}$
(respectively $M_i^{\down}$) and if one simultaneously replaces $\iota(i)$ by
$i-\iota(i)$ and vice versa $i-\iota(i)$ by $\iota(i)$, then one obtains
$M_i^{\down}$ (respectively $M_i^{\up}$). Therefore the statements of
Proposition 3.1 can be translated to the Case II by interchanging $x$ and $y$
and ``up'' and ``down'':

\begin{proposition}\label{2} 
  Let be $\rho_2>0$ and $\rho_1>0$.\\
  (a) The following cases can a priori occur: \\
  1st case: One of the $f_i$, say $f_{i_0}$ has as vice-monomial a ``lower
  empty corner'' $N_j^{\down}$. Then $\rho_2>\kappa$, thus $\cK$ is monomial
  and the $f_i$ have one of the following forms:
\begin{enumerate} 
\item $f_i=M_i^{\up}$
\item $f_i=M_i^{\down}+\alpha N_j^{\up}$
\item $f_i=M_i^{\down}+\alpha N_j^{\down}+\beta N_k^{\down}$
\item $f_i=M_i^{\down}+\alpha N_j^{\down}+\beta N_k^{\up}$
\item $f_i=M_i^{\down}+\alpha N_j^{\down}+\beta L, L\in\cL$ a monomial\\
  such that $(x,y)\cdot L\subset\ell\cK(-r-1),\ell:=\ell_0\cdots\ell_r$
\item $f_i=M_i^{\down}+\alpha N_j^{\down}+\beta N_k^{\down} \cdot (z/x),\;\alpha,\beta\in k^*$
\item $f_i=M_i^{\down}+\alpha N_j^{\down}+\beta L\cdot (z/x), L\in\cL$ a
  monomial\\ such that $(x,y)\cdot L\subset\ell\cK(-r-1),\alpha , \beta\in k^*$.
\end{enumerate} 
2nd case: Not any of the $f_i$ has as vice-monomial a `` lower empty corner'' $N_j^{\down}$. Then each of the $f_i$ has one of the following forms:
\begin{enumerate}
\item $f_i=M_i^{\down}+\alpha N_j^{\up}$
\item $f_i=M_i^{\up}$
\item $f_i=M_i^{\up}+F,F=\sum \alpha_j L_j,L_j\in\cL$ monomial,$\alpha_j\in
 k^*$, such that $yF\in\ell\cK(-r-1)$, and for suitable $\beta_j\in k^*$ and
  $G:=\sum \beta_j L_j$ one has $(x,y)\cdot G\subset\ell\cK(-r-1)$.\\
% Fehler hier \sub -> sum
  (b) The same statement as in Proposition 3.1.\\
  (c) The same statement as in Proposition 3.1, if $x$ and $y$ are interchanged.\\
  (d) The same statement as in Proposition 3.1, if one replaces $xM_i^{\up}$ by
  $yM_i^{\down}$.
\end{enumerate}
 And the order $s$ of $f_i$ fulfils the inequality
\[
  s\le \frac{\kappa}{\rho_2}-m_i
  \left(\frac{1}{\rho_2}-\frac{1}{\rho_1+\rho_2}\right) + \frac{r}{\rho_2},
\]
where $0\le i\le r$.
\hfill $\Box$
\end{proposition}

\begin{remark}\label{5}
  If $r=0$, then only the cases (2.2) and (2.3) can occur.
\end{remark}

\newpage
\hspace{-2cm}
\begin{minipage}{34cm*\real{0.5}} \label{fig:3.1}
\centering Fig. 3.1\\ 
\tikzstyle{help lines}=[gray,very thin]
\begin{tikzpicture}[scale=0.5]
 \draw[style=help lines]  grid (34,35);
 % axes 
 \draw[thick] (0,0) -- (0,35); 
 \draw[thick] (0,0) -- (34,0);
 % ticks
 {
 \pgftransformxshift{0.5cm}
 \pgftransformyshift{-0.55cm}
  \foreach \x in {0,1,2,3,4,5,6,7,8,9,10,11,12,13,14,15,16,17,18,19,20,21,22,23,24,25,26,27,28,29,30,31,32,33,0,1,2,3,4,5,6,7,8,9,10,11,12,13,14,15,16,17,18,19,20,21,22,23,24,25,26,27,28,29,30,31,32,33}
 \draw[anchor=base] (\x,0) node {$\scriptstyle{\x}$}; 
 }
 {
 \pgftransformxshift{0.5cm}
 \draw (33,33) node[above] {$\scriptstyle{M_0}$};
 \draw (26,25) node[above] {$\scriptstyle{N_2}$};
 \draw (27,26) node[above] {$\scriptstyle{M_2}$};
 \draw (23,21) node[above] {$\scriptstyle{N_3}$};
 \draw (24,22) node[above] {$\scriptstyle{M_3}$};
 \draw (16,2) node[above] {$\scriptstyle{N_5}$};
 \draw (17,2) node[above] {$\scriptstyle{M_5}$};
 \draw (20,1) node[above] {$\scriptstyle{N_4}$};
 \draw (21,1) node[above] {$\scriptstyle{M_4}$};
 \draw (29,0) node[above] {$\scriptstyle{N_1}$};
 \draw (30,0) node[above] {$\scriptstyle{M_1}$};
\draw[anchor=east] (13,-1) node[rotate=90] {$\scriptstyle{\kappa{+}6}$};
\draw[anchor=east] (17,-1) node[rotate=90] {$\scriptstyle{m_5{+}5}$};
\draw[anchor=east] (21,-1) node[rotate=90] {$\scriptstyle{m_4{+}4}$};
\draw[anchor=east] (24,-1) node[rotate=90] {$\scriptstyle{m_3{+}3}$};
\draw[anchor=east] (27,-1) node[rotate=90] {$\scriptstyle{m_2{+}2}$};
\draw[anchor=east] (30,-1) node[rotate=90] {$\scriptstyle{m_1{+}1}$};
\draw[anchor=east] (33,-1) node[rotate=90] {$\scriptstyle{m_0}$};
 }
 % steps 

 \draw[\Red,ultra thick] (30,0) -- (30,1) -- (21,1) -- (21,2) -- (17,2) -- 
 (17,3) -- (13,3)
 (12,9) -- (13,9) -- (13,11) -- (14,11) -- (14,12) -- (15,12) --
 (15,13) -- (16,13) -- (16,14) -- (17,14) -- (17,15) -- (18,15) -- (18,16) --
 (19,16) -- (19,17) -- (20,17) -- (20,18) -- (21,18) -- (21,19) -- (22,19) --
 (22,20) -- (23,20) -- (23,21) -- (24,21) -- (24,23) -- (25,23) -- (25,24) --
 (26,24) -- (26,25) -- (27,25) -- (27,27) -- (28,27) -- (28,28) -- (29,28) --
 (29,29) -- (30,29) -- (30,30) -- (31,30) -- (31,31) -- (32,31) -- (32,32) --
 (33,32) -- (33,34) -- (34,34);

 \draw[\Red,ultra thick,dotted] (13,3) -- (12,3) -- (12,4) -- (11,4) --
 (11,5) -- (12,5) -- (12,6) -- (11,6) -- (11,7) -- (12,7) -- (12,9);
 % line
  \draw[\Black, ultra thick] (0,1) -- (34,35);  
\end{tikzpicture}
\end{minipage}

\vspace{0.5cm}
%\centering
\parbox{14cm}
{
Explanation:  $\reg(\cK)\leq c \Rightarrow R_n \subset H^0(\cK(n)), n\geq c$\\
$S_{c-1}/H^0(\cK(c-1))$ has a basis of $c$ monomials, namely the monomials\\
 of $S_{c-1}-\itin H^0(\cK(c-1)) \Rightarrow \ell[S_{c-1}/H^0(\cK(c-1))]$ has a
basis \\
of $c$ monomials, namely the monomials of
$\ell S_{c-1}-\ell\cdot\itin H^0(\cK(c-1))$ \\
They generate a subspace $\cL$.
$\ell:=\ell_0\cdots \ell_r=x^{r+1-\iota(r+1)}y^{\iota(r+1)}$
}

\newpage

% ---- fig3.2 ------
\begin{minipage}{21cm*\real{0.7}} \label{fig:3.2}
\centering Fig. 3.2\\ 
\tikzstyle{help lines}=[gray,very thin]
\begin{tikzpicture}[scale=0.7]
 \draw[style=help lines]  grid (21,25);
 % axes 
 \draw[thick] (0,0) -- (0,25); 
 \draw[thick] (0,0) -- (21,0);
 % ticks
 {
 \pgftransformxshift{0.5cm}
 \draw (9,14) node[above] {?};
 \draw (8,12) node[above] {?};
 \draw (6,8) node[above] {?};
 \draw (6,9) node[above] {?};
 \draw (3,3) node[above] {?};
 \draw (6,2) node[above] {?};
 \draw (9,1) node[above] {?};
 \draw (17,23) node[above] {$y^{m_i}$};
 \draw (12,20) node[above] {$\cI_{i+1}$};
 \draw (16,22) node[above] {$N_i^\up$};
 \draw[anchor=east] (10,0) node[rotate=90] {$m_{i{+}1}$};
 \draw[anchor=east] (11,0) node[rotate=90] {$m_{i{+}1}{+}1$};
 \draw[anchor=east] (17,0) node[rotate=90] {$m_i$};
 \draw[anchor=east] (15,11) node[fill=white,inner sep=1pt,above=1pt] 
   {$\cI_i=x\cI_{i+1}(-1)+f_i\cO_{\P^2}(-m_i)$};
 }
 % steps 
 \draw[blue,ultra thick] (10,1) -- (9,1) -- (9,2) -- (6,2) -- (6,3) -- (3,3)
 -- (3,4) -- (2,4) -- (2,5) -- (3,5) -- (3,6) -- (4,6) -- (4,7) -- (5,7) --
 (5,8) -- (6,8) -- (6,11) -- (7,11) -- (7,12) -- (8,12) -- (8,14) -- (9,14) --
 (9,16) -- (10,16) -- (10,17) -- (11,17) -- (11,18) -- (12,18) -- (12,19) --
 (13,19) -- (13,20) -- (14,20) -- (14,21) -- (15,21) -- (15,22) -- (16,22) --
 (16,23) -- (17,23);
 \draw[\Red,ultra thick] (11,0) -- (11,1) -- (10,1) -- (10,2) -- (7,2) --
 (7,3) -- (4,3) -- (4,4) -- (3,4) -- (3,5) -- (4,5) -- (4,6) -- (5,6) -- (5,7)
 -- (6,7) -- (6,8) -- (7,8) -- (7,11) -- (8,11) -- (8,12) -- (9,12) -- (9,14)
 -- (10,14) -- (10,16) -- (11,16) -- (11,17) -- (12,17) -- (12,18) -- (13,18)
 -- (13,19) -- (14,19) -- (14,20) -- (15,20) -- (15,21) -- (16,21) -- (16,22)
 -- (17,22) -- (17,24) -- (18,24);
\end{tikzpicture}
\end{minipage}

%%%%%%%%%%%%%%%%%%%%%%%%%%%%%% CAHPTER 4 %%%%%%%%%%%%%%%%%%%%%%%%%%%%%%

\chapter{The $\alpha$-grade.}\label{4}

\section{Notations.}\label{4.1}
We let $\G_m$ (respectively $\G_a$) operate on $S=k[x,y,z]$ by
$\sigma(\lambda):x\to x,\, y\mapsto y,\, z\mapsto \lambda z$ (respectively by
$\pal:x\mapsto x,\, y\mapsto\alpha x+y,\, z\mapsto z$).

Let be $\GG=\Grass_m(S_n)$. If $V\in \GG(k)$, then $\mwedge V$ has the
dimension 1 and $V\mapsto\mwedge V$ defines a closed immersion
$\GG\stackrel{p}{\to}\mP(\mwedge S_n)=\mP^N, N:=\dim\mwedge S_n-1$, the so
called Pl\"ucker embedding.

If one numbers the monomials in $S_n$, then one gets a basis
$\{e_1,e_2,\cdots\}$ of $S_n$, and therefore
$e_{(i)}=e_{i_1}\wedge\cdots\wedge e_{i_m}$ is a basis of $\mwedge S_n$, where
$(i)=(i_1,\cdots,i_m)$ runs through all sequences of natural numbers, such
that $1\le i_1<\cdots <i_m\le {n+2\choose 2}$. If one puts

$\pal(e_{(i)}):=\pal (e_{i_1})
\wedge\cdots\wedge\pal (e_{i_m})$, then $\G_a$ operates in an
equivariant manner on $\GG$ and $\mP^N$, and the same is valid for the
operation $\sigma$ of $\G_m$. If $\xi\in \GG(k)$ corresponds to the vector
space $V\subset S_n$, then $C_{\xi}:=\{\mwedge \pal(V)|\alpha\in k\}^-$ is
a point or a curve in $\mP^N$, and there are polynomials $f_0,\cdots,f_N$
in one variable with coefficients in $k$, such that
$C_{\xi}=\{(f_0(\alpha):\cdots :f_N(\alpha))|\alpha\in k\}^-$. At least one
of the $f_i$ is equal to 1, and if $\xi$ is not invariant under $\G_a$,
then $C_{\xi}$ is a $\G_a$-invariant closed curve in $\mP^N$ of degree
equal to $\max \{\deg (f_i)|0\le i\le N\}$ (cf. [T1], Bemerkungen 2 und 3,
p. 11). This number is denoted by $\alpha$-grade $(V)$.

Let now be $e_i$, $1\le i\le\ell:={n+2\choose 2}$, the monomials in $S_n$,
ordered in the inverse lexicographic manner. If
$f_i=\bigsum\limits^{\ell}_{j=1}a_{ji}e_j, 1\le i\le m$, is a basis of $V$,
then $f_1\wedge\cdots\wedge f_m=\bigsum\limits_{(i)}P_{(i)}e_{(i)}$, where
$P_{(i)}=\det \left(\begin{array}{l} a_{i_1 1}\cdots a_{i_1 m}\\
    \cdots\cdots\cdots\cdots\\ a_{i_m1}\cdots a_{i_mm}\end{array} \right)$ is
the Pl\"ucker coordinate for the index $(i)=(i_1,\cdots,i_m)$. It follows that
$\pal(\mwedge V)=\langle \bigsum\limits_{(i)}P_{(i)}\pal(e_{(i)})\rangle$ and
we conclude from this:
\begin{equation}\label{1}
  \alpha-\grade\; (V)\le\mathop{\max}\limits_{(i)}\{\alpha-\grade\;
  \pal(e_{(i)}) \; | \; P_{(i)}\neq 0\}
\end{equation}
where we define the $\alpha$-grade of $\pal(e_{(i)})$ to be the $\alpha$-grade
of the monomial subspace $\langle e_{i_1},\cdots,e_{i_m}\rangle$ of $S_n$.
This can be computed as follows: Write
$\pal(e_{i_{\nu}})=\bigsum\limits^{\ell}_{j=1}p_{j\nu}(\alpha)e_j,
1\le\nu\le m$, where the $p_{\nu j}$ are polynomials in one variable with
coefficients in $\Z$. Then
$\pal(e_{(i)})=\bigsum\limits_{(j)}P_{(j)}(\alpha)e_{(j)}$, where the
$P_{(j)}(\alpha)$ are the Pl\"ucker coordinates of the vector space $\langle\pal
(e_{i_1}),\cdots,\pal(e_{i_m})\rangle$. The $P_{(j)}$ are polynomials in one
variable with coefficients in $\Z$. As $P_{(i)}(\alpha)=1$, the $\alpha$-grade
of $\langle e_{i_1},\cdots,e_{i_m}\rangle$ is equal to
$\max\limits_{(j)}\{\deg (P_{(j)})\}$.

Whereas it seems practically impossible to find a formula for the
$\alpha$-grade of an arbitrary vector space $V\subset S_n$, the $\alpha$-grade
of a monomial subspace $V\subset S_n$ can be computed as follows: Write
$V=\bigoplus\limits^n_{i=0}z^{n-i}V_i$ where $V_i\subset R_i$ is a monomial
subspace and $R=k[x,y]$. If we put $m(i):=\dim V_i$ then $V_i$ has a basis of
the form $\{x^{i-a_{ij}}y^{a_{ij}}|1\le j\le m(i)\}$, where $0\le
a_{i1}<\cdots <a_{im(i)}\le i$ is a sequence of integers. As $\alpha$-grade
$(V)=\sum\limits^n_{i=0}\alpha$-grade $(V_i)$, we can consider $V$ as a graded
vector space in $R$, which has a basis of monomials of different degrees. In
([T1], 1.3, p. 12f.) it was shown:
\begin{equation}\label{2}
  \begin{array}{l}
  \mbox{If}\; 0\le c_1<\cdots<c_r\le i\; \mbox{are integers, and}\\
  W:=\langle x^{i-c_1}y^{c_1}\cdots,x^{i-c_r}y^{c_r}\rangle \subset R_i,\\
  \mbox{then}\;\alpha\mbox{-grade}\; (W)=(c_1+\cdots+c_r)-(1+\cdots+r-1).\end{array}
\end{equation}

Later on we will need an estimate of the $\alpha$-grade of an ideal
$\cI\subset\cO_{\mP^2}$ of colength $d$, which is invariant under
$G=\Gamma\cdot T(\rho)$. This will be done by estimating the $\alpha$-grade of
the vector space $V=H^0(\cI(n))$, if $n$ is sufficiently large. By means of
(4.1) the estimate will be reduced to the computation of the $\alpha$-grade of
monomial subspaces, and because of (4.2) this can be regarded as a combinatorial
problem, the formulation of which needs some more notations.

\section{The weight of a pyramid.}\label{4.2}

\begin{definition}\label{5}
  A pyramid with frame $c$ and colength $d$, $1\le d\le c$, is a set $P$ of
  monomials in $R=k[x,y]$ with the following properties: The $i$-th ``column''
  $S_i$ of $P$ consists of monomials $x^{i-a_{ij}}y^{a_{ij}},0\le
  a_{i1}<\ldots<a_{im(i)}\le i$, for all $0\le i\le c-1$, such that the
  following conditions are fulfilled:
\begin{enumerate}
\item $\bigcup\limits^{c-1}_{i=0}S_i=P$
\item $\# (\{ x^{i-j}y^j|0\le j\le i\le c-1\} \setminus P)=d$
\end{enumerate}

Then $w(S_i):=(a_{i1}+\cdots+a_{im(i)})-(1+\cdots+m(i)-1)$ is called the weight of the $i$-th column $S_i$, and $w(P):=\bigsum\limits^{c-1}_{i=0}w(S_i)$ is called the weight of the pyramid $P$.
\end{definition}

\begin{example} Let be $\cI\subset\cO_{\mP^2}$ an ideal of colength $d$ which
  is invariant under $T(3;k)$. Then $\reg (\cI)\le d$, therefore
  $h^0(\cI(d-1))={d+1\choose 2}-d$, and we can write
  $H^0(\cI(d-1))=\bigoplus\limits^{d-1}_{i=0}z^{d-1-i}V_i$, where $V_i\subset
  R_i$ are monomial subspaces. If $S_i$ is the set of monomials in $V_i$, then
  $P:=\bigcup\limits^{d-1}_{i=0}S_i$ is a pyramid with frame and colength
  equal to $d$. From (4.2) one concludes that $w(P)=\alpha$-grade
  $(H^0(\cI(d-1))$. (N.B. One has $R_n\subset H^0(\cI(n))$ if $n\ge d$.)
\end{example}

\begin{remark}\label{1}
  Let $0\le c_1<\cdots<c_r\le i$ be a sequence of integers.
  $w(c_1,\cdots,c_r):=(c_1+\cdots+c_r)-(1+\cdots+r-1)$ is maximal if and only
  if $c_{\nu}=i-r+\nu,1\le\nu\le r$, i.e. if
  $(c_1,\cdots,c_r)=(i-r+1,\cdots,i)$.\hfill $\Box$
\end{remark}

The aim is to determine those pyramid $P$ of type $(c,d)$, i.e., with frame
$c$ and colength $d$, for which $w(P)$ is maximal. Because of Remark 4.1 we will
consider without restriction only pyramids with
$S_i=\{x^{i-a(i)}y^{a(i)},\cdots,xy^{i-1},y^i\}$, where $a(i):=i+1-m(i)$ is a
number between 0 and $i$ inclusive. We call $x^{i-a(i)}y^{a(i)}$ the
\emph{initial monomial} and $a(i)$ \emph{initial degree} of $S_i$.  For
simplicity we write $S_i=(a(i),a(i)+1,\cdots,i)$ and $P=\langle
x^{i-a(i)}y^{a(i)}|0\le i\le c-1\rangle$.

\begin{remark}\label{2}
$w(S_i)=ia(i)+a(i)-a(i)^2$.
\end{remark}

\begin{proof}
  $w(S_i)=(a(i)+\cdots +i)-(1+\cdots+i-a(i))=(1+\cdots
  +i)-(1+\cdots+a(i)-1)-(1+\cdots+i-a(i))={i+1\choose 2}-{a(i)\choose
    2}-{i-a(i)+1\choose 2}$, and a direct computation gives the assertion.
\end{proof}

Taking away $x^{i-a(i)}y^{a(i)}$ from $S_i$ and adding
$x^{j-a(j)+1}y^{a(j)-1}$ to $S_j$, if $S_i\neq\emptyset, a(j)>0$ and $j\neq
i$, then gives a pyramid
\[
  P'=P-\{x^{i-a(i)}y^{a(i)}\}\cup\{ x^{j-a(j)+1}y^{a(j)-1}\}.
\]
We express this as 
 $ S_i(P)=(a(i),\cdots,i), S_j(P)=(a(j),\cdots,j),
S_i(P')=(a(i)+1,\cdots,i),\\S_j(P')=(a(j)-1,\cdots,j)$ and we get
$w(S_i(P))=ia(i)+a(i)-a(i)^2$; $w(S_i(P'))=i(a(i)+1)+(a(i)+1)-(a(i)+1)^2;
w(S_j(P))=ja(j)+a(j)-a(j)^2, w(S_j(P'))=j(a(j)-1)+a(j)-1-(a(j)-1)^2$. It
follows that
$w(P')-w(P)=[w(S_i(P'))-w(S_i(P))]+[w(S_j(P'))-w(S_j(P))]=[i+1-2a(i)-1]+[-j-1+2a(j)-1]$.
We get the following formula:
\begin{equation}\label{3}
  w(P')-w(P)=2(a(j)-a(i))-(j-i)-2,
\end{equation}
where we have made the assumption that $i\neq j, S_i(P)\neq\emptyset$, and
$a(j)>0$.

Now let $P$ be a pyramid of type $(c,d)$, such that $w(P)$ is maximal. Then
for all deformations $P\mapsto P'$ as above, the right side of (4.3) is $\le 0$,
i.e. $a(j)-a(i)\le\frac{1}{2}(j-i)+1$. If $i<j$ (if $i>j$, respectively) this
is equivalent to
\[
 \frac{a(j)-a(i)}{j-i}\le\frac{1}{2}+\frac{1}{j-i} \quad \text{and} \quad
\frac{a(j)-a(i)}{j-i}\ge\frac{1}{2}+\frac{1}{j-i}, \quad \text{respectively}.
\]
Putting in $j=i+1$ (respectively $j=i-1$) gives $a(i+1)-a(i)\le 1.5 \quad
( a(i)-a(i-1)\ge -0.5$, respectively). As the left side of these inequalities
are integers, it follows that $a(i+1)-a(i)\le 1\quad ( a(i)-a(i-1)\ge 0$,
respectively) for all $0\le i<c-1$ (for all $0<i\le c-1$, respectively). Note
that we can apply (4.3) only under the assumption $a(i+1)>0$ (respectively
$a(i-1)>0$). But if $a(i+1)=0$ (respectively $a(i-1)=0$), then the two last
inequalities are true, too. So we get

\begin{remark}\label{3}
  If the pyramid $P$ of type $(c,d)$ has maximal weight, then one has $a(i)\le
  a(i+1)\le a(i)+1$, for all $0\le i\le c-2$.\hfill $\Box$
\end{remark}

\begin{remark}\label{4}
  If $P=\{ x^{i-a_{ij}}y^{a_{ij}}|i,j\}$ is a pyramid of type $(c,d)$, then
  $P'=yP:=\{x^{i-a_{ij}}y^{a_{ij}+1}|i,j\}$ is called shifted pyramid, and
  $w(P')=w(P)+\# P$.
\end{remark}

\begin{proof}
  The column $S_i=(a_{i1},\cdots,a_{im(i)})$ goes over to
  $S'_i=(a_{i1}+1,\cdots,a_{im(i)}+1)$, and
  $w(S'_i)=(a_{i1}+1)+\cdots+(a_{im(i)}+1)-(1+\cdots+m(i)-1)=w(S_i)+m(i)$.
  Thus $w(P')=\bigsum\limits^{c-1}_{i=0}(w(S_i)+m(i))=w(P)+\#P$.
\end{proof}

\begin{remark}\label{5}
  A pyramid with maximal weight does not contain a step of breadth $\ge 4$,
  i.e.~it is not possible that $a(j)=\cdots=a(i)>0$, if $i-j\ge 3$.
\end{remark}

\begin{proof}
  If one would have the situation shown in Figure 4.1, then one could make the
  deformation $x^{i-a(i)}y^{a(i)}\mapsto x^{j-a(j)+1}y^{a(j)-1}$ and then from
  (4.3) one would get: $w(P')-w(P)=2\cdot 0-(j-i)-2=i-j-2\ge 1$, contradiction.
\end{proof}

\begin{remark}\label{6}
  In any pyramid two consecutive ``normal'' steps can be replaced by a step of
  breadth 3, without a change of weight. In the course of this, the sum over
  the $x$-degrees of all monomials in the pyramid increases by 2, however.
\end{remark}

\begin{proof}
  (cf. Fig. 4.2). By assumption one has $a(i)+1=a(i+1)$ and $a(i+1)+1=a(i+2)$.
  If one makes the deformation $x^{i-a(i)}y^{a(i)}\mapsto
  x^{j-a(j)+1}y^{a(j)-1}$, where $j=i+2$, then one gets $w(P')-w(P)=2\cdot
  2-2-2=0$.
\end{proof}

\noindent N.B. The possibility $a(i)=0$ is not excluded.

\begin{remark}\label{7}
There is a pyramid of maximal weight, which does not contain two consecutive normal steps, that means, there is no index $i$, such that $a(i+1)=a(i)+1, a(i+2)=a(i+1)+1$. This is called a ``prepared'' pyramid. (N.B. $a(i)$ may equal zero.)
\end{remark}

\begin{proof}
Apply Remark 4.6 for several times.
\end{proof}

\begin{remark}\label{8}
A prepared pyramid of maximal weight does not contain two steps of breadth 3.
\end{remark}

\begin{proof}
  We consider two cases:\\
  {\em 1st case:} The steps of breadth $3$ are situated side by side. Then one
  makes the deformation described in Fig. 4.3 and one gets:
  $w(P')-w(P)=2[a(i-5)-a(i)]+5-2=1$, contradiction. \\
  {\em 2nd case:} Between the steps of breadth $3$ there are $\nu\ge 1$ double
  steps. One then makes the deformation described in Fig. 4.4. Putting
  $j=i-2(\nu+1)$ one gets:
  $w(P')-w(P)=2(a(j)-a(i))-(j-i)-2=2(-\nu)+2(\nu+1)-2=0$. Then $P'$ would have
  maximal weight, too. But $P'$ contains a step of breadth $4$, contradicting
  Remark 4.5.
\end{proof}

\begin{remark}\label{9}
  Each positive natural number $d$ can uniquely represented either in the form
  $d=n(n+1)-r, 0\le r<n$ (1st case) or in the form $d=n^2-r,0\le r<n$ (2nd
  case). Both cases exclude each other.
\end{remark}

\begin{proof}
  Choosing $n$ sufficiently large, then the sequence
  $n(n+1),n(n+1)-1,\cdots,n(n+1)-(n-1),
  n(n+1)-n=n^2,n^2-1,\cdots,n^2-(n-1),n^2-n=(n-1)n,\cdots$ contains any given
  set $\{ 1,\cdots,m\}\subset\N$.
\end{proof}

\noindent We now assume that $P$ is a prepared pyramid of type $(c,d)$,
with $d\ge 3$ and maximal weight. If $P$ does not contain a step of breadth 3,
then $P$ has the form described either in Fig. 4.5a or in Fig. 4.5b, according
to if either $d=n(n+1)$ or $d=n^2$. (One has
$2\cdot\bigsum\limits^n_1\nu=n(n+1)$ and $\bigsum\limits^n_1(2\nu-1)=n^2$.) If
there is 1 step of breadth 3, then from the Remarks 3,5,6,7 and 8 it follows
that $P$ has the form described either in Fig. 4.6a or in Fig. 4.6b. Here the
step of breadth 3 may lie quite left or right in Fig. 4.6a or Fig 4.6b,
respectively. One sees that Fig. 4.6a and Fig. 4.6b result from Fig. 4.5a and
Fig. 4.5b, respectively, by removing the monomials marked by $-$.

We first compute the weights of the pyramids shown in Fig. 4.5a and Fig. 4.5b,
and then the weights of the ``reduced'' pyramids $\overline{P}$ in Fig. 4.6a
and Fig. 4.6b. \\
1st case: (Fig. 4.5a)

\begin{center}
\begin{tabular}{c|c|l}
$i$ & $a(i)$ & $w(S_i)=ia(i)-a(i)(a(i)-1)$\\ \hline
$c-1$ & $n$ & $(c-1)n-n(n-1)$\\
$c-2$ & $n$ & $(c-2)n-n(n-1)$\\
$\dotfill$ & $\dotfill$ & $\dotfill$\\
$c-2\nu-1$ & $n-\nu$ & $(c-2\nu-1)(n-\nu)-(n-\nu)(n-\nu-1)$\\
$c-2\nu-2$ & $n-\nu$ & $(c-2\nu-2)(n-\nu)-(n-\nu)(n-\nu-1)$\\
$\dotfill$ & $\dotfill$ & $\dotfill$\\
$c-2n+1$ & 1 & $(c-2n+1)\cdot 1-1\cdot 0$\\
$c-2n$ & 1 & $(c-2n)\cdot 1-1\cdot 0$
\end{tabular}
\end{center}

2nd case: (Fig 4.5b)

\begin{center}
\begin{tabular}{c|c|l}
$i$ & $a(i)$ & $w(S_i)=ia(i)-a(i)(a(i)-1)$\\ \hline
$c-1$ & $n$ & $(c-1)n-n(n-1)$\\
$c-2$ & $n-1$ & $(c-2)(n-1)-(n-1)(n-2)$\\
$c-3$ & $n-1$ & $(c-3)(n-1)-(n-1)(n-2)$\\
$\dotfill$ & $\dotfill$ & $\dotfill$\\
$c-2\nu$ & $n-\nu$ & $(c-2\nu)(n-\nu)-(n-\nu)(n-\nu-1)$\\
$c-2\nu-1$ & $n-\nu$ & $(c-2\nu-1)(n-\nu)-(n-\nu)(n-\nu-1)$\\
$\dotfill$ & $\dotfill$ & $\dotfill$\\
$c-2n+2$ & 1 & $(c-2n+2)\cdot 1-1\cdot 0$\\
$c-2n+1$ & 1 & $(c-2n+1)\cdot 1-1\cdot 0$
\end{tabular}
\end{center}

\noindent 1st case: We sum up $w(S_i)$ and $w(S_{i-1})$ if $i=c-2\nu-1$ and we
get:
\begin{align*}
\sum^{n-1}_{\nu=0}(n-\nu)(2c-2n-2\nu-1)& =\sum^n_{\nu=1}\nu(2c-4n+2\nu-1)\\
 &=(2c-4n-1)\cdot\tfrac{1}{2}n(n+1)+\tfrac{1}{3}n(n+1)(2n+1)\\
 &=n(n+1)(c-2n- {0.5}+\tfrac{2}{3}n+\tfrac{1}{3})\\
 &=n(n+1)(c-\tfrac{4}{3}n-\tfrac{1}{6}).
\end{align*}
\indent In the reduced pyramid the initial terms $\overline{a}(i)$ of the
column $\overline{S}_i$, if $i=c-2,c-4,\cdots,c-2r$, then are equal to
$n-1,n-2,\cdots,n-r$. This means, $\overline{a}(i)=n-\nu$, if $i=c-2\nu$, and
$w(\overline{S}_i)=(c-2\nu)(n-\nu)-(n-\nu)(n-\nu-1), 1\le\nu\le r$. If
$i=c-2\nu$ we get:
$w(\overline{S}_i)-w(S_i)=(c-2\nu)(n-\nu)-(n-\nu)(n-\nu-1)-[(c-2\nu)(n-\nu+1)-(n-\nu+1)(n-\nu)
]=-(c-2\nu)+(n-\nu) \cdot 2=2n-c$.

It follows that $w(\overline{P})-w(P)=r(2n-c)$.\\[2mm]
2nd case: We sum up $w(S_i)$ and $w(S_{i-1})$ if $i=c-2\nu$ and we get:\\
\begin{align*}
\sum\limits^{n-1}_{\nu=1}(n-\nu)(2c-2n-2\nu+1)&=\sum\limits^{n-1}_{\nu=1}\nu(2c-4n+2\nu+1)\\
&=(2c-4n+1)\cdot\frac{1}{2}(n-1)n+\frac{1}{3}(n-1)n(2n-1)\\
&=(n-1)n(c-2n+ {0.5} +\frac{2}{3}n-\frac{1}{3})\\
& =n(n-1)(c-\frac{4}{3}n+\frac{1}{6}).
\end{align*}

Beside this, we have to add the weight $(c-1)n-n(n-1)=n(c-n)$, if $i=c-1$, and we get $w(P)=n[(c+ {0.5})n-\frac{4}{3}n^2-\frac{1}{6}]$.\\[2mm]
\indent In the reduced pyramid the initial terms $\overline{a}(i)$ of the
column $\overline{S}_i$, if $i=c-1,c-3,\cdots,c-2r+1$, then are equal to $n-1,
n-2,\cdots,n-r$. This means $\overline{a}(i)=n-\nu$, if $i=c-2\nu+1$, and
$w(\overline{S}_i)=(c-2\nu+1)(n-\nu)-(n-\nu)(n-\nu-1), 1\le\nu\le r$. If
$i=c-2\nu+1$ we get: \\
$w(\overline{S}_i)-w(S_i)=(c-2\nu+1)(n-\nu)-(n-\nu)(n-\nu-1)-[(c-2\nu+1)(n-\nu+1)-(n-\nu+1)(n-\nu)]=2n-c-1,
1\le\nu\le r$. We get $w(\overline{P})-w(P)=r(2n-c-1)$.

\begin{remark}\label{10}
The maximal weight of a pyramid $P$ of type $(c,d)$ is equal to
\begin{equation}\label{4}
n[(c-1,5)n+(c+2r-1/6)-4/3\cdot n^2]-rc
\end{equation}
if $d=n(n+1)-r$ and $0\le r<n$ and it is equal to
\begin{equation}\label{5}
  n[(c+0.5)n+(2r-1/6)-4/3\cdot n^2]-r(c+1)
\end{equation}
if $d=n^2-r$ and $0\le r<n$.
\end{remark}

\begin{proof}
  If $d\ge 3$, this follows from the foregoing computation. If $d=2$ or if
  $d=1$, then $n=1$ and $r=0$. The formula (4.4) and the formula (4.5) give the
  weights $2c-3$ and $c-1$, which is confirmed by Fig. 4.7a and Fig. 4.7b,
  respectively.
\end{proof}

\begin{remark}\label{11}
  The formulas (4.4) and (4.5) agree in the ends of the intervals. This is
  shown by putting in $r=n$ in (4.4) and $r=0$ in (4.5) and by putting in
  $n-1$ instead of $n$ and $r=0$ in (4.4) and $r=n$ in (4.5), respectively,
  and then checking equality.\hfill $\Box$
\end{remark}

We denote the maximal weight of a pyramid of type $(c,d)$ by $w(P_{c,d})$.

\begin{remark}\label{12}
 \[
w(P_{c,d})=
\begin{cases}
  -\frac{4}{3}n^3-1,5n^2+(2r-\frac{1}{6})n+dc,\;& \mbox{if}\; d=n(n+1)-r,\\

  -\frac{4}{3}n^3+0.5n^2+(2r-\frac{1}{6})n-r+dc,\;&\mbox{if}\; d=n^2-r
\end{cases}
\]
Thus $w(P_{c,d})$ is a strictly increasing function of $c\ge d$, if $d$ is
fixed.\hfill $\Box$
\end{remark}

\begin{remark}\label{13}
 Fixing the integer $c\ge 5$, then $w(P_{c,d})$ is a strictly increasing function of $1\le d\le c$.
\end{remark}

\begin{proof}
  $w(P_{c,d})=-\frac{4}{3}n^3-1,5n^2-\frac{1}{6}n+cn(n+1)+r(2n-c)$, if
  $d=n(n+1)-r,0\le r\le n$, and
  $w(P_{c,d})=-\frac{4}{3}n^3+0.5n^2-\frac{1}{6}n+cn^2+r(2n-c-1)$, if
  $d=n^2-r, 0\le r\le n$. From $n(n+1)-r=d\le c$ and from $n^2-r=d\le c$ it
  follows that $n\le\sqrt{c}$. From $c\ge 5$ we conclude $2n-c<0$. From $d\ge
  1$ it follows that $n\ge 1$. As a function of $r$ both terms for
  $w(P_{c,d})$ are strictly decreasing in the intervals $0\le r\le n$, and the
  assertion follows from Remark 4.11 and Remark 4.12.
\end{proof}

\begin{remark}\label{14}
  Fixing the integer $1\le c\le 4, w(P_{c,d})$ is an increasing function of
  $1\le d\le c$.
\end{remark}

\begin{proof}
 By drawing the possible patterns one finds:
\begin{center} %spater ansehen
\begin{tabular}{c|c|c} $d$ & 1 & 2\\ \hline $w(P_{2,d})$ & 1 & 1\end{tabular}\hfill
\begin{tabular}{c|c|c|c} $d$ & 1 & 2 & 3\\ \hline $w(P_{3,d})$ & 2 & 3 & 3\end{tabular}\hfill
\begin{tabular}{c|c|c|c|c} $d$ & 1 & 2 & 3 & 4\\ \hline $w(P_{4,d})$ & 3 & 5 & 6 & 7\end{tabular}
\end{center}
\end{proof}

We define $P_c:=P_{c,c}$ and $w(\emptyset)=0$.

\begin{proposition}\label{3}
$w(P_c)\le (c-1)^2$ for all $c\in\N$.
\end{proposition}

\begin{proof}
  1st case: $c=d=n(n+1)-r, 0\le r<n$. Because of Remark 4.12 one has to show:\\
  $-\frac{4}{3}n^3-1,5n^2+(2r-\frac{1}{6})n+c^2\le (c-1)^2\Leftrightarrow$\\
  $\frac{4}{3}n^3+1,5n^2-(2r-\frac{1}{6})n-2[n(n+1)-r]+1\ge 0\Leftrightarrow$\\
  $\frac{4}{3}n^3-0.5n^2-(2r+\frac{11}{6})n+2r+1\ge 0 \Leftarrow$\\
  $\frac{4}{3}n^3-0.5n^2-(2n+\frac{11}{6})n+2n\ge 0\Leftarrow$\\
  $\frac{4}{3}n^3-2,5n^2+\frac{1}{6}n\ge 0$. This is true if $n\ge 2$. If $n=1$, then $r=0$, and by substituting one can convince oneself that the inequality is true in this case, too.\\

  2nd case: $c=d=n^2-r, 0\le r<n$. One has to show:
  \begin{align}
    -\tfrac{4}{3}n^3+0.5n^2+(2r-\tfrac{1}{6})n-r+c^2 & \le (c-1)^2\notag\\
    \Leftrightarrow \tfrac{4}{3}n^3-0.5n^2-(2r-\tfrac{1}{6})n+r-2[n^2-r]+1 &\ge
    0
    \notag \\
    \Leftrightarrow  \tfrac{4}{3}n^3-2,5n^2-(2r-\tfrac{1}{6})n+3r+1& \ge 0.\label{eq:6}
  \end{align}
Assuming $n\ge 2$, this inequality follows from
%   $\frac{4}{3}n^3-2,5n^2-(2n-\frac{1}{6})n+3n+1\ge 0\Leftarrow$\\
%   $\frac{4}{3}n^3-4,5n^2+3\frac{1}{6}n\ge 0\Leftrightarrow
%   \frac{4}{3}n^2-4,5n+3\frac{1}{6}\ge 0$. 
\begin{align*}
  \tfrac{4}{3}n^3-2,5n^2-(2n-\tfrac{1}{6})n+3n+1& \ge 0\\
  \Leftarrow \quad \tfrac{4}{3}n^3-4,5n^2+3\tfrac{1}{6}n & \ge 0 \\
  \Leftrightarrow \quad \tfrac{4}{3}n^2-4,5n+3\tfrac{1}{6}& \ge 0.
\end{align*}
This is true if $n\ge 3$. Putting
  $n=1$ and $n=2$ in (4.6) gives the inequalities $r\ge 0$ and $2-r>0$,
  respectively, which are true by assumption. As $P_{1,1}=\emptyset$, the
  assertion is true if $c=0$ or $c=1$, too.
\end{proof}

\section{Preview}\label{4.3}

Let be $V\subset S_n$ a $m$-dimensional subspace and
$V\leftrightarrow\xi\in\GG(k)$ the corresponding point in $\GG=\Grass_m(S_n)$.
We let $\G_m$ and $\G_a$ operate on $S$ as described in (4.1). We assume $V$
not to be invariant under $\G_m$ or $\G_a$. Then $C:=\{\pal(\xi)|\alpha\in
k\}^-$ is a closed irreducible curve, which is to have the induced reduced
scheme structure and which we imagine as a curve in $\mP^N$ by means of the
Pl\"ucker embedding $p$. Let be $h$ the Hilbert polynomial of
$p(C)\subset\mP^N, \cX=\Hilb^h(\GG)\hookrightarrow\Hilb^h(\mP^N)$ and
$\sigma:\G_m\to\cX$ the morphism $\lambda\mapsto \sigma(\lambda)C$. It has an
extension $\overline{\sigma}:\mP^1\longrightarrow\cX$, which induces a family
of curves
\[
  \begin{array}{rcl}
   \cC & \hookrightarrow & \GG\times\mP^1\\
   & \stackrel{\searrow}{f} & \;\;\downarrow\; p_2\\
  & &\mP^1  \end{array}
\]
such that $f$ is flat and $\cC_{\lambda}:=f^{-1}(\lambda)=\sigma(\lambda)C$
for all $\lambda\in\mP^1-\{0,\infty\}$. As $f$ is dominant, we get (cf. [Fu],
p.15): If $\mathbf{C}_{0/\infty}:=p_1(\cC_{0/\infty}),$ then $[\mathbf{C}_0]
=[\mathbf{C}_{\infty}]$ in
$A_1(\mathbf{G})$.

Let be $\xi_{0/\infty}=\lim\limits_{\lambda\to 0/\infty}\sigma (\lambda)\xi$
and $C_{0/\infty}:=\{\pal (\xi_{0/\infty})|\alpha \in k\}^-$. The central
theme of [T1]--[T4] is the question, what is the connection of $[C_0]$ and
$[C_{\infty}]$. The essential tool, which was already used in [T1] is the
$\alpha$-grade of $V$, which is nothing else than the degree of $C$, imbedded
in $\mP^N$ by means of $p$ (cf. [T1], 1.3).

We paraphrase the recipe for estimating the $\alpha$-grade of $V$: Let
$M_1<\cdots<M_{\ell},\ell={n+2\choose 2}$, be the inverse-lexicographically
ordered monomials of $S_n$, and let be
$f_i=\sum\limits^{\ell}_{j=1}a_{ij}M_j,\\1\le i\le m$, a basis of $V$. If
$M_{j_1}<\cdots<M_{j_m}$ is a sequence of monomials in $S_n$, then the
Pl\"ucker coordinate $P_V(M_{j_1},\cdots,M_{j_m})$ is defined to be the
determinant of $(a_{ij_{\nu}})_{1\le i,\nu\le m}$.
\bigskip
In the following , $V$ is a $T(\rho)$-invariant subspace of $S_n$, and
$f_i=M_i(1+a_i^1X^{\rho}+a_i^2X^{2\rho}+\cdots+a^i_{\nu(i)}X^{\nu(i)\rho}),
1\le i\le m$, is a basis of $T(\rho)$-semi-invariants. From formula (4.1) in (4.1)
it follows that \quad $\alpha-grade (V)\le\mathop{\max}\limits_{(j)} 
\{\alpha-grade  \;\langle M_1X^{j(1)\rho},\cdots,M_mX^{j(m)\rho}\rangle \}$ \\where
$(j)=(j(1),\cdots,j(m))\in [1,\ell]^m\;\cap\;\N^m$ runs through all sequence
such that $P_V(M_1X^{j(1)\rho}, \cdots, M_mX^{j(m) \rho})\neq 0$.

It is possible to choose the semi-invariants $f_i$ so that the initial
monomials $M_i$ (the final monomials $M_iX^{\nu(i)\rho}=:N_i$, respectively)
are linearly independent, i.e. different from each other (cf. Appendix E or
the proof of Hilfssatz 6 in [T2], Anhang 1, p. 140). As the Pl\"ucker
coordinates of a subvector space, up to a factor different from zero, do not
depend on the basis which one has chosen, one has
\begin{equation}\label{7}
  P_V(M_1,\cdots,M_m)\neq 0\quad\mbox{and}\quad P_V(N_1,\cdots,N_M)\neq 0.
\end{equation}

Define $V_0=\langle M_1,\cdots,M_m\rangle \leftrightarrow\xi_0$ and $V_{\infty}:=\langle
N_1,\cdots,N_m\rangle \leftrightarrow\xi_{\infty}$. As the function $\deg$ is constant
on flat families of curves, we get
\[
  \alpha\mbox{-grade} (V)=\deg(C)=\deg(\CC_0)=\deg(\CC_{\infty})
\]
As $C_{0/\infty}\subset\CC_{0/\infty}$ we conclude that
\begin{equation}\label{8}
  \alpha\mbox{-grade} (V)\ge \max (\alpha\mbox{-grade} (V_0),\alpha\mbox{-grade}(V_{\infty})).
\end{equation}

Now it is scarcely possible to see if
$P_V(M_1X^{j(1)\rho},\cdots,M_mX^{j(m)\rho})$ is different from zero.
Therefore we introduce the number
\[
\max-\alpha\mbox{-grade}(V):= \max\limits_{(j)}\{\alpha-\grade \;\langle
a_1^{j(1)}M_1X^{j(1)\rho},\ldots, a_m^{j(m)}M_mX^{j(m)\rho}\rangle\}
\]
 where $(j)$ runs through all sequences
$(j(1),\cdots,j(m))\in\N^m$, such that $0\le j(i)\le\nu(i)$ for all $1\le i\le
m$ and $a_i^0:=1$.

\begin{remark} (a) Clearly $\alpha$-grade $(V)\le\max-\alpha$-grade $(V)$.\\
(b) In the definition, the monomials need not be ordered.\\
(c) If one coefficient $a_i^{j(i)}$ is equal to zero or if two of the monomials $M_iX^{j(i)\rho}$ are equal for different indices $i$, then the $m$-times exterior product of the monomial space and its $\alpha$-grade are zero (cf. 4.1).
\end{remark}

To say it differently, take from each semi-invariant $f_i$ a monomial
$M_iX^{j(i)\rho}$, whose coefficient $a_i^{j(i)}$ is different from zero, form
$\mathop{\wedge}\limits_{i=1}^m\pal(M_iX^{j(i)\rho})$, and determine the highest power of
$\alpha$ occurring in such an exterior product. Finally, define
$\max-\alpha$-grade $(V)$ to be the maximum of these degrees, if $(j)$ runs
through all such sequences.

Accordingly, one defines 
\[
\min-\alpha-\grade
(V):=\min\limits_{(j)}\{\alpha\mbox{-grade}\langle
a_1^{j(1)}M_1X^{j(1)\rho},\cdots,a_m^{j(m)}M_mX^{j(m)\rho}\rangle \}
\]
 where
$(j)$ runs through all sequences of the kind described above and which give
an $\alpha$-grade different from zero.

As $\alpha$-grade $(V_{0/\infty})=\deg (C_{0/\infty})$, from (4.7) we conclude that
\begin{equation}\label{9}
  \min-\alpha\mbox{-grade} (V)\le \min(\deg C_0,\deg C_{\infty}).
\end{equation}

Later on, the vector space $V$ will always be equal to $H^0(\cI(n)))$, where
$\cI\subset\cO_{\mP^2}$ is a $G=\Gamma\cdot T(\rho)$-invariant ideal of
$y$-standard form (cf. 2.4.3 Definition 2). We will see that
$\max-\alpha$-grade $(\cI):=\max-\alpha$-grade $(H^0(\cI(n)))$
and $\min-\alpha$-grade $(\cI):=\min-\alpha$-grade $(H^0(\cI(n)))$
 not only are
independent of $n\ge$ colength $(\cI)$, but also can be computed with the help
of smaller numbers $n$. The real aim of the following estimates is to prove
the following inequality: If $\cI\subset\cO_{\mP^2}$ is an ideal of
$y$-standard form and if $\reg(\cI)=m$, then
\begin{equation*}
  \label{eq:ausruf}
  Q(m-1)+\min-\alpha\mbox{-grade} (\cI)>\max-\alpha\mbox{-grade}(\cI).
 \tag{\textbf{!}}
\end{equation*}

From this it will follow that $\CC_0$ and $\CC_{\infty}$ do not contain any
$y$-standard cycle besides $C_0$ and $C_{\infty}$, respectively (cf. Lemma 9.2).

\newpage
\vspace{-2.5cm}
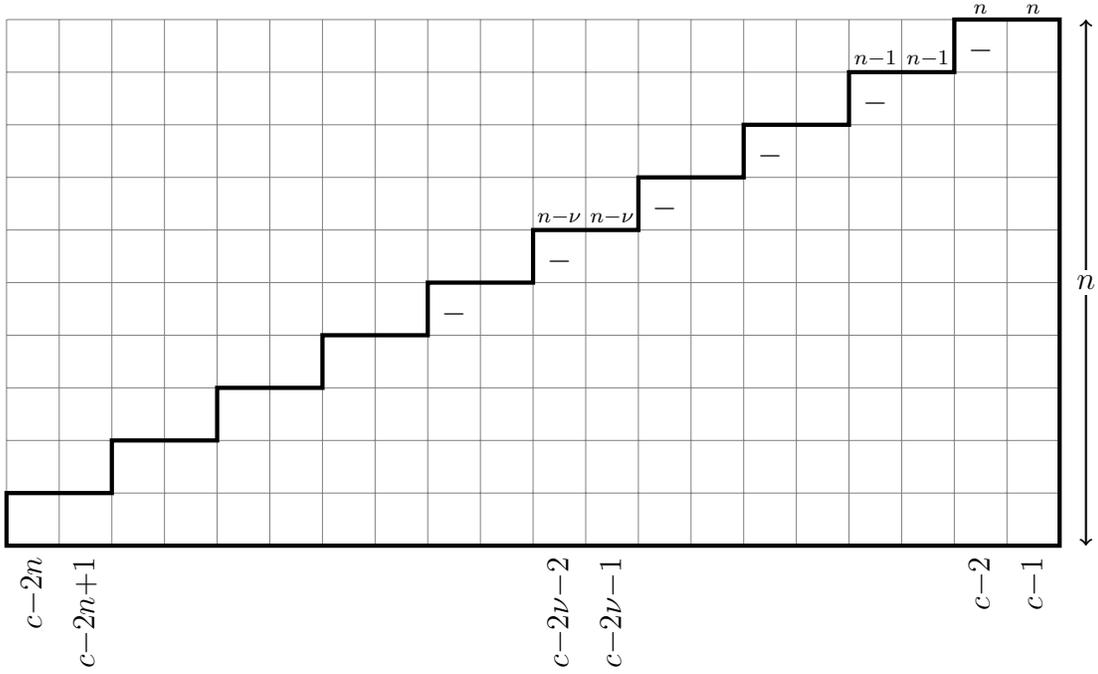
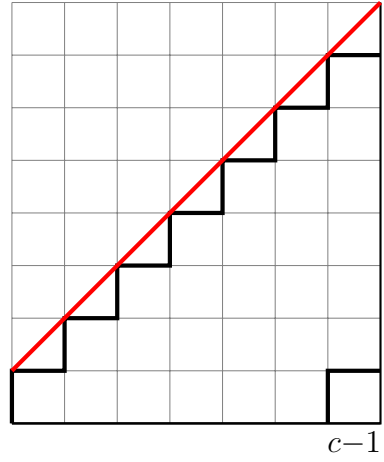
\begin{minipage}{24cm*\real{0.7}} \label{fig:4.1-4.4}
%\centering Fig.: 4.1\\ 
\tikzstyle{help lines}=[gray,very thin]
\begin{tikzpicture}[scale=0.7]
 \draw[style=help lines] (0,-4) grid  (25,5);
 % labels
 {
 \pgftransformxshift{0.5cm}
 \draw[thick,<-] (1,1.5) .. controls (2.5,0.8) and (4.1,1) .. (5,2.5);
 \draw[anchor=base] (1,2) node[above=2pt,fill=white, inner sep=0pt] 
  %{$\scriptstyle{a(j)}$};
  {$a(j)$};
 \draw[anchor=base] (5,2) node[above=2pt,fill=white, inner sep=0pt] 
 % {$\scriptstyle{a(i)}$};
  {$a(i)$};
 \draw[thick,->] (7,1.2) .. controls (8.5,1.7) and (8.5,1.5) .. (9,2.5);
 \draw[anchor=base] (7,1) node[above=2pt,fill=white, inner sep=0pt] 
  {$\scriptstyle{a(i)}$};
 \draw[anchor=base] (9,3) node[above=2pt,fill=white, inner sep=0pt] 
  {$\scriptstyle{a(i{+}2)}$};

 \draw[thick,<-] (11,1.5) .. controls (14,1.8) .. (16,3.2);
 \draw[anchor=base] (11,2) node[above=2pt,fill=white, inner sep=0pt] 
  {$\scriptstyle{a(i{-}5)}$};
 \draw[anchor=base] (16,3) node[above=2pt,fill=white, inner sep=0pt] 
  {$\scriptstyle{a(i)}$};

 \draw[thick,<-] (13,-1.5) .. controls (16.5,-1.1) and (21.2,1,6) .. (23,3.2);
 \draw[anchor=base] (13,-1) node[above=2pt,fill=white, inner sep=0pt] 
  {$\scriptstyle{a(j)}$};
 \draw[anchor=base] (23,3) node[above=2pt,fill=white, inner sep=0pt] 
  {$\scriptstyle{a(i)}$};
 }
 \draw [anchor=base] (3,5) node[above=2pt] {Fig. 4.1}; 
 \draw [anchor=base] (9,5) node[above=2pt] {Fig. 4.2}; 
 \draw [anchor=base] (14,5) node[above=2pt] {Fig. 4.3}; 
 \draw [anchor=base] (21,5) node[above=2pt] {Fig. 4.4}; 
 % steps 
 \draw[\Red,ultra thick] (1,1) -- (1,2) -- (3,2) (4,2) -- (6,2) -- (6,3)
 (7,1) -- (8,1) -- (8,2) -- (9,2) -- (9,3) -- (10,3)
 (11,1) -- (11,2) -- (14,2) -- (14,3) -- (17,3) -- (17,4)
 (9,-3) -- (10,-3) -- (10,-2) -- (13,-2) -- (13,-1) % a(j)
 -- (15,-1) -- (15,0) -- (17,0) -- (17,1) -- (19,1) -- (19,2) -- (21,2) --
 (21,3) -- (24,3) -- (24,4);
 \draw[\Red,ultra thick,dotted] (3,2) -- (4,2);
 % line
%  \draw[\Black, ultra thick] (0,1) -- (2,5);  
\end{tikzpicture}
\end{minipage}
\vspace{0.5cm}

% ---- fig4.5a ------
\begin{minipage}{20cm*\real{0.7}} \label{fig:4.5a}
\centering Fig. 4.5a\\ 
\tikzstyle{help lines}=[gray,very thin]
\begin{tikzpicture}[scale=0.7]
 \draw[style=help lines]  grid (20,10);
 % axes 
 \draw[thick] (0,0) -- (20,0); 
 %\draw[thick] (0,0) -- (0,10);
 \draw[white] (0,10) -- (0,10.5);

 % ticks
 {
 \pgftransformxshift{0.5cm}
  \draw[<->,thick] (20,0) -- (20,10);
  \draw[anchor=east] (0,0) node[rotate=90] {$c{-}2n$};
  \draw[anchor=east] (1,0) node[rotate=90] {$c{-}2n{+}1$};
  \draw[anchor=east] (10,0) node[rotate=90] {$c{-}2\nu{-}2$};
  \draw[anchor=east] (11,0) node[rotate=90] {$c{-}2\nu{-}1$};
  \draw[anchor=east] (18,0) node[rotate=90] {$c{-}2$};
  \draw[anchor=east] (19,0) node[rotate=90] {$c{-}1$};
  \draw (8,4) node[above] {$-$};
  \draw (10,5) node[above] {$-$};
  \draw (12,6) node[above] {$-$};
  \draw (14,7) node[above] {$-$};
  \draw (16,8) node[above] {$-$};
  \draw (18,9) node[above] {$-$};
 \draw[anchor=base] (10,6) node[above=2pt,fill=white, inner sep=0pt] 
  {$\scriptstyle{n{-}\nu}$};
 \draw[anchor=base] (11,6) node[above=2pt,fill=white, inner sep=0pt] {$\scriptstyle{n{-}\nu}$};
 \draw[anchor=base] (16,9) node[above=2pt,fill=white, inner sep=0pt] {$\scriptstyle{n{-}1}$};
 \draw[anchor=base] (17,9) node[above=2pt,fill=white, inner sep=0pt]
 {$\scriptstyle{n{-}1}$};
 \draw[anchor=base] (18,10) node[above=2pt,fill=white, inner sep=0pt] {$\scriptstyle{n}$};
 \draw[anchor=base] (19,10) node[above=2pt,fill=white, inner sep=0pt] {$\scriptstyle{n}$};
 \draw[anchor=mid] (20,5) node[fill=white, inner sep=2pt] {$n$};
 }
 % steps 
 \draw[\Red,ultra thick] (0,0) -- (0,1) -- (2,1) -- (2,2) -- (4,2) -- (4,3) --
 (6,3) -- (6,4) -- (8,4) -- (8,5) -- (10,5) -- (10,6) -- (12,6) -- (12,7) --
 (14,7) -- (14,8) -- (16,8) -- (16,9) -- (18,9) -- (18,10) -- (20,10) --
 (20,0) -- cycle;
 \end{tikzpicture}
\end{minipage}
\vspace{0.5cm}

% ---- fig4.5b ------
\begin{minipage}{20cm*\real{0.7}} \label{fig:4.5b}
\centering Fig. 4.5b\\ 
\tikzstyle{help lines}=[gray,very thin]
\begin{tikzpicture}[scale=0.7]
 \draw[style=help lines]  grid (19,10);
 % axes 
 \draw[white] (-1,0) -- (0,0);
 \draw[thick] (0,0) -- (19,0); 
 %\draw[thick] (0,0) -- (0,10);
 % ticks
 {
 \pgftransformxshift{0.5cm}
  \draw[<->,thick] (19,0) -- (19,10);
  \draw[anchor=east] (0,0) node[rotate=90] {$c{-}2n{+}1$};
  \draw[anchor=east] (1,0) node[rotate=90] {$c{-}2n{+}2$};
  \draw[anchor=east] (10,0) node[rotate=90] {$c{-}2\nu{-}1$};
  \draw[anchor=east] (11,0) node[rotate=90] {$c{-}2\nu$};
  \draw[anchor=east] (17,0) node[rotate=90] {$c{-}2$};
  \draw[anchor=east] (18,0) node[rotate=90] {$c{-}1$};
  \draw (8,4) node[above] {$-$};
  \draw (10,5) node[above] {$-$};
  \draw (12,6) node[above] {$-$};
  \draw (14,7) node[above] {$-$};
  \draw (16,8) node[above] {$-$};
  \draw (18,9) node[above] {$-$};
 \draw[anchor=mid] (19,4.5) node[fill=white, inner sep=2pt] {$n$};
 }

 % steps 
 \draw[\Red,ultra thick] (0,0) -- (0,1) -- (2,1) -- (2,2) -- (4,2) -- (4,3) --
 (6,3) -- (6,4) -- (8,4) -- (8,5) -- (10,5) -- (10,6) -- (12,6) -- (12,7) --
 (14,7) -- (14,8) -- (16,8) -- (16,9) -- (18,9) -- (18,10) -- (19,10) --
 (19,0) -- cycle;
 \end{tikzpicture}
\end{minipage}

% ---- fig4.6a ------
\begin{minipage}{20cm*\real{0.7}} \label{fig:4.6a}
\centering Fig. 4.6a\\ 
\tikzstyle{help lines}=[gray,very thin]
\begin{tikzpicture}[scale=0.7]
 \draw[style=help lines]  grid (20,10);
 % axes 
 \draw[thick] (0,0) -- (20,0); 
 %\draw[thick] (0,0) -- (0,10);
 \draw[white] (0,10) -- (0,10.5);

 % ticks
 {
 \pgftransformxshift{0.5cm}
 \draw[<->,thick] (20,0) -- (20,10);
  \draw[anchor=east] (0,0) node[rotate=90] {$e{-}2n$};
  \draw[anchor=east] (1,0) node[rotate=90] {$e{-}2n{+}1$};
  \draw[anchor=east] (18,0) node[rotate=90] {$e{-}2$};
  \draw[anchor=east] (19,0) node[rotate=90] {$e{-}1$};
 \draw[anchor=base] (11,6) node[above=2pt,fill=white, inner sep=0pt] 
  {$\scriptstyle{n{-}\nu}$};
 \draw[anchor=base] (12,6) node[above=2pt,fill=white, inner sep=0pt] {$\scriptstyle{n{-}\nu}$};
 \draw[anchor=base] (17,9) node[above=2pt,fill=white, inner sep=0pt] {$\scriptstyle{n{-}1}$};
 \draw[anchor=base] (18,9) node[above=2pt,fill=white, inner sep=0pt] {$\scriptstyle{n{-}1}$};
 \draw[anchor=base] (19,10) node[above=2pt,fill=white, inner sep=0pt] {$\scriptstyle{n}$};
 \draw[anchor=mid] (20,5) node[fill=white, inner sep=2pt] {$n$};
 }
 % steps 
 \draw[\Red,ultra thick] (0,0) -- (0,1) -- (2,1) -- (2,2) -- (4,2) -- (4,3) --
 (6,3) -- (6,4) -- (9,4) -- (9,5) -- (11,5) -- (11,6) -- (13,6) -- (13,7) --
 (15,7) -- (15,8) -- (17,8) -- (17,9) -- (19,9) -- (19,10) -- (20,10) --
 (20,0) -- cycle;
 \end{tikzpicture}
\end{minipage}

% ---- fig4.6b ------
\begin{minipage}{20cm*\real{0.7}} \label{fig:4.6b}
\centering Fig. 4.6b\\ 
\tikzstyle{help lines}=[gray,very thin]
\begin{tikzpicture}[scale=0.7]
 \draw[style=help lines]  grid (19,9);
 % axes 
 \draw[white] (-1,0) -- (0,0);
 \draw[thick] (1,0) -- (19,0); 
 %\draw[thick] (0,0) -- (0,9);

 % ticks
 {
 \pgftransformxshift{0.5cm}
 \draw[<->,thick] (19,0) -- (19,9);
  \draw[anchor=east] (0,0) node[rotate=90] {$e{-}2n{+}1$};
  \draw[anchor=east] (1,0) node[rotate=90] {$e{-}2n{+}2$};
  \draw[anchor=east] (17,0) node[rotate=90] {$e{-}2$};
  \draw[anchor=east] (18,0) node[rotate=90] {$e{-}1$};
  \draw[anchor=mid] (19,4.5) node[fill=white, inner sep=2pt] {$n$};
 }
 % steps 
 \draw[\Red,ultra thick] (0,0) -- (0,1) -- (2,1) -- (2,2) -- (4,2) -- (4,3) --
 (6,3) -- (6,4) -- (9,4) -- (9,5) -- (11,5) -- (11,6) -- (13,6) -- (13,7) --
 (15,7) -- (15,8) -- (17,8) -- (17,9) -- (19,9) -- (19,0) -- cycle;
 \end{tikzpicture}
\end{minipage}

\begin{minipage}{1.0\linewidth}
\centering
% ---- fig4.7a ------
\begin{minipage}[b]{7cm*\real{0.7}} \label{fig:4.7a}
\centering Fig. 4.7a\\ 
\tikzstyle{help lines}=[gray,very thin]
\begin{tikzpicture}[scale=0.7]
 \draw[style=help lines]  grid (7,8);
 % axes 
 %\draw[thick] (0,0) -- (0,8); 
 \draw[thick] (0,0) -- (7,0);
 \draw[thick] (0,0) -- (7,0) -- (7,8);
 % ticks
 {
 \pgftransformxshift{0.5cm}
 \pgftransformyshift{-0.55cm}
 \draw[anchor=base] (6,0) node {$c{-}1$};
 }
 % steps 
 \draw[\Red,ultra thick] (5,0) -- (5,1) -- (7,1) 
(0,0) -- (0,1) -- (1,1) -- (1,2) -- (2,2) -- (2,3) -- (3,3) -- (3,4) -- (4,4)
-- (4,5)  -- (5,5) -- (5,6) -- (6,6) -- (6,7) -- (7,7); 
% line
  \draw[\Black, ultra thick] (0,1) -- (7,8);  
\end{tikzpicture}
\end{minipage}
\hspace{1.5cm}  
% ---- fig4.7b ------
\begin{minipage}[b]{7cm*\real{0.7}} \label{fig:4.7b}
\centering Fig. 4.7b\\ 
\tikzstyle{help lines}=[gray,very thin]
\begin{tikzpicture}[scale=0.7]
 \draw[style=help lines]  grid (7,8);
 % axes 
 %\draw[thick] (0,0) -- (0,8); 
 \draw[thick] (0,0) -- (7,0);
 \draw[thick] (0,0) -- (7,0) -- (7,8);
 % ticks
 {
 \pgftransformxshift{0.5cm}
 \pgftransformyshift{-0.55cm}
 \draw[anchor=base] (6,0) node {$c{-}1$};
 }
 % steps 
 \draw[\Red,ultra thick] (6,0) -- (6,1) -- (7,1) 
 (0,0) -- (0,1) -- (1,1) -- (1,2) -- (2,2) -- (2,3) -- (3,3) -- (3,4) -- (4,4)
 -- (4,5) -- (5,5) -- (5,6) -- (6,6) -- (6,7) -- (7,7);
 % line
  \draw[\Black, ultra thick] (0,1) -- (7,8);  
\end{tikzpicture}
\end{minipage}
\end{minipage}

%%%%%%%%%%%%%%%%%%%%%%%%%%CHAPTER 5%%%%%%%%%%%%%%%%%%%%%%%%%%%%%%%%%%%%%%%%%%%%%%%%%%%%%%%%%%%%%

\chapter{Estimates of the $\alpha$-grade in the case $\rho_1<0, \rho_2>0$.}\label{5}

\section{Preliminary remarks.}\label{5.1}
We refer to Proposition 3.1 in section 3.5 and we treat case (I.1) at first. If
in $f_i$ the vice-monomial $N_j^{\up}$ occurs, then $\cI_k$ is monomial for
all $k\ge j+1$. Especially, $f_{j+1},\cdots f_r$ are monomials, which do not
cause a deformation of the pyramid.

We show that the $z$-degree of such a monomial $N_j^{\up}$ cannot be equal to
the $z$-degree of an initial monomial $M_k$. For then it would follow
$m_j+j-1=m_k+k$ for another index $k$, which is not possible by 
Corollary 2.4. For the same reason it is not possible that the $z$-degree of
$N_j^{\up}\cdot (z/y)$ is equal to the $z$-degree of $N_k^{\down}$ or of
$M_k$. The corresponding statements are true, if ``up'' and ``down'' are
exchanged, as it follows from the corresponding definition in (3.2) and (3.3).

Finally, if there occurs a deformation of the form (1.6) of Proposition 3.1,
then it can be that the final monomial $N_k^{\up}(z/y)$ of $f_i$ has the same
$z$-degree as the initial monomial $M_{\ell}$ of $f_{\ell}$. But then
$\ell>k>j>i$, and therefore $\cI_{\ell}$ is a monomial ideal by Lemma 3.1 .
 But then $f_{\ell}$ does not define a deformation, at all. It follows from
this remarks that one can separately consider the deformations defined by the
different $f_i$, if one wants to determine the changes of $\alpha$-grade
caused by these deformations. (N.B. This statement is analogously valid in the
situation described by Proposition 3.2, too).

At first we determine the change of the $\alpha$-grade, if in one $f_i$ the initial monomial $M_i^{\up}$ is replaces by another monomial occurring in $f_i$:\\
$1^{\circ}$ $M_i^{\up}\longmapsto N_j^{\down}$, if $0\le i<j\le r$;\\
$2^{\circ}$ $M_i^{\up}\longmapsto N_j^{\up}$, if $0\le i<j\le r$;\\
$3^{\circ}$ $M_i^{\up}\longmapsto L, L\in\cL$ monomial such that $(x,y)L\subset \ell\cK(-r-1), 0\le i\le r$;\\
$4^{\circ}$ $M_i^{\up}\longmapsto N_k^{\up}\cdot (z/y)=M_k^{\up}\cdot (z/y)^2, 0\le i<k\le r$;\\
$5^{\circ}$ $M_i^{\up}\longmapsto L\cdot (z/y), L\in\cL$ monomial such that $(x,y)L\subset\ell\cK(-r-1), 0\le i\le r$.

The deformation $4^{\circ}$ (resp. $5^{\circ}$) comes from the case 1.6 (resp.
1.7) of Proposition 3.1, and therefore there is at most one such a deformation,
whereas in the deformations $1^{\circ}$ and $2^{\circ}$ (resp. $3^{\circ}$)
the index $i$ may a priori run through all integers $0\le i<r$ (resp. $0\le
i\le r$). Then for the index $j$ in the cases $1^{\circ}$ and $2^{\circ}$
(resp. for the monomial $L$ in the case $3^{\circ}$) there are several
possibilities. But if one has chosen $i$, then one has to decide for an index
$j$ (resp. for a monomial $L$), and we will give an uniform estimate of the
corresponding changes of $\alpha$-grades.

We identify the set $\cL$, which was introduced in Section (3.3) with the
vector space  generated by the monomials in this set.

We denote by $\cL\cB$ (left domain) the vector space generated by
all monomials in $S_{m_0}$ with $z$-degree $\ge m_0-(c+r)$, i.e., generated by
all monomials $x^ay^bz^{m_0-(a+b)}$, where $a+b\le c+r$.

As to the deformation $4^{\circ}$ (resp. $5^{\circ}$), there is still the
possibility $yM_i^{\up}\mapsto yM_k^{\up}(z/y)^2$ (resp. $yM_i^{\up}\mapsto
yL\cdot (z/y)$). This is because in the case 1.6 (resp. 1.7) of Proposition 3.1,
$f_i$ has the order 1, whereas in the remaining cases $f_i$ has the order 0.

Remember that (cf. Figure 3.1 and 5.1):\\
$M_i^{\up}=x^{i-\iota(i)}y^{m_i+\iota(i)}z^{m_0-m_i-i}$\\
$N_i^{\up}=M_i^{\up}\cdot (z/y)=x^{i-\iota(i)}y^{m_i+\iota(i)-1}z^{m_0-m_i-i+1}$\\
$M_i^{\down}=x^{m_i+i-\iota(i)}y^{\iota(i)}z^{m_0-m_i-i}$\\
$N_i^{\down}=M_i^{\down}\cdot (z/x)=x^{m_i+i-\iota(i)-1}y^{\iota(i)}z^{m_0-m_i-i+1}$\\
$E_k^{\up}:=M_k^{\up}\cdot (z/y)^2=x^{k-\iota(k)}y^{m_k+\iota(k)-2}z^{m_0-m_k-k+2}$\\
$E_k^{\down}:=M_k^{\down}\cdot (z/x)^2=x^{m_k+k-\iota(k)-2}y^{\iota(k)}z^{m_0-m_k-k+2}$.

\section{Estimates in the case I.}\label{5.2}
We determine one after the other the changes of the $\alpha$-grade in the
deformations: \\
$1^{\circ}$ $M_i^{\up}\to N_j^{\down}$.\\
First we note $\varphi'(m_i+i)= m_i + 1$ and $\varphi'(m_j +j - 1 ) = m_j - 1 $ (see Fig.5.2 ).
The $\alpha$-grade of the column, in which $M_i^{\up}$ occurs changes by
$-(m_i+\iota(i))+\varphi'(m_i+i)-1 =- \iota (i)$ (cf. the formula (4.2) in 4.1). The
$\alpha$-grade of the column, to which $N_j^{\down}$ is added, changes by
$\iota (j)-\varphi'(m_j+j-1)= \iota (j) - m_j +1$ ( loc . cit.). Therefore the $\alpha$-grade changes by
$-m_i-\iota(i)+\varphi'(m_i+i)-1+\iota(j)-\varphi'(m_j+j-1)= \iota(j)-\iota(i)-m_j+1$. As $0\leq \iota (i)\leq \iota (j) \leq j$ \ , the absolute value of this difference is $\le\max (j,m_j-1)$, where $0\le i<j\le r$. \\
$2^{\circ}$ $M_i^{\up}\to N_j^{\up}$. \\
The $\alpha$-grade of the column, to which $M_i^{\up}$ belongs changes by
$-(m_i+\iota(i))+\varphi'(m_i+i)-1 =  - \iota (i) $; the $\alpha$-grade of the column, to
which $N_j^{\up}$ is added, changes by $m_j+\iota(j)-1-\varphi'(m_j+j-1) = \iota (j) $.
Therefore the change of $\alpha$-grade is equal to $0\le\iota(j)-\iota(i)\le j$, where $0\le i<j\le
r$. \\
$3^{\circ}$ $M_i^{\up}\longmapsto L\in \cL$.  \\
The $\alpha$-grade of the column, to which $M_i^{\up}$ belongs, changes by $-\iota (i)$. From the
domain $\cL$ the monomial $L$ is removed, such that by Proposition 4.1 one
gets the following estimate of the $\alpha$-grade: The $\alpha$-grade after
the deformation $3^{\circ}$ of that part of the pyramid, which belongs to the
left domain is $\le (c-1)^2+\iota(r+1)[ {c+1\choose 2}-(c-1)]$. For one has
$\# \cL=c$, and there are ${c+1\choose 2}-c$ initial monomials of $\ell
H^0(\cK(c-1))$ in the left domain $\cL\cB$. Therefore the expression in the
bracket gives the number of monomials in the corresponding part of the pyramid
after the deformation. Besides this one has to take into account that the
pyramid is pushed upwards by $\iota(r+1)$ units (cf. Remark 4.4).
 We recall
that $\cL\cB$ (resp. $\cR\cB$) is the vector space generated by the monomials
of total degree $\le c+r$ (resp. of total degree $\ge c+r+1$) in $x$ and $y$
(cf. Fig. 5.3). The change of $\alpha$-grade caused by the deformation
$3^{\circ}$ can be expressed as follows: The change in the domain $\cR\cB$ is
$-\iota(i)$; the $\alpha$-grade of the left domain of the pyramid after the
deformation is estimated as given above. \\
$4^{\circ}$ $M_i^{\up}\longmapsto E_k^{\up}$.\\
At first we consider the case that $E_k^{\up}$ belongs to the right domain,
i.e. $m_k+k-2\ge c+r+1$, and we orientate ourselves by Figure 5.1. The
deformation $4^{\circ}$ changes the $\alpha$-grade of the column of
$M_i^{\up}$ by $-\iota(i)$ (cf. $3^{\circ}$), and the $\alpha$-grade of the
column to which $E_k^{\up}$ is added changes by
$m_k+\iota(k)-2-\varphi'(m_k+k-2)$. As we have remarked above,
$\varphi'(m_k+k)=m_k+1$ and therefore $\varphi'(m_k+k-2)=m_k-2$. Therefore the
$\alpha$-grade of the column of $E_k^{\up} $ changes by $\iota(k)$. Altogether the
deformation $4^{\circ}$ causes a change of $\alpha$-grade by
$0\le\iota(k)-\iota(i)\le k$, where $0\le i<k\le r$. \\
This deformation occurs only once, yet one has to take into account the
deformation $4^{\circ} \itbis \;\, (y/z)M_i^{\up}\mapsto N_k^{\up}$ ( Proposition 
3.1c). In the column of $yM_i^{\up}$ this gives a change of the $\alpha$-grade
by $ -(m_i+\iota (i) +1 ) +\varphi'(m_i+i+1)-1 = - m_i -\iota(i) -1 +m_i +2 -1 = - \iota (i)$.
In the column of $N_k^{\up}$ the $\alpha$-grade changes by
$m_k+\iota(k)-1-\varphi'(m_k+k-1)=m_k+\iota(k)-1-(m_k-1)=\iota(k)$. Altogether
the deformation $4^{\circ}\itbis$ gives a change of $\alpha$-grade by $0\le
\iota(k)-\iota(i)\le k$.

Now to the case $m_k+k-2\le c+r$. Due to the deformation $4^{\circ}$ (resp.
$4^{\circ}\itbis$)
the $\alpha-\grade$ in the right domain of the pyramid changes by
$-\iota(i)$. In any case the deformation $4^\circ$  (resp. $4^\circ\itbis$)
 gives a change of $\alpha$-grade in the right domain
of absolute value $\le r$. \\
$5^{\circ}$ $M_i^{\up}\longmapsto L\cdot (z/y)$.\\
Removing $M_i^{\up}$ (resp. $(y/z)M_i^{\up}$) gives a change of $\alpha$-grade by
$-\iota(i)$ in the corresponding column (cf. case $4^{\circ}$). The changes in
the left domain will be estimated later on.

The deformations $1^{\circ}-5^{\circ}$ exclude each other, i.e., there are at
most $r+1$ such deformations plus two deformations $4^{\circ}\itbis$ and
$5^{\circ}\itbis$. The changes in the right domain can be estimated in the
cases $1^{\circ}$ and $2^{\circ}$ by $\max (j,m_j-1)\le r+m_{i+1}$, where $i$
runs through the numbers $0,\cdots,r-1$. The absolute value of the change in
the case $3^{\circ}$ can be estimated by $r$, and the same is true for the
deformations $4^{\circ}, 4^{\circ}\itbis, 5^{\circ}$ and $5^{\circ}\itbis$.

We now consider possible trinomials.\\
$6^{\circ}$ We assume there is a trinomial of the form 1.3. We want to
determine the change of $\alpha$-grade,if $N_j^{\up}$ is replaced by
$N_k^{\up}$, where we start from a pyramid containing $N_j^{\up}$ instead of
$M_i^{\up}$. The changes of $\alpha$-grade in the following diagram follow
from the computation in $2^{\circ}$.
\[
  \begin{array}{c}
  M_i^{\up}\\
  \iota(j)-\iota(i)\swarrow\qquad\searrow \iota(k)-\iota(i)\\
  N_j^{\up}\quad\stackrel{\delta}{\longrightarrow}\quad N_k^{\up}
  \end{array}
\]
The change of $\alpha$-degree is therefore $0\le\delta:=\iota(k)-\iota(j)\le r$.\\
$7^{\circ}$ The trinomial has the form 1.4.
\[
  \begin{array}{c}
  M_i^{\up}\\
  \qquad\qquad\iota(j)-\iota(i)\swarrow\qquad\searrow \iota(k)-\iota(i)-m_k+1\\
  N_j^{\up}\quad\stackrel{\delta}{\longrightarrow}\quad N_k^{\down}
  \end{array}
\]
(cf. $1^{\circ}$ and $2^{\circ}$) Therefore $\delta=\iota(k)-\iota(j)-m_k+1$, and as in $1^{\circ}$ we obtain the estimate $0\le |\delta|\le\max (k,m_k-1)\le m_k+k<m_{i+1}+r$.\\
$8^{\circ}$ The trinomial has the form (1.5).
\[
  \begin{array}{c}
  \qquad M_i^{\up}\\
  \iota(j)-\iota(i)\swarrow\qquad\searrow -\iota(i)\\
  \; N_j^{\up}\quad\stackrel{\delta}{\longrightarrow}\quad L
  \end{array}
\]
(cf. $3^{\circ}$) Therefore $\delta=-\iota(j)$ and $0\le |\delta|\le r$.\\
$9^{\circ}$ The trinomial has the form (1.6).\\
\[
  \begin{array}{c}
  M_i^{\up}\qquad\\
  \qquad\qquad\iota(j)-\iota(i)\swarrow\qquad\searrow \iota(k)-\iota(i) (\mbox{resp.}-\iota(i))\; \\
  \; N_j^{\up}\quad\stackrel{\delta}{\longrightarrow}\quad N_k^{\up}\cdot (z/y)
  \end{array}
\]
(cf. $4^{\circ}$) It follows $\delta=\iota(k)-\iota(j)$ (resp. $\delta=-\iota(j)$) and therefore $|\delta|\le r$.\\
$10^{\circ}$ The trinomial has the form (1.7).
\[
  \begin{array}{c}
  M_i^{\up}\; \\
  \iota(j)-\iota(i)\swarrow\qquad\searrow -\iota(i)\qquad\quad\\
  N_j^{\up}\quad\stackrel{\delta}{\longrightarrow}\quad L\cdot (z/y)
  \end{array}
\]
(cf. $5^{\circ}$) It follows that $\delta=-\iota(j)$ and $|\delta|\le r$.\\

\textbf{N.B.} Because of $N_j^{\up}\cdot (y/z)=M_j^{\up}$ the cases
$9^{\circ}\itbis$ and $10^{\circ}\itbis$ do not occur.

Summarizing the cases $1^{\circ}-10^{\circ}$ one sees that the total change of
$\alpha$-grade in the right domain has an absolute value $\le
(r+1)r+2r+\sum\limits^r_{i=1}m_i$. In order to formulate this result in a
suitable manner, we have to introduce some notations.\\

We take up the decomposition (Z) of Section (3.1) and choose a standard basis
of $ H^0(\cK(c))$.  Then we multiply the elements in this basis as well as the
forms $f_i$ by monomials in the variables $ x , y ,z $ to obtain a basis of $
T(\rho)$- semi-invariants of $H^0(\cI(n))$ with different initial monomials.
By linearily combining one can achive that the initial monomial of each
semi-invariant does not appear in any other of these semi-invariant, i.e., one
gets a standard basis. (Fig. 3.1 is to show these initial monomials.) The set
of all monomials which occur in this basis
form a pyramid, which is denoted by $\cP$. Here $n\gg 0$, e.g. $n\ge d$.

From each element of the basis we take a monomial the coefficient of
which is different from zero and such that the monomials are different from
each other .Then we compute the $\alpha$-grade of the vector space generated
by these monomials. The maximum and the minimum of the $\alpha$-grades which
one obtains in this way had been denoted by
\[
\max-\alpha-\grade (V)\quad \text{and by} \quad \min-\alpha-\grade (V),
\] 
respectively. One chooses from such a sequence of monomials, which gives the
maximal $\alpha$-grade (which gives the minimal $\alpha$- grade ,
respectively) those monomials the total degree in $x$ and $y$ of which is $\ge
c+(r+1)$. Then one forms the $\alpha$-grade of the subspaces generated by
these monomials and denotes it by $\max-\alpha$-grade $(\cP\cap\cR\cB)$ ( by
$\min-\alpha$-grade $(\cP\cap\cR\cB)$, respectively ). If one chooses from
such sequences of monomials those monomials the total degree in $x$ and $y$ of
which is $\le c+r$, then $\max-\alpha$-grade $(\cP\cap\cL\cB)$ and
$\min-\alpha$-grade $(\cP\cap\cL\cB)$ are defined analogously.

Of course this is valid in the case $\rho_1>0$, too, but the assumption $\rho_2 > 0 $ is essential .\\We make the
\begin{definition}\label{6}
$A:=\max-\alpha$-grade $(\cP\cap\cR\cB)-\min-\alpha$-grade $(\cP\cap\cR\cB)$.
\end{definition}

Then we can formulate the result obtained above as
\begin{conclusion}\label{1}
In the case I.1 one has
\[
A\le r(r+3)+\bigsum\limits^r_{i=1}m_i. \quad \quad \qed
\]
\end{conclusion}
\textbf{N.B.}. If $r=0$ one has actually $A=0$.

Now to the case I.2 (cf. Proposition 3.1, 2nd case).\\
$1^{\circ}$ $M_i^{\up}\longmapsto N_j^{\down}$ gives a change of
$\alpha$-grade of absolute value $\le\max (r,m_{i+1})$, where $0\le i\le r-1$
(cf. the case I.1). \\
$2^{\circ}$ $M_i^{\up}\longmapsto L\in\cL$ gives a change of $\alpha$-grade in
the right domain by $-\iota(i)$ (see above). Further possible deformations are
$yM_i^{\up}\mapsto \in\cL,
y^2M_i^{\up}\mapsto\in\cL,\cdots,y^{\nu }M_i^{\up}\mapsto\in\cL$, so long as
$m_i+i+\nu < m_{i-1}+(i-1)-1$ (cf. Conclusion 3.2 ). This gives
in the column of $yM_i^{\up}$ (of $y^2M_i^{\up},\cdots,y^{\nu}M_i^{\up}$,
respectively) a change of $\alpha$-grade by
\[
-(m_i+\iota(i)+1)+[\varphi'(m_i+i+1)-1]=-(m_i+\iota(i)+1)+[m_i+1]=-\iota(i)
\]
\[(\text{by}
-(m_i+\iota(i)+2)+[\varphi'(m_i+i+2)-1]=-(m_i+\iota(i)+2)+[m_i+2]=-\iota(i),\]
$\cdots,$
\[
-(m_i+\iota(i)+\nu)+[\varphi'(m_i+i+\nu)-1]=-(m_i+\iota(i)+\nu)-[m_i+\nu]=-\iota(i),
\text{respectively})
\]
 as long as $m_i+i+\nu < m_{i-1}+(i-2)$, see above.) This
procedure can be repeated at most $c$ times, until $\cL$ is full. As
$\iota(r)\le r$, the total change of $\alpha$-degree in the right domain
caused by deformations $2^{\circ}$ has an absolute value $\le cr$. If $A$ is
defined as before one gets

\begin{conclusion}\label{2}
In the case I.2 one has
\[
  A\le r(r+c)+\sum^r_{i=1}m_i.\qquad\qquad \Box
\]
\end {conclusion}
\textbf {N.B.} If $r=0$, then one really has $A=0$. For removing $M_0^{\up}=y^{m_0}$ does not change the $\alpha$-grade of the column of $z$-degree 0 (this $\alpha$-grade is equal to 0), or one has $f_0=x^{m_0}$.

\section{The case $r\ge 1$.}\label{5.3}
We start with an ideal $\cI$ of type $r\ge 0$ such that $\ell_0=y$, and we
refer to Proposition 3.1 again. The aim is to prove the inequality
(!) in (4.3), where now $V=H^0(\cI(n))$. In the course of the following
computations it will turn out that $\alpha$-grade $(V)$ is independent of $n$,
if $n$ is sufficient large, e.g. if $n\ge d$.

If one simply writes $\alpha$-grade $(\cI)$ instead of $\alpha$-grade
$(H^0(\cI(n)))$, where $n$ is sufficient large, then one has to show:
\[
Q(m_0-1)+\min-\alpha-\mbox{grade}\; (\cI)>\max-\alpha-\mbox{grade}\;
(\cI)\qquad\qquad \textbf{(!)}  %% Hier tag 
\]

We orientate ourselves by Fig. 5.4. From the Remarks 4.4 , 4.13 , 4.14 and
Proposition 4.1 it follows that
 \[
  \min-\alpha\mbox{-grade} (\cL\cB\cap\cP)\ge \iota(r+1)[{c+1\choose 2}-c]
\]
\[
\max-\alpha\mbox{-grade} (\cL\cB\cap\cP)\le (c-1)^2+\iota(r+1)[{c+1\choose
  2}-(c-s-1)]
\]
 
where $s+1=$ total number of all deformations
$M_i^{\up}\mapsto\in\cL,\cdots,y^{\nu }M_i^{\up}\mapsto\in\cL$, even with
different indices $i$ ($s=-1$ means, there are no such deformations). For
proving {\bf (!)} it is sufficient to show
\begin{equation}
  \label{eq:stern}
  Q(m_0-1)>(c-1)^2+(s+1)\cdot\iota(r+1)+A \tag{*}
\end{equation}
\noindent where $A=\max-\alpha$-grad $(\cP\cap\cR\cB)-\min-\alpha$-grade $(\cP\cap\cR\cB)$ by definition (cf. Section 5.2 for the notations).\\
\textbf{N.B.} If $c=0$ there are no deformations into the left domain of the pyramid.
Therefore the inequality (*) reduces to
\begin{equation} 
  \label{eq:starbis}
  Q(m_0-1)>A.  \tag{*\itbis}  
\end{equation}

These statements are independent of $\rho_1<0$ (Case I) or $\rho_1>0$ (Case II).\\
Unfortunately one has to distinguish these two cases in the following estimates
and we start with Case I.1 (cf. Proposition 3.1).

Because of $0\le s+1\le c, 1\le\iota(r+1)\le r+1$ and the estimates of 

 as well as Lemma 2.8) it is sufficient to show:
  \begin{align}
\tbinom{m_0+1}{2}
  -(c+\sum\limits^r_{i=0}m_i) & >(c-1)^2+c(r+1)+r(r+3)+\sum\limits^r_{i=1}m_i 
\notag \\
  \iff
  \frac{1}{2} m_0(m_0-1)& >c^2+cr+r(r+3)+1 \label{5.1} 
\end{align}
The case $r=0$ has to be treated separately (see below Section 5.4), such that
we can assume $r\ge 1$. If $c=0$, then $m_0\ge 2^{r+1}$ ( Lemma 2.8), and
(5.1) follows from $2^r(2^{r+1}-1)>r(r+3)+1$, which inequality is true
if $r\ge 1$. Therefore we can assume $c>0$. Because of $m_0\ge 2^r(c+2)$ the
inequality \label{5.1} follows from $2^{r-1}(c+2)\cdot
2^r(c+1)>c^2+cr+r(r+3)+1\iff 2^{2r-1}(c^2+3c+2)>c^2+cr+r(r+3)+1$, which is
true if $r\ge 1$ and $c>0$.

\begin{conclusion}\label{3}
In the case I.1 the inequality (*) is valid, if $r\ge 1$.\hfill $\Box$
\end{conclusion}

Now to the case I.2. Because of Conclusion 5.2 one has to replace in the inequality above $r(r+3)$ by $r(r+c)$, i.e. one has to show that $2^{2r-1}(c^2+3c+2)>c^2+2rc+r^2+1$, if $r\ge 1$. This is true, if $c\ge 0$.

\begin{conclusion}\label{4}
In the case I.2 the inequality (*) is valid, if $r\ge 1$.\hfill $\Box$
\end{conclusion}

\section{The case $r=0$.}\label{5.4}
As always we suppose $g^*(\varphi)>g(d)$. As $\ell_0=y$ and $\rho_1<0$ by
assumption, one has $f_0=x^{m_0}$ (cf. Fig. 5.5). Therefore there are no
deformations into $\cL$, i.e. $A=0$ and (*) reads $Q(m_0-1)>(c-1)^2$. If for
simplicity one writes $m$ instead of $m_0$, then one has to show
\begin{equation}\label{2}
  m(m-1)>2c^2-2c+2
\end{equation}
Now $c=$ colength $(\cK)$, $\cK$ an ideal of type $(-1)$ and the following cases can occur:\\
1. If $\psi$ is the Hilbert function of $\cK$, then $g^*(\psi)\le g(c)$. Especially one has $c\ge 5$ and by Corollary 2.2  $m\ge 2c+1$, from which (5.2) follows.\\
2. One has $c\le 4$. If $c=0$, then $\cI$ is monomial and there are no
deformations at all. It follows that $(*\itbis)$ is fulfilled.

A little consideration shows that because of $m\ge c+2$ the inequality (5.2) is fulfilled in the cases $1\le c\le 4$, too. Thus {\bf(*)} and ${\bf(*\itbis)}$ are proved in the case $r=0$, also. Using Conclusion 5.4 gives

\begin{proposition}\label{4}
  Assume that $\cI$ has the type $r\ge 0$ and has $y$-standard form; assume
  $\rho_1<0$ and $\rho_2>0$. Then (*) and $(*\itbis)$ are fulfilled,
  respectively.\hfill $\Box$
\end{proposition}
\textbf{N.B.} Hence the inequality {\bf (!)} follows (see the corresponding argumentation in 5.3).

\newpage

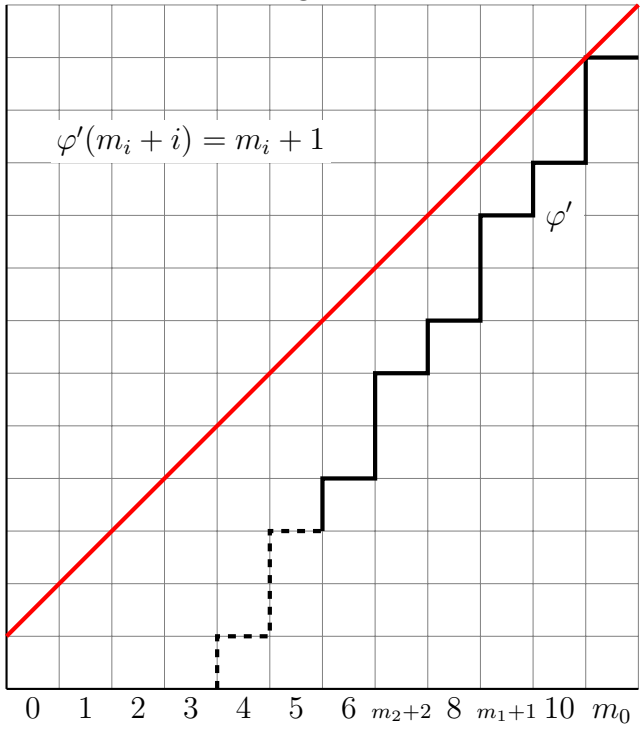
\begin{minipage}{6cm*\real{0.7}} \label{fig:5.1'}
\centering Fig. 5.1\\ 
\tikzstyle{help lines}=[gray,very thin]
\begin{tikzpicture}[scale=0.7]
 \draw[style=help lines]  grid (6,7);
 % axes 
 %  \draw[thick] (0,0) -- (0,7); 
  \draw[thick] (0,0) -- (6,0);
 % ticks
 {
 \pgftransformxshift{0.5cm}
 %\pgftransformyshift{-0.55cm}
  \draw[anchor=east] (1,0) node[rotate=90] {$m_k{+}k{-}2$};
  \draw[anchor=east] (3,0) node[rotate=90] {$m_k{+}k$};
 }
 % ticks
 {
 \pgftransformxshift{0.5cm}
 \draw (3,5) node[above=1pt,inner sep=1pt] {$M_k^\up$};
 \draw (2,4) node[above=1pt,inner sep=1pt] {$N_k^\up$};
 \draw (1,3) node[above,inner sep=1pt] {$E_k^\up$};
}
 % steps 
 \draw[\Red,ultra thick] (1,2) -- (1,3) -- (2,3) -- (2,4) -- (3,4) -- (3,6) -- (4,6);
\end{tikzpicture}
\end{minipage}
%\end{minipage}

% ---- fig5.2 ------
\begin{minipage}{12cm*\real{0.7}} \label{fig:5.2}
\centering Fig. 5.2\\ 
\tikzstyle{help lines}=[gray,very thin]
\begin{tikzpicture}[scale=0.7]
 \draw[style=help lines]  grid (12,13);
 % axes 
 \draw[thick] (0,0) -- (0,13); 
 \draw[thick] (0,0) -- (12,0);
 % ticks
 {
 \pgftransformxshift{0.5cm}
 \pgftransformyshift{-0.55cm}
  \foreach \x in {0,1,2,3,4,5,6,8,10} \draw[anchor=base] (\x,0) node {$\x$}; 
 \draw[anchor=base] (7,0) node {$\scriptstyle{m_2{+}2}$};
 \draw[anchor=base] (9,0) node {$\scriptstyle{m_1{+}1}$};
 \draw[anchor=base] (11,0) node {$m_0$};
 }
 {
 \pgftransformxshift{0.5cm}
  \draw (3,10) node[above,fill=white,inner sep=2pt] {$\varphi'(m_i+i)=m_i+1$};
  \draw (10,8.5) node[above,fill=white] {$\varphi'$};
 }
 % steps 
 \draw[\Red,ultra thick,dashed] (4,0) -- (4,1) -- (5,1) -- (5,3) -- (6,3);
 \draw[\Red,ultra thick] (6,3) -- (6,4) -- (7,4) -- (7,6) -- (8,6) -- 
 (8,7) -- (9,7) -- (9,9) -- (10,9) -- (10,10) -- (11,10) -- (11,12) -- (12,12);
 % line
  \draw[\Black, ultra thick] (0,1) -- (12,13);  
\end{tikzpicture}
\end{minipage}

\newpage
% ---- fig5.3 ------
\begin{minipage}{22cm*\real{0.7}} \label{fig:5.3}
\centering Fig. 5.3\\ 
\tikzstyle{help lines}=[gray,very thin]
\begin{tikzpicture}[scale=0.7]
 \filldraw[fill=gray,opacity=0.5] 
   (10,4) -- (10,16) -- (11,16) -- (11,4) -- cycle;
 \draw[style=help lines]  grid (22,20);
 % axes 
 \draw[thick] (0,0) -- (22,0);
 \draw[->] (0,-2.2) -- (9.8,-2.2);
 \draw[<-] (10.2,-2.2) -- (22,-2.2);
 \draw[thick,->] (12,11.2) -- (10.4,11.2);  % arrow for text monomials
 % ticks
 {
 \pgftransformxshift{0.5cm}
 \pgftransformyshift{-0.55cm}
 \draw[anchor=base] (12,0) node {$m_r{+}r$};
 \draw[anchor=base] (19,0) node {$m_0$};
 \draw[anchor=base] (7,-2) node[fill=white] {left domain};
 \draw[anchor=base] (14,-2) node[fill=white] {right domain};
 }
 {
 \pgftransformxshift{0.5cm}
 \draw[anchor=east] (9,0) node[rotate=90] {$c{+}r$};
 \draw[anchor=east] (10,0) node[rotate=90] {$c{+}r{+}1$};
 \draw (15,10) node[above=1pt,fill=white,text width=5.5cm]%,inner sep =0pt] 
      {column with $c+1$ monomials is complete};
 }
 % line
  \draw[\Black, ultra thick] (0,12) -- (8,20);  
  \draw[gray, ultra thick] (10,-2.5) -- (10,20);  % (10,20.5)  
 % steps 
 \draw[\Red,ultra thick] (19,0) -- (19,1);
% \draw[\Red,ultra thick,dashed] (19,1) -- (12,1);
 \draw[\Red,ultra thick,dashed] (19,1) -- (17,1) -- (12,4);
 \draw[\Red,ultra thick]   % (12,1) -- 
 (12,4) -- (10,4) -- (10,5) -- (9,5) -- (9,6) -- (8,6) --(8,7) -- (6,7) --
 (6,8) -- (7,8) -- (7,9) -- (8,9) -- (8,10) -- (9,10) -- (9,12) -- (8,12) --
 (8,13) -- (9,13) -- (9,15) -- (10,15) -- (10,16) -- (11,16) -- (11,17) --
 (12,17) -- (12,19) -- (13,19);
\end{tikzpicture}\\[1cm]
\end{minipage}

\noindent
$\#\{ \text{ monomials in the left domain }\}
  = h^0(\cK(c-1))=\tbinom{c+1}{2}-c$  \\
 as $\cK$ has the Hilbert polynomial $\tbinom{n+2}{2}-c$.\\[2mm]
$ \#\{ \text{ monomials in the right domain }\} = Q(m_0-1)- 
\#\{ \text{ monomials in the left domain }\}$\\
$ = \tbinom{m_0+1}{2}-d \left[\tbinom{c+1}{2}-c\right]$\\[2mm]
%right domain shifted down up to the $y$-degree $1$\,; left domain remains 

\newpage
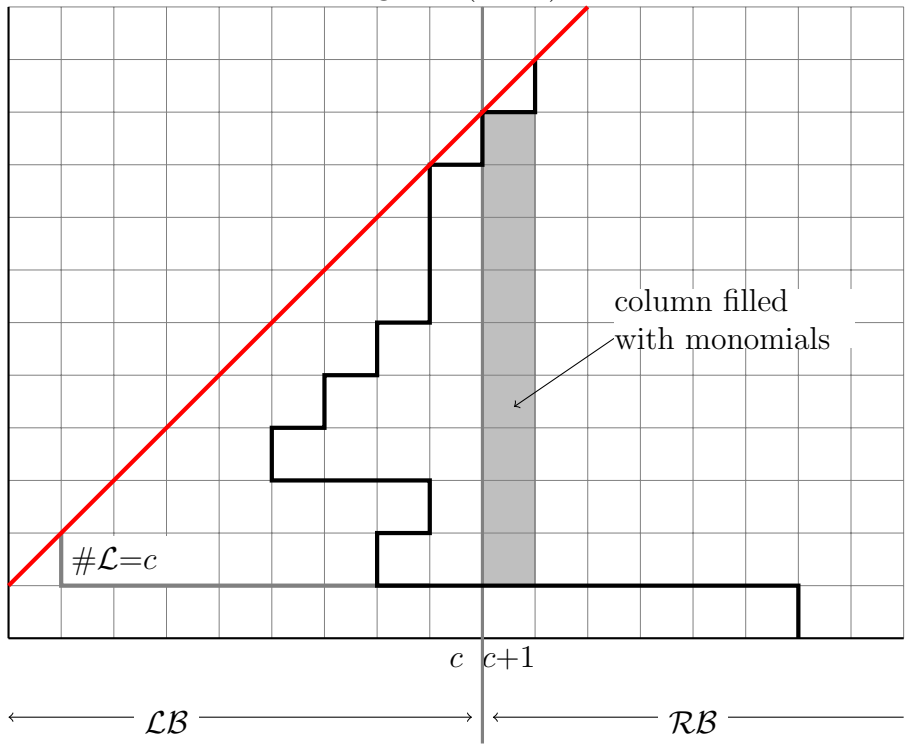
\begin{minipage}{1.0\linewidth}
% ---- fig5.4a ------
\begin{minipage}[b]{10cm*\real{0.7}} \label{fig:5.4a}
\centering Fig. 5.4a\\ 
\tikzstyle{help lines}=[gray,very thin]
\begin{tikzpicture}[scale=0.7]
 \draw[style=help lines]  grid (10,10);
 % axes 
 \draw[thick] (0,0) -- (0,10); 
 \draw[thick] (0,0) -- (10,0);
 % ticks
 {
 \pgftransformxshift{0.5cm}
 \pgftransformyshift{-0.55cm}
  \foreach \x in {0,1,2,3,4,5,7,8} \draw[anchor=base] (\x,0) node {$\cdots$}; 
 \draw[anchor=base] (6,0) node {$\scriptstyle{m_1{+}1}$};
 \draw[anchor=base] (9,0) node {$m_0$};
 }
 % steps 
 \draw[\Red,ultra thick] (9,0) -- (9,1) -- (2,1) -- (2,2) -- (3,2) -- (3,3) --
 (4,3) -- (4,4) -- (5,4) -- (5,5) -- (6,5) -- (6,7) -- (7,7) -- (7,8) -- (8,8)
  -- (8,9) -- (9,9);
 % line
  \draw[\Black, ultra thick] (0,1) -- (9,10);  
\end{tikzpicture}
\end{minipage}
\hfill
% ---- fig5.4b ------
\begin{minipage}[b]{10cm*\real{0.7}} \label{fig:5.4b}
\centering Fig. 5.4b\\ 
\tikzstyle{help lines}=[gray,very thin]
\begin{tikzpicture}[scale=0.7]
 \draw[style=help lines]  grid (10,10);
 % axes 
 \draw[thick] (0,0) -- (0,10); 
 \draw[thick] (0,0) -- (10,0);
 % ticks
 {
 \pgftransformxshift{0.5cm}
 \pgftransformyshift{-0.55cm}
  \foreach \x in {0,1,2,3,4,5,7,8} \draw[anchor=base] (\x,0) node {$\cdots$}; 
 \draw[anchor=base] (6,0) node {$\scriptstyle{m_1{+}1}$};
 \draw[anchor=base] (9,0) node {$m_0$};
 }
 % steps 
 \draw[\Red,ultra thick] (9,0) -- (9,1) -- (6,1) -- (6,2) -- 
 (2,2) -- (2,3) -- (3,3) -- (3,4) -- (4,4) -- (4,5) -- (5,5) -- (5,6) -- (6,6)
 -- (6,7) -- (7,7) -- (7,8) -- (8,8) -- (8,9); % -- (9,9);
 % line
  \draw[\Black, ultra thick] (0,1) -- (9,10);  
\end{tikzpicture}
\end{minipage}
\end{minipage}

\vspace{3cm}
%
% ---- fig5.5 ($r=0$) ------
\begin{minipage}{17cm*\real{0.7}} \label{fig:5.5 ($r=0$)}
\centering Fig. 5.5 ($r=0$)\\ 
\tikzstyle{help lines}=[gray,very thin]
\begin{tikzpicture}[scale=0.7]
\filldraw[fill=gray,opacity=0.5] (9,1) -- (10,1) -- (10,10) -- (9,10) -- cycle;
 \draw[style=help lines]  grid (17,12);
 % axes 
 \draw[thick] (0,0) -- (0,12); 
 \draw[thick] (0,0) -- (17,0);
 \draw[<->] (0,-1.5) -- (8.8,-1.5); 
 \draw[<-]  (9.2,-1.5) -- (17,-1.5); 
 \draw[->] (11.5,5.7) -- (9.6,4.4);  % arrow for text monomials
 % ticks
 {
 \pgftransformxshift{0.5cm}
  \pgftransformyshift{-0.55cm}
 \draw[anchor=base] (8,0) node {$c$};
 \draw[anchor=base] (9,0) node {$c{+}1$};
 }
 {
 \draw[anchor=south west] (11.5,5.5) 
      node[fill=white,,text width = 3.2cm,inner sep=0pt]
      {column filled with monomials};
 \draw[anchor=west] (2,1) node[above,fill=white] {$\#\cL{=}c$};
 \draw (3,-2) node[above,fill=white] {$\cL\cB$};
 \draw (13,-2) node[above,fill=white] {$\cR\cB$};
 }
 % line
 \draw[gray, very thick] (9,-2) -- (9,12);  
 % steps 
 \draw[gray,ultra thick] (1,2) -- (1,1) -- (8,1);
 \draw[\Red,ultra thick] (15,0) -- (15,1) -- (7,1) -- (7,2) -- (8,2) -- (8,3)
 -- (5,3) -- (5,4) -- (6,4) -- (6,5) -- (7,5) -- (7,6) -- (8,6) -- (8,9) --
 (9,9) -- (9,10) -- (10,10) -- (10,11); % -- (11,11);
 % line
  \draw[\Black, ultra thick] (0,1) -- (11,12);  
\end{tikzpicture}
\end{minipage}

%%%%%%%%%%%%%%%%%%%%%%%%%%CHAPTER 6%%%%%%%%%%%%%%%%%%%%%%%%%%%%%%%%%%%%%%%%%%%%

\chapter{Estimates of the $\alpha$-grade in the case $\rho_1>0, \rho_2>0$ and $r\ge 1$.}\label{6}

We refer to Proposition 3.2 in (3.6) and take over the notations from there. As has been remarked in (5.1) one can compute the changes of $\alpha$-grade by the single deformations $f_i$ separately.

\section{Estimates in the case II}\label{6.1}

At first we compute the changes of the $\alpha$-grade , if we replace the initial
monomial $M_i$ of $f_i$ by another monomial occurring in $f_i$ (cf. 5.1): \\
$1^{\circ}$ $M_i^{\down}\longmapsto N_j^{\up}$, if $0\le i<j\le r$. \\
In the column of $M_i^{\down}$ there is a change of the $\alpha$-grade by
$-\iota(i)+\varphi'(m_i+i)-1=m_i-\iota(i)$. In the column of $N_j^{\up}$  there is a
change of the $\alpha$-grade by
$m_j+\iota(j)-1-\varphi'(m_j+j-1)=m_j+\iota(j)-1-(m_j-1)=\iota(j)$. Therefore
the deformation $1^{\circ}$ gives a change of the $\alpha$-grade by
$m_i+\iota(j)-\iota(i), 0\le i<j\le r$. \\
$2^{\circ}$ $M_i^{\down}\longmapsto N_j^{\down}, 0\le i<j\le r$.\\
In the column of $M_i^{\down}$ there is a change of the $\alpha$-grade by $m_i-\iota(i)$; in
the column of $N_j^{\down}$ is a change of $\alpha$-grade by
$\iota(j)-\varphi'(m_j+j-1)=\iota(j)-m_j+1$. Therefore the deformation
$2^{\circ}$ gives a change of $\alpha$-grade, whose absolute value is
$|\iota(j)-\iota(i)+m_i-m_j+1|\le\max (|m_i+\iota(j)|,\, |m_j+\iota(i)-1|)\le
m_i+j$, where $0\le i<j\le r$. \\
$3^{\circ}$ $M_i^{\down}\longmapsto L\in\cL$.\\
In the column of $M_i^{\down}$ is a change of $\alpha$-grade by $0<m_i-\iota(i)\le
m_i+i,0\le i\le r$. The change of $\alpha$-grade in the left domain can be
estimated as in Case I. \\
$4^{\circ}$ $M_i^{\down}\longmapsto E_k^{\down}=M_k^{\down}\cdot (z/x)^2,0\le
i<k\le r$. \\
At first we consider the case that $E_k^{\down}$ belongs to the right domain,
i.e. $m_k+k-2\ge c+r+1$. The change of $\alpha$-grade in the column of
$M_i^{\down}$ is $m_i-\iota(i)$; the change of $\alpha$-grade in the column of
$E_k^{\down}$ is $\iota(k)-\varphi'(m_k+k-2)=\iota(k)-(m_k-2)$. Therefore the
absolute value of the change of $\alpha$-grade by the deformation $4^{\circ}$
is $|m_i-\iota(i)+\iota(k)-m_k+2|\le\max (|m_i+\iota(k)|,\,
|m_k+\iota(i)-2|)=m_i+\iota(k)\le m_i+k$, if $0\le i<k\le r$. This deformation
can occur only once, yet one has to take into account the deformation\\
$4^{\circ}\itbis\quad  (x/z) M_i^{\down}\mapsto N_k^{\down}$ (cf. Proposition 3.2c). \\
In the column of $xM_i^{\down}$ this gives a change of the $\alpha$-grade by
$-\iota(i)+\varphi'(m_i+i+1)-1=m_i-\iota(i)+1$. In the column of $N_k^{\down}$
the $\alpha$-grade changes by $\iota(k)-\varphi'(m_k+k-1)=\iota(k)-m_k+1$.
Thus the absolute value of the change of $\alpha$-grade in the right domain
due to $4^{\circ}\itbis$ is $|m_i-\iota(i)+1+\iota(k)-m_k+1)\le\max
(|m_i+\iota(k)|,\, |m_k+\iota(i)-2|)=m_i+\iota(k)\le m_i+k$, where $0\le
i<k<r$. This deformation occurs only once.\\
Now to the case that $m_k+k-2\le c+r$. Removing $M_i^{\down}$ ( resp. $ (x/z)M_i^{\down}$) gives a change of $\alpha$-grade by $m_i-\iota(i)$ ( by $m_i-\iota(i)+1$, respectively), whose absolute value is bounded by $m_i+r$.\\
$5^{\circ}$ $M_i^{\down}\longmapsto L\cdot (z/x)$.\\
Removing $M_i^{\down}$ (resp. $(x/z)M_i^{\down}$) causes a change of
$\alpha$-grade of the column of $M_i^{\down}$ (resp. $(x/z)M_i^{\down}$) by
$m_i-\iota(i)$ (resp. by $m_i-\iota(i)+1$), which are estimated by $m_i+i$
(resp. $m_i+i+1$), where $0\le i\le r$. The deformation $5^{\circ}$ (resp.
$5^{\circ}\itbis$) can occur only once. The changes in the left domain will be
estimated later on.

The deformation $1^{\circ}-5^{\circ}$ exclude each other, i.e. there are at
most $r+1$ such deformation plus two deformations $4^{\circ}\itbis$ and
$5^{\circ}\itbis$. The changes of $\alpha$-grade in the right domain in the
cases $1^{\circ}-3^{\circ}$ have an absolute value $\le m_i+r,0\le i\le r$.
The same estimate is valid for the deformations $4^{\circ}$ and
$4^{\circ}\itbis$, even if $E_k$ belongs to the left domain, as we have
assumed $r\ge 1$. As for the deformations $5^{\circ}$ (resp.
$5^{\circ}\itbis$) we estimate the change of the $\alpha$-grade by $m_i+r$
(resp. $m_i+r+1$).

We now consider possible trinomials.\\
$6^{\circ}$ We assume there is a trinomial of the form 1.3. Similarly as in
the Case I in Chapter 5, we have a diagram
\[
  \begin{array}{c}
    M_i^{\down}\\
    \gamma(j)\swarrow\qquad\searrow \gamma(k)\\   %%%%%%%%%%%%%%%%% Abstände sind alle ok!!!!! %%%%%%%%%%
    N_j^{\down}\quad\stackrel{\delta}{\longrightarrow}\quad N_k^{\down}
  \end{array}
\]
where $\gamma(j):=m_i-m_j-\iota(i)+\iota(j)+1$ and $\gamma(k):=m_i-m_k-\iota(i)+\iota(k)+1$ (cf. $2^{\circ}$). Therefore $\delta=m_j-m_k-\iota(j)$. It follows that $|\delta|\le\max (|m_j+\iota(k)|,\,|m_k+\iota(j)|)\le m_i+r$.\\

$7^{\circ}$ The trinomial has the form 1.4.
\[
  \begin{array}{c}
    M_i^{\down}\\
    \gamma(j)\swarrow\qquad\searrow \beta\\
    N_j^{\down}\quad\stackrel{\delta}{\longrightarrow}\quad N_k^{\up}
  \end{array}
\]
where $\beta:=m_i-\iota(i)+\iota(k)$ (cf. $1^{\circ}$ and $2^{\circ}$). It follows that $\delta=m_j-\iota(j)+\iota(k)-1$ and $|\delta|\le m_i+r$.\\

$8^{\circ}$ The trinomial has the form 1.5.
\[
  \begin{array}{c}
  M_i^{\down}\\
  \gamma(j)\swarrow\qquad\searrow \beta \\
  N_j^{\down}\quad\stackrel{\delta}{\longrightarrow}\quad L
  \end{array}
\]
where $\beta:=m_i-\iota(i)$ (cf. $3^{\circ}$). It follows that $\delta=m_j-\iota(j)-1$ and $|\delta|\le m_i+r$.\\

$9^{\circ}$ The trinomial has the form 1.6.
\[
   \begin{array}{c}
     M_i^{\down}\\
     \gamma(j)\swarrow\qquad\searrow \beta \\
     N_j^{\down}\quad\stackrel{\delta}{\longrightarrow}\quad N_k\cdot (z/x)
   \end{array}
 \]
where $\beta:=m_i-m_k-\iota(i)+\iota(k)+2$, respectively $\beta:=m_i-\iota(i)$
 (cf. $4^{\circ}$). It follows that $\delta=m_j-m_k-\iota(j)+\iota(k)+1$
 (resp. $\delta=m_j-\iota(j)-1$) and $|\delta|\le \max
 (|m_j-\iota(k)|,|m_k+\iota(j)-1|)\le m_j+r$. \\
 
 $10^{\circ}$ The trinomial has the form 1.7.
\[
  \begin{array}{c}
  M_i^{\down}\\
  \gamma(j)\swarrow\qquad\searrow m_i-\iota(i)\\       %%%%%%%% Abstände sind alle ok! %%%%%%%%%%%%%%%%%%%%
  N_j^{\down}\quad\stackrel{\delta}{\longrightarrow}\quad L\cdot (z/x)
  \end{array}
\]
(cf. $5^{\circ}$). It follows that $\delta=m_j-\iota(j)-1$ and $|\delta|<m_i+r$.\\

\textbf{Notabene}.  Because of $N_j^{\down}\cdot (x/z)=M_j^{\down}$ the case
$9^{\circ}\itbis$ or $10^{\circ}\itbis$ does not occur.

Summarizing the cases $1^{\circ}-10^{\circ}$ one sees that the total change of
$\alpha$-grade in the right domain has an absolute value $\le
(r+1)r+\sum\limits^r_{i=0}m_i$. If one estimates the changes in the cases
$4^{\circ}\itbis$ and $5^{\circ}\itbis$ by $m_i+r$ and $m_i+r+1$,
respectively, one obtains $A\le r(r+3)+1+3m_0+\sum\limits^r_{i=1}m_i$. As we
have assumed that $r\ge 1$, we have $m_0\geq c+2+m_r+\cdots+m_1$ ( Lemma 2.8)
 and we obtain

\begin{conclusion}\label{1}
 If $r\ge 1$ is assumed,in the case II.1 one has $A\le
4m_0+r(r+3)-c-1$.\qed
\end{conclusion}
Now we come to the case II.2 (cf. Proposition 3.2, 2nd case).\\

$1^{\circ}$ $M_i^{\down}\longmapsto N_j^{\up}$ again gives a change of the
$\alpha$-grade in the right domain, whose absolute value is $\le m_i+r, 0\le
i\le r-1$ (see above). \\
$2^{\circ}$ $M_i^{\down}\longmapsto L\in\cL$ gives a change of $\alpha$-grade
in the right domain by $m_i-\iota(i), 0\le i\le r$ (see above). Further
possible deformations are $xM_i^{\down}\mapsto\in\cL,
x^2M_i^{\down}\mapsto\in\cL,\cdots,x^{\nu}M_i^{\down}\mapsto\in\cL$, as long as
$m_i+i+\nu < m_{i-1}+(i-2)$ (cf. Conclusion 3.2). This gives in the column of 
$xM_i^{\down}$ (of $x^2M_i^{\down},\cdots,x^{\nu}M_i^{\down}$,
respectively) a change of $\alpha$-grade by
$-\iota(i)+\varphi'(m_i+i+1)-1=-\iota(i)+(m_i+2)-1=m_i-\iota(i)+1$ (by
$-\iota(i)+\varphi'(m_i+i+2)-1=m_i-\iota(i)+2,\cdots,-\iota(i)+\varphi'(m_i+i+\nu)-1=m_i-\iota(i)+\nu$,
respectively), as long as $m_i+i+\nu < m_{i-1}+ (i-2)$ and $\nu \leq c-1$.

\begin{remark}\label{2}
 One has $|m_i-\iota(i)|\le m_0$ for all $0\le i\le r$.
\end{remark}

\begin{proof}
  The inequality $-\iota(i)+m_i\le m_i\le m_0$ is true, and $\iota(i)-m_i\le
  i-m_i\le r-m_i\le r\le m_0$ is true if $r=0$. If $r\ge 1$ one has $m_0\ge
  c+2+m_r+\cdots +m_1>r$ (Lemma 2.8).
\end{proof}

From this we conclude: Replacing $M_i^{\down},
xM_i^{\down},\cdots,x^{\nu(i)}M_i^{\down}$ by monomials in $\cL$, even with
different indices $i$ as long as $\nu(i)\le c-1$, gives a change of
$\alpha$-grade in the right domain whose absolute value is $\leq\sum\limits^{\nu(i)}_{j=0}(m_0+j)$,
 because $|m_i-\iota(i)+j|\le m_0+j$ by the remark above . One gets
 $A\le\sum\limits^{r-1}_{i=0}(m_i+r)+\sum\limits^r_{i=0}\sum\limits^{\nu(i)}_{j=0}(m_0+j)$.
As $\nu(0)+1+\cdots +\nu(r)+1\le c$, it follows that
$A\le\sum\limits^{r-1}_{i=0}(m_i+r)+\sum\limits^{c-1}_{j=0} (m_0+j)$. As
$m_0\ge c+2+m_r+\cdots+m_1$ and $m_r\ge c+2$, we obtain
$\sum\limits^{r-1}_{i=1}m_i\le m_0-2(c+2)$. This estimate is valid if $r\ge
1$. In the case $r=0$ one only has the deformations
$M_0^{\down}\mapsto\in\cL,\cdots,x^sM_0^{\down}\mapsto\in\cL$, and $s$ can be
estimated as in ( Proposition 3.2).

If $M_0^{\down}$ occurs, the $\alpha$-grade in the column of
$M_0^{\down},\cdots,x^sM_0^{\down}$ increases by\\ $m_0,\cdots,m_0+s$,
respectively.
\begin{conclusion}\label{2}
 In the case II.2 one has $A\le (c+1)m_0+r^2-2(c+2)+{c\choose 2}$, if $r\ge
 1$. If $r=0$, if $\cI$ has $y$-standard form and $\kappa:=\reg(\cK)$, then
 $A\le (s+1)m_0+{s+1\choose 2}$, where
 $s\le\kappa/\rho_2-m_0(1/\rho_2-1/(\rho_1+\rho_2))$.
 \end{conclusion}\qed

\section{The case $r\ge 2$.}\label{6.2}
We recall that we started from an ideal $\cI$ of type $r\ge 0$ with
$y$-standard form, and the aim was to show the inequalities
\[
  Q(m_0-1)>(c-1)^2+(s+1)\iota (r+1)+A\quad\textbf{ (*)}
\]
and
\[
  Q(m_0-1)>A\quad(\textbf{*}\it{bis})
\]
respectively (cf. Section 5.3).

At first we treat the case II.1, where $A\le 4m_0+r(r+3)-c-1$, if $r\ge 1$.

\noindent{\bf Auxiliary Lemma 1.} If $\cI$ has the type $r=2$ (the type $r=1$,
respectively), then $m_0\ge 14$ ( $m_0\ge 7$, respectively).

\begin{proof}
  We use the results of Lemma 2.8 and write in the case
  $r=2:\cI=\cI_0=y\cI_1(-1)+f_0\cO_{\P^2}(-m_0),\cI_1=\ell_1\cI_2(-1)+f_1\cO_{\P^2}(-m_1),\cI_2=\ell_2\cK(-1)+f_2\cO_{\P^2}(-m_2)$.
  $\cI_2$ has the type 0, therefore colength $(\cI_2)=c+m_2\ge 5$ and it
  follows that $m_2\ge 5-c$. As $\cI_1$ has the type 1, we get $m_1\ge
  m_2+c+2\ge 7$. Because of $m_0\ge c+2+m_2+m_1$ it follows that $m_0\ge
  c+2+5-c+7$.
\end{proof}

To begin with, let $c=0$. Then $(*\itbis)$ reads
$\frac{1}{2}m_0(m_0+1)-(m_r+\cdots +m_0)>4m_0+r(r+3)-1$. Because of
$m_0-(c+2)\ge \sum\limits^r_1m_i$ it is sufficient to prove
\begin{equation}\label{1}
  m_0(m_0-11)>2r(r+3)-6
\end{equation}
If $r=2$ (respectively $r=3$) the right side of (6.1) is equal to 14 (equal to
30, respectively). As $m_0\ge 14$ by the Auxiliary Lemma 1 ($m_0\ge
2^3(c+2)=16$ by Lemma 2.8, respectively), (6.1) is true in these cases.

Let be $r\ge 4$. Then $m_0\ge 2^{r+1}$ and (6.1) follows from
$2^r(2^{r+1}-11)>r(r+3)$. Now $2^r>r$ if $r\ge 2$, and $2^{r+1}-11>r+3$ is
equivalent to $2^{r+1}>r+14$, which is true if $r\ge 4$.

Thus we can assume $c\ge 1$ in the following. Because of $1\le \iota(r+1)\le
r+1,0\le s+1\le c,\sum\limits^r_{i=1}m_i\le m_0-(c+2)$ the inequality $(*)$
will follow from: 
\begin{align}
     & \frac{1}{2}m_0(m_0+1)-(2m_0-2)>(c-1)^2+c(r+1)+4m_0+r(r+3)-c-1\notag  \\
\iff &  m_0(m_0-11)>2c^2+2(r-2)c+2r(r+3)-4 \label{(2)}
\end{align}
Because of $m_0\ge 2^r(c+2)$ it suffices to show:
\begin{equation}\label{3}
  2^r(c+2)[2^rc+2^{r+1}-11]>2c^2+2(r-2)c+2r(r+3)-4
\end{equation}

If $r=2$ this inequality reads $4(c+2)(4c-3)>2c^2+16\Leftrightarrow
7c^2+10c-20>0$ and this is true if $c\ge 2$.

In the case $r=2, c=1$, the inequality (6.2) reads $m_0(m_0-11)>18$, which is
true because $m_0\ge 14$ (cf. Auxiliary Lemma 1). Therefore we can now suppose
without restriction $r\ge 3 , c\ge 1$. But $2^{r+1}>11$ and thus (6.3) follows
from $2^r(c+2)\cdot 2^rc > 2c^2 + 2cr + 2r(r+3)$ which is equivalent to :
\begin{equation}\label{4} 
(2^{2r}-2)c^2+(2^{2r+1}-2r)c>2r(r+3)
\end{equation}

The left side of (6.4) is a monotone function of $c$, and if $c=1$, then (6.4) reads $2^{2r}-2+2^{2r+1}-2r>2r(r+3)\Leftrightarrow 2^{2r}+2^{2r+1}>2r^2+8r+2\Leftrightarrow 2^{2r-1}+2^r>r^2+4r+1$.\\
This is true if $r\ge 3$. Summarizing all subcases we obtain
\begin{conclusion}\label{3}
 In the case II.1 the inequality $(*)$ is fulfilled for all $r\ge2$.
\end{conclusion}

We now consider the case II.2 and assume $r\ge 2$. With the help of Conclusion
 6.2 and the estimates $\sum\limits^r_{i=1}m_i\le m_0-(c+2), s+1\le c,
\iota(r+1)\le r+1$ one sees that $(*)$ follows from
\[
  \frac{1}{2} m_0(m_0+1)-(2m_0-2)\ge (c-1)^2+c(r+1)+(c+1)m_0+r^2-2(c+2)+{c\choose 2}.
\]
A simple computation shows that this is equivalent to
\begin{equation}\label{5}
  m_0(m_0-2c-5)>3c^2+2cr-7c-10+2r^2.
\end{equation}
Now we have $m_0\ge 2^r(c+2)\ge 4(c+2)$. If $c=0$, then (6.5) follows from
$2^{r+1}> -10 +r^2 $,which is true for all $r \geq 2 $. Therefore we can assume $ c > 0 $. Then (6.5) follows from $2^r(c+2)(2c+3)>3c^2+2cr+2r^2\Leftrightarrow 2^r(2c^2+7c+6)>3c^2+2cr+2r^2$
which is true for all $c\ge 1$ and $r\ge 2$. We get
\begin{conclusion}\label{4}
 In the case II.2 the inequality (*) is fulfilled for all $r\ge 2$.
\end{conclusion}

\section{The case $r=1$.}\label{6.3}
Then
$\cI=y\cI_1(-1)+f_0\cO_{\P^2}(-m_0),\cI_1=\ell_1\cK(-1)+f_1\cO_{\P^2}(-m_1)$,
where $\cI_1$ has the type 0 and $\cK$ has the type $-1$.

\subsection{}\label{(6.3.1)} We start with the case II.1 of Proposition 3.2 .

{\it Subcase 1:\/} $\ell_1=y$: Then one has the situation shown in Figure 6.1 and there are the following possibilities (case II 1.5 and II 1.7, respectively):\\
$1^{\circ}$ $f_0=x^{m_0}+\alpha N_1^{\down}+\beta L, L\in\cL$ monomial such that $(x,y)L\subset \ell \cK(-2)$.\\
$2^{\circ}$ $f_0=x^{m_0}+\alpha N_1^{\down}+\beta L\cdot (z/x), L\in\cL$
monomial such that $(x,y)L\subset\ell\cK(-2)$.

We treat the case $1^{\circ}$. At first, one has the possibility
$x^{m_0}\mapsto N_1^{\down}$. The $\alpha$-grade of the column of $x^{m_0}$
changes by $m_0$; the $\alpha$-grade of the column of $N_1^{\down}$ changes by
$\iota(1)-\varphi'(m_1)=1-(m_1-1)=2-m_1$. Therefore the change of
$\alpha$-grade in the right domain is $m_0-m_1+2$. The deformation
$x^{m_0}\mapsto L$ gives a change of $\alpha$-grade by $m_0$ in the right
domain. As the order of $f_0$ is equal to 0 in the case II 1.5 (cf.
Proposition 3.2c), there are no other changes of $\alpha$-grade caused by $f_0$.

By Proposition 3.2 again, it follows that $f_1$ has the form of case
II.1.5, where $\alpha=0$, and the order of $f_1$ is equal to 0. The
deformation $M_1^{\down}\mapsto\in\cL$ gives a change of $\alpha$-grade by
$-\iota(1)+\varphi'(m_1+1)-1=-\iota(1)+m_1=m_1-1$. Thus in the case
$1^{\circ}$ one has $A\le\max (m_0-m_1+2,m_0)+m_1-1=m_0+m_1-1$, because
$m_1\ge c+2$ (Lemma 2.4). \\
$2^{\circ}$ At first, $f_0$ defines a deformation as in the case $1^{\circ}$
and gives a change of $\alpha$-grade $\le\max (m_0,m_0-m_1+2)=m_0$ in the
right domain. But as $f_0$ has the order $\le 1$ by Proposition 3.2c, there
is still the possibility $x^{m_0+1}\mapsto\in\cL$, which gives a change of
$\alpha$-grade by $m_0+1$ in the right domain. As $f_1$ again has the same
form as in the case $1^{\circ}$, it follows that $A\le 2m_0+m_1$.

Because of $s+1\le c$ the inequality \textbf{(*)} follows from
\[
  \frac{1}{2} m_0(m_0+1)-(c+m_0+m_1)>(c-1)^2+2c+2m_0+m_1.
\]
As $m_1\le m_0-(c+2)$, this inequality follows from
$m_0(m_0-9)>2c^2-2c-6$. Because of $m_0\ge 2c+4$ it suffices to show
$(2c+4)(2c-5)>2c^2-2c-6\Leftrightarrow 2c^2>14$. This inequality is fulfilled, if
$c\ge 3$. The cases $c=0,1,2$ will be treated later on (see below). \\
{\it Subcase 2:\/} $\ell_1=x$: Then from Figure 6.2 we conclude that only the
second case of Proposition 3.2 can occur, and we have

\begin{conclusion}\label{5}
 In the case II.1 the inequality (*) is fulfilled except if
  $\ell_0=y,\ell_1=y$ and $c=\{0,1,2\}$.
\end{conclusion}
  
\subsection{}\label{(6.3.2)}We now treat the case II.2 of Proposition 3.2).\\

{\it Subcase 1:} $\ell_1=y$: Figure 6.1 shows that only the case II.2.3 is
possible. Then there are $s+1$ deformations
$M_0\mapsto\in\cL,\cdots,x^sM_0\mapsto\in\cL$ ( and $t+1$ deformations
$M_1\mapsto\in\cL,\cdots,x^tM_1\mapsto\in\cL$, respectively). The changes of
$\alpha$-grade in the columns of $M_0,\cdots,x^sM_0$ (of $M_1,\cdots,x^tM_1$,
respectively) is $m_0,\cdots,m_0+s $ ( and $ m_1-1,m_1,\cdots,m_1+t-1$, respectively).
Here $s$ and $t$ fulfil the inequalities of ( Proposition 3.2d). Thus the
total change of $\alpha$-grade in the right domain fulfils the inequality:
$A\le (s+1)m_0+{s+1\choose 2}+(t+1)m_1+{t\choose 2}-1$, where
\[
  s\le\frac{\kappa}{\rho_2}-m_0(\frac{1}{\rho_2}-\frac{1}{\rho_1+\rho_2})+\frac{1}{\rho_2} \qquad\mbox{and}
\]
\[
  t\le\frac{\kappa}{\rho_2}-m_1(\frac{1}{\rho_2}-\frac{1}{\rho_1+\rho_2}).
\]
{\it Estimate of $s$:\/} Because of $\kappa\le c$ and $m_0\ge 2(c+2)$ one obtains :
\[
s\le\frac{c}{\rho_2}-2(c+2)(\frac{1}{\rho_2}-\frac{1}{\rho_1+\rho_2})+\frac{1}{\rho_2}
\]
\[
  =\frac{c+1}{\rho_2}-2(c+1)(\frac{1}{\rho_2}-\frac{1}{\rho_1+\rho_2})-2(\frac{1}{\rho_2}-\frac{1}{\rho_1+\rho_2})
\]
\[
  =(c+1)  (\frac{2}{\rho_1+\rho_2}-\frac{1}{\rho_2})-2(\frac{1}{\rho_2}-\frac{1}{\rho_1+\rho_2})
\]

We first consider the possibility $s\ge 0$. This implies $\frac{2}{\rho_1+\rho_2}-\frac{1}{\rho_2}>0$, i.e. $\rho_2>\rho_1$.

Let be $f_a(x)=\frac{2}{x+a}-\frac{1}{x}$; i.e. $x$ corresponds to $\rho_2$
and $a$ corresponds to $\rho_1$, therefore $1\le
a<x$. $f'_a(x)=-2/(x+a)^2+1/x^2<0\Leftrightarrow 2x^2>(x+a)^2\Leftrightarrow
\sqrt{2}x>x+a\Leftrightarrow x>a(1+\sqrt{2})$. It follows that $f_a(x)$ has
the maximum for $x=a(1+\sqrt{2})$, and
$f_a(a(1+\sqrt{2}))=\frac{0.171\cdots}{a}$. Therefore $s\le 0.172(c+1)$. \\

{\it Estimate of $t$:} Because of $m_1\ge c+2$ (cf. Lemma 2.4) one has
\begin{align*}
 t&\le\frac{c}{\rho_2}-(c+2)\left(\frac{1}{\rho_2}-\frac{1}{\rho_1+\rho_2}\right)\\
  &
  =\frac{c}{\rho_2}-c\left(\frac{c}{\rho_2}-\frac{1}{\rho_1+\rho_2}\right)
  -2\left(\frac{1}{\rho_2}-\frac{1}{\rho_1+\rho_2}\right) \\
 & =\frac{c}{\rho_1+\rho_2}-2\left(\frac{1}{\rho_2}-\frac{1}{\rho_1+\rho_2}\right)
\end{align*}
Therefore $t\le c/2$, if $\rho_2\le\rho_1$ and $t\le c/3$, if $\rho_2>\rho_1$.

{\it First possibility:\/} $\rho_2\le\rho_1$: Then $s<0$, i.e. there are no deformations defined by $f_0$, and $A\le (c/2+1)m_1+{c/2\choose 2}-1$.

{\it Second possibility:\/} $\rho_1<\rho_2$. Then $s\le 0.172 (c+1), t\le c/3$ and $A\le (0.172c+1,172)m_0+{0.172c+1,172\choose 2}+(c/3+1)m_1+{c/3\choose 2}-1$.

As $m_1\le m_0-(c+2)$ (Lemma 2.8), one obtains the following estimates:

{\it First possibility:\/} $A\le (c/2+1)[m_0-(c+2)]+c^2/8-1\Rightarrow A\le
(0.5c+1)m_0-3/8c^2-2c-3$

{\it Second possibility:\/} $A\le (0.172c+1,172)m_0+0.5(0.172c+1,172)^2+(c/3+1)[m_0-(c+2)]+c^2/18-1\Rightarrow A\le (0.51c+2,172)m_0-0.25c^2-1,46c-2,3$

As we have assumed that $\ell_0=y,\ell_1=y$ it follows that
$\iota(r+1)=\iota(2)=2$. As mentioned above, only the case II.2.3 of 
Proposition 3.2 can occur. If $c=0$, there are no deformations at all, so that
one can assume $c\ge 1$. Replacing $s+1$ by $c$, one sees that it suffices to
show:
\begin{equation}\label{6}
  Q(m_0-1)>c^2+1+A
\end{equation}

{\it First possibility:\/} $A\le (0.5c+1)m_1+{0.5c\choose 2}-1$ (see above). One sees that it suffices to show: \\
$\frac{1}{2}m_0(m_0+1)-(c+m_0+m_1)>c^2+1+(0.5c+1)m_1+\frac{1}{8}c^2-1\;\iff\,\frac{1}{2}m_0(m_0+1)-m_0-(0.5c+2)m_1>\frac{9}{8}c^2+c$.\\
Because of $m_1\le m_0-(c+2)$ this follows from $\frac{1}{2}m_0(m_0+1)-(0.5c+3)m_0+(0.5c+2)(c+2)>\frac{9}{8}c^2+c\;\iff\;m_0(m_0-c-5)>1,25c^2-4c-8$.\\
As $m_0\ge 2c+4$ this follows from
$(2c+4)(c-1)>1,25c^2-4c-8\;\iff\;0.75c^2+6c+4>0$. \\
This inequality is fulfilled for all $c$.

{\it Second possibility:\/} $A\le (0.51c+2,172)m_0-0.25c^2-1,46c-2,3$. Putting this into (6.6), one has to show $m_0(m_0+1)-2(c+m_0+m_1)>1,5c^2-2,92c-2,6+(1.02c+4,344)m_0$.\\
Because of $m_1\le m_0-(c+2)$ this follows from $m_0(m_0+1)-2(2m_0-2)-(1,02c+4,344)m_0>1,5c^2-2,92c-2,6\;\iff\; m_0(m_0-1,02c-7,344)>1,5c^2-2,92c-6,6$.\\
As $m_0\ge 2c+4$ this follows from $(2c+4)(0.98c-3,344)>1,5c^2-2,92c-6,6\;\iff\, 0.46c^2+0.152c-6,776>0$.\\
This inequality is fulfilled if $c>3,676$ and the cases $c=0,1,2,3$ will be
treated later on in (6.3.3).

{\it Subcase 2:\/} $\ell_1=x$. Figure 6.2 shows that in Proposition 3.2 
$f_1=M_1^{\up}$ has to be a monomial and that for $f_0$ one of the two cases
2.1 or 2.3 can occur. We first treat the case 2.1, i.e., one has the
deformation $x^{m_0}\mapsto N_1^{\up}$. The $\alpha$-grade of the column of
$x^{m_0}$ changes by $m_0$ and the $\alpha$-grade of the column of $N_1^{\up}$
changes by 1. Therefore $A\le m_0+1$. There are no further deformations in the
right domain. Now in the inequality (*) one has $s=0$ and $\iota(2)=1$ and
therefore has to show: \\
$\frac{1}{2}m_0(m_0+1)-(c+m_0+m_1)>(c-1)^2+1+m_0+1$. Because of $m_1\le
m_0-(c+2)$ this follows from $m_0(m_0-5)>2(c-1)^2$. As $m_0\ge 2c+4$ this
follows from $(2c+4)(2c-1)>2(c-1)^2\iff 2c^2+10c-6>0$. This inequality is
fulfilled, if $c\ge 1$. In the case $c=0$ one has $(*\itbis)$, i.e., one has
to prove $Q(m_0-1)>A$, i.e. to prove
$\frac{1}{2}m_0(m_0+1)-(c+m_1+m_0)>m_0+1$. \\
One sees that this follows from $m_0(m_0-5)>-2$. As one has $m_0\ge 7$ by (6.2 Auxiliary Lemma 1), this inequality is fulfilled.

\indent Now we treat the case 2.3, that means $f_0=M_0+F$ as in 
Proposition 3.2. The only possible deformations are $M_0\mapsto\in\cL,
xM_0\mapsto\in\cL,\cdots,x^sM_0\mapsto\in\cL$. As $c=0$ implies that $\cI$ is
monomial, we can assume $c>0$.

We again distinguish two cases:

{\it First possibility:\/} $\rho_2\le\rho_1$. Then $s+1=0$, i.e. there is no
deformation at all, and therefore one has $A=0$. Then \textbf{(*)} reads
$\frac{1}{2}m_0(m_0+1)-(c+m_0+m_1)>(c-1)^2$. Because of $m_1\le m_0-(c+2)$
this follows from $m_0(m_0-3)>2c^2-4c-2$. Because of $m_0\geq 2c+4$ it suffices
to show $(2c+4)(2c+1)>2c^2-4c-2\iff 2c^2+14c+6>0$ which is true for all $c$.

{\it Second possibility:\/} $\rho_1<\rho_2$. Then $s\le 0,172(c+1)$, as was
shown above. We have already remarked at the beginning of (6.3.2) that the
$s+1$ deformations noted above give a change of $\alpha$-grade $A\le
(s+1)m_0+{s+1\choose 2}$. It follows that $A\le
(0.172c+1,172)m_0+0.5(0.172c+1,172)^2$. \\
In order to prove \textbf{(*)} it suffices to show
$Q(m_0-1)>(c-1)^2+(0.172c+1,172)[1+m_0+\\0.5\cdot (0.172c+1,172)]$. Because of
$m_1\le m_0-(c+2)$ this follows from
$m_0(m_0+1)-2(2m_0-2)>2(c-1)^2+(0.344c+2,344)(m_0+0.086c+1,586)\iff
m_0(m_0-0.344c-5,344)>2,029584c^2-3,252832c+1,717584$. \\
Now $c\ge 1$ and $m_0\ge 2c+4$ so that it suffices to show $(2c+4)(1,656c-1,344)> 2,029584c^2-3,252832c+1,717584\iff 1,282416c^2+7,188832c-7,093584>0$.\\
This is true if $c\ge 1$. Therefore the inequality \textbf{(*)}is fulfilled in Subcase 2 .

\begin{conclusion}\label{6}
In the case II, if $r=1$, then (*) is fulfilled except in the case II.1 if
  $\ell_0=y,\ell_1=y$ and $c\in \{0,1,2\}$ or in the case II.2 if $\ell_0=y,
  \ell_1=y$ and $c\in \{ 0,1,2,3\}$.
\end{conclusion}

\subsection{\textbf{The cases} $0\le c\le 4$}\label{(6.3.3)}

 \fbox{$c=0$} From Figure 6.3 it follows that $M_0\mapsto N_1$ is the only possible deformation, which is the case II.1. The change of $\alpha$-grade is $A=m_0+m_1-2(m_1-1)=m_0-m_1+2$ and the inequality $(*\itbis)$ reads $Q(m_0-1)>m_0-m_1+2\iff m_0(m_0-3)>4$, and this is true, as $m_0\ge 7$ by (6.2 Auxiliary Lemma 1).\\

\fbox{$c=1$} At first, we note that $\cK$ is monomial. From Figure 6.4 it follows that there are three deformations, which can occur, with the ``total'' change of $\alpha$-grade $B$:\\
1. $M_0\mapsto L, B=m_0+2$\\
2. $M_0\mapsto N_1$ and $M_1\mapsto L,B=(m_0-m_1+2)+(2m_1-m_1-1)+2=m_0+3$\\
3. $M_1\mapsto L,B=m_1+1$.\\
Then $(*\itbis)$ follows from $Q(m_0-1)>m_0+3\iff\frac{1}{2}m_0(m_0+1)-(1+m_0+m_1)>m_0+3$. As $m_1\le m_0-3$ (Lemma 2.8), one sees that it suffices to show $m_0(m_0-5)>2$. This is true, because of $m_0\geq 2(c+2)=6$ (loc.cit.).\\

\fbox{$c=2$} Then $\cK$ is monomial, and the possible deformations are shown
in Figure 6.5 $a,b$. If the case II.1 occurs, then $f_0=m_0+\alpha N_1+\beta
F$ and $f_1=M_1+\gamma L$. One sees that the change of $\alpha$-grade in the
right domain becomes maximal, if the deformations $M_0\mapsto F$ and
$M_1\mapsto L$ occur. Therefore $A\le m_0+2m_1-m_1-1=m_0+m_1-1$. The
inequality (*) follows from $Q(m_0-1)>1^2+2\cdot
2+(m_0+m_1-1)\iff\frac{1}{2}m_0(m_0+1)-(2+m_0+m_1)>4+m_0+m_1\iff
\frac{1}{2}m_0(m_0+1)-2m_0-2m_1>6$. Because of $m_1\le m_0-(c+2)=m_0-4$ it
suffices to show $\frac{1}{2}m_0(m_0+1)-4m_0>-2\iff m_0(m_0-7)>-4$. This is
true as $m_0\geq 2c+4=8$.

\indent If the Figure 6.5a occurs, then only the case $1^{\circ}$ of (6.3.1) is
possible and the change of $\alpha$-grade in the right domain is $A\le
m_0+m_1-1$. One gets the same estimate of $A$ if Figure 6.5a occurs in the
case II.2, because $xF$ and $xL$ are elements of $\cK$ and thus the order of
$f_0$ and of $f_1$ is equal to 0. Then (*) reads $Q(m_0-1)>1^2+2\cdot
2+(m_0+m_1-1)$, and this is true as was shown just before.

\indent If Figure 6.5b occurs, then the order of $f_0$ can be equal to 1, but
it is not possible that $f_1=M_1+\alpha F$, where $\alpha \neq 0$, because
$\cK$ is monomial and $F\notin\cK$, whereas $f_1\in\cK$ by Lemma 2.6.

We want to sharpen the estimate of (6.3). $M_0\mapsto F$ and $xM_0\mapsto L$
yield $A=2m_0+1; M_0\mapsto F,M_1\mapsto L$ yield $A=m_0+m_1-1$ and
$M_0\mapsto N_1,M_1\mapsto L$ yield $A=(m_0-m_1+2)+m_1-1$. Therefore $A\le
2m_0+1$, and (*) reads $\frac{1}{2}m_0(m_0+1)-(2+m_0+m_1)>1^2+2\cdot
2+2m_0+1$. Because of $m_1\le m_0-4$ it suffices to show $m_0(m_0-7)>8$. As
$m_0\ge 2c+4=8$, the case $c=2, m_1=4, m_0=8$ remains (cf. Fig. 6.5c). One
sees that the deformations $M_0\mapsto F,xM_0\mapsto L$ give the maximal
change of ``total'' $\alpha$-grade $B=(8+9)+1+2=20$. As $Q(7)={9\choose
  2}-14=22$ the inequality (!) in (5.3) is fulfilled. \\
  
\fbox{$c=3$} Because of Conclusion 6.6 one has only to consider
deformations of the form $f_0=M_0+F,f_1=M_1+G$, where $F,G\in\cL$ (cf. 
Proposition 3.2). From the computations in (6.3.2) it follows that the order of
$f_0$ is $\le 0.172(3+1)<1$ and the order of $f_1$ is $\le 1$. Therefore the
change of $\alpha$-grade in the right domain is $A\le m_0+(m_1-1)+m_1$. If one
replaces $m_1$ by $m_0-5\ge m_1$ in the inequality (*) of (5.3), then one has
to show: \\
$\frac{1}{2}m_0(m_0+1)-(2m_0-2)>2^2+3\cdot 2+m_0+2(m_0-5)-1\iff
m_0(m_0-9)>-6$. This is true because of $m_0\ge 2\cdot (3+2)=10$ (cf. 
Lemma 2.8).

\begin{conclusion}\label{7}
  Even in the cases $c\in\{0,1,2,3\}$ the inequalities \textbf{(*)} and
  \textbf{(!)} respectively, are fulfilled.\
\end{conclusion}

Summarizing the Conclusions 6.1--6.7, we obtain:
\begin{proposition}\label{5}

If $\rho_1>0,\rho_2>0, \cI$ has the type $r\ge 1$ and has $y$-standard form, then the inequality (!) of Section (4.3) is fulfilled.
\end{proposition}

\newpage
% ---- fig6.1 ------
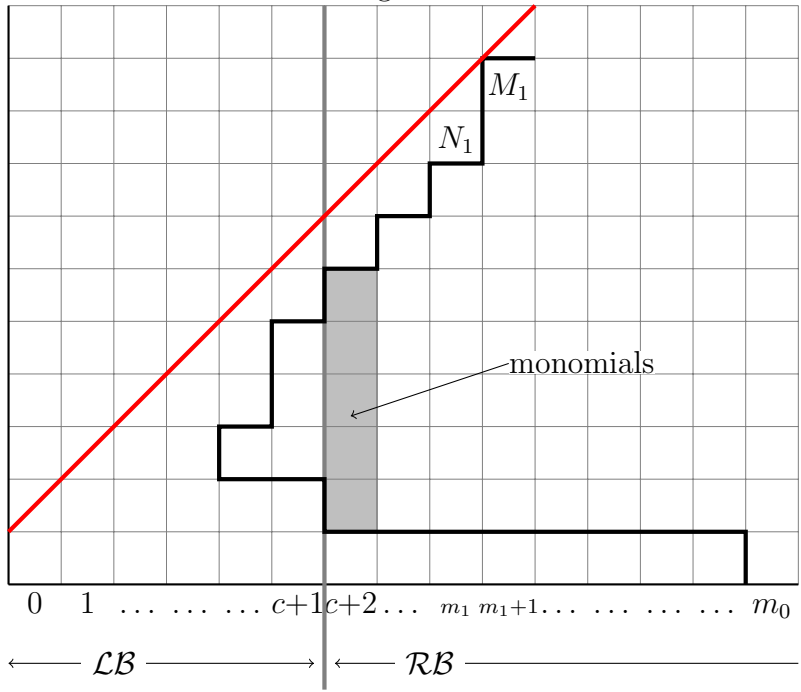
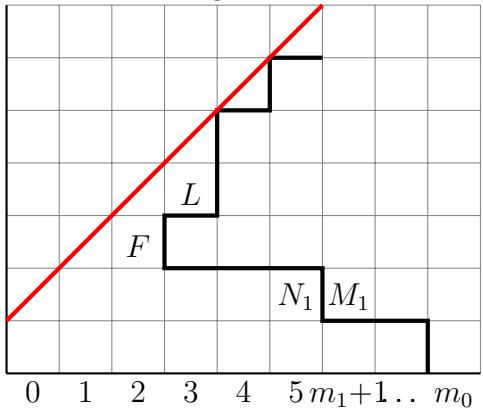
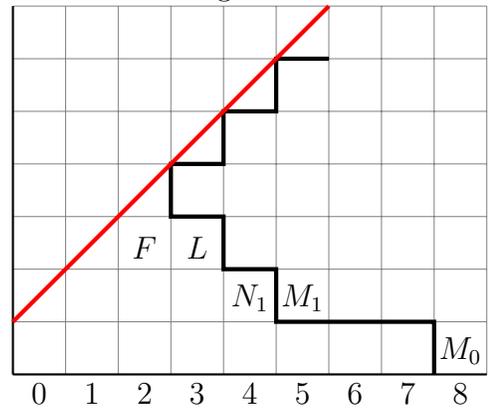
\begin{minipage}{15cm*\real{0.7}} \label{fig:6.1}
\centering Fig. 6.1\\ 
\tikzstyle{help lines}=[gray,very thin]
\begin{tikzpicture}[scale=0.7]
\filldraw[fill=gray,opacity=0.5] (6,2) -- (7,2) -- (7,7) -- (6,7) -- cycle;
 \draw[style=help lines]  grid (15,10);
 % axes 
 \draw[thick] (0,0) -- (0,10); 
 \draw[thick] (0,0) -- (15,0);
 \draw[<->] (0,-2) -- (5.8,-2); 
 \draw[<-]  (6.2,-2) -- (15,-2); 
 \draw (2,-2) node[fill=white] {$\cL\cB$};
 \draw (8,-2) node[fill=white] {$\cR\cB$};
 \draw[->] (8.5,5.2) -- (6.5,4.2);  % arrow for text monomials
 % ticks
 {
 \pgftransformxshift{0.5cm}
 \pgftransformyshift{-0.55cm}
  \foreach \x in {0,1,2,3,4,5,6,7,10,11,12,13} 
   \draw[anchor=base] (\x,0) node {$\scriptstyle{\x}$}; 
 \draw[anchor=base] (14,0) node {$m_0$};
 \draw[anchor=base] (5,-1) node {$c$};
 \draw[anchor=base] (6,-1) node {$c{+}1$};
 \draw[anchor=base] (8,0) node {$\scriptstyle{m_1}$};
 \draw[anchor=base] (9,0) node {$\scriptstyle{m_1{+}1}$};
 }
 {
 \pgftransformxshift{0.5cm}
 \draw (8,1) node[above] {$N_1$};
 \draw (9,1) node[above] {$M_1$};
 \draw (14,0) node[above] {$M_0$};
 \draw[anchor=south west] (8,5) node[fill=white,inner sep=2pt] {monomials};
 }
 % line
  \draw[gray, ultra thick] (6,-2.5) -- (6,10);  
 % steps 
  \draw[\Red,ultra thick] (14,0) -- (14,1) -- (9,1) -- (9,2) -- (6,2) -- (6,3)
  -- (4,3) -- (4,4) -- (5,4) -- (5,6) -- (6,6) -- (6,7) -- (7,7) -- (7,8) --
  (8,8) -- (8,9) -- (9,9);
 % line
  \draw[\Black, ultra thick] (0,1) -- (9,10);  
\end{tikzpicture}
\end{minipage}

\vspace{1.5cm}
% ---- fig6.2 ------
\begin{minipage}{15cm*\real{0.7}} \label{fig:6.2}
\centering Fig. 6.2\\ 
\tikzstyle{help lines}=[gray,very thin]
\begin{tikzpicture}[scale=0.7]
\filldraw[fill=gray,opacity=0.5] (6,1) -- (7,1) -- (7,6) -- (6,6) -- cycle;
 \draw[style=help lines]  grid (15,11);
 % axes 
 \draw[thick] (0,0) -- (0,11); 
 \draw[thick] (0,0) -- (15,0);
 \draw[<->] (0,-1.5) -- (5.8,-1.5); 
 \draw[<-]  (6.2,-1.5) -- (15,-1.5); 
 \draw (2,-1.5) node[fill=white] {$\cL\cB$};
 \draw (8,-1.5) node[fill=white] {$\cR\cB$};
 \draw[->] (9.5,4.2) -- (6.5,3.2);  % arrow for text monomials
 % ticks
 {
 \pgftransformxshift{0.5cm}
 \pgftransformyshift{-0.55cm}
  \foreach \x in {0,1}
   \draw[anchor=base] (\x,0) node {$\x$}; 
  \foreach \x in {2,3,4,7,10,11,12,13}
   \draw[anchor=base] (\x,0) node {$\dots$}; 
 \draw[anchor=base] (14,0) node {$m_0$};
 \draw[anchor=base] (5,0) node {$c{+}1$};
 \draw[anchor=base] (6,0) node {$c{+}2$};
 \draw[anchor=base] (8,0) node {$\scriptstyle{m_1}$};
 \draw[anchor=base] (9,0) node {$\scriptstyle{m_1{+}1}$};
 }
{
 \pgftransformxshift{0.5cm}
 \draw (8,8) node[above] {$N_1$};
 \draw (9,9) node[above] {$M_1$};
 \draw[anchor=south west] (9,4) node[fill=white,inner sep=0pt] {monomials};
}
 % line
  \draw[gray, ultra thick] (6,-2) -- (6,11);  
 % steps 
  \draw[\Red,ultra thick] (14,0) -- (14,1) -- (6,1) -- (6,2) -- (4,2) -- (4,2)
  -- (4,3) -- (5,3) -- (5,5) -- (6,5) -- (6,6) -- (7,6) -- (7,7) -- (8,7) --
  (8,8) -- (9,8) -- (9,10) -- (10,10);
 % line
  \draw[\Black, ultra thick] (0,1) -- (10,11);  
\end{tikzpicture}
\end{minipage}
\newpage
% ---- fig6.3 ------
\begin{minipage}{9cm*\real{0.7}} \label{fig:6.3}
\centering Fig. 6.3\\ 
\tikzstyle{help lines}=[gray,very thin]
\begin{tikzpicture}[scale=0.7]
 \draw[style=help lines]  grid (9,7);
 % axes 
 \draw[thick] (0,0) -- (0,7); 
 \draw[thick] (0,0) -- (9,0);
 % ticks
 {
 \pgftransformxshift{0.5cm}
 \pgftransformyshift{-0.55cm}
  \foreach \x in {0,1,2,3,4} \draw[anchor=base] (\x,0) node {$\x$}; 
 \draw[anchor=base] (5,0) node {$\scriptstyle{m_1}$};
 \draw[anchor=base] (6,0) node {$\scriptstyle{m_1{+}1}$};
 \draw[anchor=base] (7,0) node {$\dots$};
 \draw[anchor=base] (8,0) node {$m_0$};
 }
 {
 \pgftransformxshift{0.5cm}
\draw (8,0) node[above] {$M_0$};
\draw (5,1) node[above] {$N_1$};
\draw (6,1) node[above] {$M_1$};
 }
 % steps 
 \draw[\Red,ultra thick] (8,0) -- (8,1) -- (6,1) -- (6,2) -- 
(2,2) -- (2,3) -- (3,3) -- (3,4) -- (4,4) -- (4,5) -- (5,5) -- (5,6) -- (6,6);
 % line
  \draw[\Black, ultra thick] (0,1) -- (6,7);  
\end{tikzpicture}
\end{minipage}

% ---- fig6.4 ------
\begin{minipage}{9cm*\real{0.7}} \label{fig:6.4}
\centering Fig. 6.4\\ 
\tikzstyle{help lines}=[gray,very thin]
\begin{tikzpicture}[scale=0.7]
 \draw[style=help lines]  grid (9,7);
 % axes 
 \draw[thick] (0,0) -- (0,7); 
 \draw[thick] (0,0) -- (9,0);
 % ticks
 {
 \pgftransformxshift{0.5cm}
 \pgftransformyshift{-0.55cm}
  \foreach \x in {0,1,2} \draw[anchor=base] (\x,0) node {$\x$}; 
 \foreach \x in {3,4,5,7} \draw[anchor=base] (\x,0) node {$\dots$};
 \draw[anchor=base] (6,0) node {$m_1{+}1$};
 \draw[anchor=base] (8,0) node {$m_0$};
 }
 {
 \pgftransformxshift{0.5cm}
\draw (2,2) node[above] {$L$};
\draw (5,1) node[above] {$N_1$};
\draw (6,1) node[above] {$M_1$};
 }
 % steps 
 \draw[\Red,ultra thick] (8,0) -- (8,1) -- (6,1) -- (6,2) -- 
(3,2) -- (3,4) -- (4,4) -- (4,5) -- (5,5) -- (5,6) -- (6,6);
 % line
  \draw[\Black, ultra thick] (0,1) -- (6,7);  
\end{tikzpicture}
\end{minipage}

% ---- fig6.5a ------
\begin{minipage}{9cm*\real{0.7}} \label{fig:6.5a}
\centering Fig. 6.5a\\ 
\tikzstyle{help lines}=[gray,very thin]
\begin{tikzpicture}[scale=0.7]
 \draw[style=help lines]  grid (9,7);
 % axes 
 \draw[thick] (0,0) -- (0,7); 
 \draw[thick] (0,0) -- (9,0);
 % ticks
 {
 \pgftransformxshift{0.5cm}
 \pgftransformyshift{-0.55cm}
  \foreach \x in {0,1,2,3,4,5} \draw[anchor=base] (\x,0) node {$\x$}; 
 \foreach \x in {7} \draw[anchor=base] (\x,0) node {$\dots$};
 \draw[anchor=base] (6,0) node {$m_1{+}1$};
 \draw[anchor=base] (7,0) node {$\dots$};
 \draw[anchor=base] (8,0) node {$m_0$};
 }
 {
 \pgftransformxshift{0.5cm}
\draw (2,2) node[above] {$F$};
\draw (3,3) node[above] {$L$};
\draw (5,1) node[above] {$N_1$};
\draw (6,1) node[above] {$M_1$};
 }
 % steps 
 \draw[\Red,ultra thick] (8,0) -- (8,1) -- (6,1) -- (6,2) -- (3,2) -- 
 (3,3) -- (4,3) -- (4,5) -- (5,5) -- (5,6) -- (6,6);
 % line
  \draw[\Black, ultra thick] (0,1) -- (6,7);  
\end{tikzpicture}
\end{minipage}

\begin{minipage}{15cm}
% ---- fig6.5b ------
\begin{minipage}[b]{9cm*\real{0.7}} \label{fig:6.5b}
\centering Fig. 6.5b\\ 
\tikzstyle{help lines}=[gray,very thin]
\begin{tikzpicture}[scale=0.7]
 \draw[style=help lines]  grid (9,7);
 % axes 
 \draw[thick] (0,0) -- (0,7); 
 \draw[thick] (0,0) -- (9,0);
 % ticks
 {
 \pgftransformxshift{0.5cm}
 \pgftransformyshift{-0.55cm}
  \foreach \x in {0,1,2,3,4,5,6,7,8} \draw[anchor=base] (\x,0) node {$\x$}; 
 }
 {
 \pgftransformxshift{0.5cm}
\draw (2,2) node[above] {$F$};
\draw (3,2) node[above] {$L$};
\draw (8,0) node[above] {$M_0$};
\draw (5,1) node[above] {$N_1$};
\draw (6,1) node[above] {$M_1$};

 }
 % steps 
 \draw[\Red,ultra thick] (8,0) -- (8,1) -- (6,1) -- (6,2) -- 
(4,2) -- (4,3) -- (3,3) -- (3,4) -- (4,4) -- (4,5) -- (5,5) -- (5,6) -- (6,6);
 % line
  \draw[\Black, ultra thick] (0,1) -- (6,7);  
\end{tikzpicture}
\end{minipage}
\hfill
% ---- fig6.5c ------
\begin{minipage}[b]{9cm*\real{0.7}} \label{fig:6.5c}
\centering Fig. 6.5c\\ 
\tikzstyle{help lines}=[gray,very thin]
\begin{tikzpicture}[scale=0.7]
 \draw[style=help lines]  grid (9,7);
 % axes 
 \draw[thick] (0,0) -- (0,7); 
 \draw[thick] (0,0) -- (9,0);
 % ticks
 {
 \pgftransformxshift{0.5cm}
 \pgftransformyshift{-0.55cm}
  \foreach \x in {0,1,2,3,4,5,6,7,8} \draw[anchor=base] (\x,0) node {$\x$}; 
 }
 {
 \pgftransformxshift{0.5cm}
\draw (2,2) node[above] {$F$};
\draw (3,2) node[above] {$L$};
\draw (8,0) node[above] {$M_0$};
\draw (4,1) node[above] {$N_1$};
\draw (5,1) node[above] {$M_1$};

 }
 % steps 
 \draw[\Red,ultra thick] (8,0) -- (8,1) -- (5,1) -- (5,2) -- 
(4,2) -- (4,3) -- (3,3) -- (3,4) -- (4,4) -- (4,5) -- (5,5) -- (5,6) -- (6,6);
 % line
  \draw[\Black, ultra thick] (0,1) -- (6,7);  
\end{tikzpicture}
\end{minipage}  
\end{minipage}

%%%%%%%%%%%%%%%%%%%%%%%%%CHAPTER 7%%%%%%%%%%%%%%%%%%%%%%%%%%%%%%%%%%%%%%%%%%%%%

\chapter{Estimates of the $\alpha$-grade in the case $\rho_1>0, \rho_2>0$ and  $r=0$.}\label{7}

Let $\cI$ be an ideal of type $r=0$ with $y$-standard form. Then one has by
definition: $\cI=y\cK(-1)+f\cO_{\mP^2}(-m)$, colength
$(\cK)=c,\reg(\cK)=\kappa$, colength $(\cI)=d=c+m$, $\cI$ and $\cK$ invariant
under $G=\Gamma\cdot T(\rho)$, and if the Hilbert functions of $\cK$ and $\cI$
are $\psi$ and $\varphi$, respectively, then $g^*(\varphi)>g(d)$. By
definition, $\cK$ has the type $(-1)$, that means, the following cases can
occur: \\
{\it 1st case:\/} $g^*(\psi)\le g(c)$, where $c\ge 5$ by convention.\\
{\it 2nd case:\/} $0\le c\le 4$.

As usual the aim is to prove \textbf{(*)} in (5.3), and we write $m$ and $f$ instead of $m_0$ and $f_0$.

\section{The case $g^*(\psi)\le g(c)$.}\label{7.1}

\subsection{}\label{7.1.1}
We first assume that there is no deformation into $\cL$, at all. That means
$f=x^m$ (cf. Fig. 7.1), $A=0$ and one has to prove the inequality \textbf{(*)}, i.e.
$Q(m-1)>(c-1)^2$. Because of $Q(m-1)=\frac{1}{2}m(m+1)-(c+m)$, this is
equivalent to $m(m-1)>2c^2-2c+2$. As $m\ge 2c+1$ by Corollary 2.2, it
suffices to show $(2c+1)\cdot 2c>2c^2-2c+2\iff c^2+2c-1>0$, and this is true
for $c\ge 5$.

\subsection{}\label{7.1.2}
We now assume that there are deformations
$M_0=x^m\mapsto\in\cL,\cdots,x^sM_0\mapsto\in\cL$. Then by Proposition 3.2d 
 $s\le\frac{\kappa}{\rho_2}-m\left(
  \frac{1}{\rho_2}-\frac{1}{\rho_1+\rho_2}\right)$.

\noindent {\bf Auxiliary lemma 1.} If $g^*(\psi)\le g(c)$ and if there is the deformation $M_0\mapsto\in\cL$, then $\rho_1<\rho_2$.

\begin{proof}
  Because $M_0X^{\nu\rho}=L$ is a monomial in $\cL$, the slope of the line
  connecting $M_0$ and $L$ is $\le$ the slope of the line connecting $M_0$ and
  $L_0$ (see Figure 7.1, take into account the inclusion (3.1) in 3.3 and the
  following remark concerning the vice-monomial of $f$.) It follows that
  $\rho_1/\rho_2\le\kappa / (m-\kappa)$. Now $\kappa / (m-\kappa)<1$ is equivalent
  to $2\kappa<m$, and because of $\kappa\le c$ and Corollary 2.2 this is
  true.
\end{proof}

\noindent {\bf Auxiliary lemma 2.} If $g^*(\psi)\le g(c)$, then
$\kappa=\reg(\cK)\le c-2$.

\begin{proof}
  One has $\kappa\le c$ and from $\kappa=c$ it follows that
  $g^*(\psi)=(c-1)(c-2)/2>g(c)$ (cf. [T1], p. 92 and Anhang 2e, p. 96). By
  considering the figures 7.2a and 7.2b one can convince oneself that
  $\kappa=c-1$ implies $g^*(\psi)>g(c)$.
\end{proof}

It is clear that the change of $\alpha$-grade in the right domain caused by
the deformations mentioned above is equal to
\begin{equation}\label{1}
  A=m+(m+1)+\cdots+(m+s)=(s+1)m+{s+1\choose 2}.
\end{equation}
Because of $\kappa\le c-2$ and $m\geq 2c+1$ (Corollary 2.2) one has
\[
s\le \frac{c-2}{\rho_2} - (2c+1) \left(
  \frac{1}{\rho_2}-\frac{1}{\rho_1+\rho_2}\right).
\]
It follows that 
\[
s\le
\frac{c-2}{\rho_2}-2(c-2)\left(\frac{1}{\rho_2}-\frac{1}{\rho_2+\rho_2}\right)
-5\left(\frac{1}{\rho_2}-\frac{1}{\rho_2+\rho_2}\right)
\]
and therefore
\begin{equation}\label{2}
  s<(c-2) \left[ \frac{2}{\rho_1+\rho_2}-\frac{1}{\rho_2}\right]
\end{equation}

\noindent {\bf Auxiliary lemma 3.} If $g^*(\psi)\le g(c)$ and if there is a
deformation $M_0\mapsto\in\cL$, then $s<\frac{1}{6}(c-2)$.

\begin{proof}
  As $1\le\rho_1<\rho_2$ (cf. Auxiliary lemma 1), one has to find an upper
  bound for the function $f:[ \N -\{0\}]\times [\N -
  \{0\}]\rightarrow\R, f(x,y):=\frac{2}{x+y}-\frac{1}{y}$, on the domain
  $1\le x<y$. One can convince oneself that $f(1,2)$ is the maximal value.
\end{proof}

In the case $r=0$ the inequality  \textbf{(*)} reads $\frac{1}{2}m(m+1)-(c+m)>(c-1)^2+(s+1)+A$.\\
Putting in the expression (7.1), one gets
$\frac{1}{2}m(m+1)-m>c^2-c+s+2+(s+1)m+{s+1\choose 2}\iff
m(m+1)-2m>2c^2-2c+2(s+2)+2(s+1)m+s(s+1)\iff
m(m-1)>2c^2-2c+2(s+2)+2(s+1)m+s(s+1)$. Because of the Auxiliary lemma 3 it
suffices to show:
$m\left(m-\frac{1}{3}c-\frac{7}{3}\right)>\frac{73}{36}c^2-\frac{29}{18}c+\frac{28}{9}$.
As $m\ge 2c+1$ (Corollary 2.2) this follows from:
$(2c+1)\left(\frac{5}{3}c-\frac{4}{3}\right)>\frac{73}{36}c^2-\frac{29}{18}c+\frac{28}{9}\iff\frac{47}{36}c^2+\frac{11}{18}c-\frac{40}{9}>0$.
As this is true if $c\ge 2$, we have proven \textbf{(*)} in the 1st case.

\section{The cases $0\le c\le 4$.}\label{7.2}
If $c=0$, then $\cI=(y,x^m)$ is monomial and (*) is true, obviously.
Unfortunately, one has to consider the cases $c=1,\cdots,4$ separately. The
ideal $\cK$ can have the Hilbert functions noted in (2.2.2).

\subsection{\fbox{$c=1$}}\label{7.2.1}
From Figure 7.3 we see that $\deg C_0=2+\cdots+m-1={m\choose 2}-1$ and $\deg
C_{\infty}=1+\cdots+m={m+1\choose 2}$. The deformation of $C_0$ into
$C_{\infty}$ is defined by $f_0=x^m+yz^{m-1}$, and this is a simple
deformation in the sense of ([T1], 1.3).

One has $\alpha$-grade $(\cI)=\deg(C)=\max (\deg C_0, \deg C_{\infty}$)
(loc.cit. Hilfssatz 2, p. 12). In order to prove \textbf{(*)} one has to show:
$Q(m-1)>\deg C_{\infty}-\deg C_0=m+1\iff\frac{1}{2}m(m+1)-(1+m)>m+1\iff
m^2-3m-4>0$. This is true if $m\ge 5$.

\begin{conclusion}\label{1}
If $c=1$, then \textbf{(*)} is fulfilled except in the case $c=1, m=4$, which will be treated in (9.3).\hfill \qed
\end{conclusion} 

\subsection{\fbox{$c=2$}}\label{7.2.2}
There are two subcases, which are shown in Fig. 7.4a and Fig. 7.4b. Here only
simple deformations occur, again.

{\it 1st subcase\/} (Fig. 7.4a). One gets $\deg C_0=1+3+4+\cdots+m-1={m\choose 2}-2; \deg C_{\infty}=2+3+\cdots+m={m+1\choose 2}-1$. The same argumentation as in (7.2.1) shows that one has to show $Q(m-1)>m+1$, i.e. $m^2-3m-6>0$. This is fulfilled if $m\ge 5$.\\
In the case $m=4$, by the formula of Remark 2.2, it follows that
$g^*(\varphi)=4$. As $g(6)=4$ and $g^*(\varphi)>g(d)$ by assumption, the case
$m=4$ cannot occur.

{\it 2nd subcase\/} (Fig. 7.4b). $\deg C_0=2+\cdots+m-1, \deg
C_{\infty}=2+\cdots+m; Q(m-1)>\deg C_{\infty}-\deg C_0=m\iff m^2-3m-4>0$. This
is true if $m\ge 5$, and the case $m=4$ cannot occur.

\begin{conclusion}\label{2}
If $c=2$, then \textbf{(*)} is fulfilled.\hfill $\Box$
\end{conclusion}

\subsection{\fbox{$c=3$}}\label{7.2.3}
Here 5 deformations are possible, which are all simple (Fig. 7.5a-7.5e).

{\it 1st subcase\/} (Fig. 7.5a). $\deg C_0=3+\cdots+m-1={m\choose 2}-3; \deg
C_{\infty}=2+4+4+\cdots+m-1={m\choose 2}$. $Q(m-1)>3\iff m^2-m>12$. This is
true because of $m\ge c+2=5$.

{\it 2nd subcase\/} (Fig. 7.5b). $\deg C_0=3+\cdots +m-1; \deg
C_{\infty}=1+3+4+\cdots+m$. $Q(m-1)>\deg C_{\infty}-\deg C_0=m+1\iff
m^2-3m-8>0$. This is true because of $m\ge 5$.

{\it 3rd subcase\/} (Fig. 7.5c). $\deg C_0=3+\cdots+m-1; \deg
C_{\infty}=2+\cdots+m$. $Q(m-1)>\deg C_{\infty}-\deg C_0=m+2\iff
m^2-3m-10>0$. This is fulfilled if $m\ge 6$. As $m\ge c+2=5$, the case $m=5$
remains. But if $m=5$, then from Remark 2.2 it follows that
$g^*(\varphi)=8<g(8)=9$, which contradicts the assumption.

{\it 4th subcase\/} (Fig. 7.5d). $\deg C_0=1+2+4+\cdots+m-1;\deg C_{\infty}=1+3+\cdots+m; Q(m-1)>\deg C_{\infty}-\deg C_0=m+1\iff m^2-3m-8>0$. This is fulfilled if $m\ge 5$.

{\it 5th subcase\/} (Fig. 7.5e). $\deg C_0=2+4+4+\cdots+m-1; \deg
C_{\infty}=2+\cdots+m; Q(m-1)>\deg C_{\infty}-\deg C_0=m-1\iff m^2-3m>4$.
This is fulfilled if $m\ge 5$.

\begin{conclusion}\label{3}
If $c=3$, then \textbf{(*)} is fulfilled.
\end{conclusion}

\subsection{\fbox{$c=4$}}\label{7.2.4}
There are 8 subcases which are shown in Fig. 7.6$a_1$--7.6e. At first we make
the additional assumption that $m\ge 7$. The slope $\rho_1/\rho_2$ of the line
connecting the monomials denoted by + and - is equal to $1/2$ in Fig.
7.6$b_1$, is $\le 1/4$ in Fig. 7.6$b_2$ and is $\le 2/5$ in Fig. 7.6$b_3$. It
follows that the deformations of Fig. 7.6$b_1$ and of Fig. 7.6$b_2$ ( respectively the 
deformations of Fig. 7.6$b_1$ and of Fig. 7.6$b_3$ ) cannot occur
simultaneously, that means, we have only simple deformations.

{\it 1st subcase\/} (Fig. 7.6$a_1$). $\deg C_0=2+4+\cdots+m-1; \deg
C_{\infty}=1+2+4+\cdots+m; Q(m-1)>\deg C_{\infty}-\deg C_0=m+1\iff
m^2-3m-10>0$. This is fulfilled if $m\ge c+2=6$.

{\it 2nd subcase\/} (Fig. 7.6$a_2$). $\deg C_0=2+4+\cdots+m-1; \deg
C_{\infty}=3+4+\cdots+m; \deg C_{\infty}-\deg C_0=m+1$, etc, as in the first
subcase.

{\it 3rd subcase\/} (Fig 7.6$b_1$). $\deg C_0=4+4+5+\cdots+m-1={m\choose 2}-2;
\deg C_{\infty}=2+4+6+5+\cdots+m-1; Q(m-1)>\deg C_{\infty}-\deg C_0=4\iff
m(m-1)>16$. This is fulfilled if $m\ge 5$.

{\it 4th subcase\/} (Fig. 7.6$b_2$). $\deg C_0=4+4+5+\cdots+(m-1); \deg
C_{\infty}=3+\cdots+m; Q(m-1)>\deg C_{\infty}-\deg C_0=m-1\iff m^2-3m-6>0$.
This is fulfilled if $m\ge 5$.

{\it 5th subcase\/} (Fig. 7.6$b_3$). $\deg C_0=4+4+5+\cdots+m-1; \deg
C_{\infty}=2+4+4+\cdots+m; Q(m-1)>m+2\iff m^2-3m-12>0$. This is fulfilled if
$m\ge 6$.

{\it 6th subcase\/} (Fig 7.6c). $\deg C_0=3+\cdots+m-1; \deg C_{\infty}=3+\cdots+m; Q(m-1)>m\iff m^2-3m-8>0$. This is fulfilled if $m\ge 5$.

{\it 7th subcase\/} (Fig. 7.6d). $\deg C_0=1+2+3+5+\ldots+m-1; \deg C_{\infty}=1+2+4+\cdots+m; Q(m-1)>m+1\iff m^2-3m-10>0$. This is fulfilled if $m\ge 6$.

{\it 8th subcase\/} (Fig. 7.6e). $\deg C_0=2+4+6+5\cdots+m-1={m\choose 2}+2;
\deg C_{\infty}=2+4+4+\cdots+m={m+1\choose 2}; Q(m-1)>m-2\iff m^2-3m-4>0$.
This is fulfilled, if $m\ge 5$.

As $m\ge c+2=6$, the case $m=6$ remains. All inequalities are fulfilled if
$m\ge 6$; therefore the possibility remains that in Fig. 7.6$b_1$ and Fig.
7.6$b_3$, if $m=6$ and $\rho=(-3,1,2)$, the deformations
$f=(x^3y+ay^2z^2)\cdot z^2$ and $g=x^6+by^2z^4$ simultaneously occur. Now
$f\wedge g=x^3yz^2\wedge x^6+bx^3yz^2\wedge y^2z^4+ay^2z^4\wedge x^6$ and
therefore\\ $\max-\alpha$-grade $(\cI)\le \max (\deg C_0$ in Fig. 7.6$b_1$,
$\deg C_{\infty}$ in Fig. 7.6$b_3$, $\deg C_{\infty}$ in Fig. 7.6$b_1$) $=\max
\left({m\choose 2}-2,{m+1\choose 2},{m\choose 2}+2\right)$ and
$min-\alpha-grade(\cI)$ =$\min \left({m\choose 2}-2,{m+1\choose 2},{m\choose
    2} +2\right)$. Now ${m+1\choose 2}>{m\choose 2}+2$, if $m\ge 3$, and
therefore $\max-\alpha-grade (\cI)={m+1\choose 2}$ , and\\
$min-\alpha-grade (\cI)={m\choose 2}-2$. Thus \textbf{(*)} follows from
$Q(m-1)>{m+1\choose 2}-{m\choose 2}+2\iff m^2-3m-12>0$, and this is true if
$m\ge 6$.

\begin{conclusion}\label{4} 
If $c=4$, then \textbf{(*)} is fulfilled.\hfill $\Box$
\end{conclusion}

\subsection{}\label{7.2.5}
Summarizing the Conclusions 7.1--7.4 we obtain

\begin{proposition}\label{6}
If $\rho_1>0, \rho_2>0$ and $r=0$, then \textbf{(*)} is fulfilled except in the case $c=1, m=4,\cI=(y(x,y),x^4+yz^3)$.\hfill $\Box$
\end{proposition}
\textbf{Notabene.} Hence the inequality \textbf{(!)}in Section 4.3 is fulfilled with one exception .

\newpage
% ---- fig7.1 ------
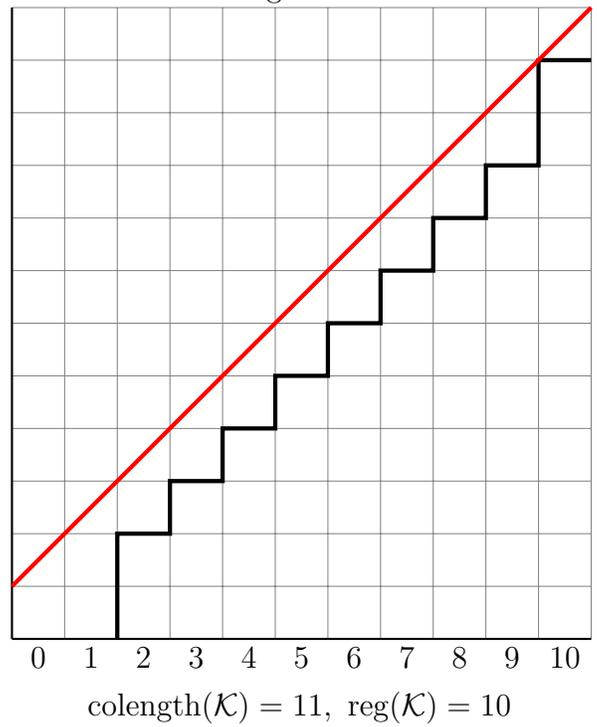
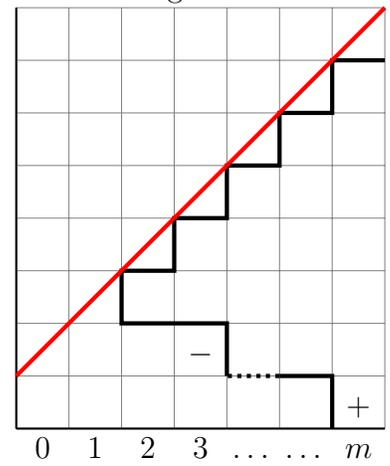
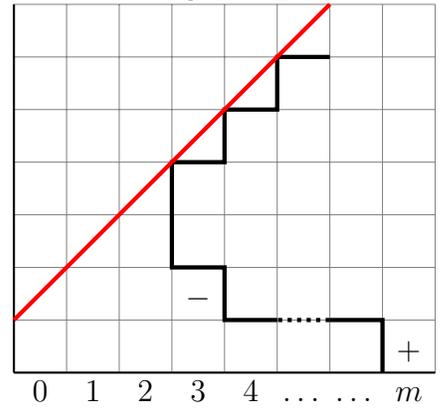
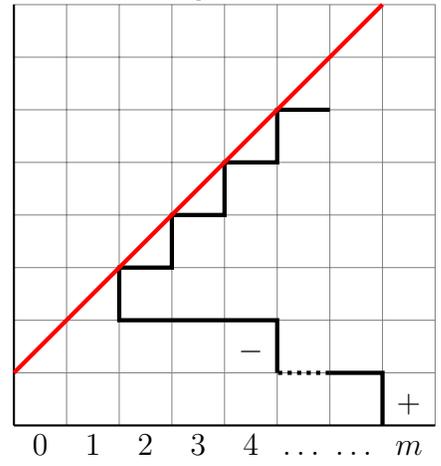
\begin{minipage}{19cm*\real{0.7}} \label{fig:7.1}
\centering Fig. 7.1\\ 
\tikzstyle{help lines}=[gray,very thin]
\begin{tikzpicture}[scale=0.7]
 \filldraw[fill=gray,opacity=0.5] 
         (10,1) -- (10,11) -- (11,11) -- (11,1)-- cycle;
 \draw[style=help lines]  grid (19,14); % (19,15); 
 % axes 
 \draw[thick] (0,0) -- (0,14);  % (0,15); 
 \draw[thick] (0,0) -- (19,0);
 \draw[<->] (0,-1.5) -- (12.8,-1.5); 
 \draw[<-]  (13.2,-1.5) -- (19,-1.5); 
 \draw[->]  (5.5,10.5) -- (10.5,8.5); 
 % ticks
 {
 \pgftransformxshift{0.5cm}
 \pgftransformyshift{-0.55cm}
  \foreach \x in {0,1,2,3,4,5,6,7,8} \draw[anchor=base] (\x,0) node {$\x$}; 
 \draw[anchor=base] (9,0) node {$\kappa$};
 \draw[anchor=base] (10,0) node {$\kappa{+}1$};
 \draw[anchor=base] (12,0) node {$c$};
 \draw[anchor=base] (13,0) node {$c{+}1$};
 \draw[anchor=base] (18,0) node {$m_0$};
 \draw[anchor=base] (4,-1) node[fill=white] {$\cL\cB$};
 \draw[anchor=base] (15,-1) node[fill=white] {$\cR\cB$};
 }
%\draw[anchor=south west] (1,10) node[ellipse,fill=white,text width=3.2cm,inner
sep=2pt]
\draw[anchor=south west] (1,10) node[fill=white,text width=3.2cm,inner sep=2pt]
   {first column filled with monomials of $H^0(\cK(c))$};
\draw[anchor=west] (14,7) node[fill=white,inner sep=1pt]
   {monomial domain};
 {
 \pgftransformxshift{0.5cm}
 \draw (9,9) node[above] {$L_0$};
 \draw (18,0) node[above] {$M_0$};
 }
 % line
 \draw[gray, ultra thick] (13,-2) -- (13,14);  % (13,15);  
 % steps 
 \draw[\Red,ultra thick] (18,0) -- (18,1) -- (6,1) -- (6,2) -- (7,2) -- (7,3)
 -- (8,3) -- (8,4) -- (6,4) -- (6,5) -- (7,5) -- (7,6) -- (8,6) -- (8,7) -- (9,7) -- (9,9) -- (10,9) -- (10,11) -- (11,11) -- (11,12) -- (12,12) -- (12,13) -- (13,13);
 % line
%  \draw[\Black, ultra thick] (0,1) -- (14,15);  
  \draw[\Black, ultra thick] (0,1) -- (13,14);  
\end{tikzpicture}
\end{minipage} 

\vspace{1cm}
\begin{minipage}{1.0\linewidth}
  % ---- fig7.2a ------
\begin{minipage}[b]{10cm*\real{0.7}} \label{fig:7.2a}
\centering Fig. 7.2a\\ 
\tikzstyle{help lines}=[gray,very thin]
\begin{tikzpicture}[scale=0.7]
 \draw[style=help lines]  grid (10,11);
 % axes 
 \draw[thick] (0,0) -- (0,11); 
 \draw[thick] (0,0) -- (10,0);
 % ticks
 {
 \pgftransformxshift{0.5cm}
 \pgftransformyshift{-0.55cm}
  \foreach \x in {0,1,2,3,4,5,6,7,8,9} \draw[anchor=base] (\x,0) node {$\x$}; 
 }
 % steps 
 \draw[\Red,ultra thick] (2,0) -- (2,2) -- (3,2) -- (3,3) -- (4,3) -- (4,4) -- (5,4) -- (5,5) -- (6,5) -- (6,6) -- (7,6) -- (7,7) -- (8,7) -- (8,8) -- (9,8) -- (9,10) -- (10,10);
 % line
  \draw[\Black, ultra thick] (0,1) -- (10,11);  
\end{tikzpicture}
$\colength(\cK)=10,\ \reg(\cK)=9$
\end{minipage}
\hfill
% ---- fig7.2b ------
\begin{minipage}[b]{11cm*\real{0.7}} \label{fig:7.2b}
\centering Fig. 7.2b\\ 
\tikzstyle{help lines}=[gray,very thin]
\begin{tikzpicture}[scale=0.7]
 \draw[style=help lines]  grid (11,12);
 % axes 
 \draw[thick] (0,0) -- (0,12); 
 \draw[thick] (0,0) -- (11,0);
 % ticks
 {
 \pgftransformxshift{0.5cm}
 \pgftransformyshift{-0.55cm}
 \foreach \x in {0,1,2,3,4,5,6,7,8,9,10}\draw[anchor=base] (\x,0) node {$\x$}; 
 }
 % steps 
 \draw[\Red,ultra thick] (2,0) -- (2,2) -- (3,2) -- (3,3) -- (4,3) -- (4,4) -- (5,4) -- (5,5) -- (6,5) -- (6,6) -- (7,6) -- (7,7) -- (8,7) -- (8,8) -- (9,8) -- (9,9) -- (10,9) -- (10,11) -- (11,11);
 % line
  \draw[\Black, ultra thick] (0,1) -- (11,12);  
\end{tikzpicture}
$\colength(\cK)=11,\ \reg(\cK)=10$
\end{minipage}
\end{minipage}

\begin{minipage}{1.0\linewidth}
  % ---- fig7.3 ------
\begin{minipage}{7cm*\real{0.7}} \label{fig:7.3}
\centering Fig. 7.3\\ 
\tikzstyle{help lines}=[gray,very thin]
\begin{tikzpicture}[scale=0.7]
 \draw[style=help lines]  grid (7,7);
 % axes 
 \draw[thick] (0,0) -- (0,7); 
 \draw[thick] (0,0) -- (7,0);
 % ticks
 {
 \pgftransformxshift{0.5cm}
 \pgftransformyshift{-0.55cm}
  \foreach \x in {0,1,2,3} \draw[anchor=base] (\x,0) node {$\x$}; 
 \draw[anchor=base] (4,0) node {$\dots$};
 \draw[anchor=base] (5,0) node {$\dots$}; 
 \draw[anchor=base] (6,0) node {$m$};
 } 
 {
  \pgftransformxshift{0.5cm}
  \draw (6,0) node[above] {$M$};
 }
 % steps 
 \draw[\Red,ultra thick] (6,0) -- (6,1) -- (5,1);
 \draw[\Red,dotted, ultra thick] (5,1) -- (4,1);
 \draw[\Red,ultra thick] (4,1) --
(2,1) -- (2,3) -- (3,3) -- (3,4) -- (4,4) -- (4,5) -- (5,5) -- (5,6) -- (6,6);
 % line
  \draw[\Black, ultra thick] (0,1) -- (6,7);  
\end{tikzpicture}
\end{minipage}
\hfill
  % ---- fig7.4a ------
\begin{minipage}{6cm*\real{0.7}} \label{fig:7.4a}
\centering Fig. 7.4a\\ 
\tikzstyle{help lines}=[gray,very thin]
\begin{tikzpicture}[scale=0.7]
 \draw[style=help lines]  grid (6,7);
 % axes 
 \draw[thick] (0,0) -- (0,7); 
 \draw[thick] (0,0) -- (6,0);
 % ticks
 {
 \pgftransformxshift{0.5cm}
 \pgftransformyshift{-0.55cm}
  \foreach \x in {0,1,2} \draw[anchor=base] (\x,0) node {$\x$}; 
 \draw[anchor=base] (3,0) node {$\dots$}; 
 \draw[anchor=base] (4,0) node {$\dots$};
 \draw[anchor=base] (5,0) node {$m$};
 } 
 {
  \pgftransformxshift{0.5cm}
   \draw (2,2) node[above] {$-$};
   \draw (5,0) node[above] {$+$};
 }
 % steps 
 \draw[\Red,ultra thick] (5,0) -- (5,1) -- (4,1);
 \draw[\Red,dotted, ultra thick] (4,1) -- (3,1);
 \draw[\Red,ultra thick] (3,1) --
(2,1) -- (2,2) -- (3,2) -- (3,4) -- (4,4) -- (4,5) -- (5,5) -- (5,6) -- (6,6);
 % line
  \draw[\Black, ultra thick] (0,1) -- (6,7);  
\end{tikzpicture}
\end{minipage}
\hfill
  % ---- fig7.4b------
\begin{minipage}{7cm*\real{0.7}} \label{fig:7.4b}
\centering Fig. 7.4b\\ 
\tikzstyle{help lines}=[gray,very thin]
\begin{tikzpicture}[scale=0.7]
 \draw[style=help lines]  grid (7,7);
 % axes 
 \draw[thick] (0,0) -- (0,7); 
 \draw[thick] (0,0) -- (7,0);
 % ticks
 {
 \pgftransformxshift{0.5cm}
 \pgftransformyshift{-0.55cm}
  \foreach \x in {0,1,2,3} \draw[anchor=base] (\x,0) node {$\x$}; 
 \draw[anchor=base] (4,0) node {$\dots$};
 \draw[anchor=base] (5,0) node {$\dots$}; 
 \draw[anchor=base] (6,0) node {$m$};
 } 
 {
  \pgftransformxshift{0.5cm}
  \draw (2,1) node[above] {$-$};
  \draw (6,0) node[above] {$+$};
 }
 % steps 
\draw[\Red,ultra thick] (6,0) -- (6,1) -- (5,1); 
\draw[\Red,dotted, ultra thick] (5,1) -- (4,1); 
\draw[\Red,ultra thick] (4,1) -- (3,1) -- (3,2) --
(2,2) -- (2,3) -- (3,3) -- (3,4) -- (4,4) -- (4,5) -- (5,5) -- (5,6) -- (6,6);
 % line
  \draw[\Black, ultra thick] (0,1) -- (6,7);  
\end{tikzpicture}
\end{minipage}

\end{minipage}

\vspace{2cm}
\begin{minipage}{1.0\linewidth}
  % ---- fig7.5a ------
\begin{minipage}[b]{7cm*\real{0.7}} \label{fig:7.5a}
\centering Fig. 7.5a\\ 
\tikzstyle{help lines}=[gray,very thin]
\begin{tikzpicture}[scale=0.7]
 \draw[style=help lines]  grid (7,6);
 % axes 
 \draw[thick] (0,0) -- (0,6); 
 \draw[thick] (0,0) -- (7,0);
 % ticks
 {
 \pgftransformxshift{0.5cm}
 \pgftransformyshift{-0.55cm}
  \foreach \x in {0,1,2,3} \draw[anchor=base] (\x,0) node {$\x$}; 
 \draw[anchor=base] (4,0) node {$\dots$};
 \draw[anchor=base] (5,0) node {$\dots$}; 
 \draw[anchor=base] (6,0) node {$m$};
 } 
 {
  \pgftransformxshift{0.5cm}
  \draw (2,2) node[above] {$-$};
  \draw (3,1) node[above] {$+$};
 }
 % steps 
\draw[\Red,ultra thick] (6,0) -- (6,1) -- (5,1); 
\draw[\Red,dotted, ultra thick] (5,1) -- (4,1); 
\draw[\Red,ultra thick] (4,1) -- (3,1) -- (3,4) -- (4,4) -- (4,5) --
 (5,5) -- (5,6);
 % line
  \draw[\Black, ultra thick] (0,1) -- (5,6);  
\end{tikzpicture}
\end{minipage}
\hfill  
  % ---- fig7.5b ------
\begin{minipage}[b]{7cm*\real{0.7}} \label{fig:7.5b}
\centering Fig. 7.5b\\ 
\tikzstyle{help lines}=[gray,very thin]
\begin{tikzpicture}[scale=0.7]
 \draw[style=help lines]  grid (7,6);
 % axes 
 \draw[thick] (0,0) -- (0,6); 
 \draw[thick] (0,0) -- (7,0);
 % ticks
 {
 \pgftransformxshift{0.5cm}
 \pgftransformyshift{-0.55cm}
  \foreach \x in {0,1,2,3} \draw[anchor=base] (\x,0) node {$\x$}; 
 \draw[anchor=base] (4,0) node {$\dots$};
 \draw[anchor=base] (5,0) node {$\dots$}; 
 \draw[anchor=base] (6,0) node {$m$};
 } 
 {
  \pgftransformxshift{0.5cm}
  \draw (2,1) node[above] {$-$};
  \draw (6,0) node[above] {$+$};
 }
 % steps 
\draw[\Red,ultra thick] (6,0) -- (6,1) -- (5,1); 
\draw[\Red,dotted, ultra thick] (5,1) -- (4,1); 
\draw[\Red,ultra thick] (4,1) -- (3,1) -- (3,4) -- (4,4) -- (4,5) --
 (5,5);
 % line
  \draw[\Black, ultra thick] (0,1) -- (5,6);  
\end{tikzpicture}
\end{minipage}
\hfill  
  % ---- fig7.5c ------
\begin{minipage}[b]{7cm*\real{0.7}} \label{fig:7.5c}
\centering Fig. 7.5c\\ 
\tikzstyle{help lines}=[gray,very thin]
\begin{tikzpicture}[scale=0.7]
 \draw[style=help lines]  grid (7,7);
 % axes 
 \draw[thick] (0,0) -- (0,7); 
 \draw[thick] (0,0) -- (7,0);
 % ticks
 {
 \pgftransformxshift{0.5cm}
 \pgftransformyshift{-0.55cm}
  \foreach \x in {0,1,2,3} \draw[anchor=base] (\x,0) node {$\x$}; 
 \draw[anchor=base] (4,0) node {$\dots$};
 \draw[anchor=base] (5,0) node {$\dots$}; 
 \draw[anchor=base] (6,0) node {$m$};
 } 
 {
  \pgftransformxshift{0.5cm}
  \draw (2,2) node[above] {$-$};
  \draw (6,0) node[above] {$+$};
 }
 % steps 
\draw[\Red,ultra thick] (6,0) -- (6,1) -- (5,1); 
\draw[\Red,dotted, ultra thick] (5,1) -- (4,1); 
\draw[\Red,ultra thick] (4,1) -- (3,1) -- (3,4) -- (4,4) -- (4,5) --
 (5,5) -- (5,6) -- (6,6);
 % line
  \draw[\Black, ultra thick] (0,1) -- (6,7);  
\end{tikzpicture}
\end{minipage}
\end{minipage}

\vspace{1cm}

 \begin{minipage}{1.0\linewidth}
\hfill
  % ---- fig7.5d------
\begin{minipage}[b]{7cm*\real{0.7}} \label{fig:7.5d}
\centering Fig. 7.5d\\ 
\tikzstyle{help lines}=[gray,very thin]
\begin{tikzpicture}[scale=0.7]
 \draw[style=help lines]  grid (7,8);
 % axes 
 \draw[thick] (0,0) -- (0,8); 
 \draw[thick] (0,0) -- (7,0);
 % ticks
 {
 \pgftransformxshift{0.5cm}
 \pgftransformyshift{-0.55cm}
  \foreach \x in {0,1,2,3} \draw[anchor=base] (\x,0) node {$\x$}; 
 \draw[anchor=base] (4,0) node {$\dots$};
 \draw[anchor=base] (5,0) node {$\dots$}; 
 \draw[anchor=base] (6,0) node {$m$};
 } 
 {
  \pgftransformxshift{0.5cm}
  \draw (3,3) node[above] {$-$};
  \draw (6,0) node[above] {$+$};
 }
 % steps 
\draw[\Red,ultra thick] (6,0) -- (6,1) -- (5,1); 
\draw[\Red,dotted, ultra thick] (5,1) -- (4,1); 
\draw[\Red,ultra thick] (4,1) -- (2,1) -- (2,2) -- (3,2) -- 
 (3,3) --  (4,3) -- (4,5) -- (5,5) -- (5,6) -- (6,6)
 -- (6,7) -- (7,7);
 % line
  \draw[\Black, ultra thick] (0,1) -- (7,8);  
\end{tikzpicture}
\end{minipage}  
\hfill
  % ---- fig7.5e------
\begin{minipage}[b]{7cm*\real{0.7}} \label{fig:7.5e}
\centering Fig. 7.5e\\ 
\tikzstyle{help lines}=[gray,very thin]
\begin{tikzpicture}[scale=0.7]
 \draw[style=help lines]  grid (7,8);
 % axes 
 \draw[thick] (0,0) -- (0,8); 
 \draw[thick] (0,0) -- (7,0);
 % ticks
 {
 \pgftransformxshift{0.5cm}
 \pgftransformyshift{-0.55cm}
  \foreach \x in {0,1,2,3} \draw[anchor=base] (\x,0) node {$\x$}; 
 \draw[anchor=base] (4,0) node {$\dots$};
 \draw[anchor=base] (5,0) node {$\dots$}; 
 \draw[anchor=base] (6,0) node {$m$};
 } 
 {
  \pgftransformxshift{0.5cm}
  \draw (3,1) node[above] {$-$};
  \draw (6,0) node[above] {$+$};
 }
 % steps 
\draw[\Red,ultra thick] (6,0) -- (6,1) -- (5,1); 
\draw[\Red,dotted, ultra thick] (5,1) -- (4,1); 
\draw[\Red,ultra thick] (4,1) -- (4,2)  -- (2,2) --
 (2,3) -- (3,3) -- (3,4) -- (4,4) -- (4,5) -- (5,5) -- (5,6) -- (6,6)
 -- (6,7) -- (7,7);
 % line
  \draw[\Black, ultra thick] (0,1) -- (7,8);  
\end{tikzpicture}
\end{minipage}  
\hfill
 \end{minipage}
\vfill

\begin{minipage}{1.0\linewidth}
\begin{minipage}{8cm*\real{0.7}} \label{fig:7.6a1}
\centering Fig. 7.6a${}_1$\\ 
\tikzstyle{help lines}=[gray,very thin]
\begin{tikzpicture}[scale=0.7]
 \draw[style=help lines]  grid (8,8);
 % axes 
 \draw[thick] (0,0) -- (0,8); 
 \draw[thick] (0,0) -- (8,0);
 % ticks
 {
 \pgftransformxshift{0.5cm}
 \pgftransformyshift{-0.55cm}
  \foreach \x in {0,1,2,3,4} \draw[anchor=base] (\x,0) node {$\x$}; 
 \draw[anchor=base] (5,0) node {$\dots$};
 \draw[anchor=base] (6,0) node {$\dots$};
 \draw[anchor=base] (7,0) node {$m$};
 }
 {
  \pgftransformxshift{0.5cm}
  \draw (2,1) node[above] {$-$};
  \draw (7,0) node[above] {$+$};
 }
 % steps 
 \draw[\Red,ultra thick] (7,0) -- (7,1) -- (6,1);
 \draw[\Red,ultra thick,dotted] (6,1) -- (5,1);
 \draw[\Red,ultra thick] (5,1) -- (3,1) -- (3,3) -- (4,3) -- (4,5) -- (5,5) --
 (5,6) -- (6,6) -- (6,7) -- (7,7);
 % line
  \draw[\Black, ultra thick] (0,1) -- (7,8);  
\end{tikzpicture}
\end{minipage} 
\hfill
 % ---- fig7.6a2 ------
\begin{minipage}{8cm*\real{0.7}} \label{fig:7.6a2}
\centering Fig. 7.6a${}_2$\\ 
\tikzstyle{help lines}=[gray,very thin]
\begin{tikzpicture}[scale=0.7]
 \draw[style=help lines]  grid (8,8);
 % axes 
 \draw[thick] (0,0) -- (0,8); 
 \draw[thick] (0,0) -- (8,0);
 % ticks
 {
 \pgftransformxshift{0.5cm}
 \pgftransformyshift{-0.55cm}
  \foreach \x in {0,1,2,3,4} \draw[anchor=base] (\x,0) node {$\x$}; 
 \draw[anchor=base] (5,0) node {$\dots$};
 \draw[anchor=base] (6,0) node {$\dots$};
 \draw[anchor=base] (7,0) node {$m$};
 }
 {
  \pgftransformxshift{0.5cm}
  \draw (3,3) node[above] {$-$};
  \draw (7,0) node[above] {$+$};
 }
 % steps 
 \draw[\Red,ultra thick] (7,0) -- (7,1) -- (6,1);
 \draw[\Red,ultra thick,dotted] (6,1) -- (5,1);
 \draw[\Red,ultra thick] (5,1) -- (3,1) -- (3,3) -- (4,3) -- (4,5) -- (5,5) --
 (5,6) -- (6,6) -- (6,7) -- (7,7);
 % line
  \draw[\Black, ultra thick] (0,1) -- (7,8);  
\end{tikzpicture}
\end{minipage} 
\end{minipage}

\vspace{4cm}
\begin{minipage}{1.0\linewidth}
  % ---- fig7.6b1 ------
\begin{minipage}{8cm*\real{0.7}} \label{fig:7.6b1}
\centering Fig. 7.6b${}_1$\\ 
\tikzstyle{help lines}=[gray,very thin]
\begin{tikzpicture}[scale=0.7]
 \draw[style=help lines]  grid (8,7);
 % axes 
 \draw[thick] (0,0) -- (0,7); 
 \draw[thick] (0,0) -- (8,0);
 % ticks
 {
 \pgftransformxshift{0.5cm}
 \pgftransformyshift{-0.55cm}
  \foreach \x in {0,1,2,3,4} \draw[anchor=base] (\x,0) node {$\x$}; 
 \draw[anchor=base] (5,0) node {$\dots$};
 \draw[anchor=base] (6,0) node {$\dots$};
 \draw[anchor=base] (7,0) node {$m$};
 }
 {
  \pgftransformxshift{0.5cm}
  \draw (2,2) node[above] {$-$};
  \draw (4,1) node[above] {$+$};
 }
 % steps 
\draw[\Red,ultra thick] (7,0) -- (7,1) -- (6,1);
 \draw[\Red,ultra thick,dotted] (6,1) -- (5,1);
 \draw[\Red,ultra thick] (5,1) -- (4,1) -- (4,2) -- 
(3,2) -- (3,4) -- (4,4) -- (4,5) -- (5,5) -- (5,6) -- (6,6);
 % line
  \draw[\Black, ultra thick] (0,1) -- (6,7);  
\end{tikzpicture}
\end{minipage}
\hfill
  % ---- fig7.6b2 ------
\begin{minipage}{8cm*\real{0.7}} \label{fig:7.6b2}
\centering Fig. 7.6b${}_2$\\ 
\tikzstyle{help lines}=[gray,very thin]
\begin{tikzpicture}[scale=0.7]
 \draw[style=help lines]  grid (8,7);
 % axes 
 \draw[thick] (0,0) -- (0,7); 
 \draw[thick] (0,0) -- (8,0);
 % ticks
 {
 \pgftransformxshift{0.5cm}
 \pgftransformyshift{-0.55cm}
  \foreach \x in {0,1,2,3,4} \draw[anchor=base] (\x,0) node {$\x$}; 
 \draw[anchor=base] (5,0) node {$\dots$};
 \draw[anchor=base] (6,0) node {$\dots$};
 \draw[anchor=base] (7,0) node {$m$};
 }
 {
  \pgftransformxshift{0.5cm}
  \draw (3,1) node[above] {$-$};
  \draw (7,0) node[above] {$+$};
 }
 % steps 
\draw[\Red,ultra thick] (7,0) -- (7,1) -- (6,1);
 \draw[\Red,ultra thick,dotted] (6,1) -- (5,1);
 \draw[\Red,ultra thick] (5,1) -- (4,1) -- (4,2) -- 
(3,2) -- (3,4) -- (4,4) -- (4,5) -- (5,5) -- (5,6) -- (6,6);
 % line
  \draw[\Black, ultra thick] (0,1) -- (6,7);  
\end{tikzpicture}
\end{minipage}
\end{minipage}

\newpage
\begin{minipage}{1.0\linewidth}
\begin{minipage}[b]{8cm*\real{0.7}} \label{fig:7.6b3}
\centering Fig. 7.6b${}_3$\\ 
\tikzstyle{help lines}=[gray,very thin]
\begin{tikzpicture}[scale=0.7]
 \draw[style=help lines]  grid (8,8);
 % axes 
 \draw[thick] (0,0) -- (0,8); 
 \draw[thick] (0,0) -- (8,0);
 % ticks
 {
 \pgftransformxshift{0.5cm}
 \pgftransformyshift{-0.55cm}
  \foreach \x in {0,1,2,3,4} \draw[anchor=base] (\x,0) node {$\x$}; 
 \draw[anchor=base] (5,0) node {$\dots$};
 \draw[anchor=base] (6,0) node {$\dots$};
 \draw[anchor=base] (7,0) node {$m$};
 }
 {
  \pgftransformxshift{0.5cm}
  \draw (2,2) node[above] {$-$};
  \draw (7,0) node[above] {$+$};
 }
 % steps 
 \draw[\Red,ultra thick] (7,0) -- (7,1) -- (6,1);
 \draw[\Red,ultra thick,dotted] (6,1) -- (5,1);
  \draw[\Red,ultra thick] (5,1) -- (4,1) -- (4,2) -- (3,2) -- (3,4) -- (4,4) --
 (4,5) -- (5,5) -- (5,6) -- (6,6) -- (6,7) -- (7,7);
 % line
  \draw[\Black, ultra thick] (0,1) -- (7,8);  
\end{tikzpicture}
\end{minipage} 
\hfill
  % ---- fig7.6c ------
\begin{minipage}[b]{8cm*\real{0.7}} \label{fig:7.6c}
\centering Fig. 7.6c\\ 
\tikzstyle{help lines}=[gray,very thin]
\begin{tikzpicture}[scale=0.7]
 \draw[style=help lines]  grid (8,7);
 % axes 
 \draw[thick] (0,0) -- (0,7); 
 \draw[thick] (0,0) -- (8,0);
 % ticks
 {
 \pgftransformxshift{0.5cm}
 \pgftransformyshift{-0.55cm}
  \foreach \x in {0,1,2,3,4} \draw[anchor=base] (\x,0) node {$\x$}; 
 \draw[anchor=base] (5,0) node {$\dots$};
 \draw[anchor=base] (6,0) node {$\dots$};
 \draw[anchor=base] (7,0) node {$m$};
 }
 {
  \pgftransformxshift{0.5cm}
  \draw (3,2) node[above] {$-$};
  \draw (7,0) node[above] {$+$};
 }
 % steps 
\draw[\Red,ultra thick] (7,0) -- (7,1) -- (6,1);
 \draw[\Red,ultra thick,dotted] (6,1) -- (5,1);
 \draw[\Red,ultra thick] (5,1) -- (3,1) -- (3,2) -- 
(4,2) -- (4,3) -- (3,3) -- (3,4) -- (4,4)-- (4,5) -- (5,5) -- (5,6) -- (6,6);
 % line
  \draw[\Black, ultra thick] (0,1) -- (6,7);  
\end{tikzpicture}
\end{minipage}
\end{minipage}

\vspace{4cm}

\begin{minipage}{1.0\linewidth}
\begin{minipage}{8cm*\real{0.7}} \label{fig:7.6d}
\centering Fig. 7.6d \\ 
\tikzstyle{help lines}=[gray,very thin]
\begin{tikzpicture}[scale=0.7]
 \draw[style=help lines]  grid (8,8);
 % axes 
 \draw[thick] (0,0) -- (0,8); 
 \draw[thick] (0,0) -- (8,0);
 % ticks
 {
 \pgftransformxshift{0.5cm}
 \pgftransformyshift{-0.55cm}
  \foreach \x in {0,1,2,3,4} \draw[anchor=base] (\x,0) node {$\x$}; 
 \draw[anchor=base] (5,0) node {$\dots$};
 \draw[anchor=base] (6,0) node {$\dots$};
 \draw[anchor=base] (7,0) node {$m$};
 }
 {
  \pgftransformxshift{0.5cm}
  \draw (4,4) node[above] {$-$};
  \draw (7,0) node[above] {$+$};
 }
 % steps 
 \draw[\Red,ultra thick] (7,0) -- (7,1) -- (6,1);
 \draw[\Red,ultra thick,dotted] (6,1) -- (5,1);
 \draw[\Red,ultra thick] (5,1) --(2,1) -- (2,2) -- (3,2) -- (3,3) -- (4,3) -- (4,4) -- (5,4) -- (5,6) -- (6,6) -- (6,7) -- (7,7);
 % line
  \draw[\Black, ultra thick] (0,1) -- (7,8);  
\end{tikzpicture}
\end{minipage} 
\hfill
\begin{minipage}{8cm*\real{0.7}} \label{fig:7.6e}
\centering Fig. 7.6e \\ 
\tikzstyle{help lines}=[gray,very thin]
\begin{tikzpicture}[scale=0.7]
 \draw[style=help lines]  grid (8,8);
 % axes 
 \draw[thick] (0,0) -- (0,8); 
 \draw[thick] (0,0) -- (8,0);
 % ticks
 {
 \pgftransformxshift{0.5cm}
 \pgftransformyshift{-0.55cm}
  \foreach \x in {0,1,2,3,4} \draw[anchor=base] (\x,0) node {$\x$}; 
 \draw[anchor=base] (5,0) node {$\dots$};
 \draw[anchor=base] (6,0) node {$\dots$};
 \draw[anchor=base] (7,0) node {$m$};
 }
 {
  \pgftransformxshift{0.5cm}
  \draw (4,1) node[above] {$-$};
  \draw (7,0) node[above] {$+$};
 }
 % steps 
 \draw[\Red,ultra thick] (7,0) -- (7,1) -- (6,1);
 \draw[\Red,ultra thick,dotted] (6,1) -- (5,1);
 \draw[\Red,ultra thick] (5,1) -- (5,2) -- (2,2) -- (2,3) -- (3,3) -- (3,4) -- (4,4) -- (4,5) -- (5,5) -- (5,6) -- (6,6);
 % line
  \draw[\Black, ultra thick] (0,1) -- (7,8);  
\end{tikzpicture}
\end{minipage}   
\end{minipage}

%%%%%%%%%%%%%%%%%%%%%%%%%CHAPTER 8%%%%%%%%%%%%%%%%%%%%%%%%%%%%%%%%%%%%%%%%%%%%

\chapter{Borel-invariant surfaces and standard cycles}\label{8}

\section{Preliminary remarks.}\label{8.1}
The group $\G L(3;k)$ operates on $S$ by matrices, thus $\G L(3;k)$ operates
on $H^d=\Hilb^d(\mP^2_k)$. We recall the subgroups $\Gamma=\left\{ \left(
    \begin{array}{ccc} 1&0&*\\0&1&*\\0&0&1\end{array}\right)\right\}<U(3;k)$
and $T(\rho)<T:=T(3;k)$, already introduced in (2.4.1). As each subspace $U\subset S_n$ is invariant under
$D:=\{ (\lambda,\lambda,\lambda)|\lambda\in k^*\}$, each ideal
$\cI\subset\cO_{\mP^2}$ and each point of $H^d$ is invariant under $D$.
Instead of the operation of $T$ on $H^d$ one can consider the operation of the
group $T/D$, which is isomorphic to $T(2;k)$.

According to the theory of Hirschowitz, the 1-cycles in $H^d$ which are
invariant under $B=B(3;k)$, form a generating system of the first Chow group
of $H^d$, and relations among them are realized in $B$-invariant surfaces
$V\subset H^d$ ([Hi], Mode d'emploi, p. 89).

We distinguish between the cases, whether $V$ is pointwise invariant under the
$\G_m$-operation $\sigma(\lambda):x\mapsto x,y\mapsto y,z\mapsto\lambda z$, or
not. This we call the \emph{homogeneous case} and the \emph{inhomogeneous
  case}, respectively.

\section{The inhomogeneous case.}\label{8.2}
We let $T=T(3;k)$ operate by diagonal matrices and let $\G_a$ operate by
$\pal:x\mapsto x, y\mapsto \alpha x + y , z\mapsto z$ on $S$. Then the group
$G=\G_a\cdot T$ operates on $H^d$. Let $V\subset H^d$ be a $G$-invariant,
closed, two-dimensional variety, which is {\it not} pointwise invariant under
$\G_a$ and is {\it not} pointwise invariant under the $\G_m$-operation
$\sigma(\lambda):x\mapsto x,y\mapsto y,z\mapsto\lambda z$. (For abbreviation
we speak of an \emph{inhomogeneous surface}.)

We suppose $k=\ol{k}$. If $\xi\in H^d(k)$ is a closed point, then $T_{\xi}$
and $G_{\xi}$ denote the inertia group of $\xi$ in $T$ and $G$, respectively.

\subsection{Auxiliary lemma 1.}\label{8.2.1}
There is a point $\xi\in V(k)$ such that $V=\overline{T\cdot\xi}$.

\begin{proof} %%%%%%%%%%% im Original ohne Numerierung
  If $\dim T\cdot\xi<2$ for all $\xi\in V(k)$, then $\xi$ is $T$-invariant or
  $\ol{T\cdot\xi}$ is a $T$-invariant, irreducible curve, $\forall\; \xi\in
  V(k)$. The first case can occur for only finitely many points; in the second
  case one has $T_{\xi}=T(\rho),\rho\in\Z^3-(0)$, such that
  $\rho_0+\rho_1+\rho_2=0$ ([T1], Bemerkung 1, p.2).

There are only finitely many $\rho$, such that there exists an ideal
$\cI\subset\cO_{\mP^2}$ of fixed colength $d$, which is invariant under
$T(\rho)$, but not invariant under $T$ (see [T2], Hilfssatz 6, p. 140). In
other words, there are only finitely many $\rho$, such that
$(H^d)^{T(\rho)}\not\supseteq (H^d)^T$. From the assumption follows $V= \bigcup\{
V^{T(\rho_i)}|i\in I\}$, where $I$ is a finite set of indices. As $V$ is
irreducible, it follows that $V=V^{T(\rho)}$ for a certain $\rho$. Now one has
\[
  G_{g(\xi)}=gG_{\xi}g^{-1}\ \forall\ \xi\in V(k),\ \forall g\in G(k),
\]
and therefore
\[
G_{g(\xi)}\supset T(\rho)\cup gT(\rho)g^{-1},\ \forall\ \xi\in
V(k),\quad\forall\quad g\in G(k).
\]
We show there are $\lambda_0\neq\lambda_1$ in $k^*$ such that
$\tau=(\lambda_0,\lambda_1,1)\in T(\rho)$. We assume that
$(\lambda_0,\lambda_1,\lambda_2)\in T(\rho)$ implies
$\lambda_0/\lambda_2=\lambda_1/\lambda_2$, i.e.
$\lambda_0=\lambda_1$. Then each element in $T(\rho)$ has the form
$(\lambda,\lambda,\lambda_2)$, thus $T(\rho)\subset\G_m\cdot D$, where $D=
\{(\lambda,\lambda,\lambda)|\lambda\in k^*\}$. But then $\rho_2 = 0$ and $V$
is pointwise $\G_m$-invariant, contradiction.

 We thus take $\tau=(\lambda_0,\lambda_1,1)\in T(\rho)$ such that
$\lambda_0\neq\lambda_1$ and take 
\[
g=\left( 
 \begin{array}{ccc} 
   1&\alpha&0\\
    0&1&0\\ 0&0&1\end{array}\right)\in \G_a<G, \quad
 \text{where} \ \alpha\neq 0.
\] 
Then  
\[
g\tau g^{-1} =  
  \left(\begin{array}{ccc}
    \lambda_0&(\lambda_1-\lambda_0)\alpha&0\\ 
     0&\lambda_1&0\\
     0&0&1
  \end{array}\right)
\]
is an element of $G_{g(\xi)}$ and thus 
\[
\tau^{-1}g\tau g^{-1}=
\begin{pmatrix}
  1&\beta&0\\
  0&1&0\\
  0&0&1
\end{pmatrix}
  \quad\beta:=\lambda_0^{-1}(\lambda_1-\lambda_0)\alpha\neq 0
\]
is an element of $G_{g(\xi)}$, too. It follows that $g(\xi)$ is invariant
under $\psi_{\beta}$, thus invariant under $\G_a$ and therefore $\xi$ is
invariant under $\G_a$. This is valid for all $\xi\in V(k)$, contradicting the
assumption.
\end{proof}

\subsection{}\label{8.2.2}
If one lets $T(2;k)$ operate on $S$ by
$x\mapsto\lambda_0x,y\mapsto\lambda_1y,z\mapsto z$ then one sees that
$G=T\cdot\G_a$ operates on $H^d$ in the same way as the group
$T(2;k)\cdot\G_a$ which for simplicity is again denoted by $G$ (and is equal
to $B(2;k)$). If $\xi$ is as in the Auxiliary lemma 1, then from
$V=\ol{G\cdot\xi}$ it follows that $G_{\xi}<G$ is 1-dimensional. There is the
decomposition $G_{\xi}=\bigudot g_iH , H:=G_{\xi}^0 $, in finitely many cosets.
 As $H$ is a 1-dimensional connected group, two cases
can occur:

{\it 1st case:} $H\simeq \G_a$. Now the unipotent elements in $B(2;k)$ are
just the matrices 
$
 \pal=\left(\begin{smallmatrix}
  1&\alpha\\
    0&1
  \end{smallmatrix}
\right)$. 
It follows that there is $\alpha\neq 0$ such that
$\pal\in G_{\xi}$. But then $\xi$ is invariant under $\G_a$. As $\G_a$ is
normalized by $T$, $T\cdot\xi$ is pointwise $\G_a$-invariant. Because of
$V=\ol{T\cdot\xi}$, $V$ is pointwise $\G_a$-invariant, contradiction.

{\it 2nd case:} $H\simeq\G_m$. As each element of $\G_m$ is semi-simple, so is
each element of the isomorphic image $H$. Thus the commutative group $H<\G
L(2;k)$ consists of semi-simple elements. Then there is an $g\in\G L(2;k)$,
such that $g^{-1}Hg$ consists of diagonal matrices ([K], Lemma 1,
p. 150). Because of $\G_m\simeq H\stackrel{\sim}{\rightarrow} g^{-1}Hg<T(2;k)$
one has a 1-psg $f:\G_m\rightarrow T(2;k)$. Let $p:T(2;k)\to\G_m$ be the
projection onto the first component. Then $p\circ f:\G_m\to\G_m$ has the form
$\lambda\to\lambda^n$, $n$ a suitable integer. Thus $g^{-1}Hg=\left\{ \left(
    \begin{array}{cc}\lambda^a&0\\ 0&\lambda^b\end{array}\right) \bigg|
  \lambda\in k^*\right\}=:T(a,b)$, $a$ and $b$ suitable integers. It follows
$H=gT(a,b)g^{-1}\subset G=\G_a\cdot T(2;k)=B(2;k)$. Writing
$g=\left(\begin{array}{cc} a_{11}&a_{12}\\a_{21}&a_{22}\end{array}\right)$
gives $g^{-1}= D^{-1}\left(\begin{array}{rr}
    a_{22}&-a_{12}\\-a_{21}&a_{11}\end{array}\right), D:=\det (g)$. We
compute: \\

$\left(\begin{array}{cc}
    a_{11}&a_{12}\\a_{21}&a_{22}\end{array}\right)\left(\begin{array}{cc}
    \lambda_a&0\\0&\lambda^b\end{array}\right)=\left(\begin{array}{cc}
    \lambda^aa_{11}&\lambda^ba_{12}\\
    \lambda^aa_{21}&\lambda^ba_{22}\end{array}\right)$\;and \\
 
$g\left(\begin{array}{cc} \lambda_a&0\\0&\lambda^b\end{array}\right)g^{-1}=  D^{-1} \left( \begin{array}{cc} \lambda^aa_{11}&\lambda^ba_{12}\\
    \lambda^aa_{21}&\lambda^ba_{22}\end{array}\right) \left( \begin{array}{rr}
    a_{22}&-a_{12}\\-a_{21}&a_{11}\end{array} \right)=\\
     D^{-1} \left(
  \begin{array}{ll} \lambda^aa_{11}a_{22}-\lambda^ba_{12}a_{21} &
    -\lambda^aa_{11}a_{12}+\lambda^ba_{11}a_{12}\\
    \lambda^aa_{21}a_{22}-\lambda^ba_{21}a_{22} &
    -\lambda^aa_{12}a_{21}+\lambda^ba_{11}a_{22}\end{array} \right)=\\
     D^{-1}\left( \begin{array}{ll}
    \lambda^aa_{11}a_{22}-\lambda^ba_{12}a_{21} &
    (\lambda^b-\lambda^a)a_{11}a_{12}\\ (\lambda^a-\lambda^b)a_{21}a_{22} &
    \lambda^ba_{11}a_{22}-\lambda^aa_{12}a_{21}\end{array} \right)$.\\
    
This matrix is an upper triangular matrix if and only if
$(\lambda^a-\lambda^b)a_{21}a_{22}=0,\ \forall\lambda\in k^*$.

{\it Subcase 1.} $a=b\Rightarrow H=T(a,a)$.

{\it Subcase 2.} $a\neq b$ and $a_{21}=0$. Write 
\[
g=\left(
  \begin{array}{cc}a_{11} & a_{12}\\ 0&a_{22}\end{array} \right)=u\tau,\quad 
\text{where} \quad
 \tau= \left( \begin{array}{cc}a_{11} & 0\\ 0&a_{22}\end{array} \right),
  u=\left( \begin{array}{cc} 1&c\\0&1\end{array} \right) \ \text{and}\ 
 c:=a_{12}/a_{22}.
\] 
Then $H=gT(a,b)g^{-1}=uT(a,b)u^{-1}$.  

{\it Subcase 3.} $a\neq b$ and $a_{21}\neq 0$. Then $a_{22}=0$ and $g\left( \begin{array}{cc}\lambda^a&0\\ 0&\lambda^b\end{array} \right)g^{-1}=$\\
$ D^{-1}\left(\begin{array}{cc}-\lambda^ba_{12}a_{21}&(\lambda^b-\lambda^a)a_{11}a_{12}\\
    0& -\lambda^aa_{12}a_{21}\end{array}\right)$. As $D =-a_{12}a_{21}$ , this matrix equals \\
    
    $\left( \begin{array}{cc}\lambda^b&(\lambda^a-\lambda^b)c\\
    0&\lambda^a\end{array} \right)$, where $c=a_{11}a_{12}/a_{12}a_{21}=a_{11}/a_{21}$.\\

Thus $H=gT(a,b)g^{-1}=\left\{ \left( \begin{array}{cc}\lambda^b &
      (\lambda^a-\lambda^b)c\\ 0& \lambda^a\end{array} \right)\bigg|
  \lambda\in k^*\right\}$, where $c:=a_{11}/a_{21}$. If $u:=\\  
  \left(
  \begin{array}{cc} 1&-c\\ 0&1\end{array} \right)$, then $u\left(
  \begin{array}{cc}\lambda^b& (\lambda^a-\lambda^b)c\\ 0 &
    \lambda^a\end{array}\right)u^{-1}=\left(\begin{array}{cc}\lambda^b & 0\\ 0
    &\lambda^a\end{array}\right)$ and thus $H=u^{-1}T(b,a)u$.\\
    
We have proved

\noindent {\bf Auxiliary lemma 2.} There is an element $u\in \G_a$ such that
$H=uT(a,b)u^{-1}$, where $a$ and $b$ are suitable integers.\hfill $\Box$

\subsection{}\label{8.2.3}
Let $\xi$ and $u$ be as in Auxiliary lemma 1 and 2, respectively. Set
$\zeta:=u^{-1}(\xi)$. Then $G_{\zeta}=u^{-1}G_{\xi}u\supset 
u^{-1}Hu=T(a,b)<T(2;k)$ and thus $\dim T(2;k)\cdot\zeta\le 1$. If this
dimension would be equal to 0, then $G_{\zeta}$ and $G_{\xi}$ would have the
dimension 2. Because of $V=\ol{G\cdot\xi}$ the dimension of $V$ would be 1,
contradiction. Thus $\dim T\cdot\zeta=1$ and (Appendix E, Corollary) gives

\noindent {\bf Auxiliary lemma 3.} The inertia group $T_{\zeta}$ of $\zeta$ in
$T(3;k)$ has the form $T(\rho)$, where $\rho=(\rho_0,\rho_1,\rho_2)$ and
$\rho_2\neq 0$.\hfill $\Box$

\subsection{Summary.}\label{8.2.4}

\begin{proposition}\label{7}
  Let $V\subset H^d$ be a closed 2-dimensional subvariety, invariant under
  $G=\G_a\cdot T(3;k)$, where $\G_a$ operates by $\pal:x\mapsto x,\,
  y\mapsto\alpha x+y,\, z\mapsto z$ and $T(3;k)$ operates by diagonal matrices
  on $S$. We suppose that $V$ is not pointwise invariant under this
  $\G_a$-operation and is not pointwise invariant under the $\G_m$-operation
  $\sigma(\lambda):x\mapsto x,y\mapsto y,z\mapsto\lambda z$. Then there is a
  point $\xi\in V(k)$ and $u \in\G_a$ such that : \\
  (i) $V=\ol{T(3;k)\cdot\xi}$ (ii) The inertia group $T_{\zeta}$ of
  $\zeta:=u(\xi)$ in $T(3;k)$ has the form $T( \rho)$, where
  $\rho=(\rho_0,\rho_1,\rho_2)$ and $\rho_2\neq 0$. (iii)
  $V=\ol{\G_m\cdot\G_a\cdot\zeta}$.
\end{proposition}

\begin{proof}
  The statements (i) and (ii) follow from (8.2.1) -- (8.2.3). Put
  $G:=\G_m\cdot\G_a=\G_m\times\G_a$. If the statement (iii) were wrong, then
  the inertia group $G_{\zeta}$ would have a dimension $\ge 1$ and thus would
  contain a subgroup $H$ isomorphic to $\G_m$ or $\G_a$. If $H\simeq\G_m$,
  then $\zeta$ would be invariant under $p_1(H)=\{ (\lambda^n,1)|\lambda\in
  k^*\}, n\in\Z-(0)$. Then $\zeta$ and $\xi$ would be invariant under $\G_m$
  and thus $V$ would be pointwise $\G_m$ - invariant, contradiction.

  \noindent If $H\simeq\G_a$ then $\zeta$ and $\xi$ would be invariant under
  $\G_a$. As $\G_a$ is normalized by $T(3;k), \\T(3;k)\cdot\xi$ would be
  pointwise $\G_a$-invariant and the same would follow for $V$.
\end{proof}

\section{The homogeneous case.}\label{8.3}
We now assume $V\subset H^d$ is a 2-dimensional subvariety, invariant under $G
: =\\ \G_a\cdot T(3;k)$, \emph{not} pointwise invariant under $\G_a$, but now
pointwise invariant under the $\G_m$-operation $\sigma$ as in (8.1). As there
are only finitely many $T(3;k)$-fixed points in $H^d$, it follows that $V$ is
not pointwise fixed by the $\G_m$-operation $\tau(\lambda):x\mapsto\lambda
x,y\mapsto y,z\mapsto z$.

Let $\xi\in V(k)$ be not $\G_a$-invariant and not $\G_m$-invariant. We define
a morphism $f$ by $f:\G_a\times\G_m\to V$,
$(\alpha,\lambda)\mapsto\pal\tau(\lambda)\xi$.

Assume that $f$ has an infinite fibre. Then there is an element
$(\beta,\mu)\in\G_a\times\G_m$ such that
$\pal\tau(\lambda)\xi=\psi_{\beta}\tau(\mu)\xi$, i.e.
\begin{equation}\label{1}
  \psi_{(\alpha -\beta)\lambda^{-1}}(\xi) = \tau(\lambda^{-1}\mu)\xi
\end{equation}
for infinitely many $(\alpha,\lambda )$ in $ \G_m \times \G_a $.  

By assumption, $C:=\{\pal (\xi)|\alpha\in k\}^-$ and $D:=\{ \tau
(\lambda)(\xi)|\lambda\in k^*\}^-$ are curves in $V$. If one assumes that only
finitely many different $\lambda$ can occur in (8.1), then on the left side of
(8.1) also only finitely many $\alpha$ can occur. For $\xi$ is not a fixed point
of $ \G_a$, so that from $\pal (\xi)=\psi_{\beta}(\xi)$ it follows that
$\alpha=\beta$. This contradicts the last assumption. Thus $C$ and $D$ have in
common infinitely many points and hence are equal (as subschemes with the
induced reduced scheme structure).

The fixed-point set of $C$ under $\G_m$ consists of the two points
$\xi_{0/\infty}:=\lim\limits_{\lambda\to 0/\infty}\tau(\lambda)\xi$. As $C$
has an unique $\G_a$-fixed-point, and as $\G_a$ is normalized by $ \G_m$, one
of the two points, say $\xi_{\infty}$, is fixed by $\G_a$. Thus $C=\{\pal
(\xi_0)|\alpha\in k\}^-$ and $\xi_0$ corresponds to a monomial ideal. There
are only finitely many $T(3;k)$-fixed points $\xi_i$, $1\le i\le n$, in $H^d$.
Set $M:=\bigcup\limits^n_1\overline{\G_a\cdot\xi_i}$. Then $C\setminus M$ is
open and non empty, and choosing $\xi\in C\setminus M$ it follows that $f$ as
defined above has no infinite fiber.

\begin{proposition}\label{8}
  Let $V\subset H^d$ be a closed 2-dimensional subvariety, invariant under
  $G=\G_a\cdot T(3;k)$, not pointwise $\G_a$-invariant, but pointwise
  invariant under the $\G_m$-operation $\sigma(\lambda):x\mapsto x, y\mapsto
  y,z\mapsto\lambda z$. Then $V$ is not pointwise invariant under the
  $\G_m$-operation $\tau(\lambda):x\mapsto\lambda x, y\mapsto y,z\mapsto z$.
  And for all closed points $\zeta$ in an open non empty subset of $V$ one has
  $V=\overline{\G_a\cdot\G_m\cdot\zeta}$.\hfill $\Box$
\end{proposition}

\section{Standard cycles.}\label{8.4}
In the following we suppose $d\ge 5$ and we recall to memory the closed
subscheme
\[
  \cH=\bigcup \{H_{\varphi}\subset H^d\;|\; g^*(\varphi)>g(d)\}
\]
of $H^d$ (cf. 1.2.2). As usual, we let $\G_a$ operate by $\psi_{\alpha} :
x\mapsto x,y\mapsto\alpha x+y,z\mapsto z$ and we let $\G_m$ operate by
$\sigma(\lambda):x\mapsto x,y\mapsto y,z\mapsto\lambda z$ (by
$\sigma(\lambda):x\mapsto\lambda x,y\mapsto y,z\mapsto z$) on $S$ in the
inhomogeneous case (in the homogeneous case, respectively).

\subsection{}\label{8.4.1}
\begin{definition}
  Let $C\subset\cH$ be a $B=B(3;k)$ -invariant curve, which is not pointwise
  $\G_a$-invariant. Then $C=\{\pal (\xi)|\alpha\in k\}^-$, where
  $\xi\leftrightarrow\cI$
  is an ideal, invariant under\\
  $T=T(3;k)$ and $\Gamma=\left\{ \left(
      \begin{array}{ccc}
        1&0&*\\0&1&*\\0&0&1\end{array}\right)\right\}<\Delta:=U(3;k)$, which is
  uniquely determined by $C$ (cf. [T1], Proposition 0, p. 3). We call $C$ a
  $x$-standard cycle respectively a $y$-standard cycle, if $\cI$ has
  $x$-standard form respectively $y$-standard form (see 2.4.3 Definition 2).
\end{definition}

\subsection{}\label{8.4.2}
Let $V\subset Z=Z(\cH)$ be a 2-dimensional, $B$-invariant subvariety, where
$Z$ is defined as in (1.2.1). We suppose, that $V$ contains a $y$-standard
cycle. Then $V$ is not pointwise $\G_a$-invariant, so that we can write
$V=\overline{\G_m\cdot\G_a\cdot\zeta}$, where $\zeta\in V(k)$ is as in
Proposition 8.1 or Proposition 8.2 , respectively. By the definition of $Z$, the
orbit $\Delta\cdot\zeta$ has a dimension $\le 1$. As $\zeta$ is not
$\Delta$-invariant, the inertia group $\Delta_{\zeta}$ of $\zeta$ in $\Delta$
has the form \\
$G(a:b)=\left\{ \left(\begin{array}{ccc}
      1&\alpha&*\\0&1&\beta\\0&0&1\end{array}\right)
  \in\Delta|a\alpha+b\beta=0\right\}$, where $a,b\in k$ and not both elements
are\\
zero (cf. Appendix D, Lemma 1). Let $\cI$ be the ideal corresponding to
$\zeta$. If $\varphi$ is the Hilbert function of $\cI$, then
$g^*(\varphi)>g(d)$ by the definition of $\cH$ and thus
$\cI=\ell\cK(-1)+f\cO_{\mP^2}(-m)$, where $\cK\subset\cO_{\mP^2}$ has the
colength $ c $, $f\in S_m, c+m=d$ and $m=\reg (\cI)\ge c+2$ ( see 
Lemma 2.1 -- Lemma 2.4). If $\nu : = \min \{ n|H^0(\cI(n))\neq (0)\}$, then
$\nu<m$ . This follows from Lemma 2.2 and Corollary 2.1 . As $G(a:b)$ is
unipotent, there is an eigenvector $f\in H^0(\cI(\nu))$. From $x\;\partial
f/\partial z\in\langle f\rangle$ and $bx\;\partial f/\partial y-ay\partial
f/\partial z\in\langle f\rangle$ we conclude that $f=x^{\nu}$, if we assume
$b\neq 0$, which we do now. Let be $\eta\in V(k)$ any point. If $\cL$ is the
corresponding ideal, then $x^{\nu}\in H^0(\cL(\nu))$. ( Proof : This is first
of all true for all $\eta\in W:=\G_m\cdot\G_a\cdot\zeta$. By means of the
mapping $\cJ\mapsto H^0(\cJ(d))$ we can regard $H^d$ as a closed subscheme of
$\GG=\Grass_{Q(d)}(S_d)$. If $\cJ\in H^d$ is any ideal, the condition
$x^{\nu}\in H^0(\cJ(\nu))$ is equivalent to the condition
$x^{\nu}S_{d-\nu}\subset H^0(\cJ(d))$. An element of $\GG(\Spec A)$ is a
subbundle $L\subset S_d\otimes A$ of rank $Q(d)$, and the condition
$x^{\nu}S_{d-\nu}\subset L$ defines a closed subscheme of $\GG$.  Thus the
condition $x^{\nu}\in H^0(\cL(\nu))$ defines a closed subset of $V$.  As
$V=\overline{W}$, this condition is fulfilled for all points of $V$.)  Assume
that $\cL$ has $y$-standard form. Then $\cL=y\cdot\cM(-1)+g\cO_{\mP^2}(-n)$,
where $e:=$ colength $(\cM), n:=\reg (\cL)\ge m$, and $e+n=d$. As $\nu<m\le n$
we get $x^{\nu}\in H^0(\cL(\nu))=yH^0(\cM(\nu-1))$, contradiction.

\begin{lemma}\label{10}
  If $V\subset Z(\cH)$ is a $B$-invariant surface which contains a
  $y$-standard cycle, then $V$ is point-wise invariant under $\Gamma$.
\end{lemma}

\begin{proof}
  From the foregoing reasoning it follows that $b=0$, i.e., $\zeta$ is
  invariant under $\Gamma=G(1:0)$. As $\Gamma$ is normalized by $\G_a$ and
  $T$, it follows that $\G_m\cdot\G_a\cdot\zeta$ is point-wise invariant under
  $\Gamma$, and the same is true for $V$.
\end{proof}

\subsection{}\label{8.4.3}
We suppose that $V\subset Z(\cH)$ is a $B$-invariant surface, which contains a
$y$-standard cycle. We represent $V$ in the form
$V=\overline{\G_m\cdot\G_a\cdot\zeta}$, according to Proposition 8.1 or
Proposition 8.2, respectively.

\begin{lemma}\label{11}
  (a) One can assume without restriction that $\cJ\leftrightarrow\zeta$ has $y$-standard form.\\
  (b) $\cI_{0/\infty}$ are monomial ideals and have $y$-standard form.
\end{lemma}

\begin{proof}
  Let be $\zeta \leftrightarrow \cI$ as in Proposition 8.1 (resp. in Proposition 8.2 ).\\
  (a) By Lemma 2.6), $\cI$ has $x$- or $y$-standard form. First we
  consider the inhomogeneous case . From $\cI=x\cK(-1)+f\cO_{\mP}(-m)$ and the
  $\Gamma$-invariance of $\cI$ (Lemma 8.1), the $\Gamma$-invariance of $\cK$
  follows. By (Appendix C, Remark 2) we have $R_c\subset H^0(\cK(c))$, and
  because of $m\ge c+2$ we get $R_{m-2}\subset H^0(\cK(m-2))$, thus
  $xR_{m-2}\subset xH^0(\cK(m-2))=H^0(\cI(m-1))$. We conclude that
  $xR_{m-2}\subset H^0(\cJ(m-1))$, hence $xR_{m-2}\cdot S_{d-m+1}\subset
  H^0(\cJ(d))$, if $\cJ\leftrightarrow\eta$ is any point of
  $\G_m\cdot\G_a\cdot\zeta$. The same reasoning as in (8.4.2) shows that
  $xR_{m-2}\cdot S_{d-m+1}\subset H^0(\cJ(d))$ and hence $xR_{m-2}\subset
  H^0(\cJ(m-1))$ is true for all $\cJ\leftrightarrow\eta\in V$. As
  $\cJ=\ell\cM(-1)+g\cO_{\mP^2}(-n), e=$ colength $(\cM), n=\reg (\cJ)\ge m,
  e+n=d$ and $\ell\in R_1-(0)$, it follows that $H^0(\cJ(m-1))=\ell
  H^0(\cM(m-2))\subset xR_{m-2}$, hence $\ell=x$. It follows, that no ideal in
  $V$ has $y$-standard form, contradiction. Thus the statement (a) is proved
  in the inhomogeneous case. If the homogeneous occurs and if we assume that
  $\zeta\leftrightarrow\cI$ would have $x$-standard form, the same
  argumentation as in the inhomogeneous case gives a contradiction. By 
  Lemma 2.2 we have a representation $\cI=\ell \cK(-1)+f\cO_{\mP^2}(-m)$, and
  if the homogeneous case occurs, because of the $\Gamma$-invariance of $\cI$,
   we conclude that $\ell= - \beta x+y$, where $\beta\in k$. Replacing $\cI$ by
  $\psi_{\-\beta}(\cI)=y\psi_{\-\beta}(\cK)(-1)+\psi_{\-\beta}(f)\cO_{\mP^2}(-m)$,
  we can assume without restriction that $\cI$ has $y$-standard form. \\
  (b) If $\cI=y\cK(-1)+f\cO_{\mP^2}(-m)$, colength $(\cK)=c, \reg (\cI)=m,
  c+m=d, m\ge c+2$, then $R_{m-2}\subset H^0(\cK(m-2))$ (Appendix C, Remark
  2), thus $yR_{m-2}\subset H^0(\cI(m-1))$, hence $yR_{m-2}\subset
  H^0(\sigma(\lambda)\cI(m-1)), \lambda\in k^*$. As $yR_{m-2}\subset
  H^0(\cJ(m-1))$ is equivalent to $yR_{m-2}\cdot S_{d-m+1}\subset
  H^0(\cJ(d))$, the condition $yR_{m-2}\subset H^0(\cJ(m-1))$ defines a closed
  subscheme of $V$, which is invariant under the $\G_m$-operation $\sigma$, as
  one sees by the same reasoning as in the proof of Lemma 8.1 . It follows that
  $yR_{m-2}\subset H^0(\cI_{0/\infty}(m-1))$. By Lemma 2.2 we can write
  $\cI_0=\ell\cM(-1)+g\cO_{\mP^2}(-n)$, where $n=\reg(\cI_0)\ge m$. But then
  $yR_{m-2}\subset H^0(\cI_0(m-1))=\ell H^0(\cM(m-2))$, showing that $\ell=y$.
  The same argument shows that $\cI_{\infty}$ has $y$-standard form, too.

  In the inhomogeneous case $\zeta$ is fixed by $T(\rho)$, where $\rho_2\neq
  0$ ( Proposition 8.1) , hence $\zeta_{0/\infty}$ are fixed by 
  $T(\rho)\cdot\G_m = T(3 ; k ) $ \\
  In the homogeneous case $\zeta_{0/\infty}$ are invariant under the two
  $\G_m$-operations $\sigma$ and $\tau$, hence are invariant under $T(3;k)$.
\end{proof}

%%%%%%%%%%%%%%%%%%%%%%%%%CHAPTER 9%%%%%%%%%%%%%%%%%%%%%%%%%%%%%%%%%%%%%%%%%%%%%

\chapter{Relations between $B$-invariant 1-cycles} 
\label{ch:chap9}

 \section{Generalities on relations between 1-cycles} \label{sec:9.1}

 We describe the method of Hirschowitz in our special situation. $B=B(3;k)$
 operates on $H^d$ and we take a closed subscheme $X\subset H^d$, which is
 $B$-invariant. $Z^B_1(X)$ is the free group generated by $B$-invariant curves
 in $X$. It is easy to see that the canonical homomorphism $Z^B_1(X)\to
 A_1(X)$ is surjective (cf. [T1], Lemma 1, p. 6). By ([Hi], Theorem 1, p.87)
 the kernel ${\rm Rat}^B_1(X)$ can be described as follows. One has to
 consider all operations of $B$ on $\P^1$ and all two dimensional subvarieties
 $\cB\subset X\mathop{\times}\limits_k\P^1$ with the following properties:

\begin{enumerate}[(i)]
\item $p_2:\cB\to\P^1$ is dominant, hence surjective and flat.
\item The operation of $B$ on $\P^1$ fixes 0 and $\infty$.
\item $\cB$ is invariant under the induced operation $(\xi,\mu)\mapsto
  (g\xi,g\mu), \xi\in X, \mu\in\P^1, g\in B$, of $B$ on $X\times_k\P^1$.
\end{enumerate}

N.B. According to a general theorem of Fogarty, the fixed-point scheme
$(\P^1)^{U(3;k)}$ is connected; hence from (iii) it follows that $U(3;k)$
operates trivially on $\P^1$.

If $\cB_{\mu}:=p_2^{-1}(\mu)$ is the fibre and $B_{\mu}:=p_1(\cB_{\mu})$ is
the isomorphic image in $X$, then one has:
\begin{align*}
  g(\cB_{\mu}) & = \Set{ (g\xi,g\mu) | \xi\in X, \; \text{such that}\; (\xi,\mu)\in\cB} \\
  & = \Set{ (\xi,g\mu)|\xi\in X, \; \text{such that}\; (\xi,g\mu)\in\cB }\\
  & = \cB_{g\mu}, \; \text{for all}\; g\in B, \mu\in\P^1
\end{align*}
\noindent We conclude that
\begin{equation}\label{eq:1}
  gB_{\mu} = B_{g\mu},\ \mbox{for all}\ g\in B, \mu\in\P^1.
\end{equation}

From property (ii) it follows that $B_0$ and $B_{\infty}$ are $B$-invariant
1-cycles in $X$. Then $\Rat_1^B(X)$ is defined to be the subgroup of
$Z_1^B(X)$ generated by all elements of the form $B_0-B_{\infty}$, and the
theorem of Hirschowitz says that $A_1^B(X):=Z_1^B(X)/\Rat_1^B(X)$ is
canonically isomorphic to $A_1(X)$ (loc.cit.). We consider $V:=p_1(\cB)$ as a
closed subscheme of $X$ with the induced reduced scheme structure. As
$\cB\subset V\times_k\P^1, \dim V\ge 1$. If $\dim V=1$, then $\cB=V\times\P^1$
and $B_0-B_{\infty}$ is equal to zero in $Z_1^B(X)$. Thus we can assume
without restriction that $V$ is a $B$-invariant, 2-dimensional subvariety of
$X$.

\section{The standard construction}\label{sec:9.2}

Let $X\subset H^d$ be a $B(3;k)$-invariant subvariety. Start with a subvariety
$\cB\subset X\times_k\P^1$ as in (9.1), such that $V:=p_1(\cB)\subset
H^d$ is a 2-dimensional subvariety, which is automatically $B(3;k)$-invariant.
Assume $V$ is \emph{not} pointwise invariant under the $\G_a$-operation as in
8.2. Write
\begin{equation}\label{eq:2}
  V=\overline{\G_a\cdot\G_m\cdot\xi}
\end{equation}
where $\xi\in V(k)$ is a suitable point and $\G_m$ operates by
$\sigma(\lambda):x\mapsto x,y\mapsto y,z\mapsto\lambda z$ or by
$\sigma(\lambda):x\mapsto\lambda x,y\mapsto y,z\mapsto z$, respectively (cf.
 Proposition 8.1 and Proposition 8.2). Then $\G_m$ is a subgroup of $T$
and fixes $0=(0:1)$ and $\infty=(1:0)$ by assumption. Assume that $\sigma$
operates trivially on $\G_m$. From (9.1) it follows that
$\sigma(\lambda)B_{\mu} = B_{\mu} , \forall\lambda\in k^*$. If $\xi$ is the
point of (9.2), then one chooses $\mu$ in such a way that
$(\xi,\mu)\in\cB_{\mu}$. Then
$\pal(\xi,\mu)=(\pal(\xi),\mu)\in\cB_{\mu},\forall \alpha\in k$, hence
$\pal(\xi)\in B_{\mu},\forall \alpha\in k$. Because of (9.1) and
 (9.2) we would get $V\subset B_{\mu}$. Then the closed subscheme
$\cB_{\mu}\subset\cB$ has the dimension two and it follows $\cB_{\mu}=\cB$,
contradiction, as $p_2$ is dominant by assumption. This argumentation also
shows $\xi\notin B_0\cup B_{\infty}$, i.e. there is $\mu\in\P^1-\{0,\infty\}$
such that $(\xi,\mu)\in\cB_{\mu}$. Thus one can find $\lambda\in k^*$ such
that $\sigma(\lambda)\mu=(1:1)$. Replacing $\xi$ and $\mu$ by
$\xi'=\sigma(\lambda)\xi$ and $\sigma(\lambda)\mu$, respectively, then
 (9.2) is fulfilled with $\xi'$ instead of $\xi$. Thus one can assume
without restriction that $\mu=(1:1)$. As $\A^1=\P^1-\{\infty\}$ is fixed by
$\sigma, \G_m$ operates by $\sigma$ on $\A^1$ and fixes $(0:1)$. Then there is
$m\in\Z$ such that $\sigma(\lambda)(a:1)=(\lambda^ma:1)$ for all $a\in
k,\lambda\in k^*$. As the action of $\G_m$ on $\P^1$ is not trivial, $m\neq
0$.

$C:=\{\pal(\xi)|\alpha\in k\}^-\subset B_{(1:1)}$ is a curve in $V$; let $h$
be its Hilbert polynomial with respect to the usual embedding of $H^d$ in a
projective space (cf. [T2], 4.1.1). The association $\lambda\mapsto\sigma
(\lambda)C$ defines a morphism $\G_m\to\mathcal{X}: =\Hilb^h(V)$. It has an
extension to a morphism $\P^1\to\mathcal{X}$, which defines a flat family of
curves
\[
  \begin{array}{ccc}
  \cC & \hookrightarrow & V\times_k\P^1\\ & &\\
  & & \;\;\;\downarrow p_2\\ & &\\
  & & \P^1
 \end{array}
\]
If $\cC_{\lambda}:=p_2^{-1}(\lambda)$, then
$C_{\lambda}:=p_1(\cC_{\lambda})=\sigma(\lambda)C,\forall \lambda\in k^*$ and
\begin{equation}\label{eq:3}
  [p_1(\cC_0)]=[C]=[p_1(\cC_{\infty})] \in A_1(V).
\end{equation}

The finite morphism $\A^1-(0)\to\A^1-(0)$ defined by $\lambda\mapsto\lambda^m$
has an extension $f:\P^1\to\P^1$ such that $(\lambda:1)\mapsto
(\lambda^m:1),\forall\lambda\in k$, and $(1:0)\mapsto (1:0)$. For simplicity, the
morphism $1_V\times f:V\times\P^1\to V\times\P^1$ is denoted by $f$, again. By
construction $C\subset B_{(1:1)}$ and hence
$\sigma(\lambda)C\subset\sigma(\lambda) B_{(1:1)}=B_{\sigma(\lambda)(1:1)}$.
The fibre $\cC_{\lambda}=\sigma(\lambda)C\times \{ (\lambda:1)\}$ is mapped by
$f$ into $\sigma(\lambda)C\times \{ (\lambda^m:1)\}=\sigma(\lambda)(C\times \{
(1:1)\} )\subset\sigma(\lambda)(B_{(1:1)}\times \{ (1:1)\}
)=\sigma(\lambda)\cB_{(1:1)}=\cB_{\sigma(\lambda)(1:1)}, \forall\lambda\in
k^*$.

This construction of the family $\cC$ is called \emph{standard construction}
and the proof in ([T1], Lemma 1, p.6) shows that $\cC\subset V\times_k\P^1$ is
a subvariety. As $\cC_0$ and $\cC_{\infty}$ are closed in $\cC,\
\overset{o}{\cC}\ : = \cC \ - (\cC_0\cup\cC_{\infty})$ is open in $\cC$.
Because $\cC$ and $\cB$ are irreducible and $f$ is projective, from
$f(\overset{o}{\cC}) \subset\cB$ we conclude that $f(\cC)=\cB$. As
$V\times \{ 0\}$ and $V\times \{\infty\}$ are mapped by $f$ into itselve,
$\cC_0\subset V\times \{0\}$ and $\cC_{\infty}\subset V\times \{\infty\}$ are
mapped by $f$ into $\cB\cap V\times \{0\}=\cB_0$ and $\cB\cap V\times
\{\infty\}=\cB_{\infty}$, respectively. As $f(\cC)=\cB$, we get
$f(\cC_0)=\cB_0$ and $f(\cC_{\infty})=\cB_{\infty}$, i.e. $\cC_0=\cB_0$ and
$\cC_{\infty}=\cB_{\infty}$. The curves $p_1(\cC_0)$ and $p_1(\cC_{\infty})$
are called the \emph{limit curves} (of the standard construction) and are
denoted by $\CC_0=\lim\limits_{\lambda\to 0}\sigma(\lambda)C$ and
$\CC_{\infty}=\lim\limits_{\lambda\to\infty}\sigma(\lambda)C$, respectively,
and one has
\begin{equation}\label{eq:4}
  B_0-B_{\infty}=\CC_0-\CC_{\infty} \quad \mathrm{in}\quad Z_1^B(X).
\end{equation}

We note that the relation (9.4) was derived under the assumption that
$V$ is not pointwise invariant under $\G_a$. If $V$ is pointwise invariant
under $\G_a$, the standard construction cannot be carried out. We get

\begin{lemma}\label{lem:12}
  Let $X$ be as before. Elements of $\Rat_1^B(X)$ are either of the form $\sum
  q_iC_i$, where $q_i\in \Q$ and $C_i\subset X$ is a $B(3;k)$-invariant 1-prime
  cycle in $X^{\G_a}$ or they occur in the following way: One considers all
  $B(3;k)$-invariant 2-dimensional subvarieties $V\subset X$, which are not
  pointwise invariant under $\G_a$. Then there are points $\xi\in V(k)$ such
  that $V=\ol{\G_m\cdot\G_a\cdot\xi}$, where $\G_m$ operates by
  $\sigma(\lambda):x\mapsto x,y\mapsto y,z\mapsto\lambda z$ on $S$
  (inhomogeneous case) or by $\sigma(\lambda):x\mapsto\lambda x, y\mapsto
  y,z\mapsto z$ (homogeneous case). $C:= \{\pal (\xi)|\alpha \in k\}^- $ is a
  curve in $V$ (with the reduced induced scheme structure). By means of the
  standard construction it defines a family of curves $\cC\subset
  V\times_k\P^1$, which is flat over $\P^1$ and with the fibres
  $\cC_{\lambda}=\sigma(\lambda)C$ for all $\lambda\in\P^1-\{0,\infty\}$.
  $\Rat_1^B(X)$ is generated by the relations in $X^{\G_a}$ noted above as
  well as by the relations $\CC_0-\CC_{\infty}$, where
  $\CC_{0/\infty}=\lim\limits_{\lambda\to 0/\infty} \sigma(\lambda) C$ are the
  limit curves of $C$ in $V$ and these are $\G_a$-invariant 1-cycles.
\end{lemma}
\begin{proof}
  As $\sigma(\lambda)\pal=\psi_{\alpha\lambda}\sigma(\lambda)$ in the
  homogeneous case, $\sigma(\lambda) C=\{\pal \sigma(\lambda)\xi|\alpha\in
  k\}^-$ is $\G_a$-invariant, $\forall \lambda\in k^*$. And the same is true
  in the inhomogeneous case .
\end{proof}

\section{Relations containing a $y$-standard cycle}\label{sec:9.3}

Suppose $d\ge 5$. The closed subscheme $\cH\subset H^d$ (cf.~Section 8.4) is
clearly $B$-invariant. As $U(3;k)$ is normalized by $T(3;k)$, the closed
subscheme $X=Z(\cH)$ is $B$-invariant, too (cf.Section 1.2). From 
Lemma 8.1 it follows that relations in $Z_1^B(X)$ containing a $y$-standard
cycle are defined by 2-dimensional $B$-invariant surfaces $V\subset X$, which
are pointwise $\Gamma$-invariant but not pointwise $\G_a$-invariant.

We can write $V=\ol{\G_m\cdot\G_a\cdot\zeta}$, where $\zeta\in V(k)$ is a
point as in Proposition 8.1 or Proposition 8.2. By 
Lemma 8.2 we can assume that $\cI\leftrightarrow\zeta$ has $y$-standard form.
Let be $\zeta_{0/\infty}=\lim\limits_{\lambda\to
  0/\infty}\sigma(\lambda)\zeta$ , where $ \sigma (\lambda) $ denotes one of the
  two $ \G_m $ - operations mentioned in Lemma 9.1 .Then carry out the standard construction by
using the curve $C:= \Set{ \pal (\zeta)|\alpha\in k }^-$. By Lemma 8.2 
$C_{0/\infty}$ are $y$-standard cycles. We conclude from Proposition 8.1
 and Proposition 8.2 that $W:=\G_m\cdot\G_a\cdot\zeta$ is
$B$-invariant and therefore $V - W$ is a union of $B$-invariant curves and
points. As $\CC_{0/\infty}$ is invariant under $\G_m$ and $\G_a$, from
$W\cap\CC_{0/\infty}\neq\emptyset$ it would follow that
$V\subset\CC_{0/\infty}$. Hence $\CC_{0/\infty}\subset V-W$, and
$\CC_{0/\infty}$ is a union of $B$-invariant curves, too. As
$\zeta_{0/\infty}\in\CC_{0/\infty}$, one has
$C_{0/\infty}\subset\CC_{0/\infty}$. Now from the Propositions 5.1, 6.1 and 7.1 it
follows that the inequality (!) in ( 4.3) is fulfilled with one exception (cf.
 Proposition 7.1 ). Let be $m=\reg (\cI)$. Because of $\max-\alpha$-grade
$(\cI)\ge\alpha$-grade $(\cI)=\deg C=\deg\CC_0=\deg \CC_{\infty}$ and $\min
(\deg C_0,\deg C_{\infty})\ge\min-\alpha$-grade $(\cI)$ (cf. the definitions
and the inequality (4.9) in Section 4.3), from \textbf{(!)} it follows that
\[ 
  Q(m-1)+\min(\deg C_0,\deg C_{\infty})>\deg \CC_0=\deg\CC_{\infty}.\tag{**}
\]
As $V=\ol{W}$, each ideal $\cJ$ in $V$ has a regularity $n\ge m$. If $\cJ$ is
monomial and has $y$-standard form , then from Remark 4.4 it follows that
the $y$-standard cycle generated by $\cJ$ has a degree $\ge Q(n-1)\ge Q(m-1)$.
\begin{lemma}\label{lem:13}
  $C_0$ and $C_{\infty}$ are the only $y$-standard cycles contained in the
  $B$-invariant cycles $\CC_0$ and $\CC_{\infty}$, respectively, and they both
  occur with the multiplicity 1.
\end{lemma}
\begin{proof}
  Except when $c=1, m=4$ the statements follow from (**) and the foregoing
  discussion. In the remaining exceptional case (cf. Proposition 7.1 ) we
  carry out the standard construction with
  $\cI=(y(x,y),x^4+yz^3)\leftrightarrow\zeta$. Then $\zeta_0\leftrightarrow
  (y(x,y),x^4),\zeta_{\infty}\leftrightarrow (y,x^5)$ and $C_0$ and
  $C_{\infty}$ are $y$-standard cycles of degrees 5 and 10, respectively, as
  one can see from Figure 9.1.

  If $\eta:=\lim\limits_{\lambda\to\infty}\pal(\zeta)$, the corresponding
  ideal $\cL$ is invariant under $\G_a$, hence invariant under $U(3;k)$. One
  sees that $xz^3\in\cL$, and hence $x\in\cL$ follows. One concludes from this
  that $\cL=(x,y^5)$.

  As $\sigma$ and $\pal$ commute, it follows that $\eta$ is the only point in
  $\sigma(\lambda)C$ for all $\lambda$, which is $U(3;k)$-invariant. From
  $V=p_1(\cC)=\bigcup \{p_1 (\cC_{\lambda})|\lambda\in\P^1\}=\bigcup \{\sigma
  (\lambda)C|\lambda\in k^*\}\cup\CC_0\cup\CC_{\infty}$ it follows that
\[
  V=W\cup\{\eta\}\cup\CC_0\cup\CC_{\infty}
\]
where $W:=\G_m\cdot\G_a\cdot\zeta$. Now
$\eta_0:=\lim\limits_{\alpha\to\infty}\pal(\zeta_0)\leftrightarrow
(x(x,y),y^4)$ and $\eta_{\infty}:=\lim\limits_{\alpha\to\infty}\pal
(\zeta_{\infty})\leftrightarrow (x,y^5)$ are the only $B(3;k)$-invariant
points in $C_0$ and $C_{\infty}$, respectively. By the theorem of Fogarty
$V^{U(3;k)}$ is connected, hence is a union of pointwise $U(3;k)$-invariant,
closed and irreducible curves. If $E$ is such a curve, then
$E\subset\CC_0\cup\CC_{\infty}$ follows. Now $\deg\CC_0=\deg\CC_{\infty}=\deg
C=10$. Hence $\CC_{\infty}=C_{\infty}$ and $E\subset\CC_0$ follows. As each
$y$-standard cycle in $V$ has a degree $\ge Q(3)=5, C_0$ is the only
$y$-standard cycle contained in $\CC_0$, and $C_0$ occurs with multiplicity 1.
\end{proof}
\begin{proposition}\label{prop:9}
  Let be $X=Z(\cH)$. Relations in $\Rat_1^B(X)$ , which are defined by a
  $B$-invariant surface in $X$ and which contain a $y$-standard cycle, are
  generated by relations of the form $C_1-C_2+Z$ where $C_1$ and $C_2$ are
  $y$-standard cycles and $Z$ is a 1-cycle in $X$, whose prime components are
  $B$-invariant but are not $y$-standard cycles.
\end{proposition}
\begin{proof}
This follows from Lemma 9.1 and Lemma 9.2 .
\end{proof}

 %%%% Graphiken

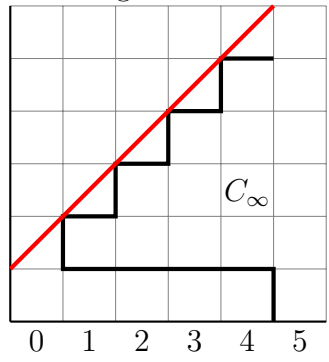
\begin{minipage}{14cm} 
% ---- figFig 9.1 ------
\begin{minipage}{5cm*\real{0.7}} \label{fig:Fig 9.1}
\centering Fig. 9.1\\ 
\tikzstyle{help lines}=[gray,very thin]
\begin{tikzpicture}[scale=0.7]
 \draw[style=help lines]  grid (5,6);
 % axes 
 \draw[thick] (0,0) -- (0,6); 
 \draw[thick] (0,0) -- (5,0);
 % ticks
 {
 \pgftransformxshift{0.5cm}
 \pgftransformyshift{-0.55cm}
  \foreach \x in {0,1,2,3,4} \draw[anchor=base] (\x,0) node {$\x$}; 
 %\draw[anchor=base] (5,0) node {$\varepsilon$};
 %\draw[anchor=base] (6,0) node {$\varepsilon{+}1$};
 %\draw[anchor=base] (10,0) node {$m$};
 }
{
 \pgftransformxshift{0.5cm}
  \draw (1,1) node[above] {$-$};
  \draw (4,0) node[above] {$+$};
  \draw (4,2) node[above] {$C_0$};
}
 % steps 
 \draw[\Red,ultra thick] (4,0) -- (4,1) -- (2,1) -- (2,3) -- (3,3) -- (3,4) --
   (4,4) -- (4,5) -- (5,5);
 % line
  \draw[\Black, ultra thick] (0,1) -- (5,6);  
\end{tikzpicture}
\end{minipage}
\hspace{1.5cm}
% ---- figFig 9.2 ------
\begin{minipage}{5cm*\real{0.7}} \label{fig:Fig 9.2}
\centering Fig. 9.2\\ 
\tikzstyle{help lines}=[gray,very thin]
\begin{tikzpicture}[scale=0.7]
 \draw[style=help lines]  grid (6,6);
 % axes 
 \draw[thick] (0,0) -- (0,6); 
 \draw[thick] (0,0) -- (6,0);
 % ticks
 {
 \pgftransformxshift{0.5cm}
 \pgftransformyshift{-0.55cm}
  \foreach \x in {0,1,2,3,4,5} \draw[anchor=base] (\x,0) node {$\x$}; 
 }
{
 \pgftransformxshift{0.5cm}
  \draw (4,2) node[above] {$C_\infty$};
}
 % steps 
 \draw[\Red,ultra thick] (5,0) -- (5,1) -- (1,1) -- (1,2) -- (2,2) -- (2,3) --
 (3,3) -- (3,4) -- (4,4) -- (4,5) -- (5,5);
 % line
  \draw[\Black, ultra thick] (0,1) -- (5,6);  
\end{tikzpicture}
\end{minipage}
\end{minipage}

%%%%%%%%%%%%%%%%%%%%%%%%%CHAPTER 10%%%%%%%%%%%%%%%%%%%%%%%%%%%%%%%%%%%%%%%%%%%%

\chapter{Proof of Theorem 1.1} \label{ch:chap10}

This had been formulated in Chapter 1 and we repeat the notations and
assumptions introduced there: $d\ge 5, g(d)=(d-2)^2/4, \cH=\bigcup\{
H_{\ge\varphi}\subset H^d|g^*(\varphi)>g(d)\}, \cA(\cH)=\Im
(A_1(\cH^{U(3;k)})\to A_1 (\cH)), c_3=\{ (\alpha x+y,x^d)|\alpha\in k\}^-$. \\
We make the assumption $[c_3]\in\cA(\cH)$. Then $[c_3]=\sum q_i[U_i]$ in
$A_1(\cH)$, where $q_i\in \mathbb{Q}$, and $U_i\subset\cH$ are $T$-invariant curves,
which are pointwise $U(3;k)$-invariant. ($A_1(\cH^{U(3;k)})$ is generated by
such curves, as follows from the theorem of Hirschowitz or more directly from
an application of Lemma 1 in [T1], p. 6.) If $X:=Z(\cH)$, then
$A_1(X)\stackrel{\sim}{\to} A_1(\cH)$ ([T2], Lemma 24, p.121) and
$A_1^B(X)\stackrel{\sim}{\to} A_1(X)$ ([Hi], Theorem 1, p.87).

Hence we can regard $\alpha:=c_3-\sum q_iU_i$ as a 1-cycle in $Z_1^B(X)$ ,whose canonical image in $A_1^B(X) :=
Z_1^B(X)/\Rat_1^B(X)$ vanishes. (We recall that $Z_1^B(X)$ is the free group
generated by all closed, reduced and irreducible curves in $X$.) Define
$\cB(X)$ to be the subgroup of $Z_1^B(X)$, which is generated by all
$B(3;k)$-invariant curves in $X$ (closed, reduced, irreducible), which are not
$y$-standard cycles. Define:
\[
  F_1(X)=Z_1^B(X)/\cB(X),\ R_1 (X)=\Rat_1^B(X)+\cB(X)/\cB(X)
\]
$F_1(X)$ is the free group generated by the $y$-standard cycles, and by 
Proposition 9.1 $R_1(X)$ is generated by elements of $F_1(X)$ of the form
$C_1-C_2$, where $C_1$ and $C_2$ are (different) $y$-standard cycles. It
follows that the canonical image of $\alpha$ in $F_1(X)/R_1(X)$ on the one
hand vanishes and on the other hand is equal to the canonical image of $c_3$
in this $\Q$-vector space. Hence $c_3\in R_1(X)$. We will show that this is
not possible. To each of the finitely many $y$-standard cycles we associate a
canonical basis element $e_i\in\Q^n, 1\le i\le n$. Especially, we associate
$c_3$ to the element $e_n=(0,\ldots,1)$. From $c_3\in R_1(X)$ it follows that
$e_n\in \langle \{ e_i-e_j|1\le i<j\le n\}\rangle$, which is not possible. All
in all we have
\begin{theorem}\label{thm:1} $[c_3]\notin\cA(\cH)$.
\qed
\end{theorem}
\begin{corollary} $\dim_\Q A_1 (\cH)\ge 3$.
\end{corollary}

\begin{proof} $E:=\{(x^2,xy,y^{d-1}+\alpha xz^{d-2})|\alpha\in k\}^-$ and
  $F:=\{ (x,y^{d-1}(\alpha y+z))|\alpha\in k\}^-$ are 1-cycles in $\cH$, which
  are pointwise invariant under $U(3;k)$ and $G(0,1)$, respectively (see the notation in Appendix D). If
  $[c_3]=q_1[E]+q_2[F], q_1,q_2\in \mathbb{Q}$, we compute the intersection numbers
  with the tautological line bundles and get
\[
  {d\choose 2} = q_1+q_2 (n-d+1)
\]
for all $n\ge d-1$. Hence $q_2=0$. If $q_1\neq 0$, then $[c_3]\in\cA(\cH)$
would follow.
\end{proof}

%%%%%%%%%%%%%%%%%%%%%%%%%CHAPTER 11 %%%%%%%%%%%%%%%%%%%%%%%%%%%%%%%%%%%%%%%%%%

\chapter{Surfaces in $\HH$ 
invariant under $\G_a\cdot T(4;k)$} 
\label{ch:chap11}

We let $\G_a$ operate by $\pal:x\mapsto x,y\mapsto\alpha x+y, z\mapsto
z,t\mapsto z$ and let $T(4;k)$ operate by diagonal matrices on $P=k[x,y,z,t]$,
hence on $\HH=\HH_Q$. Let $V\subset\HH$ be a 2-dimensional subvariety, which
is invariant under $\G_a\cdot T(4;k)$ but not pointwise invariant under
$\G_a$.

\section{The inhomogeneous case}\label{sec:11.1}
We suppose that $V$ is not pointwise invariant under the $\G_m$-operation
$\sigma(\lambda):x\mapsto x,y\mapsto y,z\mapsto z, t\mapsto \lambda t$.

\subsection{Auxiliary Lemma 1.}\label{sec:11.1.1}
There is a $\xi\in V(k)$ such that $V=\overline{T(4;k)\cdot\xi}$.

\begin{proof}
  If $\dim T(4;k)\cdot\xi\le 1$, then the inertia group $T_{\xi}\subset
  T(4;k)$ has the dimension $\ge 3$. If the dimension is 4, then $\xi$
  corresponds to a monomial ideal with Hilbert polynomial $Q$, and there are
  only finitely many such ideals. If $\dim T_{\xi}=3$, then $T_\xi=T(\rho)$
  ([T2], Hilfssatz 7, p.141). If $\xi$ corresponds to the ideal $\cI$ with
  Hilbert polynomial $Q$, then $H^0(\cI(b))$ has a basis of the form
  $f_i=m_i p_i(X^\rho)$, where $m_{} $ is a monomial and $p_i$ is a polynomial in
  1 variable. At least for one index $i$ the polynomial $p_i(X^{\rho})$
  contains a positive power of $X^{\rho}$. Then $m_iX^{\rho}\in P_b$. As
  $m_i\in P_b$ and $\rho=(\rho_0,\ldots,\rho_3)\in\Z^4-(0)$ is a vector such
  that $\rho_0+\cdots +\rho_3=0$, this implies $|\rho_i|\le b$. Hence there
  are only finitely many such vectors $\rho$ such that the fixed-point scheme
  $\HH^{T(\rho)}$ does not consist of only finitely many $T(4;k)$-fixed
  points.

  Suppose that $\dim T(4;k)\cdot\xi\le 1$ for all $\xi\in V(k)$. Then
  $V=\mathop{\bigcup}\limits^n_1 V^{T(\rho_i)}, \rho_i\in\Z^4-(0)$, 
hence $V=V^{T(\rho)},
  \rho\in\Z^4-(0)$ suitable vector. We show there is
  $\tau=(\lambda_0,\ldots,\lambda_3)\in T(\rho)$ such that
  $\lambda_0\neq\lambda_1$. If not, it follows that $T(\rho)\subset
  T(1,-1,0,0)$, hence $\nu\rho=(1,-1,0,0),\nu\in\Z$. But then $\rho_3=0$
  and $V$ would be pointwise invariant under $\G_m$, contradiction.

  Take this $\tau$ and an arbitrary $\alpha\neq 0$. Then
  $\tau^{-1}\pal\tau\pal^{-1}= \begin{pmatrix} 1 & \alpha
    (\lambda_0^{-1}\lambda_1-1)\\ 0 & 1\end{pmatrix} =\psi_{\beta}$, where
  $\beta:=\alpha (\lambda_0^{-1}\lambda_1-1)\neq 0$. The same argumentation as
  in the proof of ( 8.2.1 Auxiliary Lemma 1) gives a contradiction.
\end{proof}

\subsection{}\label{sec:11.1.2}
If $0\le i<j\le 3$ define $T_{ij}=\{
(\lambda_0,\lambda_1,\lambda_2,\lambda_3)\in T(4;k)|\lambda_{\ell}=1$ if
$\ell\neq i$ and $\ell\neq j\}$. We let $G:=\G_a\cdot T_{23}$ operate von $V$.
If $\xi\in V(k)$ is as in  Auxiliary Lemma~1, let $G_{\xi}$ be the inertia
group of $\xi$ in $G$. From $1\le\dim G_{\xi}\le 2$ it follows that there is a
1-dimensional connected subgroup $H<G_{\xi}^0$. Then $H\simeq\G_a$ or
$H\simeq\G_m$. In the first case it follows that $H=\G_a$ and hence
$T(4;k)\cdot\xi$ is pointwise $\G_a$-invariant, contradiction. Hence
$H\simeq\G_m$. In the diagram
\[
  \begin{array}{ccc}
  \G_m\simeq H\stackrel{i}{\hookrightarrow} & G=\G_a\times T_{23}\\
   & p_1\swarrow\hfill \searrow p_2\\
   & \G_a \hfill T_{23}
\end{array}
\]
$p_1\circ i$ is the trivial map, hence
$H=\{(1,1,\lambda^r,\lambda^s)|\lambda\in k^*\}=:T_{23}(r,s)$, where $r,s\in\Z$
are not both equal to zero.

\subsection{}\label{sec:11.1.3}
Let be $G:=\Set{ \begin{pmatrix}
                 \lambda & \alpha\\ 
                 0&\mu
                \end{pmatrix}
             |\lambda,\mu\in k^*, \alpha\in k}$ 
and let $\xi$ be as in Auxiliary Lemma 1.
Then there is again a subgroup $H<G_{\xi}^0$ isomorphic to $\G_m$. It has the 
form 
$H= \Set{ \begin{pmatrix} 
          \lambda^a & (\lambda^b-\lambda^a)c\\
            0&\lambda^b
         \end{pmatrix} 
        |\lambda\in k^* }$, 
where $a,b\in\Z$ are
not both equal to zero and $c\in k$ is a fixed element (cf. 8.2.2). Let be
$ u=
\begin{pmatrix}
  1&-c&0&0\\
  0&1&0&0\\
  0&0&1&0\\
  0&0&0&1
\end{pmatrix}
$.

Then $uHu^{-1}= \Set{ 
\begin{pmatrix} 
 \lambda^a & & &\\ 
  &\lambda^b & &\\
  &          & 1 & \\ 
  & & & 1
 \end{pmatrix} 
  |\lambda\in k^* }=:T_{01}(a,b)$. 
If $\zeta:=u(\xi)$ one gets $T_{\zeta}\supset T_{01}(a,b)$. From (11.1.2) it
follows $T_{\zeta}\supset T_{23}(r,s)$, too. As $T_{\zeta}$ contains the
diagonal group, $\dim T_{\zeta}\ge 3$ follows. If $\zeta$ were fixed by
$T(4;k)$, then $\xi=u^{-1}(\zeta)$ would be fixed by $T_{23}$, contradiction.
It follows that $T_{\zeta}=T(\rho)$. Here $\rho_3\neq 0$, because $\xi$ is not
invariant under $\sigma$, for otherwise $V$ would be pointwise
$\G_m$-invariant. If $c=0$, then $T_{\xi}=T(\rho)$, which contradicts the
choice of $\xi$.

\begin{conclusion}
  The element $u$ is different from ${\bf 1}$, and putting $\zeta:=u(\xi)$ one
  has $V=\overline{\G_a\cdot\G_m\cdot\zeta}$.
\end{conclusion}

\begin{proof}
  Put $G:=\G_a\cdot\G_m=\G_a\times\G_m$. If $\dim G_{\zeta}\ge 1$, then
  $G^0_{\zeta}$ would contain a subgroup $H$ isomorphic to $\G_a$ or $\G_m$.
  But then $\zeta$ would be invariant under $\G_m$ or $\G_a$, and then the
  same would be true for $\xi$, which gives a contradiction as above.
\end{proof}

\begin{conclusion}\label{con:11.2}
  $V\setminus \G_a\cdot\G_m\cdot\zeta$ is a union of points and curves which
  are invariant under $\G_a\cdot T(4;k)$.
\end{conclusion}

\begin{proof}
  If $\tau\in T(\rho)$, then
  $\tau\cdot\G_a\cdot\G_m\cdot\zeta=\G_a\cdot\tau\cdot\G_m\cdot\zeta=\G_a\cdot\G_m\cdot\tau\cdot\zeta =\G_a\cdot\G_m\cdot\zeta $,
  and if $\tau\in\G_m$, the same is true.
\end{proof}

\section{The homogeneous case}\label{sec:11.2}

\subsection{}\label{sec:11.2.1} 
We first suppose that $V$ is pointwise invariant under the $\G_m$-operation
$\sigma(\lambda):x\mapsto x,y\mapsto y,z\mapsto z,t\mapsto\lambda t$, but not
pointwise invariant under the $\G_m$-operation $\tau(\lambda):x\mapsto
x,y\mapsto y,z\mapsto\lambda z,t\mapsto t$. Then $V$ is invariant under
$T(3;k)\simeq \{ (\lambda_0,\lambda_1,\lambda_2,1)|\lambda_i\in k^*\}$ and we
have the same situation as in Chapter 8. We use the same notations introduced
there and get $V=\overline{T(2;k)\cdot\xi}$ (8.2.1 Auxiliary Lemma 1). We let
$G:=\G_a\cdot T(2;k)= \Set{
  \begin{pmatrix}\lambda&\alpha\\0&\mu\end{pmatrix}|\lambda,\mu\in k^*,
  \alpha\in k }$ operate on $V$. Then there is a subgroup $H$ of
$G_{\xi}$, which is isomorphic to $\G_a$ or $\G_m$. In the first case $\xi$
would be $\G_a$-invariant, hence $V$ would be pointwise $\G_a$-invariant. In
the second case 
$H=\Set{ 
  \begin{pmatrix} \lambda^a &
    (\lambda^b-\lambda^a)c\\
    0&\lambda^b\end{pmatrix}
 |\lambda\in k^*}$. 
The same argumentation as in (11.1.3) shows that the inertia group of
$\zeta=u(\xi)$ in $T(3;k)$ has a dimension $\ge 2$. If $T_{\zeta}=T(3;k)$,
then $\zeta$ would be invariant under the $\G_m$-operation $\tau$, hence $\xi$
invariant under $\tau$, too, and $V$ would be pointwise invariant under
$\tau$. It follows that $T_{\zeta}=T(\rho)$, where
$\rho=(\rho_0,\rho_1,\rho_2,0)$ and $\rho_2\neq 0$.  We note that $u\neq
\mathbf{1}$, for otherwise the inertia group of $\xi$ in $T(4;k)$ would have
  a dimension $\ge 3$.

\begin{conclusion}\label{con:1}
  The element $u$ is different from ${\bf 1}$ and putting $\zeta=u(\xi)$ one
  has $V=\overline{\G_a\cdot\G_m\cdot\zeta}$.
\end{conclusion}
The same argumentation as in (11.1.3) gives

\begin{conclusion}\label{con:2}
  $V\setminus \G_a\cdot\G_m\cdot\zeta$ is a union of points and curves which
  are invariant under $\G_a\cdot T(4;k)$.\qed
\end{conclusion}

\subsection{}\label{sec:11.2.2}
We now suppose $V$ is pointwise invariant under $T_{23}=\Set{
(1,1,\lambda_2,\lambda_3)|\lambda_i\in k^*}$. Then $V$ is not pointwise
invariant under the $\G_m$-operation $\sigma(\lambda):x\mapsto \lambda
x,y\mapsto y,z\mapsto z, t\mapsto t$. Let $\xi\in V\setminus (V^{\G_m}\cup
V^{\G_a})$ be a closed point, and put 
$G:=\G_a\ltimes\G_m=\Set{
  \begin{pmatrix}
    \lambda&\alpha\\ 
       0&1
 \end{pmatrix}
   | \lambda\in k^*,\alpha\in k}$.

 Assume that $\dim G_{\xi}\ge 1$. Then $H:=G^0_{\xi}$ is 1-dimensional and
 connected. As $\xi$ is not $\G_a$-invariant, $H\simeq\G_m$ hence 
$H=\Set{
   \begin{pmatrix} \lambda^a & (1-\lambda^a)c\\0&1\end{pmatrix}|\lambda\in
   k^* }, a\in\Z-(0)$.
 Putting 
  $u=\begin{pmatrix}
     1&-c&0&0\\ 
     0&1&0&0\\ 
      0&0&1&0\\
      0&0&0&1
    \end{pmatrix}$ 
    and $\zeta=u(\xi)$ we get $uHu^{-1}= \Set{ (\lambda^a,1,1,1)|\lambda\in
      k^* }$. It follows that $\zeta$ is $\G_m$-invariant, hence
    $T(4;k)$-invariant. Thus $\xi$ corresponds to an ideal $\cI$ such that
    $H^0(\cI(b))$ has a generating system of the form $M_i(y - cx )^{n_i}$, where
    $M_i$ is a monomial without $y$ and $n_i\in\N$. There are only finitely
    many points $\zeta_i\in V(k)$ which are $T(4;k)$ invariant and it follows
    that $\xi\in \bigcup\G_a\cdot\zeta_i$.

\begin{conclusion}\label{sec:11.2.3}
  If $\xi\in V\setminus [ \;\bigcup\G_a\cdot\zeta_i\cup V^{\G_a}\cup V^{\G_m}]$ is a\
  closed point, then $V=\overline{\G_a\cdot\G_m\cdot\xi}$.\qed
\end{conclusion}

\section{Summary}\label{sec:11.3}
Let $V\subset\HH_Q$ be a 2-dimensional subvariety, invariant under
$G:=\G_a\cdot T(4;k)$ but not pointwise invariant under $\G_a$. We distinguish
between three cases:

\begin{enumerate}[1.]
\item 
 $V$ is not pointwise invariant under the $\G_m$-operation
  $\sigma(\lambda):x\mapsto x,y\mapsto y,z\mapsto z,t\mapsto\lambda t$.
\item 
 $V$ is pointwise invariant under the $\G_m$-operation $\sigma$
  as in the first case, but not pointwise invariant under the $\G_m$-operation
  $\tau(\lambda):x\mapsto x,y\mapsto y,z\mapsto\lambda z,t\mapsto t$.
\item 
 $V$ is pointwise invariant under the $\G_m$-operations $\sigma$
  and $\tau$ as in the 1th and 2nd case. Then $V$ is not pointwise invariant
  under the $\G_m$-operation $\omega(\lambda):x\mapsto\lambda x,y\mapsto
  y,z\mapsto z, t\mapsto t$.
\end{enumerate}
\begin{lemma}\label{lem:11.14} 
  (a) In the 1st case (respectively in the 2nd case) there is $\zeta\in V(k)$
  with the following properties: \\
  (i) The inertia group $T_{\zeta}$ of $\zeta$ in $T(4;k)$ has the form
  $T(\rho)$ where $\rho_3>0$ (respectively $\rho_2>0$ and $\rho_3=0$). \\
  (ii) $V=\overline{\G_a\cdot\G_m\cdot\zeta}$ and
  $V\setminus\G_a\cdot\G_m\cdot\zeta$ is a union of points and curves
  invariant under $ \G_a\cdot T(4;k)$. \\
  (iii) There is $u\in \G_a$, different from ${\bf 1}$, such that
  $V=\overline{T(4;k)\cdot\xi}$, where $\xi:=u(\zeta)$. \\
  (b) In the 3rd case, if one chooses $\zeta\in V(k)$ such that $\zeta$ is
  neither $\G_a$ -- nor $\G_m$-invariant and does not lie in the set
 $\Set{\xi\in V(k)|\exists u\in\G_a \text{ and } \exists\mu\in V^{T(4;k)}(k)
 \text{ such that } \xi=u(\mu) }$,
 then $V=\overline{\G_a\cdot\G_m\cdot\zeta}$.\qed
\end{lemma}

%%%%%%%%%%%%%%%%%%%%%%%%% CHAPTER 12 %%%%%%%%%%%%%%%%%%%%%%%%%%%%%%%

\chapter{Surfaces in $\HH$ invariant under $B(4;k)$} \label{ch:chap12}

\section{The operation of the unipotent group on $\HH$}\label{sec:12.1}

\subsection{} \label{sec:12.1.1} Let be $p = (a:b:c)\in\P^2(k)$ and
\[
  G(p):=\Set{ 
\begin{pmatrix} 
 1&\alpha&*&*\\  
  0&1&\beta&*\\
  0&0&1&\gamma\\
  0&0&0&1
\end{pmatrix} | a\alpha+b\beta+c\gamma=0 } .
\]
This is a 5-dimensional subgroup of $\Delta:=U(4:k)$ and each 5-dimensional
subgroup of $\Delta$ has this form, where $p\in\P^2(k)$ is uniquely determined
(Appendix D, Lemma 1).

Especially one has the groups $G_i=G(p_i)$, where $p_1=(0:0:1), p_2=(0:1:0),
p_3=(1:0:0)$ and $G_1=\left\{ \begin{pmatrix}
    1&*&*&*\\
    0&1&*&*\\
    0&0&1&0\\
    0&0&0&1\end{pmatrix}\right\}, 
G_2=\left\{ \begin{pmatrix}
    1&*&*&*\\
    0&1&0&*\\
    0&0&1&*\\
    0&0&0&1\end{pmatrix}\right\}, 
G_3=\left\{ \begin{pmatrix}
    1&0&*&*\\
    0&1&*&*\\
    0&0&1&*\\
    0&0&0&1\end{pmatrix}\right\}.$

\begin{remark}\label{rem:12.1}
If $\pal=\begin{pmatrix} 
1&\alpha&0&0\\
0&1&0&0\\
0&0&1&0\\
0&0&0&1\end{pmatrix}$, then $\pal G(p)\pal^{-1}=G(p)$.
\end{remark}

\begin{remark}\label{rem:12.2}
  If $\tau=(\lambda_0,\lambda_1,\lambda_2,\lambda_3)\in T(4;k)$, then $\tau
  G(p)\tau^{-1}=G(\tau p)$, where $\tau p:=(a\lambda_0^{-1}\lambda_1:
  b\lambda_1^{-1}\lambda_2:c\lambda_2^{-1}\lambda_3$).
\end{remark}

\subsection{}\label{sec:12.1.2} In ([T2], 3.2.1) and ([T3], 10.2) we had
  introduced a closed, reduced subscheme $Z=Z(\HH_Q)$ of $\HH_Q$ such that
  $Z(k)=\{ x\in\HH_Q(k)|\dim\Delta\cdot x\le 1\}$. From the theorem of
  Hirschowitz it follows that $A_1(Z)\stackrel{\sim}{\longrightarrow}
  A_1(\HH_Q)$ ([T2], Lemma 24, p. 121).
  In the following we consider a surface $V\subset Z$ (i.e. a closed
  2-dimensional subvariety) which is $B(4;k)$-invariant, but not pointwise
  invariant under $\Delta = U(4;k)$.

 \subsection{Auxiliary Lemma 1.}\label{sec:12.1.3} Let be $\xi \in Z(k)$
 but $\xi$ not invariant under $\Delta$ ,hence $\Delta_{\xi}= G(p)$ .
 If $\tau \in T_{\xi}$ , then $\tau p = p$ .
  \begin{proof}
   If $\tau \in T_{\xi}$, then $\tau G(p)\tau^{-1}\tau \xi = \tau \xi$ 
   and thus $G(\tau p)\xi = \xi $. If $\tau p \neq p $ ,                                                                                                                           then $G(\tau p) \neq G(p) $ and $\xi$ would be fixed by the subgroup of
   $\Delta$ , which is generated by $G(p)$ and $G(\tau p)$ , i.e.
   fixed by $\Delta$ .
   
\end{proof}

\subsection{Auxiliary Lemma 2.}\label{sec:12.1.4}
Let $\xi$ be as in Auxiliary Lemma 1 . If $p = ( a:b:0)$ and $ a , b \neq 0$ ,
then $T_{\xi} \subset T(1,-2,1,0)$ .
\begin{proof}
This follows from Remark 12.2 and Auxiliary Lemma 1 .
\end{proof}

\section{The case $p=(a:b:c)$ where $a,b,c\neq 0$.}\label{sec:12.2}
  Let be $V\subset Z$ a $B$-invariant surface and $\xi\in V(k)$ a point such
  that $\Delta_{\xi}=G(a:b:c)$, where $a,b,c\neq 0$. Then $T_{\xi}\subset
  \Lambda:=\{
  (\lambda_0,\lambda_1,\lambda_0^{-1}\lambda_1^2,\lambda_0^{-2}\lambda_1^3)|\lambda_0,\lambda_1\in
  k^*\}$. As $\dim T(4;k)\cdot\xi\le 2$, one has $\dim T_{\xi}\ge 2$ and thus
  $T_{\xi}=\Lambda $.

  We let $G:=\G_a\cdot T_{01}=\Set{\begin{pmatrix} \lambda_0&\alpha&0&0\\
      0&\lambda_1 &0&0\\ 0&0&1&0\\ 0&0&0&1\end{pmatrix} | \alpha\in
    k,\lambda_0,\lambda_1\in k^*}$ operate on $V$.

\begin{remark}\label{rem:12.3} $T(4;k)_{\xi}\cap G=(\mathbf{1})$.
\end{remark}

\begin{proof}
  If $(\lambda_0,\lambda_1,
  \lambda_0^{-1}\lambda_1^2,\lambda_0^{-2}\lambda_1^3)\in
  G$, then $\lambda_0^{-1}\lambda_1^2=\lambda_0^{-2}\lambda_1^3=1$ which
  implies $\lambda_0/\lambda_1=1$, and then $ \lambda_0 = \lambda_1 = 1$
  follows.
\end{proof}

\begin{remark}\label{rem:12.4}
$V=\overline{T_{01}\cdot\xi}$.
\end{remark}
\begin{proof}
Because of Remark (12.3) $\dim T_{01}\cdot\xi<2$ is not possible.
\end{proof}

\indent From Remark (12.4) it follows that $G_{\xi}<G$ is 1-dimensional. If
$H:=G_{\xi}^0$ would be isomorphic to $\G_a$, then $\xi$ would be invariant
under $\G_a$, hence invariant under $\Delta$, contradiction. Thus 
$H=\Set{
  \begin{pmatrix} \lambda^m & (\lambda^n-\lambda^m)c\\ 0 &
    \lambda^n\end{pmatrix} | \lambda\in k^*}, m,n\in\Z$ not both
equal to zero, $c\in k$ (see 8.2.2 Auxiliary Lemma 2). If $u:=
\begin{pmatrix} 1&-c\\ 0&1\end{pmatrix}$, then $uHu^{-1}=\left\{
  \begin{pmatrix} \lambda^m&0\\ 0&\lambda^n\end{pmatrix}\big| \lambda\in
  k^*\right\}=:T_{01}(m,n)$. Putting $\zeta:=u(\xi)$ one gets
$T_{\zeta}\supset T_{01}(m,n)$. Because of Remark (12.1) from
$\Delta_{\xi}=G(p)$ it follows that $\Delta_{\zeta}=G(p)$, too. Hence
$T_{\zeta}=\Lambda\supset T_{01}(m,n)$, which is not possible. We have proved

\begin{lemma}\label{lem:12.15}
  If $V\subset Z$ is a $B(4;k)$-invariant surface, which is not
  pointwise$\Delta$ -invariant , then there is no point $\xi\in V(k)$, whose
  inertia group $\Delta_{\xi}$ in $\Delta$ has the form $G(a:b:c)$, where
  $a,b,c\neq 0$.\qed
\end{lemma}

\section{1-cycles of proper type 3.}\label{sec:12.3}

\subsection{Recalling the restriction morphism $h$.}\label{sec:12.3.1}
The ideals $\cI\subset\cO_{\P^3}$ with Hilbert polynomial $Q$ such that $t$ is
not a zero divisor of $\cO_{\P^3}/\cI$ form an open non empty subset
$U_t\subset\HH_Q$, and $\cI\mapsto\cI':=\cI+t\cO_{\P^3}(-1)/t\cO_{\P^3}(-1)$
defines the so called restriction morphism $h:U_t\to H^d=\Hilb^d(\P^2)$. If
$\Gamma:=\left\{ \begin{pmatrix} 1&0&0&*\\ 0&1&0&*\\ 0&0&1&*\\
    0&0&0&1\end{pmatrix}\right\}<\Delta$, then $\HH_Q^{\Gamma}$ is contained
in $U_t$.

\subsection{1-cycles of  proper type 3.}\label{sec:12.3.2}
We recall, respectively introduce, the notations: An ideal
$\cJ\subset\cO_{\P^3_k}$ with Hilbert polynomial $Q$ corresponds to a point
$\xi\in \HH_Q(k)$. $\cJ$ has the type 3, if $\cJ$ is invariant under $T(4;k)$
and $G_3=G(1:0:0)$, but not invariant under $\Delta$. The curve
$C=\overline{\G_a\cdot\xi}=\{\psi_{\alpha}(\xi)|\alpha\in k\}^-$ in $\HH_Q$ is
called the 1-cycle of type 3 defined by $\xi$. We say $C$ is a 1-cycle of
\emph{proper type} 3, if $\cI:=\cJ'=h(\cJ)$ has $y$-standard form (cf. 2.4.3
Definition 2). If $\varphi$ is the Hilbert function of $\cI\leftrightarrow
\xi'\in H^d(k)$, then $g^*(\varphi)>g(d)$ by definition, and one has
$\cI=y\cK(-1)+x^m\cO_{\P^2}(-m)$, where $\cK\subset\cO_{\P^2}$ has the
colength $c$ and is invariant under $T(3;k)$ and $G'_3:=\left\{
  \begin{pmatrix} 1&0&*\\ 0&1&*\\ 0&0&1\end{pmatrix}\right\} <U(3;k)$.
Moreover one has $d=m+c$ and $m\ge c+2$.

\indent If $Q(T)={T-1+3\choose 3}+{T-a+2\choose 2}+{T-b+1\choose 1}$ is the
Hilbert polynomial of $\cJ$, then the closed subscheme $V_+(\cJ)\subset\P^3$
defined by $\cJ$ has the ``complementary'' Hilbert polynomial $P(n)=dn-g+1$,
where $d=a-1,g=g(\cJ)=(a^2-3a+4)/2-b$ (cf. [T1], p. 92).

\begin{lemma}

Let $\cJ$ be of proper type 3 and put $ \nu:=\min \{ n|H^0(\cJ(n))\neq (0)\}$.
If $x^{\nu}\in H^0(\cJ(\nu))$, then $g(\cJ)<0$.
\end{lemma}
\begin{proof}
  We start with an ideal $\cJ$ fulfilling these conditions. There are
  subspaces $U_i\subset S_i$, invariant under $T(3;k)\cdot G'_3$ such that
  $S_1U_i\subset U_{i+1},i=0,1,2,\ldots$ and
  $H^0(\cJ(n))=\bigoplus\limits^n_{i=0}t^{n-i}U_i$, for all $n\in\N$. Besides
  this, $U_n=H^0(\cI(n))$, at least if $n\ge b$, where $\cI=\cJ'$ is the
  restriction ideal of $\cJ$ (cf. [G78], Lemma 2.9, p. 65). We replace $U_n$
  by $H^0(\cI(n))$ for all $\nu\le n\le b-1$ and we get an ideal
  $\cJ^*\supset\cJ$ such that $H^0(\cJ^*(n))=(0)$, if $n<\nu$, and
  $H^0(\cJ^*(n))=\bigoplus\limits^n_{i=\nu}t^{n-i}H^0(\cI(i))$, if $n\ge\nu$.
  (N.B.: $x^{\nu}\in H^0(\cJ(\nu))\Rightarrow x^{\nu}\in\Im
  (H^0(\cJ(\nu))\stackrel{\longrightarrow}{\mathrm{res}} H^0(\cI(\nu)))$). Then
  $(\cJ^*)'$ is equal to $\cI$, $\cJ^*$ is of proper type 3, the Hilbert
  polynomial of $\cJ^*$ is equal to $Q^*(n)={n-1+3\choose 3}+{n-a+2\choose
    2}+{n-b^*+1\choose 1}$, the length of $\cJ^*/\cJ$ is equal to
  $Q^*(n)-Q(n)=b-b^*\ge 0$, and because of $g(\cJ^*)=(a^2-3a+4)/2-b^*$ one has
  $g(\cJ^*)\geq g(\cJ)$. Thus it suffices to show $g(\cJ^*) < 0 $ We thus can
  assume without restriction $\cJ=\cJ^*$, i.e.
  $H^0(\cJ(n))=\bigoplus\limits^n_{i=\nu}t^{n-i}H^(\cI(i))$ for all $n\ge\nu$,
  and $x^{\nu}\in H^0(\cI(\nu))$. Using the terminology of [T1]--[T4], one can
  say the pyramid $E(\cJ)$ is complete up to the level $\nu$ over the platform
  $H^0(\cI(n))$, where $n\ge b$, for instance (cf. Fig. 12.1). Further note
  that
\begin{equation}\label{eq:1}
  H^0(\cI(n))=yH^0(\cK(n-1))\oplus x^mk[x,z]_{n-m}
\end{equation}

for all $n\in\N$ (cf. Lemma 2.6). We associate to $\cI$ the ideal
$\tilde{\cI}$ represented in Figure 12.2; that means $\tilde{\cI}$ arises from
$\cI$ by shifting $yH^0(\cK(n-1))$ into $yH^0(\tilde{\cK}(n-1))$, where the
Hilbert functions of $\cK$ and $\tilde{\cK}$ agree and hence $\cI$ and
$\tilde{\cI}$ have the same Hilbert function $\varphi$. Then $\tilde{\cJ}$ is
defined by
$H^0(\tilde{\cJ}(n))=\bigoplus\limits^n_{i=\nu}t^{n-i}H^0(\tilde{\cI}(i))$ and
$H^0(\tilde{\cJ}(n))=(0)$, if $n<\nu$. Then $\tilde{\cI}$ and $\tilde{\cJ}$
fulfil the same assumptions as $\cI$ and $\cJ$ do, and
$g(\cJ)=g(\tilde{\cJ})$. Thus we can assume without restriction that $\cI$ has
the shape as represented by Fig. 12.2. Then one makes the graded deformations
$\square^{\bullet}\mapsto \fbox{1}^{\bullet}$ (or $\square^{\bullet}\to
\fbox{2}^{\bullet}$, etc.) in $E(\cJ)$. Each of the orbits which are to be
exchanged, have the same length. One gets an ideal $\tilde{\cJ}$ with the same
Hilbert polynomial, which is again of proper type 3. If $\tilde{\varphi}$ is
the Hilbert function of $\tilde{\cI}:=(\tilde{\cJ})'$, then
$\tilde{\varphi}>\varphi$, hence $g^*(\tilde{\varphi})>g^*(\varphi)>g(d)$ (cf.
 Remark 2.1). As $\cJ$ and $\tilde{\cJ}$ have the same Hilbert polynomial,
one has $g(\tilde{\cJ})=g(\cJ)$. The colength of $\tilde{\cI}$ in $\cO_{\P^2}$
is the same as the colenght $d$ of $\cI$ in $\cO_{\P^2}$, hence the
coefficient $a$, remains unchanged. Now one can again pass to
$(\tilde{\cJ})^*$ and has again $x^{\nu}\in H^0((\tilde{\cJ})^*(\nu)),
\nu=\min \{ n|H^0((\tilde{\cJ})^*(n))\neq (0)\}, (\tilde{\cJ})^*$ of proper
type 3. Thus it suffices to show $g((\tilde{\cJ})^*)< 0$. Continuing in this
way one sees that one can assume without restriction: $\cJ$ is of proper type
3, $\cI=\cJ'$ has the shape of Figure 12.4,
$H^0(\cJ(n))=\bigoplus\limits^n_{i=\nu}t^{n-i}H^0(\cI(i))$, for all $n\ge\nu,
H^0(\cJ(n))=(0)$, if $n<\nu$ and $x^{\nu}\in H^0(\cJ(\nu))$, i.e., $x^{\nu}\in
H^0(\cI(\nu))$. From (12.1) it follows that $\nu\ge m$. As $m\ge c+2$, we have
$h^0(\cK(n))={n+2\choose 2}-c,n\ge m-1$, hence
$h^0(\cI(n))=h^0(\cK(n-1))+(n-m+1)={n-1+2\choose 2}-c+(n-m+1)={n-1+2\choose
  2}+{n-1+1\choose 1}+1-(c+m)={n+2\choose 2}-d, n\ge m$. But then
\begin{align*}
    h^0(\cJ(n))= \sum^n_{i=\nu}h^0(\cI(i))& =\sum^n_{i=\nu}\left[ {i+2\choose 2}-d\right]\\
    & = \sum^n_{i=0}{i+2\choose 2}-\sum^{\nu-1}_{i=0}{i+2\choose 2}-(n-\nu+1)d\\
    & = {n+3\choose 3}-{\nu+2\choose 3}-(n-\nu+1)d.  
\end{align*}

From this we get $P(n)=(n-\nu +1)d+{\nu+2\choose 3}$, thus:
\begin{equation}\label{eq:2}
  g(\cJ)=(\nu-1)d-\tbinom{\nu+2}{3}+1\;.
\end{equation}

We regard $g(\cJ)$ as a function of $\nu\ge m$, and we have to determine the
maximum of
\[
  g(x):=(x-1)d-\frac{1}{6}x^3-\frac{1}{2}x^2-\frac{1}{3}x+1,\quad x\ge m.
\]

We have $g'(x)=-\frac{1}{2}x^2-x+(d-\frac{1}{3})=0\Leftrightarrow
x=-1\pm\sqrt{2d+\frac{1}{3}}$, and we show $m\ge -1+\sqrt{2d+\frac{1}{3}}$.
This last inequality is equivalent to $(m+1)^2>2d+\frac{1}{3}\Leftrightarrow
m^2+\frac{2}{3}\ge 2c$ (because of $d=c+m$). As $m\ge c+2$, this is true.

It follows that $g'(x)\le 0$, if $x\ge m$, hence $g(x)$ is monotone decreasing
if $x\ge m$. Now
\begin{align*}
  g(m) & = (m-1)d-\frac{1}{6}m^3-\frac{1}{2}m^2-\frac{1}{3}m+1\\
  & = (m-1)(c+m)-\frac{1}{6}m^3-\frac{1}{2}m^2-\frac{1}{3}m+1\\
  & \le (m-1)(2m-2)-\frac{1}{6}m^3-\frac{1}{2}m^2-\frac{1}{3}m+1.
\end{align*}
The right side of this inequality is smaller than 0 if and only if $18<m^3-9m^2+26m$, which is true as $m\ge c+2\ge 2$.
\end{proof}

\section{$B(4;k)$-invariant surfaces containing a 1-cycle of proper type
  3.}\label{sec:12.4}

Let be $V\subset Z=Z(\HH_Q)$ a $B(4;k)$-invariant surface containing a 1-cycle
$D$ of proper type 3. Then $V$ is not pointwise invariant under the
$\G_a$-operation $\psi_{\alpha}:x\mapsto x,y\mapsto\alpha x+y,z\mapsto z,t\mapsto
t$. According to Lemma 11.1 one can write
$V=\overline{\G_a\cdot\G_m\cdot\zeta}$. The inertia group $\Delta_{\zeta}$ has
the form $G(p)$ and by Lemma 12.1 it follows that $p=(a:b:c),a,b,c\neq 0$, is
not possible.

\subsection{The case $p=(a:0:c), a,c\neq 0$}\label{sec:12.4.1}
Then $V$ is not pointwise invariant under the $\G_m$-operation $\sigma$
(notations as in Lemma 11.1), as the following argumentation will show: Let
$\cJ\leftrightarrow\zeta$ be the corresponding ideal, let $\cP$ be an
associated prime ideal, which is $G(p)$-invariant. If $t\in\cP$, then from
(Appendix D, Lemma 2) it follows that $\cP=(x,y,z,t)$, contradiction. Thus $t$
is a non zero divisor of $\cO_{\P^3}/\cJ$. If $\cJ$ would be invariant under
$\sigma$, then $\cJ$ would be generated by elements of the form $f\cdot t^n$,
where $f\in S_m,m,n\in\N$. It follows that $f\in\cJ$ and thus $\cJ$ is
invariant under $\Gamma$ (cf. 12.3.1). But by assumption
$\Delta_{\zeta}=G(a:0:c)\not\supset\Gamma$.

The inertia group $T_{\zeta}\subset T(4;k)$ has the form $T(\rho)$, where
$\rho_3>0$ ( Lemma 11.1 a). By Appendix E $H^0(\cJ(b))$ has a standard basis
$f_i=M_ip_i(X^\rho)$.

  Let be $W:=\G_a\cdot\G_m\cdot\zeta$ The morphism $h$is defined on $V\cap
  U_t\supset W$. As $\overline{W}=V$, it follows that
  $h(W)=\{\psi_{\alpha}(\zeta')|\alpha\in k\}$ is dense in $h(V\cap U_t)$,
  where $\zeta'=h(\zeta)\leftrightarrow\cJ'$. As $\zeta\in U_t$, it follows
  that $\zeta_0=\lim\limits_{\lambda\to 0}\sigma(\lambda)\zeta\in U_t$ too 
  (see [G88], Lemma 4, p. 542, and [G89], 1.6, p.14). As $\rho_3>0$ it follows
  that $\zeta'_0=\zeta'$. Now write $D=\{ \psi_{\alpha}(\eta)|\alpha \in
  k\}^-$, where $\eta\in V(k)$ is invariant under $T(4;k)$ and $G_3$, hence
  $C\subset V\cap U_t$ and $h(\eta)\in \{\psi_{\alpha}(\zeta')|\alpha\in
  k\}^-$. As $\eta$ is not $\G_a$-invariant , $h(\eta)$ is not
  $\G_a$-invariant, hence $h(\eta)\in\{\psi_{\alpha} (\zeta')|\alpha\in k\}$.
  Then (Appendix D, Lemma 3) shows $h(\eta)=\zeta'=\zeta'_0$. $H:=\left\{
    \begin{pmatrix} 1&0&*&*\\ 0&1&*&*\\ 0&0&1&0\\
      0&0&0&1\end{pmatrix}\right\}<G(p)$ is normalized by $\sigma$, hence
  $\zeta_0$ is $H$-invariant. As $\zeta_0\in U_t$ is $\G_m$-invariant , it
  follows that $\zeta_0$ is $\Gamma$-invariant. But then $\zeta_0$ is
  invariant under $G_3$, and as $\zeta'_0=\eta'$ corresponds to an ideal in
  $y$-standard form, $\cJ_0\leftrightarrow \zeta_0$ is of proper type 3.

  As $G(p)$ is unipotent there is an eigenvector $f\neq 0$ in
  $H^0(\cJ(\nu)),\nu:=\min \{ n|H^0(\cJ(n))\neq (0)\}$. From $x \partial
  f/\partial t \in\langle f\rangle$ it follows that $\partial f/\partial t=0$.
  From $y\partial f/\partial z\in\langle f\rangle$ it follows $\partial
  f/\partial z=0$ and from $cx\partial f/\partial y-az\partial f/\partial
  t\in\langle f\rangle$ we deduce $f=x^{\nu}$ (cf. Appendix D, Lemma 2). But
  then $x^{\nu}\in\cJ_0$, too. Now $h^0(\cJ_0(n))=h^0(\cJ(n)), n\in\Z$ (cf.
  [G88], Lemma 4, p.542 and [G89], 1.6, p.14). But then from Lemma 12.2 it
  follows that $g<0$.

\begin{conclusion}\label{con:12.1}
  In the case $p=(a:0:c), a,c\neq 0$, $V$ does not contain a $1$-cycle of
  proper type 3, if $g>g(d)$ is supposed.
\end{conclusion}

\subsection {The case $ p = (a:b:0), a,b\neq 0$}\label{sec:12.4.2}
As $T_{\zeta}\subset T(1,-2,1,0)$ (cf. Auxiliary Lemma 2 ), it follows
from Lemma 11.1 that $V$ is pointwise invariant under $
\sigma(\lambda):x\mapsto x,y\mapsto y,z\mapsto z,t\mapsto\lambda t$. As $V$ is
pointwise invariant under $\Gamma$, one has $V\subset G_{\Phi}$, where $\Phi$
is the Hilbert function of $\cJ\leftrightarrow\zeta$ and $G_{\Phi}$ is the
corresponding ``graded Hilbert scheme''. This had been defined in ([G4],
Abschnitt 2) as follows: $\G_m$ operates on $\HH_Q$ by $\sigma$. Then it is
shown in (loc. cit.) that $G:=(\HH_Q)^{\G_m}\cap U_t$ is a closed subscheme of
$\HH_Q$, and $G$ is a disjunct union of closed subschemes $G_{\Phi}$, where
$G_{\Phi}$ parametrizes the corresponding ideals with Hilbert function $\Phi$.

Now suppose $D=\{\psi_{\alpha}(\eta)|\alpha\in k\}^-\subset V$ is a $1$-cycle
of proper type 3, where $\eta$ corresponds to an ideal $\cL$ of proper type 3.
Then $\cL$ and $\cJ\leftrightarrow\zeta$ have the same Hilbert function.

Let $\nu$ and $f\in H^0(\cJ( \nu))$ be defined as in (12.4.1). From $x\partial
f/\partial z\in\langle f\rangle$ and $y\partial f/\partial t\in\langle
f\rangle$ it follows that $\partial f/\partial z=\partial f/\partial t=0$. But
then from $bx\partial f/\partial y-ay\partial f/\partial z\in\langle f\rangle$
it follows that $\partial f/\partial y=0$, hence $f=x^{\nu}$. (cf. Appendix D,
Lemma 2). If $\cI$ corresponds to any point $\xi\in\G_a\cdot \G_m\cdot\zeta$,
then $x^{\nu}\in H^0(\cI(\nu))$, and the same argumentation as in (8.4.2)
shows this is true for all points in $V$. But from $x^{\nu}\in H^0(\cL(\nu))$
it follows that $g<0$.

\begin{conclusion}\label{con:12.2}
  In the case $p=(a:b:0), a,b\neq 0$, $V$ does not contain a $1$-cycle of
  proper type 3, if $g>g(d)$ is supposed.\qed
\end{conclusion}

\subsection{}\label{sec:12.4.3}
If \;$V$ contains a 1-cycle of proper type 3, then $V$ cannot be pointwise
$\G_a$-invariant, so the case $p=(0:b:c)$ cannot occur, at all.

\begin{lemma}\label{lem:12.17}
  Suppose $g>g(d)$. If\;$V\subset Z(\HH_Q)$ is a $B(4;k)$-invariant surface
  containing a 1-cycle of proper type 3, then $V$ is pointwise invariant under
  $G_3=G(1:0:0)$.\qed
\end{lemma}

% ---- fig12.1 ------
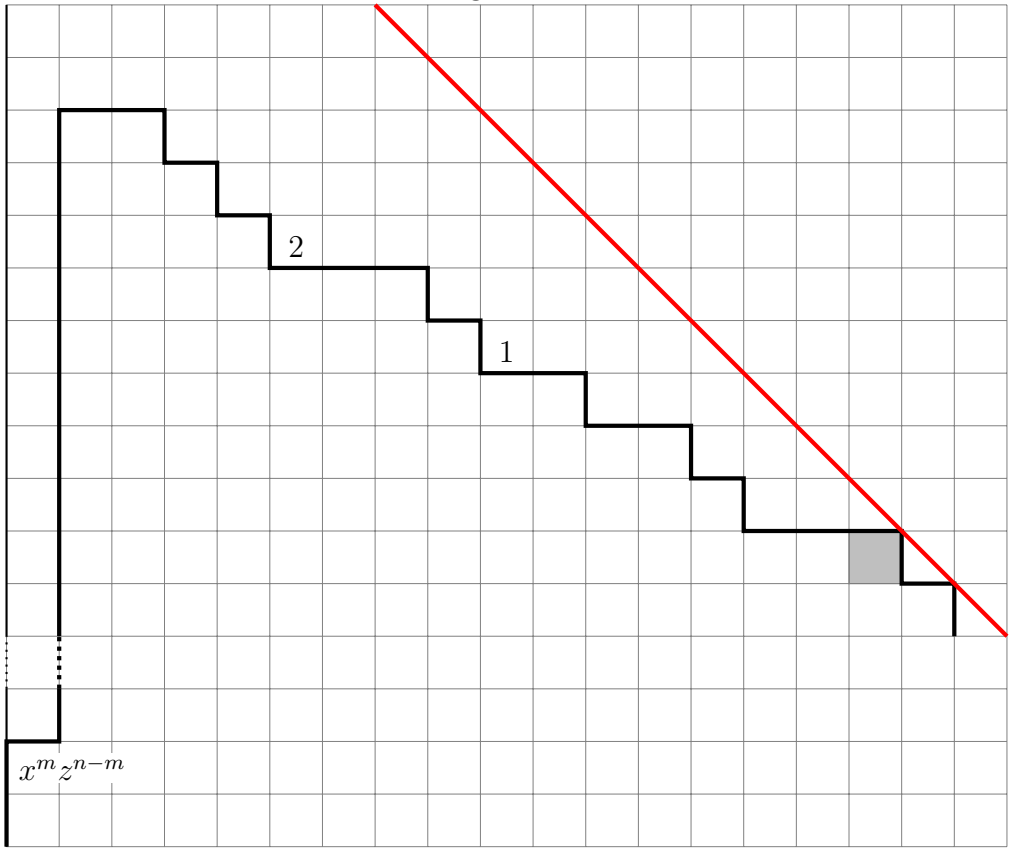
\begin{minipage}{19cm*\real{0.7}} \label{fig:12.1}
\centering Fig.: 12.1\\ 
\tikzstyle{help lines}=[gray,very thin]
\begin{tikzpicture}[scale=0.7]
 \draw[style=help lines]  grid (19,14);
 % axes 
 \draw[thick] (0,0) -- (0,4) (0,6) -- (0,14); 
 \draw[thick,dotted] (0,4) -- (0,6);
 \draw[->] (11.5,12.2) -- (8.7,11.4);  % arrow for text monomials
 % labels
 {
 \pgftransformxshift{0.5cm}
 \draw (0.75,0.2) node[above,fill=white,inner sep=1pt] {$x^mz^{n-m}$};
\draw[anchor=base] (14,12.2) node[fill=white,inner sep=1pt] 
                     {$c$ monomials are missing};
 }
 % steps 
 \draw[\Red,ultra thick] (0,0) -- (0,1) -- (1,1) -- (1,4)
   (1,6) -- (1,8) -- (2,8) -- (2,10) -- (3,10) -- (3,13) -- (5,13) -- (5,12) --
   (6,12) -- (6,11) -- (7,11) -- (7,10) -- (9,10) -- (9,9) -- (10,9) -- (10,8)
    -- (11,8) -- (11,7) -- (13,7) -- (13,6) -- (14,6) -- (14,5) -- (17,5) --
    (17,4) -- (18,4) -- (18,3);
 \draw[\Red,ultra thick,dotted] (1,4) -- (1,6);
 % line
  \draw[\Black, ultra thick] (8,14) -- (19,3);  
\end{tikzpicture}
\end{minipage}
\vspace{0.5cm}

% ---- fig12.2 ------
\begin{minipage}{19cm*\real{0.7}} \label{fig:12.2}
\centering Fig.: 12.2\\ 
\tikzstyle{help lines}=[gray,very thin]
\begin{tikzpicture}[scale=0.7]
 \draw[style=help lines]  grid (19,11);
 % axes 
 \draw[thick] (0,0) -- (0,4) (0,6) -- (0,11); 
 \draw[thick,dotted] (0,4) -- (0,6);
 % labels
 {
 \pgftransformxshift{0.5cm}
 \draw (0.75,0.2) node[above,fill=white,inner sep=1pt] {$x^mz^{n-m}$};
 }
 % steps 
 \draw[\Red,ultra thick] (0,0) -- (0,1) -- (1,1) -- (1,4)
   (1,6) -- (1,10) -- (3,10) -- (3,9) -- (4,9) -- (4,8) -- (5,8) -- (5,7) --
   (8,7) -- (8,6) -- (9,6) -- (9,5) -- (11,5) -- (11,4) -- (13,4) -- (13,3) --
   (14,3) -- (14,2) -- (17,2) -- (17,1) -- (18,1) -- (18,0);
 \draw[\Red,ultra thick,dotted] (1,4) -- (1,6);
 % line
  \draw[\Black, ultra thick] (8,11) -- (19,0);  
\end{tikzpicture}
\end{minipage}

%\newpage

% ---- fig12.3 ------
\begin{minipage}{19cm*\real{0.7}} \label{fig:12.3}
\centering Fig.: 12.3\\ 
\tikzstyle{help lines}=[gray,very thin]
\begin{tikzpicture}[scale=0.7]
\filldraw[fill=gray,opacity=0.5] (16,5) -- (16,6) -- (17,6) -- (17,5) -- cycle;
 \draw[style=help lines]  grid (19,16);
 % axes 
 \draw[thick] (0,0) -- (0,3) (0,4) -- (0,16); 
 \draw[thick,dotted] (0,3) -- (0,4); 
 % labels
 {
 \pgftransformxshift{0.5cm}
  \draw (5,11) node[above] {$2$};
  \draw (9,9) node[above] {$1$};
  \draw (0.75,1.2) node[above,fill=white,inner sep=1pt] {$x^mz^{n-m}$};
 }
 % steps 
 \draw[\Red,ultra thick] (0,0) -- (0,2) -- (1,2) -- (1,3) 
    (1,4) -- (1,14) -- (3,14) -- (3,13) -- (4,13) -- (4,12) -- (5,12) -- (5,11)
     -- (8,11) -- (8,10) -- (9,10) -- (9,9) -- (11,9) -- (11,8) -- (13,8) --
     (13,7) -- (14,7) -- (14,6) -- (17,6) -- (17,5) -- (18,5) -- (18,4);
 \draw[\Red,ultra thick,dotted] (1,3) -- (1,4);
 % line
  \draw[\Black, ultra thick] (7,16) -- (19,4);  
\end{tikzpicture}
\end{minipage}
\vspace{0.5cm}

% ---- fig12.4 ------
\begin{minipage}{14cm*\real{0.7}} \label{fig:12.4}
\centering Fig.: 12.4\\ 
\tikzstyle{help lines}=[gray,very thin]
\begin{tikzpicture}[scale=0.7]
 \draw[style=help lines]  grid (14,15);
 % axes 
 \draw[thick] (0,0) -- (0,6) (0,8) -- (0,15); 
 \draw[thick,dotted] (0,6) -- (0,8); 
 \draw[thick,->] (7.4,11.4) -- (5.7,10.3);  % arrow for text monomials
 % labels
 {
 \pgftransformxshift{0.5cm}
 \draw[anchor=base] (10,11.2) node[fill=white,inner sep=1pt]
                                   {$c$ monomials are missing};
 \draw (0.75,1.2) node[above,fill=white,inner sep=1pt] {$x^mz^{n-m}$};
 }
 % steps 
 \draw[\Red,ultra thick] (0,0) -- (0,2) -- (1,2) -- (1,6) 
 (1,8) -- (1,14) -- (2,14) -- (2,13) -- (3,13) -- (3,12) -- (4,12) -- (4,11)
 -- (5,11) -- (5,10) -- (6,10) -- (6,9) -- (7,9) -- (7,8)
%-- (8,8) -- (8,7) -- (9,7) -- 
(9,6) -- (10,6) -- (10,5) -- (12,5) -- (12,4);
 \draw[\Red,ultra thick,dotted] (7,8) -- (8,8) -- (8,7) -- (9,7) -- (9,6)
      (1,6) -- (1,8);
 % line
  \draw[\Black, ultra thick] (2,15) -- (14,3);  
\end{tikzpicture}
\end{minipage}

%%%%%%%%%%%%%%%%%%%%%%%%% CHAPTER 13 %%%%%%%%%%%%%%%%%%%%%%%%%%%%%%%

\chapter{Relations in $B(4;k)$-invariant surfaces} \label{ch:chap13} 
We suppose $g>g(d)$ and let $V\subset X=Z(\HH_Q)$ be a $B(4;k)$-invariant
surface, which contains a 1-cycle of proper type 3. From 
 it
follows that $V$ is pointwise $G_3$-invariant. Then from ([T1], Proposition 0,
p.3) we conclude that any $B$-invariant 1-prime cycle $D\subset V$ is either
pointwise $\Delta$-invariant or a 1-cycle of type 3. The aim is to describe
the relations in $Z_1^B(X)$ defined by the standard construction of Section
9.2 carried out with regard to $V$.

\section{First case : $V$ is not pointwise invariant under the
  $\G_m$-operation $\sigma$}
We use the notation of (11.1). According to 
Lemma 11.1  $V=\overline{\G_m\cdot\G_a\cdot\zeta}$, where $T_{\zeta}=T(\rho)$
and $\rho_3>0$, for otherwise $\zeta$ would be $\G_m$-invariant and hence $V$
would be pointwise $\G_m$-invariant. Let $\cJ\leftrightarrow\zeta$ and
$C:=\{\psi_{\alpha}(\zeta)|\alpha \in k\}^-$. If one chooses a standard basis
of $H^0(\cJ(b))$ consisting of $T(\rho)$-semi-invariants, then one sees that
$h(V)=C':=\{\psi_{\alpha}(\zeta')|\alpha\in k\}^-$, where $\zeta'=h(\zeta)$.
As $V$ contains a 1-cycle of proper type 3, $C'$ is a $y$-standard cycle,
generated by $\zeta'\leftrightarrow\cJ'$.

Now if $D\subset V$ is a 1-cycle of proper type 3, then $h(D)=C'$. If $D$ is
of type 3, but not of proper type 3, then $h(D)$ is not a $y$-standard cycle.
Write $D=\overline{\G_a\cdot\eta},\eta\in V(k)$ invariant under $T(4;k)$. It
follows that $\eta'=h(\eta)$ is one of the two $T(3;k)$-fixed points of $C'$
(cf. Appendix D, Lemma 3). If $\eta'=\zeta'$, then $D$ would be of proper type
3. Hence $\eta'$ is the unique point of $C'$, invariant under $B(3;k)$ and $h(D)$ equals the point $\eta'$ .\\
\indent If $D$ is pointwise $\Delta$-invariant, then $h(D)$ is pointwise invariant under $U(3;k)$ and hence equals the point $\eta'$ , too .\\

Let $\CC_{0/\infty}=\lim\limits_{\lambda\to 0/\infty}\sigma(\lambda)C$ be the
$B$-invariant limit curves of $C$ coming from the standard construction. We
write in $Z_1^B(V)$:
\[
  \CC_0=\sum m_iA_i+Z_0 \quad \text{and} \quad\CC_{\infty}=\sum n_j B_j+Z_{\infty}
\]
where $A_i$ and $B_j$ are the 1-cycles of proper type 3, occurring in $\CC_0$
and $\CC_{\infty}$, respectively, and $m_i,n_j\in\N-(0)$. All the other prime
cycles occurring in $\CC_0$ and $\CC_{\infty}$ are summed up in $Z_0$ and
$Z_{\infty}$, respectively.

Let $\cM_n$ be a tautological line bundle on $\HH_Q$. Then $(\cM_n\cdot
A_i)=\delta n+a_i, (\cL_n\cdot B_j)=\delta n+b_j$. Here $\delta\in\N-(0)$ is
independent of $i$ and $j$, as $h(A_i)=h(B_j)=C'$ for all $i,j$, whereas the
constants $a_i$ and $b_j$ still depend on $A_i$ and $B_j$, respectively. From
$[\CC_0]=[\CC_{\infty}]$ it follows that $\sum m_i(\delta n+a_i)=\sum
n_j(\delta n+b_j) + c$ for all $n\gg 0$, where $c$ depends only on $Z_0$ and $Z_{\infty}$ (see Appendix D, Lemma 4) . It follows that $\sum m_i=\sum n_j$. Therefore we can write
\begin{equation}\label{eq:1}
  \CC_0-\CC_{\infty}=\sum_k(E_k-F_k)+\sum_k G_k
\end{equation}
where $(E_1,E_2,\dots):=(A_1,\dots,A_1,A_2,\dots,A_2,\dots)$ and
$(F_1,F_2,\dots):=(B_1,\dots,B_1,B_2,\\ \dots,B_2,\dots)$. Here $A_1$
(respectively $B_1$) are to occur $m_1$-times (respectively $n_1$-times) etc.
By the way, the arbitrary association $E_k\mapsto F_k$ is possible because of
$\sum m_i=\sum n_j$. If $E_k=F_k$, the summand $E_k-F_k$ is left out. $G_k$ is
composed of either pointwise $U(4;k)$-invariant curves or 1-cycles of type 3,
which are not proper.

\section{Second case: $V$ is pointwise invariant under the $\G_m$-operation
  $\sigma$, but not pointwise invariant under the $\G_m$-operation $\tau$}\label{sec:13:2}
We use the same notations as in (11.3).

\noindent {\it 1st subcase:} $h(V)$ is not 2-dimensional. By assumption $V$
contains a 1-cycle $D$ of proper type 3, hence $h(V)=h(D)=:D'$ is a
$y$-standard cycle and all the other 1-cycles in $V$, which are of proper type
3, are mapped by $h$ onto $D'$. Then one carries out the standard construction
by means of the operation $\tau$ and one gets formally the same relations as
(13.1).

\noindent {\it 2nd subcase:} $h(V)=V'\subset H^d$ is a $B(3;k)$-invariant
surface, which is pointwise invariant under $G'_3=\left\{ \begin{pmatrix}
    1&0&*\\ 0&1&*\\ 0&0&1\end{pmatrix}\right\}<U(3;k)$. As $V'$ contains a
$y$-standard cycle, $V'$ is not pointwise $\G_a$-invariant.

At first, the standard construction is carried out in
$V=\overline{\G_m\cdot\G_a\cdot\zeta}$ (cf. Lemma 11.1):
$C:=\{\psi_{\alpha}(\zeta)|\alpha\in k\}^-$ is a closed curve in $V$ with
Hilbert polynomial $p$ and $\lambda\mapsto\tau(\lambda)C$ defines a morphism
$\G_m\to\Hilb^p(V)^{\G_a}$, whose extension to $\P^1$ defines a flat family
$\cC\hookrightarrow V\times\P^1$ over $\P^1$ such that
$\cC_{\lambda}:=p_2^{-1}(\lambda)=\tau(\lambda)C\times \{\lambda\}$, if
$\lambda\in k^*$. $\CC_{0/\infty}=:p_1(\cC_{0/\infty})$ are called the limit
curves of $C$ and $[\CC_0]-[\CC_{\infty}]$ is the ``standard relation''
defined by $V$.

\noindent Put $U : =\{ \tau(\lambda)\psi_{\alpha}(\zeta)|\alpha\in k,
\lambda\in k^*\}$.

\noindent Put $U': =\{\tau(\lambda)\psi_{\alpha}(\zeta')|\alpha\in k,
\lambda\in k^*\}$.

\noindent Then $\overline{U}=V $\;and\; $\overline{U'}=V'$. Carrying out the
standard construction by means of $C':=\{\psi_{\alpha}(\zeta')|\alpha\in
k\}^-$ one gets a flat family $\cC'\hookrightarrow
V'\times\P^1\stackrel{p_2}{\longrightarrow}\P^1$. One has a morphism
$\cC\hookrightarrow V\times\P^1\stackrel{h\times id}{\longrightarrow}
V'\times\P^1$, which is denoted by $f$.

\noindent Put $ {\mathbf{U}}:
=\{(\tau(\lambda)\psi_{\alpha}(\zeta),\lambda)|\alpha\in k, \lambda\in
k^*\}\subset\cC$.

\noindent Put $ {\mathbf{U'}}:
=\{(\tau(\lambda)\psi_{\alpha}(\zeta'),\lambda)|\alpha\in k, \lambda\in
k^*\}\subset\cC'$.

\noindent $\cC$ and $\cC'$ are reduced and irreducible (see [T1], proof of
Lemma 1, p.6). Hence $\overline{\mathbf{U}} =\cC$ and
$\overline{\mathbf{U'}}=\cC'$.  As $f(\mathbf{U})=\mathbf{U'}$ and $f$ is
projective, $f(\cC)=\cC'$ follows. As the diagram
\[
  \begin{array}{ccccc}
  \cC &&\stackrel{f}{\longrightarrow}&&\cC'\\
  &p_2\searrow&&\swarrow p_2&\\
  &&\P^1&&
  \end{array}
\]
is commutative, $f(\cC_0)=\cC'_0$ and $f(\cC_{\infty})=\cC'_{\infty}$ follows. As the diagram
\[
  \begin{array}{ccc}
  V\times\P^1 & \stackrel{h\times id}{\longrightarrow} & V'\times\P^1\\ \\
  p_1\big\downarrow && \big\downarrow p_1\\ \\
  V & \stackrel{h}{\longrightarrow} & V'
  \end{array}
\]
is commutative, it follows that $\CC'_{0/\infty}: =p_1(\cC'_{0/\infty})=p_1
f(\cC_{0/\infty})=h(\CC_{0/\infty})$. Let be $\zeta_{0/\infty}:
=\lim\limits_{\lambda\to 0/\infty}\tau(\lambda)\zeta$ and 
 $\zeta'_{0/\infty}:=\lim\limits_{\lambda\to 0/\infty}\tau(\lambda)\zeta'$.

Now from Lemma 9.2, it follows that $C'_{0/\infty}:=\{
\psi_{\alpha}(\zeta'_{0/\infty})|\alpha\in k\}^-$ are the only $y$-standard
cycles in $\CC'_{0/\infty}$ and they both occur with multiplicity 1. We want
to show that $C_{0/\infty}: =\{\psi_{\alpha}(\zeta_{0/\infty})|\alpha\in
k\}^-$ are the only 1-cycles of proper type 3 in $\CC_{0/\infty}$, and they
both occur with multiplicity 1. In order to show this, we consider the
extension of $\tau:\lambda\mapsto \tau(\lambda)\zeta$ to a morphism
$\overline{\tau}:\A^1\to V$. Then $\zeta_0=\overline{\tau}(0)$. As there is a
commutative diagram
\[
  \begin{array}{ccc}
  \G_m & \stackrel{\tau}{\longrightarrow} & V\\ \\
  & \tau'\searrow\qquad & \big\downarrow h\\ \\
  & &V'
 \end{array}
\]
where $\tau'(\lambda): =\tau(\lambda)\zeta'$ and as the extensions are
determined uniquely, it follows that $\overline{\tau'}=h\circ\overline{\tau}$,
hence $h(\zeta_0)=\zeta'_0$. By constructing the corresponding extensions to
$\P^1$, it follows in the same way that $h(\zeta_{\infty})=\zeta'_{\infty}$.
Hence we get $h(C_{0/\infty})=C'_{0/\infty}$. Therefore $C_0$ and $C_{\infty}$
are 1-cycles of proper type 3, and from $\zeta_{0/\infty}\in\CC_{0/\infty}$ we
conclude $C_{0/\infty}\subset\CC_{0/\infty}$.

Assume there is another 1-cycle $D$ of proper type 3 contained in $\CC_0$, or
assume that $C_0$ occurs with multiplicity $\ge 2$ in $\CC_0$. Then there is
an equation $ [\CC_0] = [C_0]+[D]+\cdots$ in $Z_1^B(V)$, where $D=C_0$, if $C_0$
occurs with multiplicity $\ge 2$. From Lemma 9.2 it follows that
$h(D)=C'_0$. It follows that
\begin{align*}
  h_*([\CC_0])& =h_*([C_0])+h_*([D])+\cdots\\
             &  =\deg (h|C_0)[h(C_0)]+\deg (h|D)[h(D)]+\cdots\\
             & =h_*([C])=\deg (h|C)[C'].
\end{align*}

\noindent As $C=\{\psi_{\alpha}(\zeta)|\alpha\in k\}^-$ and
$\alpha\mapsto\psi_{\alpha}(\zeta)$ is injective and the same is true for $C'$
and $\alpha\mapsto\psi_{\alpha}(\zeta') , h|C$ is an isomorphism. The same
argumentation shows that $h|C_0$ and $h|C_{\infty}$ are isomorphisms. Hence we
get $[C'_0]+[C'_0]+\cdots=[C']=[\CC'_0]$, which means that $C'_0$ occurs with
multiplicity $\ge 2$ in $\CC'_0$, contradiction. \\
\noindent The same argumentation shows that $C_{\infty}$ is the only 1-cycle
of proper type 3 in $\CC_{\infty}$, and it occurs with multiplicity 1. This
gives a relation of the form
\begin{equation}\label{eq:2}
  \CC_0-\CC_{\infty}=C_0-C_{\infty}-Z
\end{equation}
where $C_0$ and $C_{\infty}$ are 1-cycles of proper type 3 and $Z$ is a
1-cycle whose components are either pointwise $U(4;k)$-invariant curves or
1-cycles of type 3, which are not proper.

\section{Third case:$V$ is pointwise invariant under $\sigma$ and $\tau$}
Then we write $V=\overline{\G_m\cdot\G_a\cdot\zeta}$ as in the 3rd case of
 Lemma 11.1.

\noindent {\it 1st subcase :} $h(V)$ is not 2-dimensional. The same
argumentation as in the first subcase of (13.2), using the $\G_m$-operation
$\omega$ as in the third case of Lemma 11.1 instead of the
$\G_m$-operation $\tau$, gives relations of the form (13.1).

\noindent {\it 2nd subcase:} If $h(V)$ is 2-dimensional, the same
argumentation as in the second subcase of 13.2 , with $\omega$
instead of $\tau$, gives relations of the form (13.2).

\begin{proposition}
Assume $g>g(d)$ and let $V\subset X:=Z(\HH_Q)$ be a $B(4;k)$-invariant surface
containing a 1-cycle of proper type 3. Then each relation in $Z_1^B(X)$
defined by $V$, which contains a 1-cycle of proper type 3 is a sum of
relations of the form $C_1-C_2+Z$. Here $C_1$ and $C_2$ are 1-cycles of proper
type 3, and $Z$ is a 1-cycle whose prime components either are pointwise
$\Delta$-invariant or 1-cycles of type 3 , which are not proper .
\end{proposition}
\begin{proof} This follows from the foregoing discussion.
\end{proof}

%%%%%%%%%%%%%%%%%%%%%%%%% CHAPTER 14 %%%%%%%%%%%%%%%%%%%%%%%%%%%%%%%

\chapter{Necessary and sufficient conditions} \label{ch:chap14}
%% keep
We take up the notations introduced in Chapter 1. Let be $d\ge 5,
g(d)=(d-2)^2/4, \HH=H_{d,g}, \cA(\HH) = \Im (A_1(\HH^{U(4;k)})\longrightarrow
  A_1(\HH))$. Obviously, the cycle $C_3 
=\{(x^a,\alpha x+y,x^{a-1}z^{b-a+1})|\alpha\in k\}^-$, where
$d=a-1,g=(a^2-3a+4)/2-b$, is a  1-cycle of proper type 3.

\section{The necessary condition.}\label{sec:14.1}
In the proof of Theorem 1.1 in Chapter 10 we replace $\cH$, $U(3;k)$, $B(3;k)$
and ``$y$-standard cycle'' by $\HH$, $U(4;k)$, $B(4;k)$ and ``1-cycle of
proper type 3'', respectively. Then using Proposition 13.1 instead of
 Proposition 9.1 , the same reasoning as in the proof of Theorem 1.1 gives the

\begin{conclusion}\label{con:14:1} If $d\ge 5$ and $g>g(d)$, then $[C_3]\notin\cA(\HH)$.\qed
\end{conclusion}

\section{The sufficient condition.}\label{sec:14.2}

\subsection{The case $d\ge 5$ and $d$ odd.}\label{sec:14.2.1}
In ([T2], 3.3.2, Folgerung, p.129) it had been shown $[C_3]\in\cA(\HH)$, if
$a\ge 6$ and $b\ge a^2/4$. We express this condition by means of the formulas
in (1.1): 
\[
b\ge a^2/4\Leftrightarrow g\le\frac{1}{2}
[(d+1)^2-3(d+1)+4]-\frac{1}{4}(d+1)^2=\frac{1}{4}(d^2-4d+3). 
\]
As $d$ is odd, this is equivalent to $g\le g(d)=\frac{1}{4}(d-2)^2$.

\begin{conclusion}\label{con:14.2} If $d\ge 5$, $d$ is odd and $g\le g(d)$,
  then $[C_3]\in\cA(H_{d,g})$.\qed
\end{conclusion}

\subsection{The case $d\ge 6$ and $d$ even.}\label{sec:14.2.2}
We will show that the bound given in \\                                                                                                                             ([T2], 3.3.3 Folgerung, p.129) is already valid, if $a\ge 7$. \\
\noindent $1^{\circ}$ We consider the Hilbert function $\varphi$ of the
monomial ideal $(x^2,xy^{e-2},y^e),e:=d/2+1\ge 4$, which is represented in
Fig. 14.1. In Section (2.2.3) this Hilbert function had been denoted by $\chi$
and it had been shown that $g^*(\varphi)=\frac{1}{4}(d-2)^2=:g(d)$. The
figures 14.1 - 14.6 show all possible monomial ideals with Hilbert function
$\varphi$. Besides this, we consider the ideal $\cI :
=(xy,x^e,y^e,x^{e-1}+y^{e-1})$ which corresponds to a closed point $\xi$ in
the Iarrobino variety $I_{\varphi}$. This variety parametrizes all sequences
$(U_0,U_1,\dots)$ of subspaces $U_i\subset R_i$ with dimension
$\varphi'(i)=\varphi(i)-\varphi(i-1)$, such that $R_1U_i\subset
U_{i+1}, i\in\N$ . Here $R$ is the polyno - \\                                                                                                                       mial ring in two variables . \\
\noindent $2^{\circ}$ We show that $V=\overline{\G_a\cdot\G_m\cdot\xi}$ is
2-dimensional, where $\G_m$ operates by $\tau(\lambda):x\mapsto\lambda
x,y\mapsto y,z\mapsto z$. If this would not be the case, then
$\G_a\times\G_m\stackrel{f}{\longrightarrow}I_{\varphi}$ defined by
$(\alpha,\lambda)\mapsto\psi_{\alpha}\tau(\lambda)\xi$ would have a fibre with
infinitely many points. The argumentation in (8.3) then showed that one of the
points $\xi_{0/\infty}=\lim\limits_{\lambda\to 0/\infty}\tau(\lambda)\xi$ is
invariant under $\G_a$. As $\xi_0\leftrightarrow\cI_0=(xy,x^e,y^{e-1})$ and
$\xi_{\infty}\leftrightarrow\cI_{\infty}=(xy,x^{e-1},y^e)$, this is wrong as
$e\ge 4$. Thus we can carry out the standard construction with
$C=\{\psi_{\alpha}(\xi)|\alpha\in k\}^-$, if $e\ge 4$. As we have already
noted in the proof of Lemma 9.1 , the curves $\tau(\lambda)C$ are
$\G_a$-invariant, hence the construction of the family $\cC$ takes place in
$(\Hilb^p(I_{\varphi}))^{\G_a}$. Therefore the limit curves $\CC_{0/\infty}$ are
invariant under $B(3;k)$.

As usual, we put $C_{0/\infty}=\{ \psi_{\alpha}(\xi_{0/\infty})|\alpha\in k\}^-$.\\
\noindent $3^{\circ}$ Let $\cI\leftrightarrow (U_0,U_1,\cdots)$, where
$U_i=xyR_{i-2}$, if $0\le i\le e-2, U_{e-1}=xyR_{e-3}+\langle
x^{e-1}+y^{e-1}\rangle, U_i=R_i,i\ge e$. We get
\begin{align*}
  \psi_{\alpha}(U_{e-1}) & =x(\alpha x+y)R_{e-3}+\langle x^{e-1}+(\alpha x+y)(\alpha x+y)^{e-2}\rangle\\
  & =x(\alpha x+y)R_{e-3}+\langle x^{e-1}+(\alpha x+y)y^{e-2}\rangle\\
  \Rightarrow \quad\dot\wedge \psi_{\alpha}(U_{e-1})& =\alpha x^{e-1}\wedge\alpha
  x^{e-2}y\wedge\cdots\wedge\alpha x^2y^{e-3}\wedge\alpha xy^{e-2} \\
     &  \qquad+ \text{terms with smaller powers of } \alpha.
\end{align*}
Hence $\alpha$-grade $(U_{e-1})=e-1$. On the other hand, by formula (4.2)
in 4.1 we get $\alpha$-grade $(\langle x^{e-2}y,\dots,y^{e-1}\rangle)=e-1$,
too. From this it follows that $\CC_0=C_0$. \\
\noindent $4^{\circ}$ Besides $C_{\infty}$, the limit cycle $\CC_{\infty}$
contains other prime components $D_1, D_2,\dots$ which are $B(3;k)$-invariant
curves in $I_{\varphi}$.

\indent As char$(k)=0$ is supposed, if $m\le n , \GG:=\Grass_m(R_n)$ has only
one $\G_a$-fixed point, namely the subspace $\langle
x^n,\cdots,x^{n-m+1}y^{m-1}\rangle$. Then ([T1], Proposition 0, p.3) shows
that each $B(2;k)$-invariant curve in $\GG$ has the form
$\overline{\G_a\cdot\eta}$, where $\eta\in\GG(k)$ corresponds to a monomial
subspace of $R_n$. It follows that each $B(3;k)$-invariant curve in
$I_{\varphi}$ also has the form $\overline{\G_a\cdot\eta}$, where $\eta\in
I_{\varphi}(k)$ corresponds to a monomial ideal. \\
\noindent $5^{\circ}$ Figure 14.1 - Figure 14.6 show all possible monomial
ideals with Hilbert function $\varphi$, and using formula (4.2) we
compute the degrees of the corresponding 1-cycles: \\
$\deg c_2={e-2\choose 2},\deg c_3={e-2\choose 2}+(e-1),\deg c_4=2{e-2\choose 2}+(e-1),\deg c_5=2{e-2\choose 2}+(e-2),\deg c_6=1$.

\indent We see that $c_2$ (respectively $c_3$) is equal to $C_{\infty}$ (respectively $C_0$). If $C_0$ (respectively $D_i$) occurs in $\CC_{\infty}$ with the multiplicity $n$ (respectively $n_i$), then from $\deg \CC_{\infty}=\deg\CC_0=\deg C_0$ it follows that  ${e-2\choose 2}+(e-1)=n{e-2\choose 2}+n_1\deg D_1+\cdots$ where $n>0,n_i\ge 0$.\\
{\it 1st case:} $e\ge 6$. Then ${e-2\choose 2}>e-1$, hence $n=1$. From $D_1\in
\{ c_4,c_5,c_6\}$ and $e-1\ge n_1\deg D_1$ we conclude that $D_1=c_6$ and
\[
  [C_0]=[C_{\infty}]+(e-1)[c_6]
\]
{\it 2nd case:} $e=5$. Then $\deg C_0=7, \deg C_{\infty}=3,\deg c_4=10,\deg c_5=9$, and we get $[C_0]=2[C_{\infty}]+[c_6]$ or $[C_0]=[C_{\infty}]+4[c_6]$.\\
{\it 3rd case:} $e=4$. Then $\deg C_0=4, \deg C_{\infty}=1,\deg c_4=5,\deg c_5=4$. Therefore one gets
\[
  [C_0]=n[C_{\infty}]+m[c_6]
\]
where $n,m\in\N, n\ge 1,n+m=4$. \\
\noindent $6^{\circ}$ One has a closed immersion $I_{\varphi}\to H_{d,g(d)}$
defined by $\cI\mapsto\cI^*$, where $\cI^*$ is the ideal generated by $\cI$ in
$\cO_{\P^3}$. That means
$H^0(\P^3;\cI^*(n))=\bigoplus\limits^n_{i=0}t^{n-i}H^0(\P^2;\cI(i)),n\in\N$.
In any case, one gets an equation
\[
  [C_0^*]=n[C^*_{\infty}]+m[c_6^*], \quad m,n\in\N,n\ge 1.
\]
Now from ([T1], Lemma 5, p.45) it follows that one can deform $C_0^*$
(respectively $C_{\infty}^*)$ by a sequence of graded deformations of type 3
(cf. [T1], p.44) modulo $\cA(\HH)$ in the cycle $C_3$ (respectively in the
zero cycle). The cycle $c_6^*$ is equal to the cycle
$D_{\alpha},\alpha=(d+2)/2=e$, which had been introduced in ([T4], p.20f). By
([T4], Lemma 5, p.25) one gets $D_e\equiv D_2$ modulo $\cA(\HH)$, and in
([T4], Abschnitt 3.3, p.25) it had been noted that $[D_2]=[D]$, where $D$ is
the cycle introduced in (1.1). From ([T4], Satz 1, p.26) it follows
$[C_3]\in\cA(H_{d,g(d)})$. Applying the shifting morphism (cf. [T3],
Folgerung, p.55 and Proposition, p.56) we get:

\begin{conclusion}\label{sec:14.3}
If $d\ge 6$ is even and $g\le g(d)$, then $[C_3]\in\cA(H_{d,g})$.\qed
\end{conclusion}

\section{Summary.}\label{sec:14.3}
\begin{theorem}\label{thm:14.1}
Let be $d\ge 5$. Then $[C_3]\in\cA(H_{d,g})$ if and only if $g\le g(d)$.
\end{theorem}

\begin{proof} This follows from Conclusion 14.1- 14.3.
\end{proof}

 %% Graphiken 

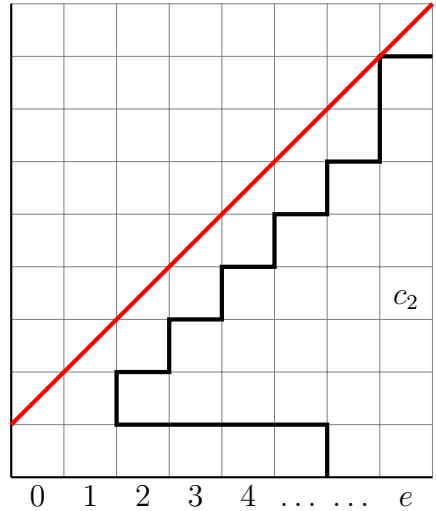
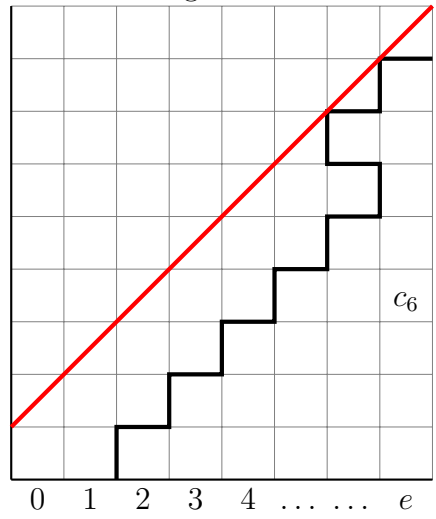
\begin{minipage}{16cm}
 % ---- fig14.1 ------
\begin{minipage}[b]{8cm*\real{0.7}} \label{fig:14.1}
\centering Fig.: 14.1\\ 
\tikzstyle{help lines}=[gray,very thin]
\begin{tikzpicture}[scale=0.7]
 \draw[style=help lines]  grid (8,9);
 % axes 
 \draw[thick] (0,0) -- (0,9); 
 \draw[thick] (0,0) -- (8,0);
 % ticks
 {
 \pgftransformxshift{0.5cm}
 \pgftransformyshift{-0.55cm}
  \foreach \x in {0,1,2,3,4} \draw[anchor=base] (\x,0) node {$\x$}; 
 \draw[anchor=base] (5,0) node {$\dots$};
 \draw[anchor=base] (6,0) node {$\dots$};
 \draw[anchor=base] (7,0) node {$e$};
 }
 % steps  
 \draw[\Red,ultra thick] (2,0) -- (2,1) -- (3,1) -- (3,2) -- (4,2) -- (4,3) --
 (5,3) -- (5,4) -- (6,4) -- (6,6) -- (7,6) -- (7,8) -- (8,8);
 % line
  \draw[\Black, ultra thick] (0,1) -- (8,9);  
\end{tikzpicture}
\end{minipage} 
\hspace{1.5cm}
% ---- fig14.2 ------
\begin{minipage}[b]{8cm*\real{0.7}} \label{fig:14.2}
\centering Fig.: 14.2\\ 
\tikzstyle{help lines}=[gray,very thin]
\begin{tikzpicture}[scale=0.7]
 \draw[style=help lines]  grid (8,9);
 % axes 
 \draw[thick] (0,0) -- (0,9); 
 \draw[thick] (0,0) -- (8,0);
 % ticks
 {
 \pgftransformxshift{0.5cm}
 \pgftransformyshift{-0.55cm}
  \foreach \x in {0,1,2,3,4} \draw[anchor=base] (\x,0) node {$\x$}; 
 \draw[anchor=base] (5,0) node {$\dots$};
 \draw[anchor=base] (6,0) node {$\dots$};
 \draw[anchor=base] (7,0) node {$e$};
 }
{
 \pgftransformxshift{0.5cm}
 \draw (7,3) node[above] {$c_2$};
}
 % steps 
 \draw[\Red,ultra thick] (6,0) -- (6,1) -- (2,1) -- (2,2) -- (3,2) -- (3,3) --
 (4,3) -- (4,4) -- (5,4) -- (5,5) -- (6,5) -- (6,6) -- (7,6) -- (7,8) --
 (8,8);
 % line
  \draw[\Black, ultra thick] (0,1) -- (8,9);  
\end{tikzpicture}
\end{minipage}
\end{minipage}
\vspace{1cm}

\begin{minipage}{16cm}
% ---- fig14.3 ------
\begin{minipage}[b]{8cm*\real{0.7}} \label{fig:14.3}
\centering Fig.: 14.3\\ 
\tikzstyle{help lines}=[gray,very thin]
\begin{tikzpicture}[scale=0.7]
 \draw[style=help lines]  grid (8,8);
 % axes 
 \draw[thick] (0,0) -- (0,8); 
 \draw[thick] (0,0) -- (8,0);
 % ticks
 {
 \pgftransformxshift{0.5cm}
 \pgftransformyshift{-0.55cm}
  \foreach \x in {0,1,2,3,4} \draw[anchor=base] (\x,0) node {$\x$}; 
 \draw[anchor=base] (5,0) node {$\dots$};
 \draw[anchor=base] (6,0) node {$\dots$};
 \draw[anchor=base] (7,0) node {$e$};
 }
{
 \pgftransformxshift{0.5cm}
 \draw (6,3) node[above] {$c_3$};
}
 % steps 
 \draw[\Red,ultra thick] (7,0) -- (7,1) -- (2,1) -- (2,2) -- (3,2) -- (3,3) --
 (4,3) -- (4,4) -- (5,4) -- (5,5) -- (6,5) -- (6,7) -- (7,7);
 % line
  \draw[\Black, ultra thick] (0,1) -- (7,8);  
\end{tikzpicture}
\end{minipage}
\hspace{1.5cm}
% ---- fig14.4 ------
\begin{minipage}[b]{8cm*\real{0.7}} \label{fig:14.4}
\centering Fig.: 14.4\\ 
\tikzstyle{help lines}=[gray,very thin]
\begin{tikzpicture}[scale=0.7]
 \draw[style=help lines]  grid (8,8);
 % axes 
 \draw[thick] (0,0) -- (0,8); 
 \draw[thick] (0,0) -- (8,0);
 % ticks
 {
 \pgftransformxshift{0.5cm}
 \pgftransformyshift{-0.55cm}
  \foreach \x in {0,1,2,3,4} \draw[anchor=base] (\x,0) node {$\x$}; 
 \draw[anchor=base] (5,0) node {$\dots$};
 \draw[anchor=base] (6,0) node {$\dots$};
 \draw[anchor=base] (7,0) node {$e$};
 }
{
 \pgftransformxshift{0.5cm}
 \draw (7,3) node[above] {$c_4$};
}
 % steps 
 \draw[\Red,ultra thick] (7,0) -- (7,1) -- (6,1) -- (6,2) -- (2,2) -- (2,3) --
 (3,3) -- (3,4) -- (4,4) -- (4,5) -- (5,5) -- (5,6) -- (6,6) -- (6,7) --
 (7,7);
 % line
  \draw[\Black, ultra thick] (0,1) -- (7,8);  
\end{tikzpicture}
\end{minipage}
\end{minipage}
\vspace{1cm}

\begin{minipage}{16cm}
% ---- fig14.5 ------
\begin{minipage}[b]{8cm*\real{0.7}} \label{fig:14.5}
\centering Fig.: 14.5\\ 
\tikzstyle{help lines}=[gray,very thin]
\begin{tikzpicture}[scale=0.7]
 \draw[style=help lines]  grid (8,8);
 % axes 
 \draw[thick] (0,0) -- (0,8); 
 \draw[thick] (0,0) -- (8,0);
 % ticks
 {
 \pgftransformxshift{0.5cm}
 \pgftransformyshift{-0.55cm}
  \foreach \x in {0,1,2,3,4} \draw[anchor=base] (\x,0) node {$\x$}; 
 \draw[anchor=base] (5,0) node {$\dots$};
 \draw[anchor=base] (6,0) node {$\dots$};
 \draw[anchor=base] (7,0) node {$e$};
 }
{
 \pgftransformxshift{0.5cm}
 \draw (6,4) node[above] {$c_5$};
}
 % steps 
 \draw[\Red,ultra thick] (6,0) -- (6,1) -- (7,1) -- (7,2) -- (2,2) -- (2,3) --
 (3,3) -- (3,4) -- (4,4) -- (4,5) -- (5,5) -- (5,6) -- (6,6) -- (6,7) --
 (7,7) -- (7,8);
 % line
  \draw[\Black, ultra thick] (0,1) -- (7,8);  
\end{tikzpicture}
\end{minipage}
\hspace{1.5cm}
% ---- fig14.6 ------
\begin{minipage}[b]{8cm*\real{0.7}} \label{fig:14.6}
\centering Fig.: 14.6\\ 
\tikzstyle{help lines}=[gray,very thin]
\begin{tikzpicture}[scale=0.7]
 \draw[style=help lines]  grid (8,9);
 % axes 
 \draw[thick] (0,0) -- (0,9); 
 \draw[thick] (0,0) -- (8,0);
 % ticks
 {
 \pgftransformxshift{0.5cm}
 \pgftransformyshift{-0.55cm}
  \foreach \x in {0,1,2,3,4} \draw[anchor=base] (\x,0) node {$\x$}; 
 \draw[anchor=base] (5,0) node {$\dots$};
 \draw[anchor=base] (6,0) node {$\dots$};
 \draw[anchor=base] (7,0) node {$e$};
 }
{
 \pgftransformxshift{0.5cm}
 \draw (7,3) node[above] {$c_6$};
}
 % steps 
 \draw[\Red,ultra thick] (2,0) -- (2,1) -- (3,1) -- (3,2) -- (4,2) -- (4,3) --
 (5,3) -- (5,4) -- (6,4) -- (6,5) -- (7,5) -- (7,6) -- (6,6) -- (6,7) -- (7,7)
 -- (7,8) -- (8,8);
 % line
  \draw[\Black, ultra thick] (0,1) -- (8,9);  
\end{tikzpicture}
\end{minipage}
\end{minipage}

%%%%%%%%%%%%%%%%%%%%%%%%% CHAPTER 15 %%%%%%%%%%%%%%%%%%%%%%%%%%%%%%%

\chapter{The case $d\ge 5$ } \label{ch:chap15}

We suppose $d\ge 5$, i.e. $a\ge 6$, and $g<{d-1\choose 2}$, i.e. $g$ not
maximal.

From ([T1], Satz 2, p.91) and ([T4], Proposition 2, p.26) it follows that
$A_1(H_{d,g})/\cA(H_{d,g})$ is generated by the so-called combinatorial cycles
$C_1,C_2,C_3$. Using the formulas in (1.1), one shows that $g\le \gamma
(d):={d-2\choose 2}$ is equivalent to $b\ge 2a-4$. \\
{\it 1st case:} $g>\gamma(d)$. This is equivalent to $b<2a-4$. By ([T4], Satz
1, p.26) one has $\cA(H_{d,g})=\langle E\rangle$. Assume there is a relation
$q_0[E]+q_1[C_1]+q_2[C_2]+q_3[C_3]=0, q_i\in\Q$. Computing the intersection
numbers with the tautological line bundles $\cL_n$ by means of the formulas
(1)-(5) in ([T2], p.134f) and the formula in ([T3], Hilfssatz 1, p.50), one
sees that $q_2=0$. Put $r:=2a-4-b$ and denote the $r$-times applied shifting
morphism by $f$ (see [T3], p.52 ). The images of $E,C_1,C_3$ under $f$ in
$H_{d,\gamma(d)}$ are denoted by $e,c_1,c_3$. By ([T2], 3.2.2 Folgerung,
p.124) and ([T3], Anhang 2, p.54f) it follows that $[f(E)]\equiv [e], \langle
f(C_1)\rangle\equiv\langle c_1\rangle$ and $\langle f(C_3)\rangle\equiv\langle
c_3\rangle$ modulo $\cA(H_{d,\gamma(d)})$. By ([T3], Proposition 4, p. 22)
$[c_1] \in \cA( H_{d,\gamma(d)}$. As $\gamma(d)>g(d)$ if $d\ge 5$, from
 Theorem 14.1 it follows that $[c_3]\notin\cA(H_{d,\gamma(d)})$. Hence
  $f(C_3)$ is not a single point, hence $\deg (f|C_3)\neq 0$. Applying $f_*$
  to the relation $q_0[E]+q_1[C_1]+q_3[C_3]=0$ then gives $q_3\deg
  (f|C_3)\cdot [c_3]\in\cA(H_{d,\gamma (d)})$. If $q_3\neq 0$ it would follow
  that $[c_3]\in\cA(H_{d,\gamma(d)})$, contradiction. But from
  $q_0[E]+q_1[C_1]=0$ it follows $q_0=q_1=0$ (cf. [T2], 4.1.3). \\
  {\it 2nd case:} $g(d)<g\le\gamma(d)$. Then $b\ge 2a-4$ and in (1.1) it was
  explained that in this case $A_1(H_{d,g})$ is generated by $E,D,C_2$ and
  $C_3$. If $q_0[E]+q_1[D]+q_2[C_2]+q_3[C_3]=0$, then $q_2=0$ by ([T2],
  loc.cit.). If one assumes that $q_3\neq 0$, it follows that $[C_3]\in\langle
  E,D\rangle\subset\cA(H_{d,g})$, contradicting Theorem 14.1. As in the
  first case we get $q_0=q_1=0$. \\
  {\it 3rd case:} $g\le g(d)$. Then $\cA(H_{d,g})=\langle E,D\rangle,
  [C_1]\in\cA(H_{d,g})$ and $[C_3]\in\cA(H_{d,g})$ (see.[T3], Proposition 4,
  p.22 ; [T4], Satz 1, p.26 ; and finally Theorem 14.1 ). Thus
  $A_1(H_{d,g})=\langle E,D,C_2\rangle$.

  All in all we get: Suppose that $d\ge 5$ and $g<{d-1\choose 2}$, i.e. $g$ is
  not maximal. Put $g(d):=(d-2)^2/4$ and $\gamma(d):={d-2\choose 2}$.

  \begin{theorem}\label{3}
\quad(i) If $g>\gamma(d)$, then $A_1(H_{d,g})$ is freely generated by
$E,C_1,C_2,C_3$. \\
\quad(ii) If $g(d)<g\le\gamma(d)$, then $A_1(H_{d,g})$ is freely generated by
$E,D,C_2,C_3$. \\
\quad(iii) If $g\le g(d)$, then $A_1(H_{d,g})$ is freely generated by $E,D,C_2$.\qed
\end{theorem}

\noindent N.B.: If $g={d-1\choose 2}$, then $A_1(H_{d,g})$ is freely generated
by $C_1:=\{(x,y^d(\alpha y+z))|\alpha\in k\}^-$ and $C_2:=\{(\alpha
x+y,x^{d+1})|\alpha\in k\}^-$ ([T1], Satz 2, p.91).

%%%%%%%%%%%%%%%%%%%%%%%%% CHAPTER 16 %%%%%%%%%%%%%%%%%%%%%%%%%%%%%%%
\chapter{The cases $d=3$ and $d=4$} \label{ch:chap16}

\section{The case $d=3$.}\label{sec:16.1}
If $g$ is not maximal, i.e., if $g\le 0$, then $Q(T)={T-1+3\choose
  3}+{T-4+2\choose 2}+{T-b+1\choose 1}$, where $b\ge 4$. If $b=4$, then in
([T2], p.137) it had been shown that $[C_3]\in\cA(\HH)$. Applying the shifting
morphism , (see [T3] ,p .54 ) then shows that this statement is true for all
$g<0$ (cf. [T3], Anhang 2, Folgerung, p.55 and Proposition, p.56). By ([T1],
Satz, p.91, [T3], Prop. 4, p.22 and Satz 1, p.26) it follows that
$A_1(H_{3,g})$ is generated by $[E],[D],[C_2]$, if $g\le 0$. The three cycles
are linear independent, as already was noted in Chapter 15.

\section{The case $d=4$}\label{sec:16.2}
If $g$ is not maximal, i.e., if $g<3$, then $Q(T)={T-1+3\choose
  3}+{T-5+2\choose 2}+{T-b+1\choose 1}$ and $b\ge 5$. If $b=5$, then $g=2$,
and then $A_1(\HH)=\Q^4$, as had been shown in ([T3], Anhang 1).

We now treat the case $b=6$, i.e., $g=1$. As the condition $b\ge 2a-4$ is
fulfilled, $\cA(\HH)=\langle E,D\rangle$ by ([T4], Satz 1, p.26). The quotient
$A_1(\HH)/\cA(\HH)$ is generated by $C_1,C_2$ and further cycles $C_3,C_4,C_5$
of type 3 (cf. [T1], Satz 2c, p.92). By ([T3], Proposition 4, p.22)
$[C_1]\in\cA(\HH)$, and we are going to simplify the reductions of
$C_3,C_4,C_5$ described in ([T3], Abschnitt 8).

\subsection{The cycle $C_3$.}\label{sec:16.2.1}
By definition a cycle of type 3 has the form $C=\overline{\G_a\cdot\xi}$,
where $\xi$ corresponds to a monomial ideal $\cJ$ with Hilbert polynomial $Q$,
which is invariant under the subgroup $G_3<U(4;k)$ (cf. Chapter 1). Here
$\G_a$ operates as usual by $\psi_{\alpha}:x\mapsto x,y\mapsto\alpha
x+y,z\mapsto z,t\mapsto t$. Let $\cI:=\cJ'\in H^4(k)$ be the image of $\cJ$
under the restriction morphism. If $\cI$ is $\G_a$-invariant, then $\cI$ is
$B(3;k)$-invariant, hence $\deg (\cL_n|C)=$ constant and $[C]\in\cA(\HH)$ (cf.
[T3], Anhang 2, Hilfssatz 2, p.50). Therefore we can assume without
restriction that $\cI$ is not $\G_a$-invariant. As the colength of $\cI$ in
$\cO_{\P^2}$ is equal to 4, the Hilbert function of $\cI$ is equal to one of
the functions represented in Figure 2.7a and Figure 2.7b. In (2.2.1) we had
obtained $g^*(\varphi_1)=1$ and $g^*(\varphi_2)=3$. The Hilbert function
$\varphi_1$ leads to two possible 1-cycles of proper type 3, namely to

\begin{align*}
  F_1&:=\{\psi_{\alpha}(\xi_1)|\alpha\in k\}^-, \quad  \xi_1\leftrightarrow
  (xy,y^2,x^3), \quad \text{and} \\
  F_2&:=\{\psi_{\alpha}(\xi_2)|\alpha\in k\}^-, \quad \xi_2\leftrightarrow (x^2,y^2).
\end{align*}
The Hilbert function $\varphi_2$ leads to different 1-cycles of proper type 3,
which however by means of admissible $G_3$-invariant deformations all can be
deformed modulo $\cA(\HH)$ in $C_3=\{\psi_{\alpha}(\xi_3)|\alpha\in
k\}^-,\xi_3\leftrightarrow (x^5,y,x^4z^2)$.

\noindent With such admissible $G_3$-invariant deformations we can deform
$C_3$ in the cycle generated by $(x^4,xy,y^2,yz^2)$, and afterwards this cycle
can be deformed by the graded deformation $(yz^2,yz^3,\dots)\mapsto
(x^3,x^3z,\dots)$ in the cycle $F_1$.

We deform $F_1$ into $F_2$ in the following way: Put
$\cK:=(x^3,x^2y,xy^2,y^3)$. If $L\subset R_2=k[x,y]_2$ is a 2-dimensional
subspace, then $\langle x,y\rangle L\subset H^0(\cK(3))=R_3$ and $z^nL\cap
H^0(\cK(n+2))=(0)$ for all $n\ge 0$. It follows that the ideal
$(L,\cK)\subset\cO_{\P^3}$ has the Hilbert polynomial $Q$, and $L\mapsto
(L,\cK)$ defines a closed immersion $\P^2\simeq\Grass_2(R_2)\to\HH_Q$. Let
$\langle xy,y^2\rangle\leftrightarrow\eta_1\in\P^1,\langle
x^2,y^2\rangle\leftrightarrow\eta_2\in\P^2,\ell_i:= 
\{\psi_{\alpha}(\eta_i)|\alpha\in k\}^-, i=1,2$. 
Because of reasons of degree one has $[\ell_1]=2[\ell_2]$ in
$A_1(\P^2)$, hence $[F_1]=2[F_2]$ in $A_1(\HH_Q)$. Now
$F_2=\{(x^2,xy^2,y^3,(\alpha x+y)^2|\alpha\in k\}^-=\{(x^2,xy^2,y^3,\lambda
xy+\mu y^2)|(\lambda:\mu)\in\P^1\}$ is equal to the cycle $D_3$ which had been
introduced in ([T4], Abschnitt 3.2, p.20). By Lemma 5 in (loc.cit.) $D_3\equiv
D_2$ modulo $\cA(\HH)$, where $D_2:=\{ (x^2,xy,y^4,\lambda xz^2+\mu
y^3)|(\lambda:\mu)\in\P^1\}$ is the cycle , which had been denoted by $D$ in
Chapter 1. It follows that $[C_3]\in\cA(\HH)$, and applying the shifting
morphism gives

\begin{conclusion}\label{con:16.1}
$[C_3]\in\cA(H_{4,g})$ for all $g\le 1$.\qed
\end{conclusion}

\begin{remark}
  If $d=3$ and $g>g(3)=1/4, C_3$ does not occur at all. If $d=4$ and
  $g>g(4)=1$, then $\cA(H_{4,2})=\langle E\rangle$ ([T4], Satz 1, p.26) and
  hence $[C_3]\notin\cA(H_{4,2})$.
\end{remark}

\subsection{}\label{sec:16.2.2}
$[C_4]\in\cA(\HH)$ is true for arbitrary $d$ and $g$ ([T4], Proposition 2,
p.26).

\subsection{The cycle $C_5$.}\label{sec:16.2.3}
In (loc.cit.) $[C_5]\in\cA(\HH)$ was shown, too, but the proof required
tedious computations, which can be avoided by the following argumentation. One
has only to treat the cases $g=0$ and $g=-2$, that means, the cases $b=7$ and
$b=9$ ([T1], Satz 2c(iii), p.92).

We start with $b=7$. Then $C_5=\overline{\G_a\cdot\xi_0}$, where
$\xi_0\leftrightarrow\cJ_0:=(y^2,\cK),\cK:=(x^3,x^2y,xy^2,\\y^3, x^2z,y^2z)$.
Put $\cJ:=(x^2+y^2,\cK)\leftrightarrow\xi$. As $\cJ_0$ and $\cJ$ have the same
Hilbert function, $\xi_0$ and $\xi$ both lie in the same graded Hilbert scheme
$\HH_{\varphi}$ (cf. Appendix A). If $\G_m$ operates by
$\tau(\lambda):x\mapsto\lambda x,y\mapsto y,z\mapsto z,t\mapsto t$, then one
sees that indeed $\xi_0=\lim\limits_{\lambda\to 0}\tau(\lambda)\xi$ and
$\xi_{\infty}:=\lim\limits_{\lambda\to\infty}\tau(\lambda)\xi\leftrightarrow\cJ_{\infty}:=(x^2,\cK)$.
Put $C:=\{\psi_{\alpha}(\xi)|\alpha\in k\}^-,
C_{0/\infty}:=\{\psi_{\alpha}(\xi_{0/\infty})|\alpha\in k\}^-$. One sees that
$\alpha$-grade $H^0(\cJ(n))=\alpha$-grade $H^0(\cJ_0(n))=\alpha$-grade
$H^0(\cJ_{\infty}(n))+2$, for all $n\ge 2$, hence $\deg C=\deg C_0=\deg
C_{\infty}+2$, relative to an appropriate imbedding of $\HH_{\varphi}$ into
$\P^N$.

Put $V:=\overline{\G_m\cdot\G_a\cdot\xi}$. As $\xi$ is invariant under
$G:=G_3\cdot T_{23}$ and as $G_3$ and
$T_{23}=\{(1,1,\lambda_2,\lambda_3)|\lambda_2,\lambda_3\in k^*\}$ are
normalized by $\G_m\cdot\G_a$, $V$ is pointwise $G$-invariant. Let $p$ be the
Hilbert polynomial of $C\subset\P^N$. Then the standard construction by means
of $V$ takes place in $\mathcal{X} : = \Hilb^p(V)^{\G_a}$, as $C$ is
$\G_a$-invariant. Thus the limit curves $\CC_{0/\infty}$ are pointwise
$G$-invariant and invariant under $\G_a$ and $\G_m$, hence they are
$B(4;k)$-invariant. Now $C_{0/\infty}\subset\CC_{0/\infty}$ and the above
computation of the degrees shows that $C_0=\CC_0 $, whereas
$\CC_{\infty}$ except from $C_{\infty}$ has to contain a further 1-cycle $Z$
of degree 2, i.e., $\CC_{\infty}$ has to contain an irreducible curve $F$ of
degree $\ge 1$, which is $B(4;k)$-invariant. As $V$ is pointwise
$G$-invariant, $F$ is either pointwise $U(4;k)$-invariant or an 1-cycle of
type 3. As $\deg (\cL_n|F)$ is constant it follows that $[F]\in\cA(\HH)$ 
(see [T3], Anhang 2, Hilfssatz 2, p.50). It follows that $Z\in\cA(\HH)$.

From $[C_5]=[C_0]=[C_{\infty}]+Z$ and $C_{\infty}=C_4$ (cf. [T1], Satz 2,
p. 91), because of (16.2.2) it follows that $[C_5]\in\cA(\HH)$.

We now treat the case $b=9$. Here $C_5=\overline{\G_a\cdot\eta}$, where
$\eta\leftrightarrow (x^3,x^2y,y^2,x^2z^3)$. By the $G_3$-admissible
deformation $y^2\mapsto x^2z^2$ $C_5$ is deformed in the cycle
$C'_5:=\overline{\G_a \cdot\xi_0}$, where
$\xi_0\leftrightarrow\cJ_0:=(x^3,x^2y,xy^2,y^3,y^2z,x^2z^2)$. By ([T1], 1.4.4)
$[C_5]=q[C'_5]$ modulo $\cA(\HH)$. Hence we can argue with $C'_5$ and $\cJ_0$
in the same way as in the case $b=7$: Let
$\cK:=(x^3,x^2y,xy^2,y^3,x^2z^2,y^2z^2)$ and
$\cJ:=(x^2z+y^2z,\cK)\leftrightarrow\xi$. If
$\xi_{0/\infty}:=\lim\limits_{\lambda\to 0/\infty}\tau(\lambda)\xi$ , then
$\xi_0$ is as above and
$\xi_{\infty}\leftrightarrow\cJ_{\infty}=(x^3,x^2y,xy^2,y^3,x^2z,y^2z^2)$.

Obviously, one has the same computation of degree as in the case $b=7$ and the
same argumentation shows that $[C'_5]=[C_{\infty}]$ modulo $\cA(\HH)$.
$C_{\infty}$ is deformed by the $G_3$-admissible deformation $y^2z^2\mapsto
x^2$ in the cycle $\overline{\G_a\cdot\zeta}$, where $\zeta\leftrightarrow
(x^2,xy^2,y^3,y^2z^3)$ As this is the cycle $C_4$, by 16.2.2 we get
$[C_5]\in\cA(\HH)$.

\begin{conclusion}\label{con:16.2}
$[C_5]\in\cA(H_{4,g})$ if $g=0$ or $g=-2$.\qed
\end{conclusion}

\subsection{Summary}\label{sec:16.3}
The results of 16.1 and 16.2 give

\begin{theorem}\label{thm:16.4}
  (i) If $g\le 0$, then $A_1(H_{3,g})$ is freely generated by $[E], [D], [C_2]$.\\
  (ii) $A_1(H_{4,2})\simeq\Q^4$ and if $g\le 1$, then $A_1(H_{4,g})$ is freely
  generated by $[E],[D],[C_2]$.\qed
\end{theorem}

%%%%%%%%%%%%%%%%%%%%%%%%% CHAPTER 17 %%%%%%%%%%%%%%%%%%%%%%%%%%%%%%%

\chapter{Correction of the results of [T4] and summary} \label{ch:chap17}

In [T4] the computation of the degree on page 28 is wrong, and hence
Conclusions~1-3 and Proposition 3 on the pages 29-32 are wrong. As well ([T4],
Lemma 7, p.33) is wrong. The statement of ([T4], Proposition 4, p.33) is
right, but with regard to Theorem 15.1(i) it is irrelevant. The
statement in ([T4], Satz 2, p.35) is wrong and has to be replaced by the
statements of Theorem 15.1 and Theorem 16.1 , respectively. The results of the sections 8 and 9 in [T4] remain valid, if one
replaces the old bound by the bound $g(d) =(d-2)^2/4$. Furthermore,
 ([T4], Satz 3, p.35) has to be replaced by:

\begin{theorem}\label{thm:5}
  Let be $d\ge 3$ and $g$ not maximal, let $\CC$ be the universal curve with degree $d$ and genus $g$ over $H_{d,g}$.\\
  (i) If $g>\gamma(d):={d-2\choose 2}$, then $A_1(\CC)$ is freely generated by $E^*,C_1^*,C_2^*,C_3^*$ and $L^*$.\\
  (ii) If $g(d)<g\le\gamma(d)$, then $A_1(\CC)$ is freely generated by $E^*,D^*,C_2^*,C_3^*$ and $L^*$.\\
  (iii) If $g\le g(d)$, then $A_1(\CC)$ is freely generated by $E^*,D^*,C_2^*$
  and $L^*$.
\end{theorem}

\indent The statements of ([T4], Satz 4 and Satz 5, p. 36) are correct, if the
bound mention there is replaced by $g(d)=(d-2)^2/4$. The reason is that the
arguments used in (loc.cit.) formally do not depend on the bound.

\indent All in all, one gets the results which had been stated in the introduction. \\

\indent{Concluding Remark}: Having arrived at this point, it is not so difficult any more to explicitly determine the cone of curves and the ample cone of $ H_{d,g}$ (and of $\CC $).

\backmatter

\appendix{}

%%%%%%%%%%%%%%%%%%%%%%%%% Appendix  A %%%%%%%%%%%%%%%%%%%%%%%%%%%%%%%
\chapter{Notations} \label{appnd:A}%\label{appnd:A}

The ground field is $\C$; all schemes are of finite type over $\C$;
$k$ denotes an extension field of $\C$. $P=k[x,y,z,t], S=k[x,y,z],
R=k[x,y]$ are the graded polynomial rings.\\
$T=T(4;k)$ group of diagonal matrices\\
$\Delta=U(4;k)$ unitriangular group\\
$B=B(4;k)$ Borel group\\
$T(\rho)$ subgroup of $T(3;k)$ or of $T(4;k)$ (cf. 2.4.1 and [T1], p.2).\\
$\Gamma=\left\{ \begin{pmatrix}
1&0&0&*\\
0&1&0&*\\
0&0&1&*\\
0&0&0&1\end{pmatrix}\right\} <U(4;k)$\\
$G_1,G_2,G_3$ subgroups of $U(4;k)$ (cf. 1.1)\\
$\HH=H_{d,g}$ Hilbert scheme of curves in $\P^3$ with degree $d\ge 1$ and genus $g$,
 i.e. $\HH=\Hilb^p(\P^3_k)$, where $P(T)=dT-g+1$. \\
$Q(T)={T+3\choose 3}-P(T)$ complementary Hilbert polynomial\\
$\HH_Q=$ Hilbert scheme of ideals $\cI\subset\cO_{\P^3}$ with Hilbert polynomial $Q(T)$, i.e. $\HH=H_{d,g}=\HH_Q$.\\
$\HH_Q\neq \emptyset$ if and only if $Q(T)={T-1+3\choose
  3}+{T-a+2\choose 2}$ or $Q(T)={T-1+3\choose 3}+{T-a+2\choose
  2}+{T-b+1\choose 1}$, where $a$ and $b$ are natural numbers $1\le
a\le b$. The first case is equivalent with $d=a$ and $g=(d-1)(d-2/2$,
i.e., equivalent with the case of plane curves. We consider only the
case $g<(d-1)(d-2)/2$. In this case we have the relations $d=a-1$ and
$g=(a^2-3a+4)/2-b$.

\indent It was not possible to reserve the letter $d$ for denoting the
degree of a curve. If necessary $d$ denotes a number large enough,
e.g. $d\ge b=$ bound of regularity of all ideals in $\cO_{\P^3}$ with
 Hilbert polynomial $Q$ (cf. [G1], Lemma 2.9, p.65). \\
$\GG=\Grass_m(P_d)$ Grassmann scheme of $m$-dimensional subspaces of
$P_d$.

\indent Let $\varphi:\N\to\N$ be a function with the following
properties: There is an ideal $\cI\subset\cO_{\P^2}$ of finite
colength with Hilbert function $h(n)=h^0(\cI(n))$, such that
$0\le\varphi(n)\le h(n)$ for all $n\in\N$ and $\varphi(n)=h(n)$ for
$n\ge d$, where $n$ is large enough, e.g. $n\ge d :=\colength(\cI)$.
On the category of $k$-schemes a functor is defined by

\indent $\HH_{\varphi}(\Spec A)=\{(U_0,\cdots,U_d)|U_n\subset S_n\otimes A$ subbundle of rank $\varphi(n)$ such that $S_1U_{n-1}\subset U_n, 1\le n\le d\}$

$\HH_{\varphi}$ is a closed subscheme of a suitable product of Grassmann schemes; it is called graded Hilbert scheme.

To each ideal $\cJ\subset\cO_{\P^3_k}$ with Hilbert polynomial $Q$ corresponds a point $\xi\in\HH(k)$, which we denote by $\xi\leftrightarrow\cJ$.

$h(\cJ)$ denotes the Hilbert function of $\cJ$, that means
$h(\cJ)(n)=\dim_kH^0(\cJ(n)), n\in\N$.

If $\varphi$ is the Hilbert function of an ideal in $\cO_{\P_k^2}$ of colength $d$, then 
\[
H_\varphi : = \{ \cI \subset \cO_{\P_k^2} | h^0(\cI(n)) = \varphi(n), n\in \N  \}
\]
\noindent is a locally closed subset of $\Hilb^d(\P^2)$, which we regard to have the induced reduced scheme structure. \\
\indent If $G$ is a subgroup of $GL(4;k)$, then $ \HH^G$ denotes the
fixed-point scheme, which is to have the induced reduced scheme structure. 
The same convention is to be valid for all fixed-point subschemes of $ H^d =
\Hilb^d(\P^2)$ .  \\
\indent If $ C\hookrightarrow \HH $ is a curve, then by means of the
Grothendieck-Pl\"ucker embedding $\HH \longrightarrow \P^N$ we can regard $C$
as a curve in a projective space, whose Hilbert polynomial has the form
$\deg(C)\cdot T + c$. Here $\deg(C)$ is defined as follows : If $\cI$
is the universal sheaf of ideals on $ X = \HH \times \P_k^3 $, then $
\cF : = \cO_X /\cI$ is the structure sheaf of the universal
curve $\CC $ over $\HH $, and the direct image $
\pi_*(\cF(n))$ is locally free on $\HH$ of rank $P(n)$ for all $n
\geq b $. The line bundles $\cM_n : = \dot\wedge\pi_*(\cF(n))$ are
called the tautological line bundles on $\HH$,which are very ample and thus
define the Grothendieck - Pl\"ucker embeddings in suitable projective
spaces. Here $\dot\wedge$ is to denote the highest exterior power.
Then $\deg(C)$ is the intersection number
$\deg(\cM_n|C) : = (\cM_n\cdot C)$. (If $C$ is a so called tautological or
basis cycle one can compute this intersection number directly, see [T2],
Section 4.1.) 

After these more or less conventional notations we introduce some notations
concerning monomial ideals. If $ \cJ \subset \cO_{\P^3}$ is $T$-invariant,
then $H^0(\P_k^3;\cJ(d)) \subset P_d $ is generated by monomials . To
each monomial $ x^{d-(a+b+c)}y^{a}z^{b}t^{c}$ in $ H^0(\cJ(d))$ we associate
the cube $ [a,a+1]\times [b,b+1]\times [c,c+1]$ in an $ y - z - t$ - coordinate
system, and the union of these cubes gives a so called pyramid, which is
denoted by $ E(\cJ(d))$. Usually we assume that $\cJ$ is invariant under
$\Delta$ or $\Gamma$. Then we can write $H^0(\cJ(d)) =
\bigoplus\limits^d_{n=0}t^{d-n}U_n $, where $ U_n \subset S_n $ are subspaces
such that $ S_1\cdot U_n \subset U_{n+1}, 0\leq n\leq d-1 $, which we call
the layers of the pyramid. (In [T1]--[T4] we made extensive use of this
concept, but here it occurs only once in 12.3.3.) \\
\indent $ A_1( - ) $ denotes the group of rational equivalence classes with
coefficients in $ \mathbb{Q}$.
\vfill
\newpage
%---- figA.1 ------
\begin{minipage}{19cm*\real{0.7}} \label{fig:A.1}
\centering Fig.: A.1\\ 
\tikzstyle{help lines}=[gray,very thin]
\begin{tikzpicture}[scale=0.7]
 \draw[style=help lines]  grid (19,13);
 % axes 
 \draw[thick] (0,0) -- (0,13); 
 \draw[thick] (0,0) -- (19,0);
 % ticks
 {
 \pgftransformxshift{0.5cm}
 \pgftransformyshift{-0.55cm}
  \foreach \x in {0,1,2} \draw[anchor=base] (\x,0) node {$\x$}; 
  \foreach \x in {3,4,5,6,7,8} \draw[anchor=base] (\x,0) node {$\dots$}; 
 }
 \draw[anchor=south east] (0,0) node {$0$};
 \draw[anchor=south east] (0,1) node {$1$};
 % steps 
 \draw[\Red,ultra thick] (9,0) -- (9,1) -- (10,1) -- (10,3) -- (11,3) --
 (11,5) -- (12,5) -- (12,6) -- (13,6) -- (13,8) -- (14,8) -- (14,9) -- (15,9)
 -- (15,12) -- (16,12) -- (16,13) -- (17,13);
 \draw[\Red,ultra thick,dotted] (15,9) -- (16,9) -- (16,10) -- (16,11) -- 
 (17,11) -- (17,12) -- (18,12);
 % line
  \draw[\Black, ultra thick] (0,1) -- (12,13);  
  \draw[\Black, ultra thick,dashed] (5,0) -- (18,13);  
\end{tikzpicture}
\end{minipage}

%%%%%%%%%%%%%%%%%%%%%%%%% Appendix  B %%%%%%%%%%%%%%%%%%%%%%%%%%%%%%%
\chapter{Hilbert functions without Uniform Position Property }\label{appnd:B}

\begin{lemma}
  Let be $ k $ be an algebraically closed field, $\cI \subset \cO_{\P_{k}^2}$
  an ideal of finite colength with Hilbert function $\varphi$ and difference
  function $\varphi'(n) = \varphi(n) - \varphi(n-1)$. Let $m$ be a natural
  number such that $\varphi'(m+1) = \varphi'(m)+1$. The ideal $\cJ \subset
  \cO_{\P_{k}^2} $ generated by $ H^0(\cI(m))$ has the following
  properties: \\
  \noindent(i) $\cJ $ is $m $-regular;\\
  \noindent (ii) $H^0(\cJ(n)) = H^0(\cI(n))$ for all $ n \leqq m+1$; \\
  \noindent (iii) If $\delta : = m+1-\varphi'(m) > 0 $, then there is a form
  $f \in S_{\delta}$ and an ideal $\cL \in \cO_{\P^2}$ of finite colength such
  that $\cJ = f\cdot \cL(-\delta)$.
\end{lemma}
\begin{proof} Let be $ I_n := H^0(\cI(n)), I : = \bigoplus
  \limits^{\infty}_{n=0} I_n $. The ideal $ I$ is called $\emph{
    Borel-normed}$, if $ \emph{in}(I) $ is invariant under $B(3;k)$, where
  $\emph{in}(I)$ is the ideal generated by the initial monomials of all forms
  in $I$. According to a theorem of Galligo, there is a $ g \in GL(3;k)$ such
  that $ g(I)$ is Borel-normed. (In [G4], Anhang IV,  in the special case of
  three variables, there is an ad-hoc-proof.) Therefore we can assume without
  restriction that $I$ is Borel-normed . Then Fig.A.1 shows not only the graph
  of $\varphi'$, but also the monomials in $H^0(\cI_0(n))$, where $\cI_0 : =
  [\emph{in}(I)]^{\sim} $ (cf. [G4], Anhang V, Hilfssatz 1, p.116). One sees
  that $S_1\emph{in}(I_{m}) = \itin(I_{m+1})$ and this implies the
  statements $(i)$ and $(ii)$ (cf. loc.cit., Lemma, p. 116 or [Gre],
  Proposition 2.28, p. 41). \\
  \indent Let $\psi$ be the Hilbert function of $\cJ$. Then $\psi$ is also
  the Hilbert function of $\cJ_0 = \emph{in}(\cJ)$ (cf. [G4], Hilfssatz 1,
  p.114), and the further development of the graph of $\psi'$ is marked by
  $\cdots$ in Fig. A.1. The line $\emph{l} : y = x-\delta+1 $ is marked by -
  - - - . If $ c$ is the number of monomials between the graphs of $\psi'$ and
  $\emph{l}$, then $\psi(n) = {n-\delta+2\choose 2} - c, n \geq m $. Then
  the Hilbert polynomial of $\cO_{\P^2}/\cJ$ is equal to $ p(n) = {n+2\choose
    2}- \psi(n) = \delta n + 1.5\delta - 0.5\delta^2+c $. Hence $ V_+(\cJ)
  \subset \P_k^2$ is $1$- codimensional and there is an irreducible component
  which is equal to a hypersurface $V(f), f\in S_\nu $ an irreducible form
  (Hauptidealsatz of Krull). From $\cJ \subset Rad(\cJ) \subset (f)$ it
  follows $\cJ = f\cdot\cK(-\nu)$, where $\cK \subset \cO_{\P^2}$ is an ideal
  with Hilbert function $ \chi(n): = \psi(n+\nu) = {n-(\delta-\nu)+2\choose
    2} - c $. If $\delta-\nu > 0$, one argues with $\cK$ in the same way as
  with $\cJ$, and one finally gets the statement $(iii)$ .
\end{proof}

\begin{lemma}
  Let the assumptions and the notations be as before. Then $\reg(\cI) = \min
  \Set{n \in \Z | \varphi'(n) = n+1 }$.
\end{lemma}

\begin{proof} As in the proof of Lemma 1, we can assume that $\cI$ is
  Borel-normed. We let $\G_m$ operate on $S$ by $\sigma (\lambda) : x \mapsto
  x, y \mapsto \lambda^{g}y, z \mapsto \lambda^{g^2}z $, where $g$ is a high
  enough natural number. Then $\lim \limits_{\lambda \to 0}\sigma(\lambda)\cI
  $ is equal to the ideal $\cI_0$ as in the proof of Lemma 1, $\cI$ and
  $\cI_0$ have the same Hilbert function, and $\reg(\cI) = \reg(\cI_0)$ 
   ( cf. [Gre], Theorem 2.27). Hence one can assume without
  restriction that $\cI$ is monomial. But then the statement follows from 
([T1], Anhang 2), for instance.
\end{proof}

%%%%%%%%%%%%%%%%%%%%%%%%% Appendix  C %%%%%%%%%%%%%%%%%%%%%%%%%%%%%%%
\chapter{Ideals with many monomials } \label{appnd:C}

If $k$ is a field, let be $S = k[x_1,\dots,x_r,t]$ and $R = k[x_1,
\dots,x_r]$. $\G_m $ operates on $S$ by $ \sigma(\lambda): x_i \mapsto x_i, 
1 \leq i \leq r$, and $t\mapsto \lambda t, \lambda \in k^*$ . Let $ \HH $
be the Hilbert scheme of ideals $\cI \subset \cO_{\P^r}$ with Hilbert polynomial
$Q$, i.e., $\HH = \Hilb ^P(\P_k^r)$, where $ P(n) = {n+r\choose r} - Q(n)$
is the complementary Hilbert polynomial of the subscheme $ V_+(\cI) \subset
\P^r $. We suppose that $\HH$ is not empty. Then the ideals $\cI \subset
\cO_{\P^r}$ with Hilbert polynomial $Q$, for which $t$ is a non-zero divisor
of $\cO_{P^r}/\cI$, form an open, non-empty subset $U_t \subset \HH$.\\

If $K/k$ is an extension field and if $\cI \in \HH(K)$, then the limit ideals
$\cI_{0/\infty} : = \lim\limits_{\lambda\to 0/\infty}\sigma(\lambda)\cI$ are
in $\HH(K)$ again, and if $\cI \in U_t$, then $\cI_0 \in U_t$, too (cf.
[G2], Lemma 4). We say that $\cI$ fulfils the $\emph{limit condition}$, if
$\cI_{\infty} \in U_t$.
\begin{remark} If $\cI$ is fixed by the subgroup $ \Gamma : x_i \mapsto x_i,
  t \mapsto \alpha_1 x_1 + \cdots + \alpha_r x_r + t $ of $U(r+1;k)$, then
  $\cI$ does fulfil the limit condition (cf. [G2], proof of Lemma 3, p. 541).
\end{remark}

If $\cI \in U_t$, then $\cI' : = \cI + t\cO_{\P^r}(-1)/ t\cO_{\P^r}(-1)$ can
be regarded as an ideal in $\cO_{\P^{r-1}}$ with Hilbert polynomial $ Q'(T) =
Q(T) - Q(T-1)$.

\subsection{Lemma} Let $\cI \in \HH(k)\cap U_t $ be an ideal which fulfils
 the limit condition.
\begin{enumerate}[(i)]
\item 
If $d \geq \max(\reg(\cI_0), \reg(\cI_{\infty}))$, then
$H^0(\P_k^r,\cI(d))\cap R_d$ has the dimension $Q'(d)$.\\
\item
If $d \geq \reg(\cI')$ and $H^0(\cI(d))\cap R_d$ has a dimension $\geq Q'(d)$,
then $d \geq \max(\reg(\cI_0), \reg(\cI_{\infty}))$.
\end{enumerate}
\begin{proof}(i) There is a basis of $M : = H^0(\cI(d))$ of the form $ g_i =
  t^{e_i}g_i^0 + t^{e_i-1}g_i^1 + \cdots $, such that $0\leq e_1 \leq e_2
  \leq \cdots \leq e_m, m := Q(d), g_i^j \in R $ and $ g_i^0 \in R_{d-e_i},
  1\leq i \leq m$, linear independent. Then $M_{\infty} : =
  \lim\limits_{\lambda\to \infty}\sigma(\lambda)M = \langle {\{ t^{e_i}g_i^0 | 1
    \leq i \leq m \}}\rangle $ (limit in $\Grass_m(S_d)$) has the dimension
  $m$. As $ d \geq \reg(\cI_{\infty})$ by assumption, it follows that $ Q(d) =
  h^0(\cI_{\infty}(d))$, and hence $ M_{\infty} = H^0(\cI_{\infty}(d))$. Now
  $t$ is a non-zero divisor of $S /\bigoplus\limits_{n\geq
    0}H^0(\cI_{\infty}(n)$ by assumption, Thus it follows that
  $H^0(\cI_{\infty}(n)) = \langle{ \{ t^{e_i-(d-n)}g_i^0 | e_i \geq d-n
    \}}\rangle$ for all $ 0 \leq n \leq d $. If $ n = d-1 $ one gets
  $H^0(\cI_{\infty}(d-1)) = \langle\{t^{e_i-1}g_i^0 | e_i \geq 1 \}\rangle$,
  hence $Q(d-1) = |\{ i | e_i \geq 1\}| $. It follows that $Q'(d) = |\{i | e_i
  = 0\} |$. Thus $M\cap R_d \supset \langle \{g_i^0 |e_i = 0\}\rangle$ has a
  dimension $\geq Q'(d)$. Because of $ \reg(\cI') \leq \reg(\cI)$ one has
  $h^0(\cI'(d)) = Q'(d)$ and the canonical restriction mapping $\rho_d : M =
  H^0(\cI(d)) \longrightarrow H^0(\cI'(d))$ is injective on $ M\cap R_d$. It
  follows that the dimension of $M \cap R_d $ can not be greater than $Q'(d)$. \\
\noindent (ii) From the exact sequence 
 \begin{equation}\label{eq:A1}
   0 \longrightarrow \cI(-1) \stackrel{\cdot t}{\longrightarrow} \cI \longrightarrow \cI' \longrightarrow 0
 \tag{1}
\end{equation}

\noindent it follows that $H^i(\cI(n-i)) = (0) $ if $i \geq 2$ and $n \geq e : = \reg(\cI')$ (see [M], p.102). The sequence
\[
0 \longrightarrow H^0(\cI(d-1))\longrightarrow H^0(\cI(d))
\stackrel{\rho_d}{\longrightarrow}
H^0(\cI'(d))\longrightarrow H^1(\cI(d-1))\longrightarrow
H^1(\cI(d))\longrightarrow 0
\]
is exact as $ d \geq e$, where $\rho$ is induced by the canonical restriction
mapping $ S \longrightarrow S/tS(-1) = R$. As $\rho_d$ is injective on
$H^0(\cI(d))\cap R_d$ and $h^0(\cI'(d)) = Q'(d)$, it follows from the
assumption that $\rho_d$ is surjective. From the $e$ - regularity of $\cI'$
it follows that $ R_1H^0(\cI'(n)) = H^0(\cI'(n+1))$, for all $n\geq e$. Hence
$\rho_n$ is surjective for all $n\geq d$. Hence $ 0 \longrightarrow
H^1(\cI(n-1)) \longrightarrow H^1(\cI(n)) \longrightarrow 0$ is exact for all
$n\geq d$, thus $H^1(\cI(n-1)) = (0)$ for all $n\geq d$. It follows that
$\reg(\cI) \leq d$. One again has the exact sequences:
\begin{equation}\label{eq:A2}
0 \longrightarrow \cI_{0/\infty}(-1) \stackrel{\cdot t}{\longrightarrow}
\cI_{0/\infty} \longrightarrow \cI'_{0/\infty} \longrightarrow 0
\tag{2}
\end{equation}

\noindent 
As $(\cI')_{0/\infty} = (\cI_{0/\infty})' \supset \cI' $ and all
these ideals have the Hilbert polynomial $Q'$, it follows that
$(\cI')_{0/\infty} = \cI'$. As $H^0(\cI(d))\cap R_d$ is fixed by
$\sigma(\lambda)$, it follows that $ H^0(\cI(d))\cap R_d \subset
H^0(\cI_{0/\infty}(d)) $. Then one argues as before, using (2) instead of (1).
\end{proof}

\begin{remark} Let $\cI \subset \cO_{\P^2}$ be an ideal of colength $d$, let
  be $S = k[x,y,z], R = k[x,y]$, and let $ \G_m $ operate by $
  \sigma(\lambda) : x \mapsto x, y \mapsto y, z \mapsto \lambda z$ . We
  assume $\cI$ to be invariant under $\Gamma $ (see above). As $d \geq
  \reg(\cI)$ for all ideals $\cI \subset \cO_{\P^2}$ of colength d, the
  assumption of part (i) of the lemma is fulfilled, hence $H^0(\cI(n))\cap
  R_n$ has the dimension $ Q'(n) = {n+1\choose 1}$ for all $ n\geq d$ and
  therefore:
\begin{equation}\label{eq:A3}
 H^0(\cI(n)) \supset R_n \quad \text{for all} \quad n\geq d.
\tag{3}
\end{equation}
\noindent This inclusion has been used in the text for several times, e.g. in Section (2.2).
 \end{remark}

 \begin{remark} Let $\cI \subset \cO_{\P^2}$ be an ideal of finite colength,
   with Hilbert function $\varphi$, which is invariant under $\Gamma\cdot
   T(\rho)$. Let $\G_m$ operate on $S$ by $\sigma(\lambda) : x \mapsto x, y
   \mapsto y,z \mapsto \lambda z$.  If $\emph{in}(\cI)$ is the initial ideal
   with regard to the inverse lexicographical order, then $\emph{in}(\cI)$ is
   equal to the limit ideal $\cI_0 = \lim\limits_{\lambda\to
     0}\sigma(\lambda)\cI$. As $h^0(\cI_0(n)) = \varphi(n)$, it follows that
   $\emph{in}(H^0(\cI(n))) = H^0(\cI_0(n))$, for all $ n\in \N $ (cf. Appendix E and [G2],
   Lemma 3 and Lemma 4, pp. 541).  Thus the number of the initial monomials
   and of the monomials in $H^0(\cI_0(n))$, which are represented in our
   figures, can be determined by means of the Hilbert function, alone.
\end{remark}

%%%%%%%%%%%%%%%%%%%%%%%%% Appendix  D %%%%%%%%%%%%%%%%%%%%%%%%%%%%%%%
\chapter{Unipotent groups acting on polynomial rings }
\label{appnd:D}
\begin{lemma}
The 5-dimensional subgroups of $\Delta = U(4;k)$ have the form  
\[
  G(p) : = \Set{ 
\begin{pmatrix} 
 1&\alpha&*&*\\  
  0&1&\beta&*\\
  0&0&1&\gamma\\
  0&0&0&1
\end{pmatrix} | a\alpha+b\beta+c\gamma=0}
\]
where $p = (a:b:c) \in \P^2(k) $ is uniquely determined.
\end{lemma}

\begin{proof} $N : = \left\{\begin{pmatrix} 
1&0&*&*\\
0&1&0&*\\
0&0&1&0\\
0&0&0&1\end{pmatrix}\right\}\subset \Delta $ is a normal subgroup. Let $G$ be
a 5-dimensional subgroup of $\Delta$. Then $G/G \cap N \longrightarrow
\Delta/N $ is an injective homomorphism and $ \Delta /N \simeq \G_a^3 $ . \\
\noindent $\emph{First case}$: dim $G \cap N = 2$. Then $ G\cap N = \Set{\begin{pmatrix} 
1&0&x&y\\
0&1&0&z\\
0&0&1&0\\
0&0&0&1\end{pmatrix} | x,y,z \in k, ax+by+cz = 0}$ where $(a:b:c) \in \P^2(k)$ is a suitable point. It follows that $ G = \left\{\begin{pmatrix} 
1&\alpha&x&y\\
0&1&\beta&z\\
0&0&1&\gamma\\
0&0&0&1\end{pmatrix}\right\}$, where $ \alpha,\beta, \gamma $ are any element of $k$, and $x,y,z \in k$ have to fulfil the conditions noted above. If $ \left\{\begin{pmatrix} 
1&\alpha'&x'&y'\\
0&1&\beta'&z'\\
0&0&1&\gamma'\\
0&0&0&1\end{pmatrix}\right\}$ is any other element of $G$, then \\
\[
a(x'+\alpha\beta'+x) + b(y'+\alpha z'+x\gamma'+y) + c(z'+\beta\gamma'+z) = 0.
\]
As $ \alpha, \beta, \gamma, \alpha', \beta', \gamma'$ are any elements of $k$,
we conclude that $ a = b= c = 0 $, contradiction. \\
$\emph{Second case}$: $N \subset G$. Then $ G/N \hookrightarrow \G_a^3$ is
2-dimensional, and one concludes from this that $G$ has the form noted above
and that $p \in\P^2(k)$ is uniquely determined. Furthermore it is easy to see
that $G(p)$ is a subgroup.
\end{proof}

\begin{lemma}
  Let be $P = k[x,y,z,t]$, let $ V \subset P$ be a subspace which is invariant
  under $ G(p)$. If $f \in P $ is invariant under $G(p)$
  modulo $V$, then the polynomials \\
  $x \partial f/\partial z, y \partial f/\partial t, x \partial f/\partial t,
   bx \partial f/\partial y - ay \partial f/\partial z, cx \partial
  f/\partial y - az \partial f/\partial t, cy \partial f/\partial z - bz
  \partial f/\partial t $ all lie in $V$.
\end{lemma}
 \begin{proof} If $ g = \begin{pmatrix}
 1&\alpha&0&0\\
 0&1&\beta&0\\
 0&0&1&\gamma\\
 0&0&0&1\end{pmatrix} \in G(p)$ 
and $ M = x^ry^sz^mt^n $, then $g(M) - M = x^r(\alpha x+y)^s(\beta
y+z)^m(\gamma z+t)^n - M = \alpha s x^{r+1}y^{s-1}z^mt^n + \beta m
x^ry^{s+1}z^{m-1}t^n + \gamma nx^ry^sz^{m+1}t^{n-1}\\ + $ terms containing $
\alpha^i, \beta^i, \gamma^i$, where $ i \geq 2 $. \\

If $g(f) - f \in V $, then it follows that $\alpha x\partial f/\partial y +
\beta y\partial f/\partial z + \gamma z\partial f/\partial t + $ terms
containing $ \alpha^{i}, \beta^{i}, \gamma^{i} $, where $i\geq 2 $, lies in
$V$. \\
$\emph{First case}$: $ c \neq 0$. Then $\gamma = -(a\alpha/c +b\beta/c)$,
hence $ \alpha x \partial f/\partial y + \beta y\partial f/\partial z -
(a\alpha/c + b\beta/c)z\partial f/\partial t $ + terms containing $\alpha^2,
\alpha\beta,\beta^2 \cdots \in V$.
It follows that $ \alpha (cx\partial f/\partial y - az\partial f/\partial t)
+ \beta(cy\partial f/\partial z - bz\partial f/\partial t) + $ terms
containing $\alpha^2, \alpha\beta, \beta^2,\dots  \in V$. Put $\alpha = 0 $
and $\beta = 0 $, respectively. It follows that $ cy\partial f/\partial z -
bz\partial f/\partial t \in V$ and $cx\partial f/\partial y - az\partial
f/\partial t \in V $. Multiplication by $a$ and $b$, respectively, and then
subtracting the two polynomials from each other gives $ c(bx\partial
f/\partial y - ay\partial f/\partial z) \in V$. As $ c \neq 0$, the last
three statements of the assertion follow. \\
$\emph{ Second case}$ : $ c = 0, b \neq 0 $. Then we can choose $\gamma \in k
$ arbitrarily and $ -a\alpha/b = \beta$. It follows that $\alpha x\partial
f/\partial y - (a\alpha/b)y\partial f/\partial z + \gamma z\partial f/\partial
t $ + terms containing $\alpha^2, \gamma^2, \cdots \in V$. Putting $\alpha = 0
$ gives $z\partial f/\partial t \in V$. Putting $\gamma = 0$ gives $bx\partial
f/\partial y - ay\partial f/\partial z \in V $. \\
$\emph{Third case}$: $ b= 0, c = 0$. Then $ a = 1$ and $\beta$ and $ \gamma$
are any elements of $k$, whereas $\alpha = 0$. This gives $y\partial
f/\partial z$ and $z\partial f/\partial t \in V$. As 
$ N = \left\{\begin{pmatrix}
 1&0&*&*\\
 0&1&0&*\\
 0&0&1&0\\
 0&0&0&1\end{pmatrix}\right\}\subset G(p)$, the same reasoning as in the proof
of ([T2], Hilfssatz 1, p. 142) shows that $ x\partial f/\partial z, x\partial f/\partial t, y\partial f/\partial t \in V $.
\end{proof}

\begin{lemma} Let $\cI \subset \cO_{\P^2}$ be a monomial ideal of
colength $d > 0 $ and let $\xi$ be the corresponding point in $H^d(k)$. If
$\xi$ is not invariant under the $\G_a$-operation $\psi_{\alpha} : x \mapsto
x, y \mapsto \alpha x +y, z \mapsto z$, then $ C : = \overline{
  \G_a\cdot\xi}$ contains exactly two fixed points under $T(3;k)$, namely the
point $\xi$ and the $\G_a$ - fixed point $\psi_{\infty}(\xi) : =
\lim\limits_{\alpha \longrightarrow \infty}\psi_{\alpha}(\xi)$.
\end{lemma}
\begin{proof} Embedding $H^d$ in $\Grass^q(S_d)$, where $ q : = \tbinom{d+2}{2}
  - d$, one sees that it is sufficient to prove the corresponding statement
  for a $T(3;k)$-invariant $q$ - dimensional subspace $U \subset S_d$ and
  the corresponding point in $\Grass^q(S_d)$. As one can write $U =
  \bigoplus\limits^d_{i=0}z^{d-i}U_i$, where $U_i \subset R_i$, is a subspace,
  it suffices to prove the corresponding statement for an $r$-dimensional
  subspace $U \subset R_n$, which is invariant under $T(2;k)$, but not
  invariant under $\G_a$. As
  $\lim\limits_{\alpha\longrightarrow\infty}\psi_{\alpha}(U)$ is a $\G_a$ -
  invariant subspace and as $ char(k) = 0$, it follows that this subspace is
  equal to $\langle x^n, x^{n-1}y, \dots,x^{n-r+1}y^{r-1} \rangle$. It
  follows that $C$ has two fixed-points under $T(2;k)$. In order to prove
  there are no more fixed-points, it suffices to show the following: If there
  is an element $\alpha \neq 0$ in $k$ such that $\psi_{\alpha}(U)$ is
  $T(2;k)$-invariant, then $\psi_{\alpha}(U)$ is $T(2;k)$-invariant for
  all $\alpha \neq 0 $. If one takes a monomial $ M = x^{n-r}y^r \in U $,
  then $ \psi_{\alpha}(M) = x^{n-r}(\alpha x+y)^r \in \psi_{\alpha}(U)$. As
  this is a monomial subspace by assumption, it follows that
  $x^{n-r}x^iy^{r-i} \in \psi_{\alpha}(U)$, $0\leq i\leq r$. Thus one has $ M
  \in \psi_{\alpha}(U)$ and it follows that $ U = \psi_{\alpha}(U)$, But this
  implies $\psi_{n\alpha}(U) = U $ for all $n \in \N$ and thus $
  \psi_{\alpha}(U) = U $ for all $ \alpha \in k$.
\end{proof} 

\begin{lemma} Suppose $C\subset \HH_Q$ is a $B$-invariant, irreducible, reduced curve, which is pointwise invariant under $\Gamma$ such that the image of $C$ under the restriction morphism $h$ is a single point. Then $(\cM_n\cdot C)$ is constant for $n\gg 0$ .
\end{lemma}
\begin{proof} From ([T1], Proposition 0, p.3) it follows that $C$ is either
  pointwise $\Delta$-invariant or a 1-cycle of type 3.  We consider the
  first case. Then $C = \{\sigma(\lambda)\xi |\lambda \in k^*\}^-$, where
  $\sigma $ is a suitable $\G_m$-operation and $\xi \in \HH(k)$ corresponds
  to a $\Delta$-invariant ideal $\cJ$. Let be $\cI = h(\cJ)$ and $\cI^*$
  the ideal on $\P^3$ ,which is generated by $\cI$ (cf. 1.2.2). Then $
  H^0(\cJ(n)) \subset \mathop{\oplus}\limits^n_{i=0}t^{n-i}H^0(\cI(i))$ for
  all $n$ (cf. [G5], Hilfssatz 3.2, p.295).
  Now the assertion follows, as $\cI$ is fixed by $\sigma$.\\
  \indent In the second case on has $ C = \{\psi_{\alpha}(\xi) | \alpha\in
  k\}^-$, where $\psi_{\alpha}$ is the usual $\G_a$-operation and
  $\xi\in\HH(k)$ corresponds to a monomial ideal $\cJ$. Then one argues as
  in the first case, with $\psi_{\alpha}$ instead of $\sigma$.

\end{proof}

%%%%%%%%%%%%%}%%%%%%%%%%%% Appendix  E %%%%%%%%%%%%%%%%%%%%%%%%%%%%%%%
\chapter{Standard bases }\label{appnd:E}
Let $k$ be an extension field of $\C$. If $\rho = (\rho_0, \dots,\rho_r)
\in\Z^{r+1} - (0)$, then $T(\rho) : =\Set{\lambda = (\lambda_0,\cdots,
\lambda_r) |  \lambda_i \in k^* \; \text{and}\;\lambda^{\rho} : = 
\lambda_0^{\rho_0}\cdots\lambda_r^{\rho_r} = 1}$ is a subgroup of dimension
$r$ of $\G_m^{r+1}$.\\

\noindent{\bf Auxiliary lemma:} If $\sigma, \tau \in \Z^{r+1}-(0)$ such that
$T(\sigma) \subset T(\tau)$, then there is an integer $n$ such that $\tau =
n\cdot\sigma$.
\begin{proof} Write $\sigma = (a_0,\cdots,a_r),\tau =(b_0,\cdots,b_r)$ . As
  the dimension of $T(\sigma)$ is equal to $r$, there is an index $i$ such
  that $a_i \neq 0$ and $b_i \neq 0$.Without restriction one can assume that
  $ a_{0} \neq 0 $ and $ b_{0} \neq 0$ . Choose $ p, q \in\Z$ such that $p
  a_0 = q b_0 $ and $ (p,q) = 1 $. Then
  $\lambda_1^{pa_1-qb_1}\cdots\lambda_r^{pa_r-qb_r} = 1 $ for all $\lambda
  \in T(\sigma)$ follows. Because of $\dim\,T(\sigma) = r $ one gets
  $pa_i-qb_i = 0 $, $ 0 \leq i \leq r $, and thus $\sigma = q\rho,\,\tau
  =p\rho$, where $\rho \in \Z^{r+1}-(0)$ is a suitable vector. If $\epsilon$
  is any q-th root of unity in $\C$, one can choose $\lambda \in
  (\C^*)^{r+1}$ such that $\lambda^{\rho} = \epsilon$. From $\lambda^{q\rho}
  = \lambda^{\sigma}= 1 $ it follows that $ \epsilon^{p} = \lambda^{p\rho} =
  \lambda^{\tau} = 1 $, too, and $ q = 1 $ follows.
\end{proof}

We let $\G_m^{r+1}$ operate on $ S = k[ X_0,\cdots,X_r]$ by $ X_i \mapsto
\lambda_iX_i $. If $ \rho = (\rho_0, \dots,\rho_r)$, then $ X^{\rho} :=
X_0^{\rho_0} \cdots X_r^{\rho_r}$.

\begin{lemma}
  Let $ V \subset S_d $ be a $ T(\rho)$-invariant subspace. Then V has a
  basis of the form $ f_i = m_i\cdot p_i(X^{\rho})$, where the $m_i$ are
  different monomials, $ p_i $ is a polynomial in one variable with constant
  term $ 1 $ and $ m_i $ does not appear in $ m_j\cdot p_j(X^{\rho})$ if $i
  \neq j $.
\end{lemma}

\begin{proof} By linearly combining any basis of $V$ one obtains a basis $f_i
  = m_i + g_i$, where the $m_i$ are different monomials, each $g_i$ is a sum
  of monomials, each of which is greater than $ m_i $ in the inverse
  lexicographic order, and $ m_i $ does not appear in $ g_j$. If $g \in
  T(\rho) $, then $ g(f_i)$ contains the same monomials as $ f_i$ and from $
  g(f_i) \in \langle\{f_i\}\rangle $ we conclude that each $ f_i$ is a
  semi-invariant, i.e,
  $\langle g(f_i)\rangle = \langle f_i\rangle$ for all $ g \in T(\rho) $ .\\[2mm]
  Now let be $ f = \bigsum a_i X^{\alpha_i} $ any $ T(\rho)$-semi-invariant.
  Let be $ \lambda \in T(\rho)$. Then $\bigsum a_i
  \lambda^{\alpha_i}X^{\alpha_i} = c(\lambda)\cdot \bigsum a_i X^{\alpha_i}$,
  where $ \lambda^{\alpha_i} := \lambda_0^{\alpha_i(0)} \dots
  \lambda_r^{\alpha_i(r)}$. It follows that $ \lambda^{\alpha_i} =
  c(\lambda)$, if $a_i\neq 0$, and therefore $ \lambda^{\alpha_i -\alpha_j} =
  1 $ for all $ i, j $, such that $ a_i \neq 0$ and $ a_j \neq 0$. Thus $
  T(\rho) \subset T(\alpha_i- \alpha_j)$, and the Auxiliary lemma gives $
  \alpha_i - \alpha_j = n_{ij}\rho $, $n_{ij} \in \Z $, if $ a_i \neq 0 $ and
  $ a_j \neq 0 $. One sees that there is an exponent $ \alpha_0
  \in{\{\alpha_i}\} $ and natural numbers $ n_i$, such that $ f =
  X^{\alpha_0}\cdot\bigsum a_i X^{n_i\rho}$.
\end{proof}

\begin{corollary}
  Let $ V \subset S_d $ be a $m$-dimensional subspace, let $ x \in
  \Grass_m(S_d) $ be the closed point of $\Grass_m(S_d) $ defined by $V$. If
  the orbit $ T\cdot x $ has the dimension 1, then the inertia group $ T_x $
  of $x$ has the form $T(\rho)$, where $\rho \in \mathbf{Z}^{r+1}-(0) $.
\end{corollary}

\begin{proof}This follows by similar argumentations as before (see[T2], Hilfssatz 7,p.141).
\end{proof}

%%%%%%%%%%%%%%%%%%%%%%%%%%%%%%%%%%%%%%%%%%%%%%%%%%%%%%%%%%%%%%%%%%%%%%%%%%%%%

 \address{Gerd Gotzmann , Isselstrasse 34 , D--48431 Rheine }
\email{ g.gotzmann@t-online.de }

\end{document}